\documentclass[11pt]{article}
\usepackage{amssymb}
\usepackage{graphicx}
\usepackage{latexsym}
\usepackage{amsmath}
\usepackage{mathrsfs}
\usepackage{booktabs}
\usepackage{xcolor}
\usepackage{dsfont}
\usepackage{subfig}
\usepackage[english]{babel}
\usepackage[letterpaper,top=2cm,bottom=2cm,left=3cm,right=3cm,marginparwidth=1.75cm]{geometry}
\usepackage{siunitx}
\usepackage[numbers,sort&compress]{natbib}
\usepackage{amsthm}
\usepackage{setspace}
\usepackage{microtype}
\usepackage{lineno,hyperref}
\usepackage{cleveref}
\usepackage{tikz}
\usepackage{bm}
\usetikzlibrary{shapes.geometric}  
\usetikzlibrary{positioning, arrows.meta}
\newtheorem{example}{Example}[section]
\newtheorem{theorem}{Theorem}[section]

\newtheorem{proposition}{Proposition}[section]
\newtheorem{remark}{Remark}[section]

\newtheorem{assumption}{Assumption}[section]

\numberwithin{equation}{section}
\allowdisplaybreaks

\newcommand{\setd}{{ d \kern -.15em l}}
\newcommand{\hatsetd}{ d \hat{\kern -.15em l }}
\newcommand{\dd}{\mathsf {d\kern -0.07em l}} 

\newcommand{\vt}{{\vartheta}}

\newcommand{\bgeqn}{\begin{eqnarray}}
\newcommand{\edeqn}{\end{eqnarray}}
\newcommand{\bgeq}{\begin{eqnarray*}}
\newcommand{\edeq}{\end{eqnarray*}}
\newcommand{\bec}{\begin{center}}
\newcommand{\enc}{\end{center}}
\newcommand{\R}{{\rm I\!R}}

\newcommand{\inmat}[1]{\mbox{\rm {#1}}}

\title{
Stability Analysis of an Integrated Multistage Stochastic Programming and Markov Decision Process Problem\footnote{Dedicated to the memory of Professor Werner R\"omisch, a pioneer in stability analysis of stochastic optimization.}
}

\author{Zhiyao Yang\thanks{School of Mathematics and Statistics, Xi’an Jiaotong University, Xi’an, Shaanxi, P. R. China\\ Research Center for Optimization Technology and Intelligent Game, Xi’an International Academy for Mathematics and Mathematical Technology, Xi’an, P. R. China, \texttt{yzy9907@stu.xjtu.edu.cn}}, Zhiping Chen\thanks{School of Mathematics and Statistics, Xi’an Jiaotong University, Xi’an, Shaanxi, P. R. China\\ Research Center for Optimization Technology and Intelligent Game, Xi’an International Academy for Mathematics and Mathematical Technology, Xi’an, P. R. China, \texttt{zchen@mail.xjtu.edu.cn}} and Huifu Xu\thanks{Department of Systems Engineering and Engineering Management, The Chinese University of Hong Kong, Hong Kong, \texttt{h.xu@cuhk.edu.hk}}}

\begin{document}

\maketitle



\begin{abstract}
In this paper, we consider  an 
integrated MSP-MDP framework  which
captures features of Markov decision process (MDP) and multistage stochastic programming (MSP). 
The integrated framework allows one to  
study a dynamic decision-making process that involves 
both transition of system states and dynamic change of the stochastic environment affected respectively by
potential endogenous uncertainties
and exogenous uncertainties.
The new model allows one to distinguish effects arising from different sources of uncertainty.
We begin by 
deriving a dynamic nested reformulation
of the problem and  
the Lipschitz continuity and convexity of 
the stagewise optimal value functions.
We then move on to 
investigate stability of 
the
problem in terms of the optimal value and the set of optimal solutions 
under
the perturbations of 
the probability distributions of 
the endogenous uncertainty and the exogenous uncertainty.
Specifically, we
quantify the effects of the perturbation of the two uncertainties 
on the optimal values and optimal solutions  
by deriving the error bounds in terms of 
Kantorovich metric and a weighted transport discrepancy of the probability distributions of the respective uncertainties.
These results differ from the existing stability results established 
in terms of the filtration distance \cite{heitsch2009scenario} or the nested distance \cite{pflug2012distance}. We use 
some examples to explain the differences via
tightness of the error bounds and applicability of the stability results.
The results complement the existing stability results and provide 
new theoretical grounding for emerging integrated MSP-MDP
models.
\end{abstract}

\textbf{Key words.}
Multistage stochastic optimization, 
Markov decision process,
time-consistency, endogenous uncertainty,
exogenous uncertainty,
stability analysis.
\section{Introduction}
Multistage stochastic programming (MSP) and Markov decision process (MDP) are two 
important dynamic stochastic optimization models for making sequential decisions in uncertain environments. 
Over the past few decades, various
specific forms of MSP/MDP models have been proposed and 
relevant computational methods and underlying theory have been developed accordingly. For a complete treatment, see monographs \cite{shapiro2021lectures,pflug2014multistage,puterman2014markov,sutton2018reinforcement} and references therein.
In MSP models, the uncertainties are progressively revealed over time and decisions are made at each stage based on observation of 
the realizations of the uncertainties 
up to the current stage (\cite{birge2011introduction, kall1994stochastic}).
In MDP models, the uncertainty arises from the state transition, which is typically assumed to be independent across stages. Decisions are always made at each stage based only on the observation of the current state, 
without depending on historical information.
Driven by advances in computer science and artificial intelligence, reinforcement learning (RL) based on MDP(~\cite{gosavi2009reinforcement}) recently 
demonstrates its remarkable potential in practical applications. 
Meanwhile, as decision-making environments become increasingly complex, the demand for robust and stable decision-making under uncertainty increases substantially, leading to renewed interest in MSP from both academia and industry.
Over the past decade, 
MSP and MDP have been extensively used to solve a wide range of dynamic decision-making problems under uncertainty in operations research and management science
such as inventory control (\cite{puterman2014markov,tan2023production}), logistics (\cite{jamali2022multi}), healthcare (\cite{zhang2021two}), autonomous driving (\cite{kretzschmar2022detection,liu2022markov}), energy management (\cite{tsaousoglou2022multistage}), and financial management (\cite{bertsekas2012dynamic}).

An important issue here is how to 
distinguish different types of uncertainties. 
Haghighat et al.~\cite{haghighat2023robust} introduce a two-stage robust optimization model that simultaneously considers different uncertainties for microgrid capacity planning. They classify uncertainties into decision-dependent and decision-independent uncertainties. Sinclair et al.~\cite{sinclair2023hindsight} propose an approach for an MDP model with additional random inputs, which efficiently learns policies for resource allocation problems by leveraging historical samples of the external disturbances; but the authors attribute all uncertainties to exogenous inputs, assuming that the state transition and reward functions are deterministic. 
Wang et al.~\cite{wang2023aleatoric} characterize uncertainty as epistemic uncertainty and aleatoric uncertainty to distinguish the uncertainty in model parameter estimation caused by limited data and the inherent randomness of the environment. 
This classification 
offers a suitable way to distinguish different sources of uncertainty. 
Ma et al.~\cite{ma2024bayesian} further propose a Bayesian MDP model dividing the uncertainty as epistemic uncertainty and aleatoric uncertainty, and advance the solution of dynamic decision-making problems based on the consideration of multiple sources of uncertainty. However, aleatoric uncertainty itself could be further decomposed according to its 
sources.

Recently, there have been a few studies integrating MDPs with MSPs to solve complex decision-making problems.
Wang et al.~\cite{wang2010security} apply a stochastic optimization technique to MDP/RL policy iteration to dynamically optimize economic dispatch under security constraints. 
Nevertheless, they simply apply the MSP algorithm to the policy iteration of MDP,  without integrating the two frameworks.
Zhang et al.~\cite{zhang2021two} introduce a two-step surgical scheduling framework: an MDP for weekly patient selection to minimize long-term costs, followed by an MSP for detailed daily scheduling. Here MDP and MSP modeling approaches are applied successively, instead of simultaneously. 
Jaimungal et al.~\cite{jaimungal2022reinforcement} combine RL with stochastic optimization to tackle sequential decision-making problems by learning optimal policies in uncertain and dynamic environments. 
This approach enables 
learning optimal solutions directly from data without paying particular attention to the  
random variation of the system. 
Kiszka et al.~\cite{kiszka2022stability} incorporate Markov processes into the MSP framework by integrating state variables, and establish a linear stochastic dynamic programming problem based on Markov processes.
The model is limited to random data process with Markov property,
without accounting for the potential randomness in state transitions.
Bhattarai and Song~\cite{bhattarai2025multistage} study a
Markovian multistage stochastic programming 
 model for 
hurricane evacuation and relief logistics planning,
where the hurricane evolution is modeled 
by an endogenous 
stochastic process.
From both modeling and theoretical perspectives, the above studies 
are yet to 
formally integrate MDP and MSP 
while accounting for the distinct types of uncertainties in real-world problems. 
In this paper, we follow the strand of research 
to propose an integrated MSP-MDP model 
which allows one to 
distinguish endogenous uncertainties
and exogenous uncertainties 
and lay down some mathematical foundation 
for the integrated model in terms of 
dynamic reformulation and stability analysis.

In this paper, 
we concentrate on 
stability analysis of the integrated MSP-MDP model with respect to (w.r.t.)
the perturbation of the underlying uncertainties.
In the literature of MSP, 
Heitsch et al.~\cite{HRS2006stability} 
introduce a
filtration distance to measure
the perturbation of the stochastic process in 
a multistage linear stochastic program 
and use it to quantify 
the impact of the perturbation 
on the optimal value. 
K\"uchler ~\cite{kuchler2008stability} proposes a quantitative stability analysis of the value function for a class of  
linear MSP problems 
without taking into account the perturbation to stagewise random variables.
Jiang et al.~\cite{jiang2018quantitative} extend the analysis 
by introducing new 
forms of calm modification.
Pflug et al.~\cite{pflug2012distance} introduce a
nested distance 
which captures the perturbation of a stochastic process and its distribution simultaneously.
They use it 
to investigate the impact on the optimal value of a convex multistage stochastic program when the underlying stochastic process is perturbed.
Kern et al.~\cite{kern2020first}
consider the perturbation of 
the transition kernel of an MDP 
at a particular 
episode and 
use the so-called S-derivative to study
its effect on the value functions of
the MDP in the remaining stages.
In this paper, we follow the strand of research by 
studying effects of 
the perturbation in endogenous uncertainties
and exogenous uncertainties
on 
the optimal value and optimal solutions of the integrated MSP-MDP model. The main challenge is to tackle
interactions between state variables and the set 
of feasible solutions
related to the two uncertainties. 
Moreover, how to quantify the impact of  inter-stage correlation of exogenous random variables under suitable conditions is also challenging. 
The main contributions of this paper can be summarized as follows.

\begin{itemize} 
\item \textbf{Modeling framework.} 
We develop an integrated MSP-MDP model
which 
covers a wide range of 
problems
including classical MSP, MDP, contextual MDP, and MSP with side information. This is 
primarily motivated by 
distinguishing the underlying uncertainties in dynamic decision-making problems
according to their
sources/nature such as 
endogenous and exogenous uncertainty. 
The former refers to uncertainties that evolve within the system itself, often characterized by history-independent random variables whereas the latter refers to uncertainties originating from complex external environments, 
usually described 
by history-dependent random processes. 
The new framework provides a modeling perspective that explicitly distinguishes the roles of exogenous and endogenous uncertainties and facilitates their separate and joint stability analyses.

\item \textbf{Structural property.}
We derive a nested reformulation of the integrated MSP-MDP model. 
The reformulation facilitates a tractable dynamic programming representation. 
Under 
some moderate conditions such as continuity and convexity of stagewise cost functions and boundedness of feasible sets at 
each stage, instead of the relatively complete recourse condition (see \cite{HRS2006stability})
or convexity of objective function and feasible set (see \cite{pflug2012distance}), 
we prove the existence of optimal solutions and establish the continuity of the value function. 
We show that the value function is convex with respect to pair of state-decision variables under appropriate convexity and monotonicity assumptions. This is challenging since the nonlinear state dynamics brings the inter-stage coupling and historical dependence of random variables.
Furthermore, we prove that the value function is Lipschitz continuous in the state-decision variables,
provided that both the objective function and the state transition mapping satisfy Lipschitz continuity in these variables. These findings extend classical results in MSP and MDP problems.

\item \textbf{Stability analysis.}
We first investigate
the impact of 
the perturbation in the distribution of 
endogenous uncertainty on
the integrated MSP-MDP model. 
Under some moderate conditions,
we derive error bounds 
for both the optimal value and the set of optimal solutions 
in terms of the stagewise Kantorovich metrics between the distributions before and after the perturbation (Theorem \ref{Theo-stab-endo}).
Next, we move on to study the
effect of perturbation in the distribution of the exogenous uncertainty on the model. 
Unlike the stability analysis in 
Heitsch et al.~\cite{HRS2006stability} and Pflug and Pichler~\cite{pflug2012distance}, 
here we have to 
tackle the challenges 
arising from intrinsic 
interactions between state and decision variables, 
as well as intertemporal dependence of exogenous random variables. 
Under the Lipschitz continuity conditions on the conditional distributions of exogenous random variables with respect to historical information, we derive error bounds for the optimal value and the set of optimal solutions (Theorems \ref{Theo-stab-exo-optvalue}, \ref{Theo-stab-exo}, \ref{Theo-stab-exo-stage} and \ref{Theo-stab-exo-optsolu-stage}) in terms of discrepancy and Kantorovich metric. 
Directly imposing conditions on the involved functionals and stagewise feasible sets, 
our quantitative stability analysis avoids the complex filtration distance (\cite{HRS2006stability}) or the nested distance (\cite{pflug2012distance}) and the obtained results 
subsume the main conclusions of \cite{kuchler2008stability}
\cite{HRS2006stability}, and \cite{pflug2012distance}.  

\end{itemize}

The 
rest of this paper is organized as follows. Section 2 
introduces the new integrated MSP-MDP model. Section 3 analyzes the fundamental properties of the proposed integrated MSP-MDP model, including its well-definedness, convexity, Lipschitz continuity, and other structural characteristics. Sections 4 and 5 establish quantitative stability for the optimal value and optimal solution set of the integrated MSP-MDP model with respect to distributional perturbations of the endogenous random variables and exogenous random variables respectively.
Section 6 concludes the paper and outlines directions for future research.

Throughout the paper, we use the following notation. 
By convention, we use 
$\mathbb{R}^n$ to denote $n$-dimensional 
Euclidean space and 
$d(a,b)$ to denote the distance
between two points $a, b \in \mathbb{R}^n$.
We define
\( d(a, B) := \underset{b \in B}{\min} \text{ } d(a, b) \) as the distance from a point \( a \) to a compact set \( B \) and 
\( \mathbb{D}(A, B) := \underset{a \in A} {\max} \text{ } d(a, B) \) the excess of set $A$ over set $B$.
The Hausdorff distance between sets \( A \) and \( B \) 
in $\mathbb{R}^n$
 is then given by \( \mathbb{H}(A, B) := \max\{ \mathbb{D}(A, B), \mathbb{D}(B, A) \} \).
Unless specified otherwise,
we use  
$\|a\|$ 
to represent 
the infinity norm of a vector $a$.
We use terminologies probability 
measure and probability distribution interchangeably depending on the context. 
Finally, we use $\bm{x}$ to denote a random policy and use $x$ to denote a given solution.

\section{An integrated MSP-MDP model}
\subsection{Setup}
In many multistage 
decision-making problems, 
the underlying uncertainties have distinct characteristics in terms of their sources and effects.
Some of them arise from random changes in external (exogenous) environment over a time horizon which have a major impact on decision-making at each stage whereas others occur in the
internal (endogenous) decision-making process. The former 
is represented by
a random process in  standard multistage stochastic programming models while the latter is described 
in MDP models. Here we consider both.
For $t = 1, 2, \cdots, T$,
let random vector $\xi_t: \Omega_1 \to \mathbb{R}^{m_{1,t}}$ denote
the exogenous uncertainty 
with support set $\Xi_t$ and 
$\zeta_t: \Omega_2 \to \mathbb{R}^{m_{2,t}}$ 
denote the endogenous uncertainty
with support set
$Z_t$.
Let $s_t \in S_t \subseteq \mathbb{R}^{\hat{n}_t}$ denote the system state vector at stage $t$, and $x_t \in \mathbb{R}^{n_t}$ 
denote the decision vector at stage $t$.
Figure \ref{liuchengtu} illustrates the decision-making process when a decision maker(DM) faces 
both types of uncertainties.

\begin{figure}[htbp]
\begin{tikzpicture}[node distance=1.5cm]
    \tikzset{
        mydiamond/.style={
            draw, diamond, 
            minimum size=0.8cm, 
            inner sep=1pt,        
            font=\scriptsize      
        }
    }
    
    \node[draw, circle, minimum size=0.5cm] (s0) at (0, 0) {$s_0$};
    \node[draw, rectangle, minimum size=0.5cm] (x0) at (1.2, 0) {$x_0$};
    \node[mydiamond] (omega0) at (2.4, 0) {$\zeta_0$};      
    \node[draw, circle, minimum size=0.5cm] (s1) at (3.6, 0) {$s_1$};
    \node[mydiamond] (xi1) at (4.8, 0) {$\xi_1$};            
    \node[draw, rectangle, minimum size=0.5cm] (x1) at (6, 0) {$x_1$};
    \node[mydiamond] (omega1) at (7.2, 0) {$\zeta_1$};       
    \node[] (ellipsis1) at (8.4, 0) {$\cdots$};
    \node[draw, circle, minimum size=0.5cm] (st) at (9.6, 0) {$s_T$};
    \node[mydiamond] (xit) at (10.8, 0) {$\xi_T$};            
    \node[draw, rectangle, minimum size=0.5cm] (xt) at (12, 0) {$x_T$};
    \node[mydiamond] (omegat) at (13.2, 0) {$\zeta_T$};      
    
    \draw[->] (s0) -- (x0);
    \draw[->] (x0) -- (omega0);
    \draw[->] (omega0) -- (s1);
    \draw[->] (s1) -- (xi1);
    \draw[->] (xi1) -- (x1);
    \draw[->] (x1) -- (omega1);
    \draw[->] (omega1) -- (ellipsis1);
    \draw[->] (ellipsis1) -- (st);
    \draw[->] (st) -- (xit);
    \draw[->] (xit) -- (xt);
    \draw[->] (xt) -- (omegat);
\end{tikzpicture}
\caption{Chronology of states, random variables, and decision variables}
\label{liuchengtu}
\end{figure}
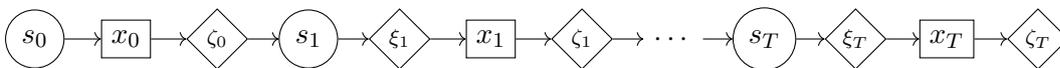
In what follows, we develop a mathematical model which precisely describes the process. We begin by introducing some notation. Let $(\xi,\zeta)$ be a joint stochastic process defined on the product probability space $(\Omega_1\times\Omega_2,\mathcal{F} \otimes \mathcal{G},\mathbb{P})$ which describes the evolution of exogenous uncertainty $(\xi_t)$ and endogenous uncertainty $(\zeta_t)$ over time. Throughout this paper, we assume that the processes $\xi$ and $\zeta$ are independent. 
Let $\mu:=\mathbb{P}\circ(\xi,\zeta)^{-1}$,
$P:=\mathbb{P}\circ\zeta^{-1}$
and 
$Q:=\mathbb{P}\circ \xi^{-1}$.
Then $\mu = P \otimes Q$.
By convention, let $\xi_{[t]}:=(\xi_1,\xi_2,\cdots,\xi_t)$ and $\mathcal{F}_t:=\sigma(\xi_{[t]})$ denote the filtration induced by $\xi_{[t]}$. Then $\mathcal{F}_t$ satisfies $\mathcal{F}_0:=\{\emptyset,\Omega_1\}\subseteq\mathcal{F}_1\subseteq\cdots\subseteq\mathcal{F}_T=\mathcal{F}$. Let $\Xi_{[t]}\subseteq\mathbb{R}^{m_{1,1}}\times\mathbb{R}^{m_{1,2}}\times\cdots\times\mathbb{R}^{m_{1,t}}$ denote the support set of $\xi_{[t]}$ and $Q_{[t]}=\mathbb{P}\circ(\xi_{[t]})^{-1}$ the probability measure induced by $\xi_{[t]}$. Write $\mathscr{P}(\Xi_{[t]})$ for the set of all probability measures defined over $\Xi_{[t]}$. For each $t=2,\ldots,T$, let $Q_t(\cdot\mid\xi_{[t-1]})$ denote a regular conditional distribution of $\xi_t$ given $\mathcal{F}_{t-1}$, and set $Q_1 := Q_{[1]}$. 
Such regular conditional distributions exist since the random variables take values in Borel subsets of finite-dimensional Euclidean spaces.
Likewise, let $\zeta_{[t]}:=(\zeta_0, \zeta_1,\zeta_2,\cdots,\zeta_t)$ and $\mathcal{G}_t$ the filtration induced by $\zeta_{[t]}$, let $\mathcal{Z}_{[t]} \subseteq\mathbb{R}^{m_{2,0}}\times\mathbb{R}^{m_{2,1}}\times\cdots\times\mathbb{R}^{m_{2,t}}$ denote the support set of $\zeta_{[t]}$, $P_t=\mathbb{P}\circ(\zeta_t)^{-1},\ t=0,1,\ldots,T$. We assume that $\zeta_0,\zeta_1,\ldots,\zeta_T$ are mutually independent. Thus, $\mu$ admits the decomposition $\mu(d\xi_{[T]},d\zeta_{[T]})=Q_1(d\xi_1)\prod_{t=2}^{T}Q_t(d\xi_t\mid\xi_{[t-1]})\prod_{t=0}^{T}P_t(d\zeta_t)$.
With the notation in place, we are ready to introduce the following integrated MSP-MDP model for the decision-making problem under both uncertainties:
\begin{subequations}
\label{eq:mixed-MSP-MDP}
\begin{eqnarray}
\label{eq:mixed-MSP-MDPa}
\inmat{(MSP-MDP)}\quad \underset{\bm x \in \mathcal{X}}{\min}&& \mathbb{E}_{\mu} \left[ C_0(s_0, x_0, \zeta_0) + \sum_{t=1}^T C_t(s_t, x_t, \xi_{[t]}, \zeta_t) \right] \\
\label{eq:mixed-MSP-MDPb}
\textrm{s.t.} &  & g_{t,i}(s_t, x_t, x_{t-1}, \xi_{[t]}) \leq 0, \quad i \in I_t, \quad t = 1, 2, 3, \cdots, T ; \\
\label{eq:mixed-MSP-MDPc}
 &  & s_{t + 1} = S^M_t(s_{t}, x_{t}, \xi_{t}, \zeta_{t}), \quad t = 1, \cdots, T - 1, \\
 & & s_1 = S^M_0(s_0, x_0, \zeta_0).
\end{eqnarray}
\end{subequations}
In this setup, the DM aims to find an admissible policy
${\bm x}(\cdot):=(x_0,\bm x_1(\cdot),\ldots,\bm x_T(\cdot))$
over a finite time horizon $T$ that minimizes the overall expected cost,
where $x_0$ is deterministic and $\bm x_t(\cdot)$ is the stage-$t$
decision rule for $t=1,\ldots,T$. This policy corresponds to a control in the literature of stochastic optimal control\cite{lan2024numerical,bertsekas2012dynamic}. 
Here $C_t(s_t, x_t, \xi_{[t]}, \zeta_t):
\mathbb{R}^{\hat{n}_t} \times \mathbb{R}^{n_t} \times \mathbb{R}^{m_{1, [t]}} \times \mathbb{R}^{m_{2, t}} \to \mathbb{R}$, $m_{1, [t]} := \sum\limits_{k = 1}^t m_{1, k}$, represents
the cost at stage $t$ incurred from action $x_t$. The cost
depends on the current state $s_t$, the historical path $\xi_{[t]}$ (with $\xi_0$ being deterministic)
representing exogenous uncertainty and 
the endogenous uncertainty $\zeta_t$. 
The solution $x_t$ at stage $t$ is subject to constraints \eqref{eq:mixed-MSP-MDPb} 
where the constraint function $g_{t,i}(s_t, x_t, x_{t-1}, \xi_{[t]}) : \mathbb{R}^{\hat{n}_t} \times \mathbb{R}^{n_t} \times \mathbb{R}^{n_{t - 1}} \times \mathbb{R}^{m_{1, [t]}} \to \mathbb{R}$ 
depends on the decision $x_{t-1}$ at stage $t - 1$, 
 the historical path of exogenous uncertainty 
$\xi_{[t]}$ and  
the current state $s_t$; the transition of states is specified 
by \eqref{eq:mixed-MSP-MDPc} 
where the transition mapping $S^M_t:\mathbb{R}^{\hat{n}_t} \times \mathbb{R}^{n_t} \times \mathbb{R}^{m_{1,t}} \times \mathbb{R}^{m_{2,t}} \to \mathbb{R}^{\hat{n}_{t + 1}}$ depends on the state $s_t$ at the stage $t$, the decision $x_t$ at the previous stage, the exogenous uncertainty $\xi_t$ and the endogenous uncertainty $\zeta_t$.

Unlike standard MSP models, the stagewise cost function depends
on state $s_t$ and $\zeta_t$ to emphasize the effect of
endogenous uncertainty of a system. Likewise, 
the state $s_t$ may also affect 
the feasibility of $x_t$. From MDP perspective, 
model \eqref{eq:mixed-MSP-MDP} differs from the standard MDP models in that
the cost function $C_t$ depends on the historical path $\xi_{[t]}$
and transition of states depends on exogenous uncertainty $\xi_t$.
To 
ease the exposition, let
\begin{eqnarray}
\mathcal{X}_t :=\mathcal{X}_t(s_t, x_{t - 1}, \xi_{[t]}) := \{ x_t \in \mathbb{R}^{n_t} \mid g_{t,i}(s_t, x_t, x_{t-1}, \xi_{[t]}) \leq 0, i \in I_t \}
\label{eq:feaset-X_t}
\end{eqnarray}
for $t=1,\cdots,T$ and $\mathcal{X} = \mathcal{X}_0 \times \mathcal{X}_1 \times \mathcal{X}_2 \times \cdots \times \mathcal{X}_T$, where \(\mathcal{X}_0\) denotes the deterministic set of feasible solutions for the initial stage decision \(x_0\); and for $t > 0$, \(\mathcal{X}_t(s_t, x_{t - 1}, \xi_{[t]}) : \mathbb{R}^{\hat{n}_t} \times \mathbb{R}^{n_{t - 1}} \times \mathbb{R}^{m_{1, [t]}} \to \mathbb{R}^{n_t}\) denotes the set of feasible solutions 
of $x_t$ that are measurable with respect to $(\xi_{[t]}, \zeta_{[t - 1]})$, 
where $I_t$ is a finite index set.
An admissible policy is a sequence 
$\bm x(\cdot)=(x_0,\bm x_1(\cdot),\ldots,\bm x_T(\cdot))$
with $x_0\in\mathcal X_0$ and $\bm x_t(s_t, x_{t - 1}, \xi_{[t]})\in\mathbf{\mathcal X}_t(s_t,x_{t-1},\xi_{[t]})$ 
for every $(\xi_{[t]}, \zeta_{[t - 1]}) \in \Xi_{[t]} \times \mathcal{Z}_{[t - 1]}$.
Since $\bm x_t$ depends only on $(s_t,x_{t-1},\xi_{[t]})$ that are observed when a decision $x_t$ is made, every admissible policy satisfies the nonanticipativity requirement (the decision at any time $t$ may depend on available probability distribution for future uncertainty, but not their scenarios).

Note that the decision-making model \eqref{eq:mixed-MSP-MDP} is generally history-dependent and hence non-Markovian w.r.t.~$s_t$. 
Thus, it differs from the Markovian models in stochastic optimal control \cite{bertsekas2012dynamic,puterman2014markov}. 
Nevertheless, by treating $(s_t,x_{t-1},\xi_{[t]})$ as an augmented state, the model can be represented as a Markov decision process. 
We retain the separate representation of these components to distinguish the respective effects of different sources of uncertainties and subsequent stability analysis.
Furthermore, the state transition equations at stage \( t \) do not directly affect the current stage's decision \( x_t \), but indirectly influence the decision-making process through
state \( s_{t + 1} \) (see \eqref{eq:value-t}).
For given \(\xi_t\), the distribution \(P_t\) of \(\zeta_t\) induces a probability transition kernel. Specifically, for every Borel subset \(A\) of \(S_{t+1}\),
$$
\mathbb{P}\left(s_{t+1}\in A\mid s_t,x_t,\xi_t\right) = P_t\left(
\left\{ z\in Z_t: S_t^M(s_t,x_t,\xi_t,z)\in A \right\} \right), \qquad t=1,\ldots,T-1.
$$
Similarly, for every Borel subset \(A\) of \(S_1\),
\[
\mathbb{P}\left(s_1\in A\mid s_0,x_0\right)
=
P_0\left(
\left\{
z\in Z_0:
S_0^M(s_0,x_0,z)\in A
\right\}
\right).
\]
Thus, the probability law of the state transition is the pushforward of the distribution of \(\zeta_t\) under the state-transition mapping.

The integrated MSP-MDP model complements 
the existing MDP models and MSP models 
with some benefits.
One is that 
it takes into account 
endogenous and exogenous uncertainties
simultaneously but 
 distinguishes them explicitly.
The new framework provides a modeling approach that makes the different information timing and structural roles of exogenous and endogenous uncertainties explicit.
For example, in power system management, it 
is desirable to distinguish the uncertainties arising in 
transmission loss and (dis)charging efficiency
in the system
from uncertainty in weather conditions such as 
wind speed and intensity of sunlight
in that these uncertainties may have different effects on the operation of the system and management. 
Distinguishing them explicitly 
may facilitate decision makers to investigate the impact of these uncertainties separately and take relevant management decisions accordingly, we will come back to this in Sections 4 and 5 in terms of stability analysis.
It might also be helpful to consider this distinction from a learning perspective. For instance, in inventory control problems, information about exogenous uncertainties such as price is usually revealed over time with the accumulation of observations and thus learnable. In contrast, endogenous uncertainties such as losses during transportation are inherent to the system and thus not necessary to learn dynamically. 
For a mature logistics system, the loss rate during transportation is mainly dominated by aleatory factors
which are usually regarded as independent random variables (\cite{daganzo2005logistics}), meaning that the loss rate information from previous stages does not provide useful insight for the current stage.
Furthermore,
the separation of the uncertainties may facilitate the DM to 
take respective robust actions 
to address risks arising from 
Knightian  
uncertainties (ambiguity of the distributions of the uncertainties).

\subsection{Examples}
We give a few examples where  
some 
specifically structured 
MSP and MDP models 
can 
be viewed as special instances of model \eqref{eq:mixed-MSP-MDP}.

\begin{example}[Observable-context MDP]
Hallak et al.~\cite{hallak2015contextual} consider the following 
contextual Markov decision process (CMDP)
problem:
\begin{subequations}
\label{eq:ex-CMDP}
\begin{eqnarray}
\min_{\bm x \in \mathcal{X}} && \mathbb{E}_{\theta,\zeta_0, \zeta_1, \cdots, \zeta_T}\left[\sum_{t=0}^T C_t\left(s_{t}, x_{t}, \zeta_{t}, \theta\right)\right] \label{eq:ex-CMDP-obj} \\
\rm{s.t.} && x_t \in \mathcal{X}_t\left(s_t, \theta \right), \label{eq:ex-CMDP-cons} \\
&& s_t = S^{M}_{t - 1}\left(s_{t-1}, x_{t-1}, \zeta_{t-1}, \theta\right), t = 1, 2, \cdots, T,  \label{eq:ex-CMDP-tran}
\end{eqnarray}
\end{subequations}
where $\theta$ represents
contextual uncertainty which 
differs significantly from the 
state transition uncertainty 
$\zeta_t$.
To clarify the distinction between $\theta$ and $\zeta_t$, consider a dynamic decision-making problem formulated as an MDP in a given country where $\theta$ represents the observable prevailing macroeconomic conditions (e.g., interest rate and unemployment rate) and/or specific geopolitical environment, while $\zeta_t$ captures the uncertainty that directly affects the decision-making process. 
The context $\theta$ can be incorporated into an augmented state in which case \eqref{eq:ex-CMDP} reduces to a standard MDP. We nevertheless retain it as a separate component to distinguish its role from that of the state $s_t$ and the state transition uncertainty $\zeta_t$.
This setting is also
closely related to MDPs with continuous side information \cite{modi2018markov}.
When $\theta$ is fixed, it means that the underlying contextual environment remains unchanged  throughout the duration of the MDP.
In the case when $\theta$ is time-dependent (see \cite{hamadanian2023online}), we can 
formulate a CMDP
with non-stationary contexts as  
\begin{subequations}
\label{eq:ex-CMDP-td}
\begin{eqnarray}
\min_{\bm x \in \mathcal{X}} && \mathbb{E}_{\theta_1, \cdots, \theta_T,\zeta_0, \zeta_1, \cdots, \zeta_T}\left[C_0(s_0, x_0, \zeta_0) + \sum_{t=1}^T C_t\left(s_{t}, x_{t}, \zeta_{t}, \theta_t \right)\right] 
\\
\rm{s.t.} && x_t \in \mathcal{X}_t\left(s_t, \theta_{[t]}\right), \label{CMDP-theta_t} \\
&& s_t = S^{M}_{t - 1}\left(s_{t-1}, x_{t-1}, \zeta_{t-1}, \theta_{t - 1}\right), t = 2, \cdots, T,  
\\
&& s_1 = S_0^M(s_0, x_0, \zeta_0), 
\end{eqnarray}
\end{subequations}
where the cost function at each stage and the state transition mapping after the first stage depend on the context variable $\theta_t$ at stage $t$. 
By 
setting $\xi_t := \theta_t$ (exogenous uncertainty) for $t = 1, 2, \cdots, T$ 
and $g_{t, i}(s_t, x_t, x_{t - 1}, \xi_{[t]}) := g_t(s_t, x_t, \theta_{[t]})$, 
we can represent \eqref{eq:ex-CMDP-td}
as \eqref{eq:mixed-MSP-MDP}. 


\end{example}

\begin{example}[Stochastic optimization with side information]
Bertsimas et al.~\cite{bertsimas2023dynamic} 
recently 
propose a dynamic stochastic optimization model with so-called side information, like:
\begin{eqnarray}
    \min_{{x}_t: \Xi_{[t]} \rightarrow \mathcal{X}_t} \mathbb{E}_{\xi} \left[ \sum\limits_{t = 1}^T c_t \left(x_t, {\xi}_t \right) \,\middle|\, {\zeta} = \tilde{{\zeta}} \right],  
    \label{eq:ex-side-info}
\end{eqnarray}
where ${\zeta}$ represents the side information such as product attributes (e.g., brand, style, color of new clothing items in retail). The side information allows the DM to better 
predict the future uncertainties and make more informed decisions.
We may extend the model by allowing $\zeta$ to be time-dependent and 
subsequently obtain the following model:
\begin{eqnarray}
\label{eq:ex-side-info-t}
\quad \underset{x_t : \Xi_{[t]} \to \mathcal{X}_t}{\min}&& \mathbb{E}_{\xi_1, \xi_2, \cdots, \xi_T, \zeta_0, \zeta_1, \cdots, \zeta_T} \left[ c_0(x_0, \zeta_0) + \sum_{t=1}^T c_t(x_t, \xi_{[t]}, \zeta_t) \right],
\end{eqnarray}
which is a special case of \eqref{eq:mixed-MSP-MDP}.
In this setup, $\zeta_t$ represents side information 
which is not necessarily endogenous uncertainty, but we distinguish it from $\xi_t$. Note also that in this model, there is no state variable.
\end{example}

\begin{example}[Inventory control]\label{inventory}
    Consider an 
    inventory control problem 
    over a finite time horizon 
    where 
    uncertainties arise from  
    demand, sale price, purchase
    price, delivery of an order 
    and 
    customer's 
    dissatisfaction.
    The optimal policy is to 
    set appropriate 
    order quantities at each stage such that 
    the expected overall cost is minimized. 
    We use an integrated MSP-MDP model to describe the problem:
\begin{subequations}
\label{eq:ex-inv-cotl}
\begin{eqnarray}
\min _{x_t} && \mathbb{E}_{p,d,\eta,\delta} \left[ h_0 s_0 + p_0 x_0 + \sum_{t=1}^T h_t [s_t]^+ + x_t p_t + l_t [- s_t]^+ \right]  \\
\rm{ s.t.} && s_{t+1} = s_t + (1 - \eta_t) x_t - (1 - \delta_t)d_t, \quad t = 0, 1, \cdots, T - 1;  \label{eq:ex-inv-tran}\\
&& p_t x_t \leq b_t, \quad x_t \leq M - s_t,  \quad t = 0, 1, 2, \cdots, T. \label{eq:ex-inv-cons1} 
\end{eqnarray}
\end{subequations}
 The objective function at stage $t = 0, 1,\cdots, T$ comprises three terms: holding cost $h_t [s_t]^+$, purchase cost $x_t p_t$ and backorder cost $l_t [-s_t]^+$, where $l_t$, $h_t$, and $p_t$ represent the unit backorder cost, holding cost, and purchase cost
respectively, assuming that there 
is no backorder at initial stage; $[s_t]^+ := \max\{ 0, s_t\}$. 
Constraints \eqref{eq:ex-inv-tran} characterize 
the changes of inventory levels between stages. 
Especially, the coefficients 
$(1-\eta_t)$ and $(1-\delta_t)$
signify the order delivery rate 
 and demand delivery rate. 
 Parameters $\eta_t$ and $\delta_t$ are
 nonnegative random variables 
 which represent rates of deliveries.
 The randomness of these parameters 
 capture uncertainties such as 
transportation and 
customer's dissatisfaction.
Constraints \eqref{eq:ex-inv-cons1}  
are related to budget and capacity constraints at
each stage,
i.e., the stagewise purchase cost cannot exceed the budget \( b_t \) and the total inventory after procurement \( x_t \) cannot exceed the maximum warehouse capacity \( M \).
In this model, the optimal decision-making 
at each stage is
dependent on the historical path
(due to the nature of 
exogenous uncertainty $ p_t $) 
and 
there is an explicit specification of 
transition of states depending on the decision and endogenous uncertainties ($ d_t, \eta_t, \delta_t$). 
In doing so, we effectively fit
\eqref{eq:ex-inv-cotl} into  
the integrated MSP-MDP framework \eqref{eq:mixed-MSP-MDP},
which is a departure 
from the existing MSP model 
(\cite{gokbayrak2023single}) and MDP models (\cite{powell2015tutorial} and \cite{toktay2000inventory}) for the problem.
\end{example}


\section{
Specifications, reformulation and properties of problem \texorpdfstring{\eqref{eq:mixed-MSP-MDP}}{(mixed-MSP-MDP)}}

In this section, we give detailed specifications
on 
problem \eqref{eq:mixed-MSP-MDP},  
derive its recursive formulation 
and investigate main properties including 
convexity and Lipschitz continuity
of the value functions. 
To this end, we 
impose the following assumptions.

\begin{assumption}[Integrable cost functions]
For $t = 1, 2, \cdots, T$, (a) there exists 
a nonnegative
integrable
function $h_t(\xi_{[t]}, \zeta_{t})$
such that $C_t(s_{t}, x_{t}, \xi_{[t]}, \zeta_{t}) \geq -h_t(\xi_{[t]}, \zeta_{t})$ for all $(s_t, x_{t})$;
(b) there exists a feasible solution
$\hat{\bm x}_t(s_t, x_{t - 1}, \xi_{[t]})$ 
which is measurable in $(\xi_{[t]}, \zeta_{[t - 1]})$
such that $\mathbb{E}_{\xi_{[t]}, \zeta_{[t]}} \left[ C_t(s_t, \hat{x}_t(s_t, x_{t - 1}, \xi_{[t]}), \xi_{[t]}, \zeta_t) \right] < \infty$ for any given $(s_t, x_{t-1})$.

\label{Assu:welldefinedness}
\end{assumption}
This type of assumptions is commonly used 
in the MSP literature, where integrability conditions are 
needed to ensure 
the finiteness of expected costs. 
  Assumption~\ref {Assu:welldefinedness}
  holds if 
  $C_t(s_{t}, x_{t}, \xi_{[t]}, \zeta_{t})$ is Lipschitz continuous w.r.t. $(s_t, x_t, \xi_{[t]}, \zeta_{t} )$, 
    and 
$\mathbb{E}_{\xi_{[t]}} \left[ \Vert \xi_{[t]} \Vert \right] < \infty$, 
$\mathbb{E}_{\zeta_{t}} \left[ \Vert \zeta_{t} \Vert \right] < \infty$.

\begin{assumption}[Continuity of the underlying functions]
    For \( t = 1, 2, \cdots, T \), (a) \( C_t\left(s_t, x_t, \xi_{[t]}, \zeta_t\right) \) is 
    continuous w.r.t. $(s_t, x_t, \xi_{[t]}, \zeta_t)$; (b) $g_{t, i}(s_t, x_t, x_{t - 1}, \xi_{[t]}), i \in I_t$ is 
    continuous w.r.t. $(s_t, x_t,x_{t - 1}, \xi_{[t]})$; (c) $S^M_{t - 1}(s_{t - 1}, x_{t - 1}, \xi_{t - 1}, \zeta_{t - 1})$ is continuous w.r.t. $(s_{t - 1}, x_{t - 1}, \xi_{t - 1}, \zeta_{t - 1})$.
    \label{Assu:continuity}
\end{assumption}

It is possible to weaken the continuity assumption 
to lower semicontinuity, see e.g. \cite{nemirovski2009robust,shapiro2021lectures,shapiro2021distributionally,lecarpentier2019non,le2021metrics}
for MSP and MDP models.
We make the assumption
so that we may concentrate
on key stability analysis in this paper.

\begin{assumption}[Boundedness of feasible sets]
Let $\mathcal{X}_t(s_t, x_{t - 1}, \xi_{[t]})$ be
 defined as in \eqref{eq:feaset-X_t}.
    For \( t = 1, 2, \cdots, T \), there exists a bounded set \( \mathcal{X}_t^0 \) such that \( \mathcal{X}_t(s_t, x_{t - 1}, \xi_{[t]}) \subseteq \mathcal{X}_t^0 \) for any given \( (s_t, x_{t - 1}, \xi_{[t]}) \).  
    \label{Assu:compactness}
\end{assumption}
Assumptions~\ref{Assu:continuity} and ~\ref{Assu:compactness} guarantee that 
the feasible sets \( \mathcal{X}_t(s_t, x_{t - 1}, \xi_{[t]}) \), $t = 1, 2, \cdots, T$
are compact.
It is possible to weaken the condition by replacing it 
with inf-compactness condition but 
here we make it easier to simplify 
the analysis in the forthcoming discussions.
With the assumptions, we are ready to state 
the dynamic formulation of problem \eqref{eq:mixed-MSP-MDP} in the next theorem.

\begin{theorem}[Nested reformulation of \eqref{eq:mixed-MSP-MDP}]
Consider problem:
\begin{subequations}
\label{eq:model-nested}
\begin{eqnarray}
\min _{x_0 \in \mathcal{X}_0} && \mathbb{E}_{\zeta_0} \bigg[ C_0(s_0, x_0, \zeta_0) + \mathbb{E}_{\xi_1} \bigg[ \min _{x_1 \in \mathcal{X}_1(s_1, x_0, \xi_1)} \mathbb{E}_{\xi_2 | \xi_1, \zeta_1} \bigg[  C_1\left(s_1, x_1, \xi_1, \zeta_1 \right) \nonumber \\
&& + \min _{x_2 \in \mathcal{X}_2\left(s_2, x_1, \xi_{[2]}\right)} \mathbb{E}_{\xi_3 \mid \xi_{[2]}, \zeta_2} \bigg[C_2\left(s_2, x_2, \xi_{[2]}, \zeta_2\right) + \cdots \nonumber \\
&& + \min _{x_T \in \mathcal{X}_T \left(s_T, x_{T-1}, \xi_{[T]}\right)} 
\mathbb{E}_{\zeta_T}\left[C_T\left(s_T, x_T, \xi_{[T]}, 
\zeta_T\right)\right]\bigg] \cdots \bigg]  \\
\rm{s.t.} && \quad  s_{t + 1} = S^M_t (s_{t}, x_{t}, \xi_{t}, \zeta_{t}), \quad  t = 0, 1, \cdots, T - 1,
\end{eqnarray}
\end{subequations}
where $S_0^M(\cdot)$ only depends on $(s_0, x_0, \zeta_0)$ as shown in \eqref{eq:mixed-MSP-MDP}, \( \mathbb{E}_{\xi_{t + 1} | \xi_{[t]}, \zeta_{t}}[\cdot] \) denotes the 
expectation with respect to joint probability distribution of 
\( \xi_{t + 1} \) 
conditional on \( \xi_{[t]} \),
and 
\( \zeta_{t} \).
Let
\begin{eqnarray}
&&v_T\left(s_{T-1}, x_{T-1}, \xi_{[T]}, \zeta_{T - 1}\right) := \min _{x_T \in \mathcal{X}_T(s_T, x_{T - 1}, \xi_{[T]})} \mathbb{E}_{\zeta_T} \left[ C_T\left(s_T, x_T, \xi_{[T]}, \zeta_T\right) \right] \label{eq:value-T} 
 \end{eqnarray}
and 
for
\( t = 1, 2, \cdots, T - 1 \),
let
\begin{eqnarray}
 &&v_t\left(s_{t-1}, x_{t-1}, \xi_{[t]}, \zeta_{t - 1}\right) \nonumber  \\
&:= & \min_{x_t \in \mathcal{X}_t(s_t, x_{t - 1}, \xi_{[t]})} \mathbb{E}_{\xi_{t+1} \mid \xi_{[t]}, \zeta_{t}} \left[ C_t\left(s_t, x_t, \xi_{[t]}, \zeta_t\right) + v_{t+1}\left(s_t, x_t, \xi_{[t+1]}, \zeta_{t}\right)\right],  \label{eq:value-t} \\
&& v_0(s_0) := \underset{x_0 \in \mathcal{X}_0}{\min} \mathbb{E}_{\zeta_0} \left[ C_0(s_0, x_0, \zeta_0) + \mathbb{E}_{\xi_1} \left[ v_1(s_0, x_0, \xi_1, \zeta_0) \right] \right], \label{eq:value-0}
\end{eqnarray}
where $\mathcal{X}_t(s_t, x_{t - 1}, \xi_{[t]})$ is defined as in \eqref{eq:feaset-X_t}.
    Under Assumptions \ref{Assu:welldefinedness}- 
    \ref{Assu:compactness},
problem \eqref{eq:mixed-MSP-MDP} 
    can be reformulated as problem  \eqref{eq:model-nested} or equivalently as problems \eqref{eq:value-T}-\eqref{eq:value-0}.
    \label{Theo-timecons}
\end{theorem}

A key step towards the proof of the theorem is to establish continuity of  
$v_t$ in $(s_{t - 1}, x_{t - 1})$.
The next proposition addresses this.

\begin{proposition}[Lower semicontinuity 
of $v_t$ w.r.t. $(s_{t - 1}, x_{t - 1})$ and well-definedness of problems \eqref{eq:value-T}-\eqref{eq:value-0}]
    Under Assumptions \ref{Assu:welldefinedness}- \ref{Assu:compactness},
    $v_t$ is well-defined and lower semicontinuous w.r.t. $(s_{t - 1}, x_{t - 1})$  for 
    $t = 1, \cdots, T$.
    \label{Prop:continuity}
\end{proposition}

\begin{proof}
    We prove the 
    lower semicontinuity of $v_t$
    by 
    induction from $t=T$. 
Under Assumption~\ref{Assu:continuity},
$\mathcal{X}_T(s_T, x_{T - 1}, \xi_{[T]})$ is upper hemicontinuous in $(s_T, x_{T - 1})$
and under Assumption~\ref{Assu:compactness}, it is compact set-valued.
By the continuity of \(C_T\) and Fatou's lemma, $\mathbb{E}_{\zeta_T} \left[ C_T\left(s_T, x_T, \xi_{[T]}, \zeta_T\right) \right]$ is lower
semicontinuous in $(s_T, x_T)$. 
Since $s_T = S^M_{T - 1}(s_{T - 1}, x_{T - 1}, \xi_{T - 1}, \zeta_{T - 1})$ and $S^M_{T - 1}$ is continuous, 
then $\mathbb{E}_{\zeta_T} \left[ C_T\left(s_T, x_T, \xi_{[T]}, \zeta_T\right) \right]$ is lower semicontinuous in $(s_{T - 1}, x_{T - 1}, x_T)$ and $\mathcal{X}_T(s_T, x_{T - 1}, \xi_{[T]})$ is  upper hemicontinuous in $(s_{T - 1}, x_{T - 1})$. Consequently, we can use Berge's maximum theorem, \cite[Theorem 17.31]{aliprantis2006infinite},
to assert that $v_T$ is lower semicontinuous in $(s_{T - 1}, x_{T - 1})$.


    Assume now 
    that $ v_{t+1}(\cdot, \cdot, \xi_{[t+1]}, \zeta_t) $ is lower semicontinuous. 
    We prove that $ v_t $ is lower semicontinuous in $ (s_{t-1}, x_{t-1}) $. Looking at the minimization problem \eqref{eq:value-t}, we note that the feasible set-valued mapping is upper hemicontinuous in $(s_t, x_{t - 1})$
    and the objective function 
   $\mathbb{E}_{\xi_{t+1} \mid \xi_{[t]}, \zeta_{t}} [C_t\left(s_t, x_t, \xi_{[t]}, \zeta_t\right) + v_{t+1}\left(s_t, x_t, \xi_{[t+1]}, \zeta_{t}\right)]$ is lower semicontinuous in $(s_t,x_{t})$. Following a similar argument to the first part of the proof, we can establish the lower semicontinuity of \(v_t\) in $ (s_{t-1}, x_{t-1})$. 
    
    Next, we 
    prove the well-definedness of problems \eqref{eq:value-T}-\eqref{eq:value-0} by induction. At stage $T$, by Assumption~\ref{Assu:welldefinedness}(a),  
    \begin{eqnarray}
        &&\mathbb{E}_{\xi_T | \xi_{[T - 1]}} \left[ v_T(s_{T - 1}, x_{T - 1}, \xi_{[T]}, \zeta_{T - 1}) \right] \nonumber \\
        &= & \mathbb{E}_{\xi_T | \xi_{[T - 1]}} \left[ \underset{x_T \in \mathcal{X}_T(s_T, x_{T - 1}, \xi_{[T]})}{\min} \mathbb{E}_{\zeta_T} \left[ C_T(s_T, x_T, \xi_{[T]},\zeta_T) \right] \right] \nonumber \\
        & \geq& \mathbb{E}_{\xi_T | \xi_{[T - 1]}, \zeta_T} \left[ -h_T(\xi_{[T]}, \zeta_T) \right] >- \infty
    \end{eqnarray}
    uniformly for all $(s_{T - 1}, x_{T - 1})$ almost surely. Moreover, by Assumption~\ref{Assu:welldefinedness}(b),
    \begin{eqnarray}
        &&\mathbb{E}_{\xi_T | \xi_{[T - 1]}} \left[ v_T(s_{T - 1}, x_{T - 1}, \xi_{[T]}, \zeta_{T - 1}) \right] \nonumber \\
        &= & \mathbb{E}_{\xi_T | \xi_{[T - 1]}} \left[ \underset{x_T \in \mathcal{X}_T(s_T, x_{T - 1}, \xi_{[T]})}{\min} \mathbb{E}_{\zeta_T} \left[ C_T(s_T, x_T, \xi_{[T]},\zeta_T) \right] \right] \nonumber \\
        &= & \mathbb{E}_{\xi_T | \xi_{[T - 1]}, \zeta_T} \left[ C_T(s_T, x^*_T, \xi_{[T]},\zeta_T) \right] \nonumber \\
        &\leq & \mathbb{E}_{\xi_T | \xi_{[T - 1]}, \zeta_T} \left[ C_T(s_T, \hat{x}_T(s_T, x_{T - 1}, \xi_{[T]}), \xi_{[T]},\zeta_T) \right] < + \infty
    \end{eqnarray} 
    almost surely. Thus,
      $\mathbb{E}_{\xi_T | \xi_{[T - 1]}} \left[ v_T \right]$ is finite-valued.
      Together with the lower semicontinuity of $v_{T}$ and the compactness of $\mathcal{X}_T(s_T, x_{T - 1}, \xi_{[T]})$, 
      we conclude that problem \eqref{eq:value-T} is well defined. 
    Assume now that $\mathbb{E}_{\xi_{t + 1} | \xi_{[t]}} \left[ v_{t + 1} \right]$ is finite-valued for $1 \leq t \leq T - 1$. Analogous to the above proof for stage $T$, we can derive, under Assumption~\ref{Assu:welldefinedness}(a), that
    \begin{eqnarray}
        &&\mathbb{E}_{\xi_t | \xi_{[t - 1]}} \left[ v_t(s_{t - 1}, x_{t - 1}, \xi_{[t]}, \zeta_{t - 1}) \right] \nonumber \\
        &= & \mathbb{E}_{\xi_t | \xi_{[t - 1]}} \left[\underset{x_t \in \mathcal{X}_t(s_t, x_{t - 1}, \xi_{[t]})}{\min} \mathbb{E}_{\zeta_t} \left[ C_t(s_{t}, x_{t}, \xi_{[t]}, \zeta_{t}) + \mathbb{E}_{\xi_{t + 1} | \xi_{[t]}} \left[ v_{t + 1}(s_t, x_t, \xi_{[t + 1]}, \zeta_t) \right] \right] \right] \nonumber \\
        &= & \mathbb{E}_{\xi_t | \xi_{[t - 1]}, \zeta_t} \left[ C_t(s_{t}, x^*_{t}, \xi_{[t]}, \zeta_{t}) \right] + \mathbb{E}_{\xi_{t + 1}, \xi_t | \xi_{[t - 1]}, \zeta_t} \left[ v_{t + 1}(s_t, x^*_t, \xi_{[t + 1]}, \zeta_t) \right] \nonumber \\
        &\geq & \mathbb{E}_{\xi_t | \xi_{[t - 1]}, \zeta_t} \left[ -h_t(\xi_{[t]}, \zeta_t) \right] + \mathbb{E}_{\xi_{t + 1}, \xi_t | \xi_{[t - 1]}, \zeta_t} \left[ v_{t + 1}(s_t, x^*_t, \xi_{[t + 1]}, \zeta_t) \right] > - \infty,
    \end{eqnarray}
    where $x_t^*$ is an optimal solution to problem \eqref{eq:value-t}. The last inequality is obtained by Assumption~\ref{Assu:welldefinedness}(a) and the fact that $\mathbb{E}_{\xi_{t + 1}, \xi_t | \xi_{[t - 1]}, \zeta_t} \left[ v_{t + 1} \right]$ is finite-valued due to the induction assumption. On the other hand, by Assumption~\ref{Assu:welldefinedness}(b), 
\begin{eqnarray}
    &&\mathbb{E}_{\xi_t | \xi_{[t - 1]}} \left[ v_t(s_{t - 1}, x_{t - 1}, \xi_{[t]}, \zeta_{t - 1}) \right] \nonumber \\
    &= & \mathbb{E}_{\xi_t | \xi_{[t - 1]}, \zeta_t} \left[ C_t(s_{t}, x^*_{t}, \xi_{[t]}, \zeta_{t}) \right] + \mathbb{E}_{\xi_{t + 1}, \xi_t | \xi_{[t - 1]}, \zeta_t} \left[ v_{t + 1}(s_t, x^*_t, \xi_{[t + 1]}, \zeta_t) \right] \nonumber \\
    &\leq & \mathbb{E}_{\xi_{t} | \xi_{[t - 1]}, \zeta_{t}} \left[ C_{t}(s_{t}, \hat{x}_{t}(s_t, x_{t - 1}, \xi_{[t]}), \xi_{[t]}, \zeta_{t}) + \mathbb{E}_{\xi_{t + 1} | \xi_{[t]}} \left[ v_{t + 1}(s_{t}, \hat{x}_{t}(s_t, x_{t - 1}, \xi_{[t]}), \xi_{[t + 1]}, \zeta_{t}) \right] \right] \nonumber \\
    &= & \mathbb{E}_{\xi_{t} | \xi_{[t - 1]}, \zeta_{t}} \left[ C_{t}(s_{t}, \hat{x}_{t}(s_t, x_{t - 1}, \xi_{[t]}), \xi_{[t]}, \zeta_{t}) \right] \nonumber \\
    &&+ \mathbb{E}_{\xi_{t}, \xi_{t + 1} | \xi_{[t - 1]}, \zeta_{t}} \left[ v_{t + 1}(s_{t}, \hat{x}_{t}(s_t, x_{t - 1}, \xi_{[t]}), \xi_{[t + 1]}, \zeta_{t}) \right]. 
    \label{eq:well-defined-vt2}
\end{eqnarray}
We estimate the second term
\begin{eqnarray}
    && \mathbb{E}_{\xi_{t},\xi_{t + 1} | \xi_{[t - 1]}, \zeta_{t}} \left[ v_{t + 1}(s_{t}, \hat{x}_{t}(s_t, x_{t - 1}, \xi_{[t]}), \xi_{[t + 1]}, \zeta_{t}) \right] \nonumber \\
    &= & \mathbb{E}_{\xi_{t}, \xi_{t + 1} | \xi_{[t - 1]}, \zeta_{t}, \zeta_{t + 1}} \Big[ C_{t + 1}(s_{t + 1}, {x}_{t + 1}^*, \xi_{[t + 1]}, \zeta_{t + 1}) \nonumber \\
    && + \mathbb{E}_{\xi_{t + 2} | \xi_{[t + 1]}} \left[ v_{t + 2}(s_{t + 1}, {x}_{t + 1}^*, \xi_{[t + 2]}, \zeta_{t + 1}) \right] \Big] \nonumber \\
    &\leq & \mathbb{E}_{\xi_{t}, \xi_{t + 1} | \xi_{[t - 1]}, \zeta_{t}, \zeta_{t + 1}} \Big[ C_{t + 1}(s_{t + 1}, \hat{x}_{t + 1}(s_{t + 1}, \hat{x}_{t},\xi_{[t + 1]}), \xi_{[t + 1]}, \zeta_{t + 1}) \nonumber \\
    &&  + \mathbb{E}_{\xi_{t + 2} | \xi_{[t + 1]}} \left[ v_{t + 2}(s_{t + 1}, \hat{x}_{t + 1}(s_{t + 1}, \hat{x}_{t},\xi_{[t + 1]}), \xi_{[t + 2]}, \zeta_{t + 1}) \right] \Big], \nonumber \\
    &\leq & \cdots \leq \sum\limits_{k = t + 1}^T \mathbb{E}_{\xi_{[t, k]} | \xi_{[t - 1]}, \zeta_{t}, \zeta_{t + 1}, \cdots, \zeta_{k}} \left[ C_{k}(s_{k}, \hat{x}_{k}, \xi_{[k]}, \zeta_{k})\right], 
    \label{eq:nested-<infty}
\end{eqnarray}
where for $t \leq k \leq T - 1$, $s_{k + 1} = S^M_{k}(s_{k}, \hat{x}_{k}, \xi_{k}, \zeta_{k})$, $\hat{x}_{k+1} := \hat{x}_{k + 1}(s_{k + 1}, \hat{x}_{k}, \xi_{[k + 1]})$ is from Assumption~\ref{Assu:welldefinedness}(b), $x_{k + 1}^*$ is the optimal solution to problem \eqref{eq:value-t}.
It is known from Assumption~\ref{Assu:welldefinedness}(b) that 
$$
\mathbb{E}_{\xi_{[t, k]} | \xi_{[t - 1]}, \zeta_{t}, \zeta_{t + 1}, \cdots, \zeta_{k}} \left[ C_{k}(s_{k}, \hat{x}_{k}, \xi_{[k]}, \zeta_{k})\right] < + \infty
$$
almost surely.
Thus, the rhs of \eqref{eq:well-defined-vt2} $< + \infty$. Together with the lower semicontinuity of $v_t$ and the compactness of $\mathcal{X}_t(s_t, x_{t - 1}, \xi_{[t]})$, we have that $\mathbb{E}_{\xi_t | \xi_{[t - 1]}} \left[ v_{t} \right]$ is finite-valued and problems \eqref{eq:value-t}-\eqref{eq:value-0} are well defined.
The proof is completed.   
\end{proof}

We are now ready to prove Theorem \ref{Theo-timecons}.
\begin{proof}[Proof of Theorem \ref{Theo-timecons}]
    Let \( \bm x^* = (x_0^*, \bm x_{1}^*(s_1^*, x_0^*, \xi_1), \bm x_2^*(s_2^*, \bm x_1^*, \xi_{[2]}), \cdots, \bm x_T^*(s_T^*, \bm x_{T - 1}^*, \xi_{[T]})) \) denote the optimal policy of problem \eqref{eq:mixed-MSP-MDP}, i.e.,
    \begin{eqnarray}
    && \mathbb{E}_{\xi_1, \xi_2 | \xi_{1}, \cdots, \xi_T | \xi_{[T - 1]}, \zeta_0, \zeta_1, \zeta_2, \cdots, \zeta_T}  \left[ C_0(s_0, x_0^*, \zeta_0) + \sum_{t=1}^T C_t (s_t,{x}^*_t,\xi_{[t]},\zeta_t )\right] \nonumber \\
    &= & \underset{\bm x(\cdot) \in \mathcal{X}(\cdot)}{\text{min }}
    \mathbb{E}_{\xi_1, \xi_2 | \xi_{1}, \cdots, \xi_T | \xi_{[T - 1]}, \zeta_0, \zeta_1, \zeta_2, \cdots, \zeta_T} \left[ C_0(s_0, x_0, \zeta_0) + \sum_{t=1}^T C_t (s_t, {x}_t,\xi_{[t]},\zeta_t )\right]. \nonumber
    \end{eqnarray}
    Note that a policy
    \( \bm x(\cdot) : \Xi_{[T]} \times \mathcal{Z}_{[T - 1]} \to \mathbb{R}^{n_{[T]}} \) maps each realization of the stochastic data process \( { \xi, \zeta } \) to a feasible solution \( x \) of problem \eqref{eq:mixed-MSP-MDP}, where $n_{[T]} = {\sum\limits_{t = 0}^T n_t}$.
    To ease the exposition, we write \( \bm x_t \) for
    the decision at stage \( t \).
    For
    \( t = 1, 2, \cdots, T \), we
    prove that \( x_t^* = \bm x_t^*(s_t, x_{t - 1}^*, \xi_{[t]}) \) is an optimal solution to problems \eqref{eq:value-t} and
    \eqref{eq:value-T}, i.e.,
    \begin{eqnarray}
    x_t^* \in  \underset{x_t \in \mathcal{X}_t(s_t, x_{t-1}^*, \xi_{[t]})}{\text{argmin }} \left\{\mathbb{E}_{\xi_{t+1} | \xi_{[t]}, \zeta_{t}} \left[C_t(s_t, x_t, \xi_{[t]}, \zeta_t) +  v_{t+1}(s_t, x_t, \xi_{[t+1]}, \zeta_{t}) \right]\right\}, t = 1, 2, \cdots, T, \nonumber
    \end{eqnarray}
    where
    we define \( \xi_{T + 1} \) as a constant and
    set \( v_{T+1}(s_T, x_T, \xi_{[T+1]}, \zeta_{T}) \equiv 0 \).
    
    We prove the theorem by induction. 
    Observe that
    for a given feasible solution \( x_{t - 1} \) 
    at stage \( t - 1 \), 
    the feasible solution at stage \( t \), \( \bm x_t(\cdot) = \bm x_t(s_t , x_{t - 1}, \cdot) \in \mathbf{\mathcal{X}}_t(s_t, x_{t - 1}, \cdot): \Xi_{[t]} \to \mathbb{R}^{n_t} \) 
    maps each realization of \( (\xi_{[t]}) \) to a feasible decision \( x_t \) at stage $t$, where 
    $\mathbf{\mathcal{X}}_t(s_t, x_{t - 1}, \cdot)$
    represents 
    the feasible set 
    of all measurable mappings defined over $\Xi_{[t]}$ and 
    $\mathbf{\mathcal{X}}_t(s_{t}, x_{t - 1}, \xi_{[t]}) \subset \mathcal{X}_{t}^0$ is a subset of  
    $\mathbb{R}^{n_t}
    $.
    At stage $T$, 
    \begin{subequations}
    \begin{eqnarray}
       && \underset{\bm x_T \in \mathcal{X}_T(s_T, x^*_{T - 1}, (\xi_{[T - 1]}, \cdot))}{\min} \mathbb{E}_{\xi_{T} | \xi_{[T-1]}, \zeta_T} \left[ C_T(s_T, \bm x_T, \xi_{[T]}, \zeta_T) \right] \nonumber \\
       &= & \mathbb{E}_{\xi_{T} | \xi_{[T-1]}} \left[ \underset{x_T \in \mathcal{X}_T(s_T, x^*_{T - 1}, \xi_{[T]})}{\min} \mathbb{E}_{\zeta_T} \left[ C_T(s_T, x_T, \xi_{[T]}, \zeta_T) \right] \right]\\ 
       &= &\mathbb{E}_{\xi_{T} | \xi_{[T-1]}} \left[ v_T(s_{T - 1}, x^*_{T - 1}, \xi_{[T]}, \zeta_{T - 1}) \right],
    \end{eqnarray}  
    \label{eq:timecons-T}
    \end{subequations}
    where $x_{T - 1}^* = \bm x_{T - 1}^*(s_{T - 1}, x_{T - 2}^*, \xi_{[T - 1]})$ is the optimal solution at stage $T - 1$ given $\xi_{[T - 1]}$ and $\zeta_{[T - 1]}$,
    the first equality 
    is due to the interchangeability principle~\cite[Lemma 3]{liu2025preference}. To see this, we verify the conditions of the lemma. 
    Assumption \ref{Assu:continuity} ensures that $C_T$ is continuous w.r.t.~$(s_T, x_T)$ and the set-valued mapping $\mathcal{X}_T(s_T, x_{T - 1}^*, \cdot)$ is a closed-valued measurable mapping due to the continuity of $g_{T, i}(s_T, x_T, x_{T - 1}, \xi_{[T]}), i \in I_T$ with respect to $(s_T, x_T,x_{T - 1}, \xi_{[T]})$. Assumption \ref{Assu:compactness} means that the feasible decision $\bm x_T(s_T, x_{T - 1}^*, \cdot)$ is an integrable mapping. Assumption \ref{Assu:welldefinedness} ensures that there is a feasible solution $\hat{\bm x}_t(s_T, x_{T - 1}, \cdot)$ such that $\mathbb{E}_{\xi_T | \xi_{[T - 1]}, \zeta_T} \left[ C_T(s_T, \hat{x}_T(s_T, x_{T - 1}, \xi_{[T]}), \xi_{[T]}, \zeta_T) \right] < \infty$.
    The second equality follows from the definition of $v_T$.
    
    Next, 
    let $t_0$ be such that $1 \leq t_0 \leq T - 1$. At stage $t = t_0 + 1$, let
    \begin{eqnarray}
    && \bm x_{[t_0+1, T]}(\cdot) := (\bm x_{t_0 +1}(s_{t_0 + 1}, x_{t_0}^*, \cdot),  \cdots, \bm x_T(s_T, x_{T - 1}, \cdot)), \nonumber \\
    && \mathbf{\mathcal{X}}_{[t_0 + 1, T]}(\cdot) := \mathbf{\mathcal{X}}_{t_0 + 1}(s_{t_0 + 1}, x_{t_0}^*, \cdot) \times \cdots \times \mathbf{\mathcal{X}}_{T}(s_T, x_{T - 1}, \cdot). \nonumber
    \end{eqnarray} 
    For given \( s_{t_0 + 1}, x_{t_0}^* \) and \( \xi_{[t_0 + 1]} \), by induction 
    \begin{eqnarray}
    && \underset{\bm x_{[t_0+1, T]}(\cdot) \in \mathbf{\mathcal{X}}_{[t_0+1,T]}(\cdot)}{\min} \mathbb{E}_{\xi_{t_0+1}, \cdots, \xi_T | \xi_{[t_0]}, \zeta_{t_0+1}, \cdots, \zeta_T} \left[\sum_{t=t_0 + 1}^T C_t (s_t,\bm x_t,\xi_{[t]},\zeta_t )\right] \nonumber \\
    &= & \mathbb{E}_{\xi_{t_0+1} | \xi_{[t_0]}} \big[ v_{t_0 + 1}(s_{t_0 }, x_{t_0}^*, \xi_{[t_0 + 1]}, \zeta_{t_0}) \big]. 
    \label{eq:nested-hypo}
    \end{eqnarray}  
    Under the Assumptions~\ref{Assu:welldefinedness}-~\ref{Assu:compactness}, 
    Proposition \ref{Prop:continuity} guarantees that
    $v_t(s_{t-1}, x_{t-1}, \xi_{[t]}, \zeta_{t - 1})$ is finite-valued and lower semicontinuous for $t = 1, 2,\cdots, T$,  
    and there exists a $\hat{\bm x}_t(s_{t_0}, x_{t_0 - 1}, \cdot)$ such that 
    $$
    \mathbb{E}_{\xi_{t_0} | \xi_{[t_0 - 1]}, \zeta_{t_0}} \left[ C_{t_0}(s_{t_0}, \hat{x}_{t_0}, \xi_{[t_0]}, \zeta_{t_0}) + \mathbb{E}_{\xi_{t_0 + 1} | \xi_{[t_0]}} \left[ v_{t_0 + 1}(s_{t_0}, \hat{x}_{t_0}, \xi_{[t_0 + 1]}, \zeta_{t_0}) \right] \right] < \infty
    $$ 
    almost surely(by \eqref{eq:nested-<infty}).
    Therefore, for the optimal solution \( x_{t_0 - 1}^* \) at stage \( t_0 - 1 \), we have
    \begin{eqnarray}
    && \underset{\boldsymbol{x}_{[t_0, T]}(\cdot) \in \mathbf{\mathcal{X}}_{[t_0, T]}(\cdot)}{\min} \mathbb{E}_{\xi_{t_0}, \xi_{t_0+1}, \cdots, \xi_T | \xi_{[t_0 - 1]}, \zeta_{t_0}, \zeta_{t_0+1}, \cdots, \zeta_T} \left[\sum_{t = t_0}^T C_t (s_t,\boldsymbol{x}_t,\xi_{[t]},\zeta_t )\right] \nonumber \\
    &= & \underset{\boldsymbol{x}_{t_0}(s_{t_0}, x^*_{t_0 - 1}, \cdot) \in \mathbf{\mathcal{X}}_{t_0}(s_{t_0}, x_{t_0 - 1}^*, \cdot)}{\min} \mathbb{E}_{\xi_{t_0} | \xi_{[t_0 - 1]},\zeta_{t_0}} \Bigg[ C_{t_0} (s_{t_0},\boldsymbol{x}_{t_0},\xi_{[t_0]},\zeta_{t_0} )   \nonumber\\
    &&  +  \underset{\boldsymbol{x}_{[t_0 + 1, T]}(\cdot) \in \mathbf{\mathcal{X}}_{[t_0 + 1,T]}(\cdot)}{\min} \mathbb{E}_{\xi_{t_0+1}, \cdots, \xi_T | \xi_{[t_0]}, \zeta_{t_0+1}, \cdots, \zeta_T} \left[ \sum_{t=t_0 + 1}^T C_t (s_t,\boldsymbol{x}_t,\xi_{[t]},\zeta_t )
    \right]\Bigg] \nonumber \\
    &= & \underset{\boldsymbol{x}_{t_0}(s_{t_0}, x^*_{t_0 - 1}, \cdot) \in \mathbf{\mathcal{X}}_{t_0}(s_{t_0}, x_{t_0 - 1}^*, \cdot)}{\min} \mathbb{E}_{\xi_{t_0} | \xi_{[t_0 - 1]},\zeta_{t_0}} \Bigg[ C_{t_0} (s_{t_0},\boldsymbol{x}_{t_0},\xi_{[t_0]},\zeta_{t_0} )  \nonumber \\
    && +  \mathbb{E}_{\xi_{t_0+1} | \xi_{[t_0]}} \Bigg[ v_{t_0 + 1}(s_{t_0 }, \boldsymbol{x}_{t_0}^*, \xi_{[t_0 + 1]}, \zeta_{t_0}) \Bigg] \Bigg] 
    \quad \text{(by \eqref{eq:nested-hypo})}
    \nonumber \\
    &= & \mathbb{E}_{\xi_{t_0} | \xi_{[t_0 - 1]}} \bigg[ \underset{{x}_{t_0}(s_{t_0}, x^*_{t_0 - 1}, \xi_{[t_0]}) \in \mathbf{\mathcal{X}}_{t_0}(s_{t_0}, x_{t_0 - 1}^*, \xi_{[t_0]})}{\min} \mathbb{E}_{\zeta_{t_0}} \big[ C_{t_0} (s_{t_0},{x}_{t_0},\xi_{[t_0]},\zeta_{t_0} )  \nonumber \\
    && +  \mathbb{E}_{\xi_{t_0+1} | \xi_{[t_0]}} \big[ v_{t_0 + 1}(s_{t_0 }, {x}_{t_0}^*, \xi_{[t_0 + 1]}, \zeta_{t_0}) \big] \big] \bigg] \nonumber \\
    &= & \mathbb{E}_{\xi_{t_0} | \xi_{[t_0 - 1]}} \big[ v_{t_0}(s_{t_0 - 1}, x_{t_0 - 1}^*, \xi_{[t_0]}, \zeta_{t_0 - 1}) \big], \nonumber
    \end{eqnarray}
    where the first equality follows from the fact that $ C_{t_0}(s_{t_0}, x_{t_0}, \xi_{[t_0]}, \zeta_{t_0}) $ is independent of $ x_{[t_0 + 1, T]} $.  
    The second equality is obtained by 
    \eqref{eq:nested-hypo}.
    The third equality 
    is based on 
    the interchangeability principle 
    ~\cite[Lemma 3]{liu2025preference} in that $ C_{t_0} + \mathbb{E} \left[ v_{t_0 + 1} \right] $ is lower semicontinuous w.r.t.~$ (s_{t_0}, x_{t_0}) $(by Proposition \ref{Prop:continuity}),
    the set-valued mapping $\mathcal{X}_{t_0}(s_{t_0}, x_{t_0 - 1}^*, \cdot)$ is measurable (by Assumption \ref{Assu:continuity}), and $\bm x_t(s_{t_0}, x_{t_0 - 1}^*, \cdot)$ is integrable (by Assumption \ref{Assu:compactness}).
    Summarizing the discussions above, we can conclude that  
    \begin{eqnarray}
    && \underset{\bm x(\cdot) \in \mathbf{\mathcal{X}}(\cdot)}{\text{min }}
    \mathbb{E}_{\xi_1, \xi_2, \cdots, \xi_T, \zeta_0, \zeta_1, \zeta_2, \cdots, \zeta_T} 
    \left[ C_0(s_0, x_0,\zeta_0) + \sum_{t=1}^T C_t (s_t,\boldsymbol{x}_t,\xi_{[t]},\zeta_t )\right] \nonumber \\
    &= & \underset{{x_0} \in \mathcal{X}_0}{\text{min }} \mathbb{E}_{\zeta_0} \left[ C_0(s_0, x_0,\zeta_0) + \mathbb{E}_{\xi_{1}} \left[ v_1(s_0, x_0, \xi_{1}, \zeta_0) \right] \right]. \nonumber
    \end{eqnarray}  
    The proof is completed.
\end{proof}

Theorem \ref{Theo-timecons} 
provides a theoretical guarantee to solve problem \eqref{eq:mixed-MSP-MDP} 
by 
solving problems \eqref{eq:value-T}-\eqref{eq:value-0}
recursively.
With this, we proceed to investigate  
basic properties of 
the value function $ v_t $. To this end, we introduce the following technical assumptions. 
Throughout the paper, convexity of a vector-valued mapping is understood componentwise; that is, \(f=(f_{1},\cdots,f_{m})\) is convex if each component \(f_{i}\), $1 \le i \le m$, is convex.

\begin{assumption}[Convexity]
    For \( t = 1, 2, \cdots, T \) and given \( \xi_{[t]} \) and \( \zeta_t \),  
    (a) \( C_t\left(s_t, x_t, \xi_{[t]}, \zeta_t\right) \) is 
    convex w.r.t.~\( \left(s_t, x_t\right) \);  
    (b) \( S^M_t\left(s_t, x_t, \xi_t, \zeta_{t}\right) \) is 
    convex w.r.t.~\( \left(s_t, x_t\right) \);
    (c) \( g_{t, i}\left(s_t, x_t, x_{t-1}, \xi_{[t]}\right) \) is
    convex w.r.t. \( \left(s_t, x_t, x_{t-1}\right) \).  
    \label{Assu:convexity}
\end{assumption}
It should be noted that at each stage $ t = 1, 2, 3, \ldots, T $, the state variable $ s_t $ directly influences feasible decisions at that stage through the constraints $ g_{t,i}(s_t, x_t, x_{t-1}, \xi_{[t]}) \leq 0$, $i \in I_t$. On the other hand, $ s_t $ is directly affected by the decision $ x_{t-1} $ at the previous stage through $ s_t = S^M_{t-1}(s_{t-1}, x_{t-1}, \xi_{t-1}, \zeta_{t-1}) $. Therefore, we need to consider $ s_t $ and $ x_t $  simultaneously. 
It is a standard assumption in the MSP literature that the objective functions and feasible sets are convex. For example, Chapter 6 of \cite{pflug2012distance} by Pflug and Pichler considers an MSP  problem where the objective function is convex in the decision variables, and the feasible set is also convex.
Furthermore, for the characterization of many decision problems including inventory problems, state transitions can often be described as linear mappings (see e.g. \cite{shapiro2021distributionally}). Therefore, the convexity assumption on the state transition mapping is also reasonable here.

\begin{assumption}[Monotonicity w.r.t. state variables]
    For \( t = 1, \cdots, T \) and given \( x_t, x_{t - 1}, \xi_{[t]}\) and \( \zeta_t \),  
    (a) \( C_t\left(s_t, x_t, \xi_{[t]}, \zeta_t\right) \) is non-decreasing in \( s_t \);  
    (b) \( S^M_t\left(s_t, x_t, \xi_t, \zeta_{t}\right) \) is non-decreasing in \( s_t \);  
    (c) \( g_{t, i}\left(s_t, x_t, x_{t-1}, \xi_{[t]}\right) \) is non-decreasing in \( s_t \).  
    \label{assumption4}
\end{assumption}

Monotonicity is a typical premise for convexity verification in many studies on nonlinear models (\cite{bauerle2022distributionally,kircher2019convexity}). Considering the practical meaning of state variables, Assumption \ref{assumption4} is very often automatically satisfied. Take the inventory problem as an example: given other factors, the state variable \( s_t \) at stage \( t \) is clearly monotonically increasing with respect to \( s_{t - 1} \) at stage \( t - 1 \). As for its constraints, the most important one is the warehouse capacity constraint in the form of \( s_t + x_t \leq M \), as that in \eqref{eq:ex-inv-cotl}, which indicates that the warehouse capacity cannot exceed \( M \). It obviously satisfies monotonicity. 
Since the objective function typically represents the sum of holding costs, shortage costs, and purchasing costs, its monotonicity with respect to \( s_{t} \) is also natural.

We know that convexity is a basic setting for optimization theory. 
For this reason, we 
establish the convexity of the value function \( v_t(s_{t-1}, x_{t-1}, \xi_{[t]}, \zeta_{t - 1}), 1 \leq t \leq T \).
\begin{proposition}[Convexity]
    Suppose that Assumptions 
    \ref{Assu:welldefinedness} - \ref{assumption4} hold. 
    Then for \( t = 1, 2, \cdots, T \) and each \( (\xi_{[t]}, \zeta_{t - 1}) \),  
    \( v_t\left(s_{t-1}, x_{t-1}, \xi_{[t]}, \zeta_{t - 1}\right) \)  
    is jointly convex w.r.t. \( \left(s_{t-1}, x_{t-1}\right) \).  
    \label{Theo-convexity}
\end{proposition}
\begin{proof}
    We prove the proposition by backward induction from $t = T$.
    We show that for any \( (s_{T - 1, 1}, x_{T - 1}) \) ,\( (s_{T - 1, 2}, y_{T - 1}) \) and \( \alpha \in [0, 1] \), the following inequality holds:  
    \begin{equation}
    \begin{aligned}
    & \alpha v_T\left(s_{T-1,1}, x_{T-1}, \xi_{[T]}, \zeta_{T - 1}\right)+(1-\alpha) v_T\left(s_{T-1,2}, y_{T-1}, \xi_{[T]}, \zeta_{T - 1}\right) \\
    & \geq v_T\left(\alpha s_{T-1,1}+(1-\alpha) s_{T-1,2}, \alpha x_{T-1}+(1-\alpha) y_{T-1}, \xi_{[T]}, \zeta_{T - 1}\right).
    \end{aligned}
    \label{eq:conv-T1}
    \end{equation}  
    For $t = T$, 
    under given state-decision pairs at the previous stage, the corresponding states at stage $t$ are
    \begin{subequations}
    \begin{eqnarray} 
    s_{T, 1}&= &S^M_{T - 1}\left(s_{T-1,1}, x_{T-1}, \xi_{T-1}, \zeta_{T - 1}\right),  s_{T, 2}=S^M_{T - 1}\left(s_{T-1,2}, y_{T-1}, \xi_{T-1}, \zeta_{T - 1}\right), \\
    s_{T, \alpha}&= &S^M_{T - 1}\left(\alpha s_{T-1,1} + (1-\alpha) s_{T-1,2}, \alpha x_{T-1} + (1-\alpha) y_{T-1}, \xi_{T-1}, \zeta_{T - 1}\right).
    \end{eqnarray}
    \label{eq:conv-s_t}
    \end{subequations}
    Under Assumption \ref{Assu:continuity},
    $\mathcal{X}_T(s_T, x_{T - 1}, \xi_{[T]})$ is a closed subset of $\mathcal{X}_T^0$. This and Assumption \ref{Assu:compactness} ensure $\mathcal{X}_T(s_T, x_{T - 1}, \xi_{[T]})$ is compact. 
    Thus, problem \eqref{eq:value-T} 
    has an optimal solution.
    Existence of the optimal solution to problem \eqref{eq:value-t} at stage $t = 1, 2, \cdots, T - 1$ can be 
    established analogously.  
    Let 
    \begin{eqnarray}
    x_T^* &\in & \underset{x_T \in \mathcal{X}_T(s_{T,1}, x_{T - 1}, \xi_{[T]})}{\arg\min} \mathbb{E}_{\zeta_T} \left[ C_T(s_{T,1}, x_T, \xi_{[T]}, \zeta_T) \right], \nonumber \\
    y_T^* &\in & \underset{y_T \in \mathcal{X}_T(s_{T,2}, y_{T - 1}, \xi_{[T]})}{\arg\min} \mathbb{E}_{\zeta_T} \left[ C_T(s_{T,2}, y_T, \xi_{[T]}, \zeta_T) \right]. \nonumber
    \end{eqnarray}
    By Assumption \ref{Assu:convexity}, we have  
    \begin{subequations}
    \begin{eqnarray}
    && \alpha v_T\left(s_{T-1,1}, x_{T-1}, \xi_{[T]}, \zeta_{T - 1}\right)+(1-\alpha) v_T\left(s_{T-1,2}, y_{T-1}, \xi_{[T]}, \zeta_{T - 1}\right) \nonumber \\
    & =&\alpha \mathbb{E}_{\zeta_T} \left[ C_T\left(s_{T, 1}, x_T^*, \xi_{[T]}, \zeta_T\right) \right] + (1-\alpha) \mathbb{E}_{\zeta_T} \left[ C_T\left(s_{T, 2}, y_T^*, \xi_{[T]}, \zeta_T\right) \right] \nonumber \\
    & \geq& \mathbb{E}_{\zeta_T} \left[ C_T\left(\alpha s_{T, 1}+(1-\alpha) s_{T, 2}, \alpha x_T^*+(1-\alpha) y_T^*, \xi_{[T]}, \zeta_T \right) \right] \quad \textrm{(convexity)}  \\
    &\geq & \mathbb{E}_{\zeta_T} \big[ C_T\left(s_{T, \alpha}, \alpha x_T^*+(1-\alpha) y_T^*, \xi_{[T]}, \zeta_T\right) \big] \quad \textrm{(monotonicity \& Assumption \ref{Assu:convexity}(b))} \nonumber \\
    &\geq & \min _{x_T \in \mathcal{X}_T(s_{T, \alpha}, \alpha x_{T - 1} + (1 - \alpha) y_{T - 1}, \xi_{[T]})} \mathbb{E}_{\zeta_T} \big[ C_T\left(s_{T, \alpha}, x_T, \xi_{[T]}, \zeta_T\right) \big] \nonumber \\
    &= & v_T\left(\alpha s_{T-1,1}+(1-\alpha) s_{T-1,2}, \alpha x_{T-1}+(1-\alpha) y_{T-1}, \xi_{[T]}, \zeta_{T - 1}\right). \label{eq:conv-T2}
    \end{eqnarray}
    \end{subequations}
    The second inequality holds because $\alpha s_{T, 1} + (1 - \alpha) s_{T, 2} \geq s_{T, \alpha}$, the monotonicity of $C_T$ w.r.t. $s_T$ and the convexity of $g_{T, i}$ w.r.t. $(s_T, x_T, x_{T - 1})$, i.e.,
    \begin{eqnarray}
    && g_{T, i} (s_{T, \alpha}, \alpha x_T + (1 - \alpha) y_T, \alpha x_{T-1} + (1 - \alpha) y_{T-1}, \xi_{[T]} ) \nonumber \\
    &\leq & g_{T, i} (\alpha s_{T,1} + (1 - \alpha) s_{T,2}, \alpha x_T + (1 - \alpha) y_T, \alpha x_{T-1} + (1 - \alpha) y_{T-1}, \xi_{[T]} ) \nonumber \\
    &\leq & \alpha g_{T, i} (s_{T,1}, x_{T}, x_{T-1}, \xi_{[T]}) + (1 - \alpha) g_{T, i} (s_{T,2}, y_{T}, y_{T-1}, \xi_{[T]})  \leq  0, i \in I_T.
    \nonumber
    \end{eqnarray}  
    The inequalities above and \eqref{eq:conv-T2} establish the convexity of \( v_T(s_{T-1}, x_{T-1}, \xi_{[T]}, \zeta_{T - 1}) \) with respect to \( (s_{T-1}, x_{T-1}) \).
    
   Next, assume that the convexity
   holds for stage \( t + 1 \). 
   We prove that \( v_t(s_{t-1}, x_{t-1}, \xi_{[t]}, \zeta_{t - 1}) \) is convex. 
   Let $s_{t, 1}, s_{t, 2}, s_{t, \alpha}$ be defined as \eqref{eq:conv-s_t} by using $t$ instead of $T$. First, by the definition of \( v_t \), we have
    \begin{eqnarray}
    && v_x := v_t\left(s_{t-1,1}, x_{t-1}, \xi_{[t]}, \zeta_{t - 1}\right)\nonumber \\
    &= &\min _{x_t \in \mathcal{X}_t(s_{t,1}, x_{t-1}, \xi_{[t]})}\mathbb{E}_{\xi_{t+1} \mid \xi_{[t]}, \zeta_{t}} \left[C_t\left(s_{t, 1}, x_t, \xi_{[t]}, \zeta_t\right) + v_{t+1}\left(s_{t, 1}, x_t, \xi_{[t+1]}, \zeta_{t}\right)\right], \nonumber \\[10pt]
    && v_y := v_t\left(s_{t-1,2}, y_{t-1}, \xi_{[t]}, \zeta_{t - 1}\right) \nonumber \\
    &= &\min _{y_t \in \mathcal{X}_t(s_{t,2}, y_{t-1}, \xi_{[t]})}\mathbb{E}_{\xi_{t+1} \mid \xi_{[t]}, \zeta_{t}} \left[C_t\left(s_{t, 2}, y_t, \xi_{[t]}, \zeta_t\right) + v_{t+1}\left(s_{t, 2}, y_t, \xi_{[t+1]}, \zeta_{t}\right)\right], \nonumber \\[10pt]
    && v_z := v_t\left(\alpha s_{t-1,1}+(1-\alpha) s_{t-1,2}, \alpha x_{t-1}+(1-\alpha) y_{t-1}, \xi_{[t]}, \zeta_{t - 1}\right)  \nonumber \\
    &= &\min _{z_t \in \mathcal{X}_t(s_{t, \alpha}, \alpha x_{t-1} + (1-\alpha)y_{t-1}, \xi_{[t]})} \mathbb{E}_{\xi_{t+1} \mid \xi_{[t]}, \zeta_{t}} \left[C_t\left(s_{t, \alpha}, z_t, \xi_{[t]}, \zeta_t\right) + v_{t+1}\left(s_{t, \alpha}, z_t, \xi_{[t+1]}, \zeta_{t}\right)\right]. \nonumber
    \end{eqnarray}
    It suffices to prove 
    $   \alpha v_x+(1-\alpha) v_y \geq v_z$. 
    To this end, we need to prove that \( v_t \) is also monotonically non-decreasing with respect to \( s_{t-1} \). We show this by using backward induction from stage $T$. At stage \( T \), let \( s_{T-1,2} \geq s_{T-1,1} \). 
    By \eqref{eq:conv-s_t} and the monotonicity of $S^M_{T - 1}$, we have \( s_{T,2} \geq s_{T,1} \). Therefore, by the monotonicity of \( C_T \) w.r.t. $s_T$, we obtain
    $$
    \mathbb{E}_{\zeta_T} \left[ C_T(s_{T,2}, x_{T,2}^*, \xi_{[T]}, \zeta_T) \right] \geq \mathbb{E}_{\zeta_T} \left[ C_T(s_{T,1}, x_{T,2}^*, \xi_{[T]}, \zeta_T) \right] \geq \mathbb{E}_{\zeta_T} \left[ C_T(s_{T,1}, x_{T,1}^*, \xi_{[T]}, \zeta_T) \right].
    $$
    This means the monotonicity of \( v_T \) with respect to \( s_{T - 1} \). Assume that the monotonicity holds for stage \( t + 1 \), i.e., for \( s_{t,2} \geq s_{t,1} \),  we have 
    \begin{eqnarray}
        v_{t+1}(s_{t,2}, x_{t}, \xi_{[t+1]}, \zeta_{t}) \geq v_{t+1}(s_{t,1}, x_{t}, \xi_{[t+1]}, \zeta_{t}). \nonumber
    \end{eqnarray}  
    Then, for \( s_{t-1,2} \geq s_{t-1,1} \),
    \begin{eqnarray}
             &&v_t(s_{t-1,2}, x_{t-1}, \xi_{[t]}, \zeta_{t - 1}) \nonumber \\
             &= &  \mathbb{E}_{\xi_{t+1} | \xi_{[t]}, \zeta_{t}} \left[C_t(s_{t,2}, x_{t,2}^*, \xi_{[t]},\zeta_t)   + v_{t+1}(s_{t,2}, x_{t,2}^*, \xi_{[t+1]},\zeta_{t})\right] \nonumber \\
             &\geq & \mathbb{E}_{\xi_{t+1} | \xi_{[t]}, \zeta_{t}} \left[C_t(s_{t,1}, x_{t,2}^*, \xi_{[t]},\zeta_t) + v_{t+1}(s_{t,1}, x_{t,2}^*, \xi_{[t+1]},\zeta_{t})\right] \nonumber \\
             &\geq & \mathbb{E}_{\xi_{t+1} | \xi_{[t]}, \zeta_{t}} \left[C_t(s_{t,1}, x_{t,1}^*, \xi_{[t]},\zeta_t) + v_{t+1}(s_{t,1}, x_{t,1}^*, \xi_{[t+1]},\zeta_{t})\right] \nonumber \\
             &= & v_{t} (s_{t-1,1}, x_{t-1}, \xi_{[t]}, \zeta_{t - 1}). \nonumber
    \end{eqnarray}  
    Thus, the proof of the monotonicity of \( v_t \) with respect to \( s_{t - 1} \) is completed.

    Now, we return to the proof of convexity of $v_t$  
    \begin{eqnarray}
    && \alpha v_x+(1-\alpha) v_y \nonumber \\
    &= &\mathbb{E}_{\xi_{t+1} \mid \xi_{[t]}, \zeta_{t}}\big[ \alpha C_t\left(s_{t, 1}, x_t^*, \xi_{[t]}, \zeta_t\right)+(1-\alpha) C_t\left(s_{t, 2}, y_t^*, \xi_{[t]}, \zeta_t\right) \nonumber \\
    && + \alpha v_{t+1}\left(s_{t, 1}, x_t^*, \xi_{[t+1]}, \zeta_{t}\right)+(1-\alpha) v_{t+1}\left(s_{t, 2}, y_t^*, \xi_{[t+1]}, \zeta_{t}\right)\big] \nonumber \\
    &\geq& \mathbb{E}_{\xi_{t+1} \mid \xi_{[t]}, \zeta_{t}}\big[ C_t\left(\alpha s_{t, 1} + (1 - \alpha) s_{t , 2}, \alpha x_t^*+(1-\alpha) y_t^*, \xi_{[t]}, \zeta_t\right) \nonumber \\
    && +v_{t+1}\left(\alpha s_{t, 1} + (1 - \alpha) s_{t , 2}, \alpha x_t^*+(1-\alpha) y_t^*, \xi_{[t+1]}, \zeta_{t}\right)\big] \quad \textrm{(convexity)} \nonumber \\
    & \geq& \mathbb{E}_{\xi_{t+1} \mid \xi_{[t]}, \zeta_{t}}\big[ C_t\left(s_{t, \alpha}, \alpha x_t^*+(1-\alpha) y_t^*, \xi_{[t]}, \zeta_t\right) \nonumber \\
    && +v_{t+1}\left(s_{t, \alpha}, \alpha x_t^*+(1-\alpha) y_t^*, \xi_{[t+1]}, \zeta_{t}\right)\big] \quad \textrm{(monotonicity \& Assumption \ref{Assu:convexity}(b))}\nonumber \\
    & \geq& \min _{z_t \in \mathcal{X}_t(s_{t, \alpha}, \alpha x_{t - 1} + (1 - \alpha) y_{t - 1}, \xi_{[t]})}\left\{\mathbb{E}_{\xi_{t+1} \mid \xi_{[t]}, \zeta_{t}} \left[C_t\left(s_{t, \alpha}, z_t, \xi_{[t]}, \zeta_t\right) + v_{t+1}\left(s_{t, \alpha}, z_t, \xi_{[t+1]}, \zeta_{t}\right)\right]\right\} \nonumber \\
    & =&v_z, \nonumber
    \end{eqnarray}
    where the first inequality follows from the convexity of \( C_t \) and \( v_{t+1} \) with respect to \( (s_t, x_t) \) and the monotonicity of \( C_t \). The proof is completed.  
\end{proof}
Due to the nonlinear state 
transition mapping, the state $s_{t, \alpha}$ at the current stage under the convex combination   $(\alpha s_{t - 1, 1} + (1 - \alpha) s_{t - 1, 2}, \alpha x_{t - 1, 1} + (1 - \alpha) x_{t - 1, 2})$ of the previous stage's state-decision pairs cannot be expressed as a convex combination of $s_{t, 1}$ and $s_{t, 2}$. Consequently, we need Assumption \ref{assumption4} to  guarantee the convexity of $v_t$ with respect to $(s_{t - 1}, x_{t - 1})$.
In fact, in the proof of Proposition \ref{Theo-convexity}, the condition we really need is that the feasible solution set at each stage is jointly convex with respect to the $(s_t, x_{t - 1})$. To ensure that the feasible set defined by the inequality constraints \( g_{t, i}(\cdot) \leq 0, i \in I_t \) is jointly convex with respect to \( (s_t, x_t) \), we require that \( g_{t, i}\left(s_t, x_t, x_{t-1}, \xi_{[t]}\right) \) is jointly convex with respect to \( (s_t, x_t, x_{t-1}) \) and is monotonically non-decreasing with respect to \( s_t \). Of course, there are other conditions that can also guarantee the convexity of the feasible solution set; as shown in \cite{shapiro2021distributionally}, when \( S^M_t(\cdot) \) is linear with respect to \( (s_t, x_t) \), the monotonicity requirements of $C_t(\cdot)$ and $g_{t, i}(\cdot), i \in I_t, 1 \leq t \leq T,$ with respect to \( s_t \) are not necessary. In contrast, our proof ensures the convexity of the integrated MSP-MDP model \eqref{eq:mixed-MSP-MDP} for general nonlinear state transition mappings.

With the above convexity and 
continuity of the integrated MSP-MDP model, we can examine the existence and global optimality of optimal solutions. These properties are important for the qualitative analysis of problem \eqref{eq:mixed-MSP-MDP}, but they are not enough for quantitative analysis. 
For the latter, we further need the Lipschitz continuity of the value function \( v_t, 1 \leq t \leq T \). The following additional   assumptions are needed to establish the Lipschitz continuity.

\begin{assumption}[Slater condition]
    There exists a positive constant \( \rho \) 
    and \( \bar{x}_t \in \mathcal{X}_t\left(s_t, x_{t-1}, \xi_{[t]}\right) \) such that  
    \begin{eqnarray}
    g_{t, i}\left(s_t, \bar{x}_t, x_{t-1}, \xi_{[t]}\right) \leq-\rho, 
    \forall s_t, x_{t-1}, \xi_{[t]}
    \label{eq:assu-Slater}
    \end{eqnarray}  
    for $i \in I_t$, 
    $t = 1, 2, \cdots, T$.
    \label{assumption7}
\end{assumption}

The assumption ensures that 
problems \eqref{eq:value-T} and \eqref{eq:value-t}
satisfy the Slater condition.
Condition \eqref{eq:assu-Slater}
requires the inequality constraint
at each stage 
has a strictly feasible solution with 
residual bounded by a time independent constant $-\rho$.
Without such a negative constant margin, the feasible set may be highly sensitive
to arbitrarily small perturbations of the problem data. In particular, a feasible
problem may become infeasible after a small perturbation. Conversely, a degenerate
feasible set, such as a singleton, may also be perturbed into a nondegenerate
feasible set with nonempty interior.
We use a simple example 
to illustrate this.
Consider
\[
    v(\xi)=\min_{x\in[0,1]\subset \R} x
    \quad \textrm{s.t.} \quad
    1-\xi x\leq 0.
\]
When \(P(\xi=1) = 1\), the feasible set is
$
    \mathcal{X}(\xi)=\{1\},
$
in which case 
the Slater condition fails.
Consider an arbitrarily small perturbation $\epsilon$. 
Let 
$    \tilde \xi=1-\varepsilon,  \varepsilon>0,
$ be a perturbation from $\xi=1$. 
Then the constraint becomes
$
    (1-\varepsilon)x\geq 1,
$
which requires \(x\geq 1/(1-\varepsilon)>1\). This contradicts the constraint
\(x\leq 1\), and hence the perturbed problem becomes infeasible. On the other
hand, if
$
    \tilde \xi=1+\varepsilon,
$
then the feasible set becomes
\(
    \mathcal{X}(\tilde{\xi})
    =
    \left[\frac{1}{1+\varepsilon},\,1\right],
\)
which 
has nonempty interior. 
This example
shows that, without a uniform strict feasibility margin, the 
set of feasible solutions
may change abruptly under small perturbation of $\xi$. 

 \begin{assumption}[Lipschitz continuity]
    For $t = 1, 2, \cdots, T$, 
    
    \noindent $\mathrm{(C)}$ \( C_t(s_t, x_t, \xi_{[t]}, \zeta_t) \) is Lipschitz continuous in \( (s_t, x_t) \) with Lipschitz modulus \( L_{C,t} \), i.e., 
    \begin{eqnarray*}
    && | C_t(s_{t,1}, x_{t,1}, \xi_{[t]}, \zeta_t) - C_t(s_{t,2}, x_{t,2}, \xi_{[t]}, \zeta_t) | \nonumber \\
    &\leq & L_{C,t} (\Vert s_{t,1} - s_{t,2} \Vert + \Vert x_{t,1} - x_{t,2} \Vert), \quad \forall (s_{t, 1}, x_{t, 1}),(s_{t, 2}, x_{t, 2});
    \end{eqnarray*}
    $\mathrm{(S)}$ \( S^M_{t - 1}(s_{t - 1}, x_{t - 1}, \xi_{t - 1}, \zeta_{t - 1}) \) is Lipschitz continuous in \( (s_{t - 1}, x_{t - 1}) \) with Lipschitz modulus \( L_{S, t - 1} \), i.e.,
    \begin{eqnarray*}
    && \Vert S^M_{t - 1}(s_{t - 1,1}, x_{t - 1,1}, \xi_{t - 1}, \zeta_{t - 1}) - S^M_{t - 1}(s_{t - 1,2}, x_{t - 1,2}, \xi_{t - 1}, \zeta_{t - 1}) \Vert \nonumber \\
    &\leq & L_{S, t - 1} (\Vert s_{t - 1,1} - s_{t - 1,2} \Vert + \Vert x_{t - 1,1} - x_{t - 1,2} \Vert), \quad \forall (s_{t - 1, 1}, x_{t - 1, 1}), (s_{t - 1, 2}, x_{t - 1, 2});
    \end{eqnarray*}   
    $\mathrm{(G)}$ \( g_{t, i}\left(s_t, x_t, x_{t-1}, \xi_{[t]}\right), i \in I_t\) is Lipschitz continuous in \( (s_t, x_{t-1}) \), i.e., 
    \begin{eqnarray*}
    && | g_{t,i}(s_{t,1}, x_{t}, x_{t-1,1}, \xi_{[t]}) - g_{t,i}(s_{t,2}, x_{t}, x_{t-1,2}, \xi_{[t]}) | \nonumber \\
    &\leq & L_{g, t} (\Vert s_{t,1} - s_{t,2} \Vert + \Vert x_{t-1,1} - x_{t-1,2} \Vert), \quad \forall (s_{t, 1}, x_{t - 1, 1}), (s_{t, 2}, x_{t - 1, 2}).
    \end{eqnarray*}
    \label{assumption-lip}
Let $L_S := \underset{t = 0, 1, \cdots, T - 1}{\max}\{L_{S, t}\}$, $L_C := \underset{t = 1, 2, \cdots, T}{\max}\{L_{C, t}\}$, and $L_{g} := \underset{t = 1, 2, \cdots, T}{\max} L_{g, t}$. 
\end{assumption}

Lipschitz continuity ensures that the value function would not change drastically within its domain. This is key to guaranteeing the quantitative stability of the integrated MSP-MDP model
with respect to small perturbations in endogenous randomness or exogenous randomness and their distributions. As for the Lipschitz continuity of \( v_t \) with respect to \( (s_{t - 1}, x_{t - 1}), 1 \leq t \leq T \), we have:

\begin{theorem}[Lipschitz continuity of the value function]
    Let  
    Assumptions \ref{Assu:welldefinedness} - \ref{Assu:compactness},
    \ref{assumption7} and \ref{assumption-lip} hold and for $1 \leq t \leq T, i \in I_t$ \( g_{t, i}\left(s_t, x_t, x_{t-1}, \xi_{[t]}\right) \) is convex in \( x_t \).
    Then, \( v_{t}\left(s_{t-1}, x_{t-1}, \xi_{[t]}, \zeta_{t - 1}\right) \) is Lipschitz continuous with respect to \( (s_{t-1}, x_{t-1}) \), i.e., there exists a constant \( L_{t} > 0 \) such that  
    \begin{eqnarray}
    &&| v_t(s_{t-1,1}, x_{t-1,1}, \xi_{[t]}, \zeta_{t - 1}) - v_t(s_{t-1,2}, x_{t-1,2}, \xi_{[t]}, \zeta_{t - 1}) | \nonumber \\
    &\leq & L_t (\Vert s_{t-1,1} - s_{t-1,2} \Vert + \Vert x_{t-1,1} - x_{t-1,2} \Vert), \label{eq:Lipschitz-v_t}
    \end{eqnarray}  
    where  
    $
    L_T := L_{C,T} (L_S + L_{X,T} + L_{X,T} L_S),
    $  
    and for \( t = 1, 2, \cdots, T-1 \),  
    \begin{eqnarray}
    L_t := (L_{C,t}+L_{t+1})
    \left(L_S+L_{X,t}+L_{X,t}L_S\right),
    \label{eq:Lip-Lt}
    \end{eqnarray}
\( L_{X,t} := \frac{A}{\rho} L_{g, t} \), 
\( A \) is the maximum of the diameters of the feasible sets \( \mathcal{X}_t(s_t, x_{t-1}, \xi_{[t]}), 1 \leq t \leq T \).  
    \label{Theo-Lipschitz}
\end{theorem}

\begin{proof}
    Let 
    \( g_t(\cdot) = (g_{t, 1}(\cdot), \cdots, g_{t,|I_t|}(\cdot))^{\top} : \mathbb{R}^{\hat{n}_t} \times \mathbb{R}^{n_t} \times \mathbb{R}^{n_{t - 1}} \times \mathbb{R}^{m_{1, [t]}} \to \mathbb{R}^{|I_t|} \).
    For any 
    \( y_t \in \mathbb{R}^{n_t} \), define
    $$
  \gamma := \Vert g_t(s_{t,1}, y_t, x_{t-1,1}, \xi_{[t]})_+ \Vert 
    $$  
    and 
    \( z_t := \frac{\gamma}{\rho+\gamma} \bar{x}_t + \frac{\rho}{\rho+\gamma} y_t \). 
    By the convexity of $g_t(s_t, x_t, x_{t - 1}, \xi_{[t]})$ in $x_t$, 
    \begin{eqnarray}
    g_{t,i}(s_{t,1}, z_t, x_{t-1,1}, \xi_{[t]}) &\leq& \frac{\gamma}{\rho+\gamma} g_{t,i}(s_{t,1}, \bar{x}_t, x_{t-1,1}, \xi_{[t]}) + \frac{\rho}{\rho+\gamma} g_{t,i}(s_{t,1}, y_t, x_{t-1,1}, \xi_{[t]}) \nonumber \\
    &\leq& - \frac{\gamma}{\rho+\gamma} \rho + \frac{\rho}{\rho+\gamma} \gamma = 0, i \in I_t, \nonumber
    \end{eqnarray}
    which implies \( z_t \in \mathcal{X}_t(s_{t,1}, x_{t-1,1}, \xi_{[t]}) \). Consequently, we have  
    \begin{eqnarray}
    d\left(y_t, \mathcal{X}_t(s_{t,1}, x_{t-1,1}, \xi_{[t]})\right) &\leq& \Vert y_t - z_t \Vert = \frac{\gamma}{\rho} \Vert z_t - \bar{x}_t \Vert. \label{eq:Lipschitz-dyx1}
    \end{eqnarray}  
    For \( t=1,2,\cdots,T \), by Assumption \ref{Assu:compactness}, \( \mathcal{X}_t(s_{t,1}, x_{t-1,1}, \xi_{[t]}) \) is bounded. 
    Let  $
    A_t 
    $  
    denote the diameter of the set $\mathcal{X}_t(s_{t, 1}, x_{t - 1, 1}, \xi_{[t]})$ and
    \( A = \underset{1 \leq t \leq T}{\max} A_t \). It 
    follows from
    \eqref{eq:Lipschitz-dyx1} that  
    \begin{equation}
    d(y_t, \mathcal{X}_t(s_{t, 1}, x_{t - 1, 1}, \xi_{[t]})) \leq \frac{\gamma A}{\rho}. \label{eq:Lipschitz-dyx2}
    \end{equation}
    For any \( x_{t,2} \in \mathcal{X}_t(s_{t, 2}, x_{t-1, 2}, \xi_{[t]}) \), 
    by \eqref{eq:Lipschitz-dyx2} and Assumption \ref{Assu:compactness}, we have  
    \begin{eqnarray}
     d(x_{t, 2}, \mathcal{X}_t(s_{t, 1}, x_{t-1, 1}, \xi_{[t]})) 
    &\leq& \frac{A}{\rho} \Vert g_t(s_{t, 1}, x_{t, 2}, x_{t - 1, 1}, \xi_{[t]})_+ \Vert \nonumber \\
    &\leq& \frac{A}{\rho} \Vert g_t(s_{t, 1}, x_{t, 2}, x_{t - 1, 1}, \xi_{[t]})_+ - g_t(s_{t, 2}, x_{t, 2}, x_{t - 1, 2}, \xi_{[t]})_+ \Vert \nonumber \\
    &\leq& \frac{A}{\rho} \Vert g_t(s_{t, 1}, x_{t, 2}, x_{t - 1, 1}, \xi_{[t]}) - g_t(s_{t, 2}, x_{t, 2}, x_{t - 1, 2}, \xi_{[t]}) \Vert \nonumber \\
    &\leq& \frac{A}{\rho} L_{g, t} (\Vert s_{t,1} - s_{t,2} \Vert + \Vert x_{t-1, 1} - x_{t-1 ,2} \Vert). \nonumber
    \end{eqnarray}  
    Since \( x_{t, 2} \) is arbitrarily chosen from \( \mathcal{X}_t(s_{t, 2}, x_{t - 1, 2}, \xi_{[t]}) \), it follows that  
    $$
    \mathbb{D}(\mathcal{X}_t(s_{t, 2}, x_{t-1, 2}, \xi_{[t]}), \mathcal{X}_t(s_{t, 1}, x_{t-1, 1}, \xi_{[t]})) \leq \frac{A}{\rho} L_{g, t} (\Vert s_{t,1} - s_{t,2} \Vert + \Vert x_{t-1, 1} - x_{t-1 ,2} \Vert),
    $$
    where $\mathbb{D}(A, B)$ denotes the deviation from $A$ to $B$. 
    The conclusion above remains valid if we swap \( x_{t - 1, 1} \) and \( x_{t - 1, 2} \) in the feasible sets. Therefore, we obtain  
    \begin{eqnarray} 
    \mathbb{H}(\mathcal{X}_t(s_{t, 1}, x_{t-1, 1}, \xi_{[t]}), \mathcal{X}_t(s_{t, 2}, x_{t-1, 2}, \xi_{[t]})) \leq L_{X,t} (\Vert s_{t,1} - s_{t,2} \Vert + \Vert x_{t-1, 1} - x_{t-1 ,2} \Vert),
    \end{eqnarray}   
    where \( L_{X,t} = \frac{A}{\rho} L_{g, t} \). Thus, we have proven that the feasible solution set at stage \( t (1 \leq t \leq T) \) satisfies the Lipschitz property under the Hausdorff distance.

    Now, we establish the Lipschitz continuity of \( v_T \). Since Assumption \ref{assumption-lip}
    implies the continuity of $C_t$ and 
    $g_{t, i}$,
    it is clear that problem \eqref{eq:value-T} at stage $T$ has at least one optimal solution. Similarly, there also exists at least one optimal solution to problem \eqref{eq:value-t} at stage $t = 1, 2, \cdots, T - 1$. 
    Let \( x_{T,1}^* \) and \( x_{T,2}^* \) satisfy 
    \begin{eqnarray}
    && x_{T, 1}^* \in \underset{x_{T, 1} \in \mathcal{X}_T(s_{T, 1}, x_{T - 1, 1}, \xi_{[T]})}{\text{arg min}} C_T(s_{T,1}, x_{T,1}, \xi_{[T]}, \zeta_T), \nonumber \\
    && x_{T, 2}^* \in \underset{x_{T, 2} \in \mathcal{X}_T(s_{T, 2}, x_{T - 1, 2}, \xi_{[T]})}{\text{arg min}} C_T(s_{T,2}, x_{T,2}, \xi_{[T]}, \zeta_T). \nonumber 
    \end{eqnarray}  
    Then
    \begin{eqnarray}
&& v_{T}\left(s_{T-1,1}, x_{T-1,1}, \xi_{[T ]}, \zeta_{T - 1}\right) - v_{T}\left(s_{T-1,2}, x_{T-1,2}, \xi_{[T]}, \zeta_{T - 1}\right) \nonumber \\
&= & \mathbb{E}_{\zeta_T} \left[ C_T(s_{T,1}, x^*_{T,1}, \xi_{[T]}, \zeta_T) - C_T(s_{T,2}, x^*_{T,2}, \xi_{[T]}, \zeta_T) \right]  \nonumber   \\
&\leq & \mathbb{E}_{\zeta_T} \left[ C_T(s_{T,1}, x_{T,1}, \xi_{[T]}, \zeta_T) - C_T(s_{T,2}, x^*_{T,2}, \xi_{[T]}, \zeta_T) \right]  \nonumber \\
&\leq & L_{C,T} (\Vert s_{T,1}-s_{T,2} \Vert + \mathbb{H}(\mathcal{X}_T(s_{T,1}, x_{T-1,1}, \xi_{[T]}),\mathcal{X}_T(s_{T,2}, x_{T-1,2}, \xi_{[T]}))) \nonumber \\
&\leq & L_{C,T}(L_S (\Vert x_{T-1,1}-x_{T-1,2} \Vert + \Vert s_{T-1,1}-s_{T-1,2} \Vert) \nonumber \\
&& + L_{X,T} (\Vert x_{T-1,1}-x_{T-1,2} \Vert + \Vert s_{T,1}-s_{T,2} \Vert)) \nonumber \\
&\leq & L_{C,T}(L_S + L_{X,T} + L_{X,T}L_S) (\Vert x_{T-1,1}-x_{T-1,2} \Vert + \Vert s_{T-1,1}-s_{T-1,2} \Vert), \nonumber
\end{eqnarray}
where \( x_{T, 1} \) is 
the projection of \( x_{T, 2}^* \) onto \( \mathcal{X}_T\left(s_{T, 1}, x_{T-1,1}, \xi_{[T]}\right) \). The first inequality is obtained from the definition of the optimal solution. The second inequality is due to the Lipschitz property of $C_T$ in $(s_{T, 1}, x_{T, 1})$ and
\( \Vert x_{T,1} - x_{T,2}^* \Vert \leq \mathbb{H}(\mathcal{X}_T(s_{T,1}, x_{T-1,1}, \xi_{[T]}), \mathcal{X}_T(s_{T,2}, x_{T-1,2}, \xi_{[T]})) \). Other inequalities follow directly from the assumptions and the Lipschitz continuity of the feasible set.

By swapping \( \left(s_{T-1,1}, x_{T-1,1}\right) \) and \( \left(s_{T-1,2}, x_{T-1,2}\right) \), it is easy to see that the conclusion above still holds. Therefore, we conclude that \( v_T(s_{T-1}, x_{T-1}, \xi_{[T]}, \zeta_{T - 1}) \) is Lipschitz continuous with respect to \( (s_{T - 1}, x_{T - 1}) \) with a Lipschitz modulus of \( L_T := L_{C,T}(L_S + L_{X,T} + L_{X,T}L_S) \).

For any  \( 1 \leq t \leq T - 1 \), since both \( C_{t}(s_t, x_t, \xi_{[t]}, \zeta_t) \) and \( \mathbb{E}_{\xi_{t+1} \mid \xi_{[t]}} \left[v_{t+1} (s_t, x_t, \xi_{[t+1]}, \zeta_{t})\right] \) are Lipschitz continuous with respect to \( (s_t, x_t) \), the specific argument for stage \( t \) is then the same as the proof for the stage \( T \) above. By the principle of induction, \( v_t(s_{t-1}, x_{t-1}, \xi_{[t]}, \zeta_{t - 1}) \) is then Lipschitz continuous with respect to \( (s_{t - 1}, x_{t - 1}) \) with the Lipschitz modulus being \( L_{C,t} + L_{t+1} \), and we have:  
$$
L_t := (L_{C,t} + L_{t+1}) (L_S + L_{X,t} + L_{X,t}L_S).
$$  
Finally, the proof by induction is completed.
\end{proof}

Theorem~\ref{Theo-Lipschitz} integrates the results in \cite{tao2025generalized} and \cite{hinderer2005lipschitz}. The former derives the Lipschitz properties of the feasible set of an one-stage parametric  stochastic programming problem 
under Slater condition.
The latter establishes  
the Lipschitz continuity of the value function for standard MDPs. 
We extend these results by considering an integrated  model where the complex constraints
depend not only on the state variable $ s_t $, but also on the decision $ x_{t - 1} $ from the previous stage. The Lipschitz conditions given in Theorem \ref{Theo-Lipschitz} are commonly used in studies, as seen in \cite{hinderer2005lipschitz, dentcheva2021subregular}.

It should be pointed out that, unless otherwise specified, \( \Vert \cdot \Vert \) in this paper denotes the infinity norm. By now, we have investigated the structural properties of problems \eqref{eq:value-T}, \eqref{eq:value-t} and \eqref{eq:value-0} under 
mild assumptions. Thanks to the time-consistency, these properties about $v_t$ not only help us to deeply understand the behavior of the proposed integrated MSP-MDP model under different conditions, but also ensure the stability 
of the solutions derived from it.

In practice, due to the complexity of the uncertain environment, the observed values and distributions of endogenous and exogenous random variables may contain errors. Additionally, as the dynamic decision-making environment constantly changes, the values and distributions of these two types of random variables may also vary. The impact of these facts on the solution of the integrated MSP-MDP model can be attributed to the quantitative stability analysis of problem \eqref{eq:mixed-MSP-MDP} with respect to changes in endogenous and exogenous random variables. Therefore, in the next section, we will analyze the stability of the optimal value and set of optimal solutions of problem \eqref{eq:mixed-MSP-MDP} with respect to changes in the two types of random variables and their distributions. 

\section{Quantitative stability with respect to endogenous uncertainty}

In this section and the next section, we will establish the quantitative stability of the integrated MSP-MDP model \eqref{eq:mixed-MSP-MDP}
in terms of the optimal values and the optimal solutions w.r.t.~perturbation of the underlying uncertainties.
Based on the structure of the integrated MSP-MDP model, we  assume without much loss of generality that the randomness from the endogenous system (\( \zeta_t \)) and the randomness from exogenous sources (\( \xi_{[t]} \)) are independent. 
This 
prompts 
us to  study the 
stability w.r.t.~two sources of uncertainties separately (in the next section for the latter).

In both cases, we will
use Monge-Kantorovich functional
to measure perturbation of probability distributions.
Specifically, let $\Xi$ denote the support set of $\xi_{[T]}$. We consider 
the Kantorovich metric
\begin{equation}\label{eq:FMp}
    \dd_{\mathrm{K}}(\mu_1, \mu_2) := \sup_{f \in \mathcal{F}(\Xi)} \left( \int_{\Xi} f(\xi) \mu_1(\mathrm{d}\xi) - \int_{\Xi} f(\xi) \mu_2(\mathrm{d}\xi) \right),
\end{equation}
where $\mathcal{F}(\Xi)$ 
denotes the set of functions satisfying 
    $| f(\xi) - f(\tilde{\xi}) | \leq \|\xi - \tilde{\xi}\|$,
see \cite{rachev2002quantitative,romisch2003stability} for details.
We will also use a slightly different ``metric'' to measure the discrepancy between \(\mu_1\) and \(\mu_2\). Let \(\Pi(\mu_1,\mu_2)\) denote the set of all joint probability measures over \(\Xi \times \Xi\) with marginals \(\mu_1\) and \(\mu_2\) that respect the stagewise information of both processes: at each stage, the coupling uses only the information observed up to that stage. This construction is closely related to the nested distance of Pflug and Pichler~\cite{pflug2012distance} and coincides with the class of bicausal couplings in \cite[Definition~2.1 and Proposition~5.1]{backhoff2017causal}. Define
\begin{eqnarray}
\label{eq:weighted_transport}
    \Delta_{p}(\mu_1, \mu_2)
    :=
    \inf_{\pi\in\Pi(\mu_1, \mu_2)}
    \int_{\Xi\times\Xi}
    \max\left\{
        1,\|\xi\|^{p-1},\|\widetilde\xi\|^{p-1}
    \right\}
    \|\xi-\widetilde\xi\|
    \,\pi(\mathrm d\xi,\mathrm d\widetilde\xi).
\end{eqnarray}
This is a specific form of Monge--Kantorovich functional, see \cite{dupacova2003scenario,heitsch2007note,romisch2003stability} for general forms.
In the case when \(p=1\),  \(\Delta_1(\mu_1,\mu_2)\) reduces to the Kantorovich metric \(\dd_K(\mu_1, \mu_2)\).
Since
\(
\max\left\{1,\|\xi\|^{p-1},\|\widetilde\xi\|^{p-1}\right\}\geq 1,
\)
then
\(
    \dd_K(\mu_1,\mu_2)\leq \Delta_p(\mu_1,\mu_2),
\) which means the latter dominates the former in the sense that $\Delta_p(\mu_1, \mu_2) \to 0$ implies that $\dd_K(\mu_1, \mu_2) \to 0$.
On the other hand, if \(\mu_1,\mu_2\in\mathcal P_p(\Xi)\), then H\"older's inequality implies that
\begin{equation}
\label{eq:Delta_controlled_by_Wp}
    \Delta_p(\mu_1,\mu_2)
    \leq
    \left[
        1+
        \left(\int_{\Xi}\|\xi\|^p \mu_1(\mathrm d\xi)\right)^{\frac{p-1}{p}}
        +
        \left(\int_{\Xi}\|\xi\|^p \mu_2(\mathrm d\xi)\right)^{\frac{p-1}{p}}
    \right]
    W_p(\mu_1,\mu_2),
\end{equation}
where \(\mathcal P_p(\Xi)\) denotes the set of all probability measures on
\(\Xi\) with finite \(p\)-th moment, and $W_p$ represents the $p$-th order Wasserstein metric.
Inequality \eqref{eq:Delta_controlled_by_Wp} means that \(W_p\) dominates \(\Delta_p\).
More generally, we may consider a nonnegative continuous function
\(L:\Xi\times\Xi\to\mathbb R_+\) and
define
\begin{eqnarray}
    \Delta_{L}(\mu_1, \mu_2) := \inf_{\pi \in \Pi} \int_{\Xi\times\Xi}
    L(\xi, \tilde{\xi})
    \|\xi-\widetilde\xi\|
    \,\pi(\mathrm d\xi,\mathrm d\widetilde\xi).
    \label{eq:def-DeltaL}
\end{eqnarray}
By convention, we call
the joint probability distribution attaining the infimum at the rhs of \eqref{eq:def-DeltaL}, written $\pi^*(\mu_1, \mu_2)$, the optimal coupling of $\mu_1$ and $\mu_2$.
By Kantorovich--Rubinstein duality, 
$\Delta_L(\mu_1, \mu_2)$ can be equivalently written as
\begin{equation}\label{eq:FMp-dual}
\begin{aligned}
    \Delta_{L}(\mu_1,\mu_2)
    :=
    \sup_{(f_1,f_2)\in\mathcal G_L}
    \left\{
        \int_{\Xi} f_1(\xi)\mu_1(d\xi)
        -
        \int_{\Xi} f_2(\xi)\mu_2(d\xi)
    \right\},
\end{aligned}
\end{equation}
where
\[
    \mathcal G_L
    :=
    \left\{
    (f_1,f_2):
    f_1(\xi)-f_2(\widetilde{\xi})
    \leq
    L(\xi, \tilde{\xi}) \|\xi - \tilde{\xi}\|,
    \quad
    \forall \xi,\widetilde{\xi}\in\Xi
    \right\},
\]
see ~\cite{santambrogio2015optimal, villani2009optimal}.
Note that when $L(\xi, \tilde{\xi})$ is symmetric in $\xi$ and $\tilde{\xi}$, i.e., $L(\xi, \tilde{\xi}) = L(\tilde\xi, {\xi})$, 
we have $\Delta_L(\mu_1, \mu_2) = \Delta_L(\mu_2, \mu_1)$.
In the forthcoming discussions, we will use both \eqref{eq:def-DeltaL} and \eqref{eq:FMp-dual}.

We are now ready to  investigate the 
effects of perturbation of endogenous uncertainty on the optimal value and the optimal policy of problem \eqref{eq:mixed-MSP-MDP}.
Let $\mathscr{P}(\R^{m_{2, t}})$
denote the set of all probability measures 
on $\R^{m_{2, t}}$.
In the rest of the paper, we write $P_t\in \mathscr{P}(\R^{m_{2, t}})$
for the true probability distribution of $\zeta_t$ and 
$\tilde{P}_t
 \in \mathscr{P}(\R^{m_{2, t}})$
for its perturbation.
Let $P := P_0 \otimes P_1\otimes \ldots\otimes P_T$ and
$\tilde{P} := \tilde{P}_0 \otimes \tilde{P}_1 \otimes \ldots \otimes \tilde{P}_T$.
Consequently, we 
write $\tilde{\zeta}_t$ 
for the perturbation of $\zeta_t$
which is
a random variable mapping from $(\Omega_2, \mathcal{G}, \mathbb{P}^2)$ to $\mathbb{R}^{m_{2, t}}$ 
with distribution $\tilde{P}_t$.
Let $\tilde{\zeta} = \left\{ \tilde{\zeta}_0, \tilde{\zeta}_1, \cdots, \tilde{\zeta}_T\right\}$ be a perturbation of
$\zeta = \{ \zeta_0, \zeta_1, \cdots, \zeta_T\}$.
We consider
\begin{eqnarray} 
        \vt(P)  = \underset{x_0 \in \mathcal{X}_0}{\min} \mathbb{E}_{\zeta_0} \left[ C_0(s_0, x_0, \zeta_0) + \mathbb{E}_{\xi_1} \left[ v_1(s_0, x_0, \xi_1, \zeta_0) \right] \right]
        \label{eq:v-zeta}
        \end{eqnarray} 
and its perturbation
    \begin{eqnarray}    
        \vt(\tilde{P}) = \underset{x_0 \in \mathcal{X}_0}{\min} \mathbb{E}_{\tilde{\zeta}_0} \left[ C_0(s_0, x_0, \tilde{\zeta}_0) + \mathbb{E}_{\xi_1} \left[ \tilde{v}_1(s_0, x_0, \xi_1, \tilde{\zeta}_0) \right] \right]. 
        \label{eq:v-zeta-ptb}
    \end{eqnarray}
    To facilitate the stability analysis in this subsection, we
    write $\mathcal{X}(\zeta)$
    for the set of feasible policies to problem \eqref{eq:v-zeta} to emphasize that the policy is induced by $\zeta$,
    instead of ${\cal X}$ 
    as in 
    the follow-up of \eqref{eq:feaset-X_t}. Likewise, we use 
    $\mathcal{X}(\tilde{\zeta})$ to  denote the set of feasible policies to problem 
    \eqref{eq:v-zeta-ptb}.

Endogenous uncertainty in model \eqref{eq:mixed-MSP-MDP} 
is mainly concerned with the stochastic state transition within the system 
which 
affects the objective function at the current stage. 
Change of the distributions of \( \zeta_t, 0 \leq t \leq T \) 
will  
affect the subsequent system transition paths 
and the values of the objective functions 
at later stages.
Moreover, a small perturbation 
of 
the endogenous uncertainty at a specific stage may be cumulated over subsequent stages, 
leading to cumulative errors. The
cumulative effect is 
significant in large scale or long-term problems. 
Such perturbation may arise from problem data and/or 
numerical computation via dynamic 
recursive formulations 
\eqref{eq:value-T}-
\eqref{eq:value-0}.
Thus, it 
will be instrumental to quantify the overall effect
of the perturbations on the optimal values and optimal solutions.
As 
discussed 
earlier, this kind of research essentially corresponds to the stability analysis of optimal value functions and policies in MDPs under
perturbations in the transition probability kernel or the  transition probability distribution in the continuous case. 
As far as we know, the only research in this regard is conducted by Z\"ahle et al.~\cite{kern2020first},  
who carried out first-order sensitivity analysis of MDPs w.r.t. the state transition kernel.

It should be noted that, as discussed in Section 2, \( \zeta_t, 0 \leq t \leq T \) at different stages are independent. Therefore, we will establish the quantitative stability of problem \eqref{eq:mixed-MSP-MDP} w.r.t. \( \zeta_t,0 \leq t \leq T \) by considering the distribution perturbations at each stage. To this end, we need the following technical assumption:  

\begin{assumption}[Lipschitz continuity]
For \( t = 1, 2, \cdots, T \),  

\noindent $\mathrm{(C^{\zeta})}$ \( C_t(s_t, x_t, \xi_{[t]}, \zeta_t) \) is Lipschitz continuous w.r.t.~\( (s_t, x_t, \zeta_t) \) with Lipschitz modulus \( L_{C, t} \), let \( L_C = \underset{t \in \{ 1, 2, \cdots, T \}}{\max} L_{C, t} \);  

\noindent $\mathrm{(S^{\zeta})}$ \( S^M_t(s_t, x_t, \xi_{[t]}, \zeta_{t}) \) is Lipschitz continuous w.r.t.~\( (s_t, x_t, \zeta_{t}) \) with Lipschitz modulus \( L_{S,t} \), let \( L_S = \underset{t = 0,1,\cdots,T-1}{\max} L_{S,t} \);  

\noindent $\mathrm{(G^{\zeta})
}$ \( g_{t, i}(s_t, x_t, x_{t - 1}, \xi_{[t]}) \) is Lipschitz continuous w.r.t.~\( (s_t, x_{t - 1}) \) with Lipschitz modulus \( L_{g, t} \), let \( L_g = \underset{t \in \{ 1, 2, \cdots, T \}}{\max} L_{g, t} \).  
\label{assumption6}
\end{assumption}

\begin{theorem}[Stability 
of the optimal value and the 
optimal policy w.r.t.~perturbation of $\zeta$
]
Consider problems \eqref{eq:v-zeta} and   \eqref{eq:v-zeta-ptb}.
Under
Assumptions \ref{Assu:welldefinedness}- \ref{Assu:compactness}, \ref{Assu:convexity}(c),
    \ref{assumption7} and \ref{assumption6},
    the following 
    assertions hold.  
    \begin{itemize}
    \item[(i)] Let $\vt(P)$ and $\vt(\tilde{P})$ be defined as in  \eqref{eq:v-zeta} and   \eqref{eq:v-zeta-ptb}. Then
    \begin{eqnarray}
    |\vt(P) - \vt(\tilde{P})| 
    \leq  \sum\limits_{t = 0}^{T - 1} \hat{L}_{t + 1} \dd_{K}(P_t, \tilde{P}_t) + L_C \dd_{K}(P_T, \tilde{P}_T). 
    \label{eq:stab-endo}
    \end{eqnarray}  
    where $\hat{L}_t := L_{C} L_S + L_{C} L_{X,t} + L_{C} + L_{t + 1} L_S + L_{t + 1} L_{X, t}$, 
    and \( L_t, 1 \leq t \leq T \), is given in Theorem \ref{Theo-Lipschitz}, \( L_{T + 1} = 0 \). 
    \item[(ii)] If, in addition, there exist constants $\beta > 0, \nu \geq 1$ such that
    \begin{eqnarray}
    &&\mathbb{E}_{\mu} \left[ \sum\limits_{t = 0}^T C_t(s_t^{\bm x}, x_t, \xi_{[t]}, \zeta_t) \right] - \mathbb{E}_{\mu} \left[ \sum\limits_{t = 0}^T C_t(s_t, x_t^*, \xi_{[t]}, \zeta_t) \right] \nonumber\\
    &\geq& \beta \left( \mathbb{E}_{\mu} \left[ d({\bm x}(\zeta), \mathcal{X}^*(\zeta)) \right] \right)^\nu, \forall {\bm x}(\zeta) \in \mathcal{X}(\zeta),
    \label{eq:growth-omega}
    \end{eqnarray}  
    where $d(\bm x, \mathcal{X}^*(\zeta))$ denotes the distance from $\bm x$ to 
    the set of optimal solutions 
    $\mathcal{X}^*(\zeta)$, 
    $s_t^{\bm x} = S^M_{t - 1}(s_{t - 1}^{\bm x}, x_{t - 1}, \xi_{[t - 1]}, \zeta_{t - 1})$,
    then there exists a sequence of positive constants \( H_t, L_{X,k,t}, 1 \leq t \leq T \), such that  
    \begin{eqnarray}
    \mathbb{E}_{\pi^*(P, \tilde{P})} \left[ \mathbb{H}({\mathcal{X}}^*(\tilde{\zeta}), \mathcal{X}^*(\zeta)) \right] \leq \sum\limits_{t = 0}^T \left( \sum\limits_{k = t + 1}^T L_{X, k, t} \right) \dd_K(P_t, \tilde{P}_t) + \left( \frac{1}{\beta} \sum\limits_{t = 0}^T H_t \dd_K(P_t, \tilde{P}_t) \right)^\frac{1}{\nu},
    \label{eq:stab-zeta-optsolu}
    \end{eqnarray}
    where $\pi^*(P,\tilde{P}):=\pi_0^*(P_0,\tilde{P}_0) \otimes \cdots \otimes \pi_T^*(P_T,\tilde{P}_T)$, \(\pi_t^*(P_t,\tilde{P}_t)\) denotes an optimal coupling of \(P_t\) and \(\tilde{P}_t\) in \eqref{eq:weighted_transport}.
\end{itemize}  
    \label{Theo-stab-endo}
\end{theorem}

\begin{proof}
    By Theorem \ref{Theo-timecons}, 
    we may
    derive the stability
    via 
    recursive formulations \eqref{eq:value-T}, \eqref{eq:value-t} and \eqref{eq:value-0}.  
    
   \underline{Part (i)}. 
For $t=0,1, \cdots, T$, let $\tilde{\zeta}_{[0:t]} := \left\{ \tilde{\zeta}_0, \tilde{\zeta}_1, \cdots, \tilde{\zeta}_t\right\}$, 
    $\zeta_{[t : T]}
:= 
\{ \zeta_{t}, \zeta_{t + 1}, \cdots, \zeta_T\}$ and
\begin{eqnarray} 
\vt\left(\tilde{P}_{[0:t - 1]}, P_{[t : T]}\right)
 &:= & \underset{\bm x \in \mathcal{X}}{\min} \mathbb{E}_{\tilde{\zeta}_0, \tilde{\zeta}_1, \cdots, \tilde{\zeta}_{t - 1}, \zeta_{t}, \cdots, \zeta_T} \Big[ C_0(s_0, x_0, \tilde{\zeta}_0) \nonumber \\
 && + \sum_{k = 1}^{t - 1} C_k(s_k, x_k, \xi_{[k]}, \tilde{\zeta}_k) + \sum_{k = t}^T C_k(s_k, x_k, \xi_{[k]}, \zeta_k)\Big].
\end{eqnarray} 
Consider $\vt\left(\tilde{P}_{[0:t - 1]}, P_{[t:T]}\right) -
\vt\left(\tilde{P}_{[0 : t]}, P_{[t + 1 : T]}\right)$.
We may regard 
$
\vt\left(\tilde{P}_{[0 : t]}, P_{[t + 1 : T]}\right)$
as a perturbation of
$\vt\left(\tilde{P}_{[0 : t - 1]}, P_{[t : T]}\right)$
when $\zeta_{t}$ is perturbed to 
$\tilde{\zeta}_{t}$.
For $k = t, \cdots, T$, let
\begin{eqnarray*}
    && w_k(s_{k - 1}, x_{k - 1}, \xi_{[k]}, \zeta_{k - 1}) \\ 
    &:= & \underset{x_k \in \mathcal{X}_k(s_k, x_{k - 1}, \xi_{[k]})}{\text{min}} \mathbb{E}_{\zeta_k} \left[ C_k(s_k, x_k, \xi_{[k]}, \zeta_k) + \mathbb{E}_{\xi_{[k + 1]} | \xi_{[k]}} \left[ w_{k + 1}(s_k, x_k, \xi_{[k + 1]}, \zeta_k) \right] \right],
\end{eqnarray*}
with $w_{T + 1}(\cdot) = 0$. Let 
\begin{eqnarray*}
    \tilde{w}_{t + 1}(s_{t}, x_{t}, \xi_{[t + 1]}, \tilde{\zeta}_{t}) &:= & w_{t + 1}(s_{t}, x_{t}, \xi_{[t + 1]}, \tilde{\zeta}_{t}), \\
    \hat{w}_{t}(s_{t - 1}, x_{t - 1}, \xi_{[t]}, \tilde{\zeta}_{t - 1}) &:= & w_{t}(s_{t - 1}, x_{t - 1}, \xi_{[t]}, \tilde{\zeta}_{t - 1}).
\end{eqnarray*}
Then for $k = t, t-1, \cdots, 1$, define
    \begin{eqnarray*}
    &&\tilde{w}_k(s_{k - 1}, x_{k - 1}, \xi_{[k]}, \tilde{\zeta}_{k - 1}) \\
    &:= & \underset{x_k \in \mathcal{X}_k(s_k, x_{k - 1}, \xi_{[k]})}{\text{min}} \mathbb{E}_{\tilde{\zeta}_k} \left[ C_k(s_k, x_k, \xi_{[k]}, \tilde{\zeta}_k) + \mathbb{E}_{\xi_{[k + 1]} | \xi_{[k]}} \left[ \tilde{w}_{k + 1}(s_k, x_k, \xi_{[k + 1]}, \tilde{\zeta}_k) \right] \right],  
    \end{eqnarray*}
and  for $k = t-1, t-2, \cdots, 1$, let
    \begin{eqnarray*}
    && \hat{w}_k(s_{k - 1}, x_{k - 1}, \xi_{[k]}, \tilde{\zeta}_{k - 1}) \\ 
    &:= & \underset{x_k \in \mathcal{X}_k(s_k, x_{k - 1}, \xi_{[k]})}{\text{min}} \mathbb{E}_{\tilde{\zeta}_k} \left[ C_k(s_k, x_k, \xi_{[k]}, \tilde{\zeta}_k) + \mathbb{E}_{\xi_{[k + 1]} | \xi_{[k]}} \left[ \hat{w}_{k + 1}(s_k, x_k, \xi_{[k + 1]}, \tilde{\zeta}_k) \right] \right].
    \end{eqnarray*}
Then
\begin{eqnarray}
        \vt(\tilde{P}_{[0 : t - 1]}, P_{[t : T]}) &= & \underset{x_0 \in \mathcal{X}_0}{\min}\mathbb{E}_{\tilde{\zeta}_0} \left[ C_0(s_0, x_0, \tilde{\zeta}_0) + \mathbb{E}_{\xi_1} \left[ \hat{w}_1(s_0, x_0, \xi_1, \tilde{\zeta}_0) \right] \right], 
    \label{eq:vt-optval-zeta-before}\\
    \vt(\tilde{P}_{[0 : t]}, P_{[t + 1 : T]}) &=& \underset{x_0 \in \mathcal{X}_0}{\min}\mathbb{E}_{\tilde{\zeta}_0} \left[ C_0(s_0, x_0, \tilde{\zeta}_0) + \mathbb{E}_{\xi_1} \left[ \tilde{w}_1(s_0, x_0, \xi_1, \tilde{\zeta}_0) \right] \right]. 
    \label{eq:vt-optval-zeta-after}
\end{eqnarray}
Let $\hat{x}_0^*$ and $\tilde{x}_0^*$ be the optimal solutions to problems \eqref{eq:vt-optval-zeta-before} and \eqref{eq:vt-optval-zeta-after} , respectively.
\begin{eqnarray}
    &&\vt(\tilde{P}_{[0 : t - 1]}, P_{[t : T]}) - \vt(\tilde{P}_{[0 : t]}, P_{[t + 1 : T]}) \nonumber \\
    &=& \mathbb{E}_{\tilde{\zeta}_0} \left[ C_0(s_0, \hat{x}^*_0, \tilde{\zeta}_0) + \mathbb{E}_{\xi_1} \left[ \hat{w}_1(s_0, \hat{x}^*_0, \xi_1, \tilde{\zeta}_0) \right] \right] \nonumber \\
    && - \mathbb{E}_{\tilde{\zeta}_0} \left[ C_0(s_0, \tilde{x}^*_0, \tilde{\zeta}_0) + \mathbb{E}_{\xi_1} \left[ \tilde{w}_1(s_0, \tilde{x}^*_0, \xi_1, \tilde{\zeta}_0) \right] \right] \nonumber \\
    &\leq &\mathbb{E}_{\xi_{1}, \tilde{\zeta}_0} \left[ \hat{w}_1(s_0, \tilde{x}^*_0, \xi_{1}, \tilde{\zeta}_0) \right] - \mathbb{E}_{\xi_{1}, \tilde{\zeta}_0} \left[ \tilde{w}_1(s_0, \tilde{x}^*_0, \xi_{1}, \tilde{\zeta}_0) \right] \nonumber \\
    &= & \mathbb{E}_{\xi_{1}, \tilde{\zeta}_0} \mathbb{E}_{\xi_{2} | \xi_1, \tilde{\zeta}_1} \Big[ \Big[ C_1(s_1, \hat{x}_1^*, \xi_{1}, \tilde{\zeta}_1) - C_1(s_1, \tilde{x}_1^*, \xi_1, \tilde{\zeta}_1) \nonumber \\
    && +  \hat{w}_2(s_1, \hat{x}_1^*, \xi_{[2]}, \tilde{\zeta}_1) - \tilde{w}_2(s_1, \tilde{x}_1^*, \xi_{[2]}, \tilde{\zeta}_1) \Big]  \Big] \nonumber \\
    &\leq & \mathbb{E}_{\xi_{1}, \tilde{\zeta}_0} \mathbb{E}_{\xi_{2} | \xi_1, \tilde{\zeta}_1} \left[ \hat{w}_2(s_1, \tilde{x}_1^*, \xi_{[2]}, \tilde{\zeta}_1) - \tilde{w}_2(s_1, \tilde{x}_1^*, \xi_{[2]}, \tilde{\zeta}_1) \right] \nonumber \\
    &\leq & \cdots \leq \mathbb{E}_{\xi_{1}, \tilde{\zeta}_0} \mathbb{E}_{\xi_{2} | \xi_1, \tilde{\zeta}_1} \cdots \mathbb{E}_{\xi_{t} | \xi_{[t - 1]}, \tilde{\zeta}_{t - 1}}  \nonumber \\
    && \Big[ \mathbb{E}_{\xi_{t + 1} | \xi_{[t]}, \zeta_{t}} \left[C_{t}(s_{t}, \tilde{x}_{t}^*, \xi_{[t]}, \zeta_{t}) + w_{t + 1}(s_{t}, \tilde{x}_{t}^*, \xi_{[t + 1]}, \zeta_{t}) \right] \nonumber \\
    && - \mathbb{E}_{\xi_{t + 1} | \xi_{[t]}, \tilde{\zeta}_{t}} \left[ C_{t}(s_{t}, \tilde{x}_{t}^*, \xi_{[t]}, \tilde{\zeta}_{t}) + w_{t + 1}(s_{t}, \tilde{x}_{t}^*, \xi_{[t + 1]}, \tilde{\zeta}_{t}) \right] \Big], \label{eq:stab-optvalue-zeta-step1}
    \end{eqnarray}
    where $\hat{x}_k^*$ and $\tilde{x}_k^*$, $k = 1,\cdots, t$ are the optimal solutions to problem \eqref{eq:value-t} before and after the perturbation at stage $t$. To estimate the difference at the right hand side of \eqref{eq:stab-optvalue-zeta-step1}, we first examine the direct influence of the perturbation of $\zeta_t$ on the feasible set. 
    Let $\tilde{s}_{t + 1} = S^M_t(s_t, x_t, \xi_t, \tilde{\zeta}_t)$. 
    Then by the Slater condition (Assumption \ref{assumption7}), 
    there exist \( \bar{x}_{t + 1} \) and \( \tilde{x}_{t + 1} \) such that  
      \begin{subequations}
    \begin{eqnarray}
    && g_{t + 1, i}(s_{t + 1}, \bar{x}_{t + 1}, x_{t}, \xi_{[t + 1]}) \leq -\rho, \quad i \in I_{t + 1}; \label{eq:xt+1-slater} \\
    && g_{t + 1, i}(\tilde{s}_{t + 1}, \tilde{x}_{t + 1}, x_{t}, \xi_{[t + 1]}) \leq 0, \quad i \in I_{t + 1}. 
    \label{eq:xt+1-feasible}
    \end{eqnarray}
    \end{subequations}
    Existence of $\tilde{x}_{t+1}$ is guaranteed by the fact that under Assumption \ref{assumption6} ($G^\zeta$), 
$
\|\tilde{s}_{t + 1} - s_{t + 1}\| \leq L_S \| \tilde{\zeta}_t - \zeta_t \|$. Together with Assumption \ref{assumption6}, we have
$$
 g_{t + 1, i}(\tilde{s}_{t + 1}, \bar{x}_{t + 1}, x_{t}, \xi_{[t + 1]}) 
 \leq g_{t + 1, i}(s_{t + 1}, \bar{x}_{t + 1}, x_{t}, \xi_{[t + 1]}) + L_g L_S \| \tilde{\zeta}_t - \zeta_t \|
 \leq -\rho + L_g L_S \| \tilde{\zeta}_t - \zeta_t \|.
$$
When $L_g L_S \| \tilde{\zeta}_t - \zeta_t \|\leq \rho$,
we can choose $\tilde{x}_{t + 1} =\bar{x}_{t + 1}$.   
    On the other hand, we can use 
    Assumption \ref{assumption6}($G^\zeta$) to establish
    \begin{eqnarray}
    g_{t + 1,i}(s_{t + 1}, \tilde{x}_{t + 1}, x_{t}, \xi_{[t + 1]}) 
    &\leq& g_{t + 1,i}(s_{t + 1}, \tilde{x}_{t + 1}, x_{t}, \xi_{[t + 1]}) - g_{t + 1,i}(\tilde{s}_{t + 1}, \tilde{x}_{t + 1}, x_{t}, \xi_{[t + 1]})  \nonumber \\
    &\leq & L_{g} \Vert s_{t + 1} - \tilde{s}_{t + 1} \Vert \leq L_g L_S \Vert \zeta_{t} - \tilde{\zeta}_{t} \Vert, \quad \quad i \in I_{t + 1}. 
    \label{eq:Lip-gti}
    \end{eqnarray}
    Let $z_{t + 1} := \frac{\rho \tilde{x}_{t + 1} + L_g L_S \Vert \zeta_{t} - \tilde{\zeta}_{t} \Vert \bar{x}_{t + 1}}{\rho + L_g L_S \Vert \zeta_{t} - \tilde{\zeta}_{t} \Vert}$. 
    By the convexity of  \( g_{t + 1,i} \)
in \( x_{t + 1} \), \eqref{eq:xt+1-slater} and \eqref{eq:Lip-gti},
we obtain 
\begin{eqnarray}
    && g_{t + 1,i}\left(s_{t + 1}, z_{t + 1}, x_{t}, \xi_{[t + 1]}\right) \nonumber \\
    &\leq & \frac{\rho g_{t + 1, i}(s_{t + 1}, \tilde{x}_{t + 1}, x_{t}, \xi_{t + 1})}{\rho + L_g L_S \Vert \zeta_{t} - \tilde{\zeta}_{t} \Vert} + \frac{L_g L_S \Vert \zeta_{t} - \tilde{\zeta}_{t} \Vert g_{t + 1, i}(s_{t + 1}, \bar{x}_{t + 1}, x_{t}, \xi_{t + 1})}{\rho + L_g L_S \Vert \zeta_{t} - \tilde{\zeta}_{t} \Vert}\nonumber \\
    &\leq & \frac{\rho L_g L_S \Vert \zeta_{t} - \tilde{\zeta}_{t} \Vert - L_g L_S \Vert \zeta_{t} - \tilde{\zeta}_{t} \Vert \rho}{\rho + L_g L_S \Vert \zeta_{t} - \tilde{\zeta}_{t} \Vert} = 0, \nonumber
    \end{eqnarray}
 which implies that $z_{t + 1} \in \mathcal{X}_{t + 1}(s_{t + 1}, x_{t}, \xi_{[t + 1]})$.
Consequently
    \begin{eqnarray}
   &&  \mathbb{D}(\mathcal{X}_{t+1}(\tilde{s}_{t+1}, x_{t}, \xi_{[t+1]}), \mathcal{X}_{t+1}({s}_{t+1}, x_{t}, \xi_{[t+1]})) \nonumber \\
    &= & \underset{\tilde{x}_{t+1} \in \mathcal{X}_{t+1}(\tilde{s}_{t+1}, \tilde{x}_{t}, \xi_{[t+1]})}{\max} 
    d(\tilde{x}_{t+1}, \mathcal{X}_{t+1}({s}_{t+1}, x_{t}, \xi_{[t+1]})) \nonumber \\ 
    &\leq & \underset{\tilde{x}_{t+1} \in \mathcal{X}_{t+1}(\tilde{s}_{t+1}, \tilde{x}_{t}, \xi_{[t+1]})}{\max} d(\tilde{x}_{t+1}, z_{t+1}) \nonumber \\
    &= & \frac{ L_g L_S \Vert \zeta_{t} - \tilde{\zeta}_{t} \Vert}{\rho} 
    \underset{\tilde{x}_{t+1} \in \mathcal{X}_{t+1}(\tilde{s}_{t+1}, \tilde{x}_{t}, \xi_{[t+1]})}{\max} d(\bar{x}_{t+1}, z_{t+1}) \nonumber \\
    &\leq & A L_g L_S \Vert \zeta_{t} - \tilde{\zeta}_{t} \Vert / \rho, 
    \label{eq:hoffman}
    \end{eqnarray}
    where $A$ is defined in Theorem \ref{Theo-Lipschitz}. 
    
    Next, we estimate 
    $
    \mathbb{E}_{\xi_{t+1} | \xi_{[t]}, \zeta_{t}} \left[ w_{t+1}(s_t, x_t, \xi_{[t+1]}, \zeta_t) \right] 
    - \mathbb{E}_{\xi_{t+1} | \xi_{[t]}, \tilde{\zeta}_t} \left[ w_{t+1}(s_t, x_t, \xi_{[t+1]}, \tilde{\zeta}_t) \right]
    $
    at the rhs of \eqref{eq:stab-optvalue-zeta-step1}.
    Let \( x_{t+1}^* \) be an optimal solution to the optimization problem \eqref{eq:value-t}, and \( \tilde{x}_{t+1}^* \) be an optimal solution to the perturbed problem \eqref{eq:value-t}. Let \( y_{t+1} \) be the projection of \( \tilde{x}_{t+1}^* \) onto \( \mathcal{X}_{t+1}(s_{t+1}, x_t, \xi_{[t+1]}) \). By Theorem \ref{Theo-Lipschitz},
    \begin{eqnarray}
    && w_{t+1}(s_t, x_t, \xi_{[t+1]}, \zeta_t) - w_{t+1}(s_t, x_t, \xi_{[t+1]}, \tilde{\zeta}_t) \nonumber \\
    &= & \mathbb{E}_{\xi_{t+2} | \xi_{[t+1]}, \zeta_{t+1}} \left[ C_{t+1}(s_{t+1}, x_{t+1}^*, \xi_{[t+1]}, \zeta_{t+1}) + 
    w_{t+2}(s_{t+1}, x_{t+1}^*, \xi_{[t+2]}, \zeta_{t+1}) \right] \nonumber \\
    && - \mathbb{E}_{\xi_{t+2} | \xi_{[t+1]}, \zeta_{t+1}} \left[ C_{t+1}(\tilde{s}_{t+1}, \tilde{x}_{t+1}^*, \xi_{[t+1]}, \zeta_{t+1}) + 
    w_{t+2}(\tilde{s}_{t+1}, \tilde{x}_{t+1}^*, \xi_{[t+2]}, \zeta_{t+1}) \right] \nonumber \\
    &\leq & \mathbb{E}_{\xi_{t+2} | \xi_{[t+1]}, \zeta_{t+1}} \left[ C_{t+1}(s_{t+1}, y_{t+1}, \xi_{[t+1]}, \zeta_{t+1}) + 
    w_{t+2}(s_{t+1}, y_{t+1}, \xi_{[t+2]}, \zeta_{t+1}) \right] 
    \nonumber \\
    && - \mathbb{E}_{\xi_{t+2} | \xi_{[t+1]}, \zeta_{t+1}} \left[ C_{t+1}(\tilde{s}_{t+1}, \tilde{x}_{t+1}^*, \xi_{[t+1]}, \zeta_{t+1}) + 
    w_{t+2}(\tilde{s}_{t+1}, \tilde{x}_{t+1}^*, \xi_{[t+2]}, \zeta_{t+1}) \right] \nonumber \\
    &\leq & L_{C,t+1} (\Vert s_{t+1} - \tilde{s}_{t+1} \Vert + \Vert y_{t+1} - \tilde{x}_{t+1}^* \Vert) + 
    L_{t+2} (\Vert s_{t+1} - \tilde{s}_{t+1} \Vert + \Vert y_{t+1} - \tilde{x}_{t+1}^* \Vert) \quad \text{(by \eqref{eq:Lipschitz-v_t})} \nonumber \\
    &\leq & (L_{C,t+1} L_S + L_{C,t+1} L_{X,t+1} + L_{t+2} L_S + L_{t+2} L_{X,t+1}) \Vert \zeta_t - \tilde{\zeta}_t \Vert, \quad \text{(by \eqref{eq:hoffman})} \nonumber \\
    &:= & \widehat{L}_{t+1} \Vert \zeta_t - \tilde{\zeta}_t \Vert. \label{eq:lip-w}
    \end{eqnarray}
    where $y_{t+1}$ is feasible to \eqref{eq:value-t}, $L_{X, t+1} = A L_g L_S / \rho$. \eqref{eq:lip-w} implies that  
    $w_{t+1}(s_t, x_t, \xi_{[t+1]}, \cdot)$ is Lipschitz continuous  over $\mathcal{Z}_t$ with modulus 
    $\widehat{L}_{t+1}$.
    By the dual representation of the 
    Kantorovich metric,
    \begin{eqnarray}
    && \mathbb{E}_{\xi_{t+1} | \xi_{[t]}, \zeta_t} \left[ w_{t+1}(s_t, x_t, \xi_{[t+1]}, \zeta_t) \right] 
    - \mathbb{E}_{\xi_{t+1} | \xi_{[t]}, \tilde{\zeta}_t} \left[ w_{t+1}(s_t, x_t, \xi_{[t+1]}, \tilde{\zeta}_t) \right] \nonumber \\
    &= & \widehat{L}_{t+1} \int_{\mathbb{R}^{m_{2, t}}} \left[ w_{t+1}(s_t, x_t, \xi_{[t+1]}, y) / \widehat{L}_{t+1} \right] P_t(dy) \nonumber \\
    && - \int_{\mathbb{R}^{m_{2, t}}} \left[ w_{t+1}(s_t, x_t, \xi_{[t+1]}, y) / \widehat{L}_{t+1} \right] \tilde{P}_t(dy) \nonumber \\
    &\leq & \widehat{L}_{t+1} \underset{f \in \mathcal{F}(\Xi)}{\sup} \left( 
    \int_{\mathbb{R}^{m_{2, t}}} f(y) P_t(dy) - \int_{\mathbb{R}^{m_{2, t}}} f(y) \tilde{P}_t(dy) \right) \nonumber \\
    &\leq & \widehat{L}_{t+1} \dd_K(P_t, \tilde{P}_t), 
    \label{eq:Lip-zeta}
    \end{eqnarray}
Combining \eqref{eq:Lip-zeta} with the rhs of \eqref{eq:stab-optvalue-zeta-step1}, we obtain that
    \begin{eqnarray}
        \vt(\tilde{P}_{[0 : t - 1]}, P_{[t : T]}) - \vt(\tilde{P}_{[0 : t]}, P_{[t + 1 : T]}) &\leq & \mathbb{E}_{\xi_{1}, \tilde{\zeta}_0} \mathbb{E}_{\xi_{2} | \xi_1, \tilde{\zeta}_1} \cdots \mathbb{E}_{\xi_{t} | \xi_{[t - 1]}, \tilde{\zeta}_{t}} \left[ (\widehat{L}_{t + 1} + L_{C,t}) \dd_K(P_{t}, \tilde{P}_{t}) \right] \nonumber \\
        &\leq &(\widehat{L}_{t + 1} + L_{C}) \dd_K(P_{t}, \tilde{P}_{t}). \nonumber
    \end{eqnarray}
    In view of the inter-stage independence of \( \zeta_t \), $0 \leq t \leq T$, we obtain by summing the above differences of the optimal values perturbed at individual stages, $t=0,1, \cdots, T$, that  
    \begin{eqnarray}
     && \vt(P) - \vt(\tilde{P}) \nonumber \\
     &= & \vt(P) - \vt((\tilde{P}_0, P_{[1:T]})) +
    \vt((\tilde{P}_0, P_{[1:T]})) - \vt((\tilde{P}_{[0:1]}, P_{[2:T]})) + 
    \cdots + \vt((\tilde{P}_{[0:T - 1]}, P_{T})) - \vt(\tilde{P}) \nonumber \\
    &\leq &
    \sum\limits_{t = 0}^T (\widehat{L}_{t + 1} + L_{C}) \dd_K(P_t, \tilde{P}_t).
    \label{eq:opt-value-omega}
    \end{eqnarray}
    By exchanging positions between $P$ and $\tilde{P}$, we can derive the same bound for $\vt(\tilde{P}) - \vt(P)$. This completes the
proof of Part (i).    

\underline{Part (ii)}.
To obtain the quantitative stability of the 
set of optimal solutions, we need to first establish the quantitative stability of the set of feasible solutions of problem \eqref{eq:mixed-MSP-MDP}. It can be seen from \eqref{eq:hoffman} that the feasible set at the first stage satisfies that
    \begin{eqnarray}
    \mathbb{H}(\mathcal{X}_1(s_1, x_0, \xi_1), \mathcal{X}_1(\tilde{s}_1, x_0, \xi_1)) \leq \frac{L_{g,1} L_S A_1}{\rho} \Vert \zeta_0 - \tilde{\zeta}_0 \Vert := L_{X, 1, 0} \Vert \zeta_0 - \tilde{\zeta}_0 \Vert. \nonumber
    \end{eqnarray}
    Assume that for the feasible set at stage $k$ ($k < t$), there exist 
  positive  constants $L_{X,k,j}$ satisfying
    \begin{eqnarray}
    \mathbb{H}(\mathcal{X}_k(s_k^z, z_{k - 1}, \xi_{[k]}), \mathcal{X}_k(\tilde{s}_k, \tilde{x}_{k - 1}, \xi_{[k]})) \leq \sum\limits_{j = 0}^{k - 1} L_{X, k, j} \Vert \tilde{\zeta}_j - \zeta_j \Vert, \label{eq:ind-ass-Lipzeta}
    \end{eqnarray}
    where $s_k^z = S^M_{k - 1}(s_{k - 1}^z, z_{k - 1}, \xi_{k - 1}, \zeta_{k - 1})$, $\tilde{s}_k = S^M_{k - 1}(s_{k - 1}, \tilde{x}_{k - 1}, \xi_{k - 1}, \tilde{\zeta}_{k - 1})$, $z_k$ is the projection of $\tilde{x}_k$ onto $\mathcal{X}_k(s_k^z, z_{k - 1}, \xi_{[k]})$, $s_0^z = \tilde{s}_0 = s_0$, and $z_0 = \tilde{x}_0 = x_0$. 
    Following a similar argument to that in Part (i), we assert that there exist feasible solutions $\bar{x}_t$ and $\tilde{x}_t$ such that  
    \begin{eqnarray}
    && g_{t,i}(\tilde{s}_t, \tilde{x}_t, \tilde{x}_{t - 1}, \xi_{[t]}) \leq 0,\ i \in I_t, \nonumber \\
    && g_{t,i}(s_t^z, \bar{x}_t, z_{t - 1}, \xi_{[t]}) \leq -\rho,\ i \in I_t. \nonumber
    \end{eqnarray}
 Moreover, 
 by the Lipschitz continuity of $ g_{t,i} $ and $S_t^M$,
    \begin{subequations}
    \begin{eqnarray}
    && g_{t,i}(s_t^z, \tilde{x}_t, z_{t - 1}, \xi_{[t]}) \leq g_{t,i}(s_t^z, \tilde{x}_t, z_{t - 1}, \xi_{[t]}) - g_{t,i}(\tilde{s}_t, \tilde{x}_t, \tilde{x}_{t - 1}, \xi_{[t]}) \label{eq:omega-g-a} \\
    &\leq & L_{g, t} (\Vert \tilde{s}_t - s_t^z \Vert + \Vert z_{t - 1} - \tilde{x}_{t - 1} \Vert) \label{eq:omega-g-b} \\
    &\leq & L_{g, t} L_S (\Vert \tilde{s}_{t - 1} - s_{t - 1}^z \Vert + \Vert z_{t - 1} - \tilde{x}_{t - 1} \Vert + \Vert \tilde{\zeta}_{t - 1} - \zeta_{t - 1} \Vert) + L_{g, t} \Vert z_{t - 1} - \tilde{x}_{t - 1} \Vert \label{eq:omega-g-c} \\
    &\leq & \cdots \leq L_{g, t}\sum\limits_{k = 1}^t L_S^k (\Vert \tilde{\zeta}_{t - k} - \zeta_{t - k} \Vert + \Vert z_{t - k} - \tilde{x}_{t - k} \Vert) + L_{g, t} \Vert z_{t - 1} - \tilde{x}_{t - 1} \Vert \label{eq:omega-g-d} \\
    &\overset{\eqref{eq:ind-ass-Lipzeta}}{\leq} & L_{g, t}(\sum\limits_{k = 1}^t L_S^k \Vert \tilde{\zeta}_{t - k} - \zeta_{t - k} \Vert + \sum\limits_{k = 1}^t L_S^{t - k} \sum\limits_{j = 0}^{k - 1} L_{X, k, j} \Vert \tilde{\zeta}_j - \zeta_j \Vert)  \nonumber \\
    && + L_{g, t} \sum\limits_{j = 0}^{t - 2} L_{X, t - 1, j} \Vert \tilde{\zeta}_j - \zeta_j \Vert \label{eq:omega-g-e} \\
    &\leq & L_{g, t} \left( \sum\limits_{j = 0}^{t - 2} L_{X, t - 1, j} \Vert \tilde{\zeta}_j - \zeta_j \Vert +  \sum\limits_{j = 0}^{t - 1} \left(L_S^{t - j} + \sum\limits_{k = j + 1}^{t - 1} L_S^{t - 1 - k} L_{X, k, j}\right) \Vert \tilde{\zeta}_j - \zeta_j \Vert \right), \label{eq:omega-g-f} \nonumber \\
    \end{eqnarray}  
    \label{eq:omega-g}
    \end{subequations}
    where \eqref{eq:omega-g-a} is due to $\tilde{x}_t \in \mathcal{X}_t(\tilde{s}_t, \tilde{x}_{t - 1}, \xi_{[t]})$, \eqref{eq:omega-g-b} and \eqref{eq:omega-g-c} are obtained by Assumption~\ref{assumption6},  \eqref{eq:omega-g-d} is derived through repeated applying Assumption \ref{assumption6}, \eqref{eq:omega-g-e} is due to the induction assumption \eqref{eq:ind-ass-Lipzeta}.  \eqref{eq:omega-g-f} comes from  interchanging the order of summations.
    Denote the right-hand side of \eqref{eq:omega-g} as \( G_\zeta \). Let \( z_t = \frac{\rho \tilde{x}_t + G_{\zeta} \bar{x}_t}{G_{\zeta} + \rho} \). Then, by the convexity of \( g_{t, i} \), we have
    \begin{eqnarray}
    g_{t, i}(s_t^z, z_t, z_{t - 1}, \xi_{[t]}) \leq \frac{G_{\zeta} \rho - \rho G_{\zeta}}{G_{\zeta} + \rho} = 0,  i \in I_t .  \nonumber
    \end{eqnarray}
    On the other hand, for any \( \tilde{x}_t \in \mathcal{X}_t(\tilde{s}_t, \tilde{x}_{t - 1}, \xi_{[t]}) \), 
    since \( z_t \in \mathcal{X}_t(s_t^z, z_{t - 1}, \xi_{[t]}) \), then
    \begin{eqnarray}
    d(z_t, \tilde{x}_t) = \frac{G_{\zeta}}{\rho} d(z_t, \bar{x}_t) \leq \frac{A}{\rho} G_{\zeta}. \nonumber
    \end{eqnarray}
    Since \( G_\zeta \) is a nonnegative linear combination of \( \Vert \zeta_k - \tilde{\zeta}_k \Vert, k \in \{0, 1, 2, \cdots, t - 1 \} \), we can express \( \frac{A}{\rho} G_{\zeta} \) as \( \sum\limits_{j = 0}^{t - 1} L_{X, t, j} \Vert \zeta_j - \tilde{\zeta}_j \Vert \), where
    \begin{eqnarray*}
        L_{X, t, j} = \frac{A L_{g, t} \left( L_{X, t - 1, j} +  L_S^{t - j} + \sum\limits_{k = j}^{t - 1} L_S^{t - 1 - k} L_{X, k, j}\right)}{\rho} \quad \text{for } j = 0, 1, \cdots, t - 2
    \end{eqnarray*}
    and $L_{X, t, t - 1} = \frac{A L_{g, t} L_S}{\rho}$. Then
    \begin{eqnarray}
    d(z_t, \tilde{x}_t) \leq \sum\limits_{j = 0}^{t - 1} L_{X, t, j} \Vert \zeta_j - \tilde{\zeta}_j \Vert. \nonumber
    \end{eqnarray}
    
    As \( \tilde{x}_t \) is chosen arbitrarily, and in the derivation above, \( {\zeta}_t \) can be replaced by \( \tilde{\zeta}_t \), it follows that the feasible solution set satisfies that
    \begin{eqnarray}
    \mathbb{H}(\mathcal{X}_t(s_t^z, z_{t - 1}, \xi_{[t]}), \mathcal{X}_t(\tilde{s}_t, \tilde{x}_{t - 1}, \xi_{[t]})) \leq \sum\limits_{j = 0}^{t - 1} L_{X, t, j} \Vert \zeta_j - \tilde{\zeta}_j \Vert. \label{eq:feas-Lip-zeta}
    \end{eqnarray}
    By the definition of Kantorovich metric (see e.g. \cite{romisch2003stability}), this implies that
    \begin{eqnarray}
    \mathbb{E}_{\pi^*(P_{[t]}, \tilde{P}_{[t]})} \left[\mathbb{H}(\mathcal{X}_t(s_t^z, z_{t - 1}, \xi_{[t]}), \mathcal{X}_t(\tilde{s}_t, \tilde{x}_{t - 1}, \xi_{[t]}))\right] \leq \sum\limits_{j = 0}^{t - 1} L_{X, t, j} \dd_K(P_j, \tilde{P}_j), \nonumber
    \end{eqnarray}
    where $\pi^*(P_{[t]},\tilde{P}_{[t]}) := \left( \pi_0^*(P_0,\tilde{P}_0), \ldots, \pi_{t-1}^*(P_{t-1},\tilde{P}_{t-1}) \right),$
    \(\pi_j^*(P_j,\tilde{P}_j)\) denotes an optimal coupling of $P_{j}$ and $\tilde{P}_{j}$ in \eqref{eq:def-DeltaL}.
    Let \( \mathcal{X}(\zeta) := \mathcal{X}_0 \times \cdots \times \mathcal{X}_T \) be the feasible set of problem \eqref{eq:mixed-MSP-MDP}. Then from the inequality above, we obtain
    \begin{eqnarray}
    \mathbb{E}_{\pi^*(P, \tilde{P})} \left[ \mathbb{H}(\mathcal{X}(\zeta), \mathcal{X}(\tilde{\zeta})) \right] \leq \sum\limits_{t = 1}^T \sum\limits_{j = 0}^{t - 1} L_{X, t, j} \dd_K(P_j, \tilde{P}_j).
    \label{eq:omega-feasibleset}
    \end{eqnarray}
    
    Let \(\tilde{\bm x}^*\) be an optimal policy for the perturbed problem.
    Fix \(\pi^*(P,\tilde P)\in\Pi(P,\tilde P)\). Set
    \(z_0:=\tilde x_0^*\in\mathcal X_0\) and \(z_t\) be a measurable
    projection of \(\tilde x_t^*(\tilde s_t,\tilde x_{t-1}^*,\xi_{[t]})\) onto \(\mathcal X_t(s_t^z,z_{t-1},\xi_{[t]})\) for \(t=1,\ldots,T\). Write \(\bm z:=(z_0,\ldots,z_T)\).
    To this end, we first consider the difference between the states \( s_t^z \) and \( \tilde{s}_t \), \( 0 \leq t \leq T \). It is known from the assumed Lipschitz continuity that
    \begin{eqnarray}
    && \Vert s_t^z - \tilde{s}_t \Vert \leq L_S (\Vert s_{t - 1}^z - \tilde{s}_{t - 1} \Vert + \Vert z_{t - 1} - \tilde{x}_{t - 1} \Vert + \Vert \zeta_{t - 1} - \tilde{\zeta}_{t - 1} \Vert) \nonumber \\
    &\leq & L_S(\Vert z_{t - 1} - \tilde{x}_{t - 1} \Vert + \Vert \zeta_{t - 1} - \tilde{\zeta}_{t - 1} \Vert) \nonumber \\
    && + L_S^2(\Vert s_{t - 2}^z - \tilde{s}_{t - 2} \Vert + \Vert z_{t - 2} - \tilde{x}_{t - 2} \Vert + \Vert 
    \zeta_{t - 2} - \tilde{\zeta}_{t - 2} \Vert) \nonumber \\
    &\leq & \cdots \leq \sum\limits_{k = 1}^t L_S^k \left(\Vert z_{t - k} - \tilde{x}_{t - k} \Vert + \Vert \zeta_{t - k} - \tilde{\zeta}_{t - k} \Vert\right). \label{eq:s-zeta-Lip}
    \end{eqnarray}
    Then, we obtain
    \begin{eqnarray}
    \hspace{-2em} && \left| \mathbb{E}_{\pi^*(P,\tilde P)} \left[ \sum\limits_{t = 0}^T C_t(s_t^z, z_t, \xi_{[t]}, \zeta_t) \right] - \mathbb{E}_{\tilde{\zeta}} \left[\sum\limits_{t = 0}^T C_t(\tilde{s}_t, \tilde{x}^*_t, \xi_{[t]}, \tilde{\zeta}_t) \right] \right| \nonumber \\
    \hspace{-2em} &\leq & \mathbb{E}_{\pi^*(P, \tilde{P})} \left[ \sum\limits_{t = 0}^T L_{C, t}(\Vert s_t^z - \tilde{s}_t \Vert + \Vert z_t - \tilde{x}^*_t \Vert + \Vert \zeta_t - \tilde{\zeta}_t \Vert) \right] \nonumber \\
    \hspace{-2em} &\leq & \mathbb{E}_{\pi^*(P, \tilde{P})}  \left[ \sum\limits_{t = 0}^T L_{C, t}(\Vert z_t - \tilde{x}^*_t \Vert + \Vert \zeta_t - \tilde{\zeta}_t \Vert)  + \sum\limits_{t = 0}^T L_{C, t} \sum\limits_{k = 1}^t L_S^k(\Vert z_{t - k} - \tilde{x}_{t - k}^* \Vert + \Vert \zeta_{t - k} - \tilde{\zeta}_{t - k} \Vert) \right] \nonumber \\
    \hspace{-2em} &\leq &  L_C\sum\limits_{t = 1}^T \sum\limits_{j = 0}^{t - 1} L_{X, t, j} \dd_K(P_j, \tilde{P}_j) + L_C \sum\limits_{t = 0}^T \dd_K(P_t, \tilde{P}_t) + L_C\sum\limits_{t = 1}^T \sum\limits_{k = 1}^t L_S^k \bigg(\sum\limits_{j = 1}^{t - k} L_{X, t - k, {j - 1}} \dd_K(P_{j - 1}, \tilde{P}_{j - 1}) \nonumber \\
    \hspace{-2em} && + \dd_K(P_{t - k}, \tilde{P}_{t - k})\bigg) \nonumber \\
    \hspace{-2em} &= & L_C \sum\limits_{t = 1}^T \left(1 + \sum\limits_{k = t}^T L_{X, k, t - 1}\right) \dd_K(P_{t - 1}, \tilde{P}_{t - 1}) + L_C \sum\limits_{t = 1}^{T} \left(\sum\limits_{k = 1}^{T - t} \sum\limits_{l = k + t}^T L_S^k L_{X, l - k, t}\right) \dd_K(P_{t - 1}, \tilde{P}_{t - 1}) \nonumber \\
    \hspace{-2em} && + L_C \sum\limits_{t = 1}^T \left(\sum\limits_{k = 1}^{T - t + 1} L_S^k\right) \dd_K(P_{t - 1}, \tilde{P}_{t - 1}) \nonumber \\
    \hspace{-2em} &= & L_C \sum\limits_{t = 0}^{T - 1} \left(1 + \sum\limits_{k = t + 1}^T L_{X, k, t} + \sum\limits_{k = 1}^{T - 1 - t} \sum\limits_{l = k + t + 1}^T L_S^k L_{X, l - k, t} + \sum\limits_{k = 1}^{T - t} L_S^k\right) \dd_K\left(P_t, \tilde{P}_t\right) + L_C \dd_K\left(P_T, \tilde{P}_T\right). \nonumber \\ 
    \label{eq:optvalue-projection}
    \end{eqnarray}
    The first inequality follows from the marginal property of \(\pi^*(P,\tilde P)\) and the Lipschitz property of \(C_t\).
    The second inequality follows from \eqref{eq:s-zeta-Lip}. The third inequality is obtained by \eqref{eq:feas-Lip-zeta}, Assumption \ref{assumption7} and the definition of Kantorovich metric. 
    
    With the above preparations, we can specifically explore the quantitative stability of the optimal solution set. According to \eqref{eq:opt-value-omega} and \eqref{eq:optvalue-projection}, we have
    \begin{eqnarray}
    && \left| \mathbb{E}_{\pi^*(P,\tilde{P})} \left[ \sum\limits_{t = 0}^T C_t(s_t^z, z_t, \xi_{[t]}, \zeta_t) \right] - \mathbb{E}_{\zeta_{[T]}} \left[ \sum\limits_{t = 0}^T C_t(s_t, x_t^*, \xi_{[t]}, \zeta_t) \right] \right| \nonumber \\
    &\leq & \left| \mathbb{E}_{\pi^*(P,\tilde{P})} \left[ \sum\limits_{t = 0}^T C_t(s_t^z, z_t, \xi_{[t]}, \zeta_t) \right] - \mathbb{E}_{\tilde{\zeta}_{[T]}} \left[\sum\limits_{t = 0}^T C_t(\tilde{s}_t, \tilde{x}^*_t, \xi_{[t]}, \tilde{\zeta}_t) \right] \right| \nonumber \\
    && + \left| \mathbb{E}_{\tilde{\zeta}_{[T]}} \left[\sum\limits_{t = 0}^T C_t(\tilde{s}_t, \tilde{x}^*_t, \xi_{[t]}, \tilde{\zeta}_t) \right] - \mathbb{E}_{\zeta_{[T]}} \left[ \sum\limits_{t = 0}^T C_t(s_t, x_t^*, \xi_{[t]}, \zeta_t) \right] \right| \nonumber \\
    &\leq & \sum\limits_{t = 0}^{T - 1} (L_{C} L_S + L_{C} L_{X,t} + L_{C} + L_{t + 1} L_S + L_{t + 1} L_{X, t} + L_C + L_C \sum\limits_{k = t + 1}^T L_{X, k, t} \nonumber \\
    && + L_C \sum\limits_{k = 1}^{T - 1 - t} \sum\limits_{l = k + t + 1}^T L_S^k L_{X, l - k, t} + L_C \sum\limits_{k = 1}^{T - t} L_S^k) \dd_K(P_t, \tilde{P}_t) + L_C \dd_K(P_T, \tilde{P}_T). \nonumber \\
    &: = & \sum\limits_{t = 0}^T H_t \dd_K(P_t, \tilde{P}_t),
    \label{eq:diff-proj-opt}
    \end{eqnarray}
    where 
    \begin{eqnarray*}
    H_t &:= & L_{C} L_S + L_{C} L_{X,t} + L_{C} + L_{t + 1} L_S + L_{t + 1} L_{X, t} + L_C + L_C \sum\limits_{k = t + 1}^T L_{X, k, t} \nonumber \\
    && + L_C \sum\limits_{k = 1}^{T - 1 - t} \sum\limits_{l = k + t + 1}^T L_S^k L_{X, l - k, t} + L_C \sum\limits_{k = 1}^{T - t} L_S^k
    \end{eqnarray*}
    for $t = 0, 1, \cdots, T - 1$ and $H_T := L_C$. The second inequality is obtained by \eqref{eq:opt-value-omega} and \eqref{eq:optvalue-projection}.
    
    For every \(A\in\mathcal B(\mathbb R^{n_t})\), where
    \(\mathcal B(\mathbb R^{n_t})\) denotes the Borel \(\sigma\)-algebra on
    \(\mathbb R^{n_t}\), define
    \bgeqn 
    \mathsf K_t^{z,\pi^*}(A\mid s_t^z,z_{t-1},\xi_{[t]})
    :=
    \mathbb E_{\pi^*(P,\tilde P)}
    \left[
    \mathds{1}_A(z_t)\mid s_t^z,z_{t-1},\xi_{[t]}
    \right],
    \edeqn 
    where \(\mathds{1}_A\) denotes the indicator function of \(A\).
    Let ${\mathsf K^{z, \pi^*}} := \left(\mathsf K_1^{z,\pi^*},\ldots,\mathsf K_T^{z,\pi^*}\right)$ denote the decision kernels of randomized policy \(\bm z^{\pi^*}\), 
    $
    \bm z^{\pi^*}:=(z_0,\bm z_1,\ldots,\bm z_T)
    $
    be the randomized policy, where \(\bm z_t\) follows the distribution
    $
    \mathsf K_t^{z,\pi^*}
    (\cdot\mid s_t^z,z_{t-1},\xi_{[t]}).
    $
    By the definition of \(\Pi(P,\tilde P)\) and the measurable projections,
    \(\bm z^{\pi^*}\) is a measurable, feasible, and nonanticipative
    randomized policy for the original problem. 
    Observe that
    \begin{eqnarray}
    \mathbb E_{\pi^*(P,\tilde P)}\left[\sum_{t=0}^T C_t(s_t^z,z_t,\xi_{[t]},\zeta_t)\right]-\vt(P)
    &=&\mathbb  E_{\mathsf K^{z, \pi^*}} \left[\mathbb E_P\left[\sum_{t=0}^T C_t(s_t^{z},z_t,\xi_{[t]},\zeta_t)\right]-\vt(P)\right]\nonumber\\
    &\overset{\eqref{eq:growth-omega}}{\geq}&\mathbb E_{\mathsf K^{z, \pi^*}}\Big[ \beta\mathbb E_P\left[d(\bm z,\mathcal X^*(\zeta))^\nu\right]\Big] \nonumber \\
    &= &\beta\mathbb E_{\pi^*(P,\tilde P)}\Big[d(\bm z,\mathcal X^*(\zeta))^\nu\Big]\nonumber\\
    &\geq&\beta\left(\mathbb E_{\pi^*(P,\tilde P)}\left[d(\bm z,\mathcal X^*(\zeta))\right]\right)^\nu,
    \label{eq:growth-omega-randomized}
    \end{eqnarray}
    where the last inequality is due to Jensen's inequality for $\nu \geq 1$.
    Assume for the sake of a contradiction that
    \begin{eqnarray}
    \mathbb{E}_{\pi^*(P,\tilde{P})}\left[d(\bm z,\mathcal{X}^*(\zeta))\right]>\left(\frac{1}{\beta}\sum\limits_{t=0}^T H_t\dd_K(P_t,\tilde{P}_t)\right)^\frac{1}{\nu}.
    \end{eqnarray}
    Then by the growth condition \eqref{eq:growth-omega}, we obtain
    \begin{eqnarray}
    \mathbb{E}_{\pi^*(P,\tilde{P})}\left[\sum\limits_{t=0}^T C_t(s_t^z,z_t,\xi_{[t]},\zeta_t)\right]-\mathbb{E}_{\zeta}\left[\sum\limits_{t=0}^T C_t(s_t,x_t^*,\xi_{[t]},\zeta_t)\right]>\sum\limits_{t=0}^T H_t\dd_K(P_t,\tilde{P}_t),\nonumber
    \end{eqnarray}
    which leads to a contradiction to \eqref{eq:diff-proj-opt}. Thus
    \begin{eqnarray}
    \mathbb{E}_{\pi^*(P,\tilde{P})}\left[d(\bm z,\mathcal{X}^*(\zeta))\right]\leq\left(\frac{1}{\beta}\sum\limits_{t=0}^T H_t\dd_K(P_t,\tilde{P}_t)\right)^\frac{1}{\nu}.
    \label{eq:solu-proj-opt}
    \end{eqnarray}
    Combining \eqref{eq:omega-feasibleset} and \eqref{eq:solu-proj-opt}, we conclude that for any \(\tilde{\bm x}^*\in\mathcal{X}^*(\tilde{\zeta})\),
    \begin{eqnarray}
    &&\mathbb{E}_{\pi^*(P,\tilde{P})}\left[d(\tilde{\bm x}^*,\mathcal{X}^*(\zeta))\right]\nonumber\\
    &\leq&\mathbb{E}_{\pi^*(P,\tilde{P})}\left[d(\tilde{\bm x}^*,\bm z)+d(\bm z,\mathcal{X}^*(\zeta))\right]\leq\mathbb{E}_{\pi^*(P,\tilde{P})}\left[\mathbb{H}(\mathcal{X}(\tilde{\zeta}),\mathcal{X}(\zeta))+d(\bm z,\mathcal{X}^*(\zeta))\right]\nonumber\\
    &\leq&\sum\limits_{t=0}^T\left(\sum\limits_{k=t+1}^T L_{X,k,t}\right)\dd_K(P_t,\tilde{P}_t)+\left(\frac{1}{\beta}\sum\limits_{t=0}^T H_t\dd_K(P_t,\tilde{P}_t)\right)^\frac{1}{\nu}.\nonumber
    \end{eqnarray}
    Since \(\tilde{\bm x}^*\) is arbitrarily chosen from \(\mathcal{X}^*(\tilde{\zeta})\), we obtain \eqref{eq:stab-zeta-optsolu}.
\end{proof}

Theorem \ref{Theo-stab-endo} quantifies the changes in the optimal value and the optimal solution set of problem \eqref{eq:mixed-MSP-MDP} when the distributions of endogenous random variables are perturbed. These variations are controlled by a weighted sum of the Kantorovich metrics between the distributions before and after perturbations at individual stages. Since \(\zeta_t, 0 \leq t \leq T,\) are independent across stages, perturbing the distributions at all stages can be decomposed into perturbations at each stage. Therefore, as shown above, the quantitative stability results follow by summing the stagewise bounds.
The growth condition \eqref{eq:growth-omega} holds under mild conditions. 
For instance, it is satisfied when the objective function of problem \eqref{eq:mixed-MSP-MDP} is strongly convex w.r.t.~\(\bm x\). We demonstrate this in Proposition \ref{Prop-strongly-convex} in Appendix \ref{appendixproof}.

Kern et al.~\cite{kern2020first} investigated the first-order sensitivity of the value function in 
 MDPs w.r.t.~transition kernel perturbations. Their analyses rely on bounding functions and the Hadamard differentiability of the cost (reward) function w.r.t.~the transition kernel, and the perturbation is along a specific direction. Unlike Kern et al.~\cite{kern2020first} and usual MDP literature, we represent the state transition process in the form of transition functions and its randomness w.r.t.~endogenous random variables. This enables us to establish quantitative stability results for both the optimal value and the optimal solution set under arbitrary perturbations rather than perturbations along a specific direction. Moreover, in Theorem \ref{Theo-stab-endo}, we only need some fundamental assumptions on the objective function and transition functions, which are easily satisfied and verified in comparison to the Hadamard differentiability.
 
The perturbation of the endogenous random variable \(\zeta_t\) naturally induces a perturbation of the transition kernel. Indeed, for a given transition mapping
\(
    s_{t+1}=S_t^M(s_t,x_t,\xi_t,\zeta_t),
\)
if the law of \(\zeta_t\) is \(\mu_t\), then the induced transition kernel under given $\xi_t$ is
\(
    P_t^\mu(B\mid s_t,x_t,\xi_t)
    =
    \mu_t\big(\{z:S_t^M(s_t,x_t,\xi_t,z)\in B\}\big).
\)
Thus, replacing \(\mu_t\) by \(\tilde\mu_t\) leads to a perturbed kernel \(P_t^{\tilde\mu}\). For example, if
\(
    s_{t+1}=s_t+x_t+\zeta_t,
\)
then
\(
    P_t^\mu(B\mid s_t,x_t)=\mu_t(B-s_t-x_t),
\)
so changing the distribution of \(\zeta_t\) directly changes the transition kernel.

However, such perturbations form only a structured class of MDP kernel perturbations, since the transition mapping \(S_t^M\) is specified beforehand. Hence, they cannot represent arbitrary changes from \(P_t(\cdot\mid s,a)\) to \(Q_t(\cdot\mid s,a)\). On the other hand, classical MDP stability bounds often control value perturbations by worst-case kernel distances~\cite{muller1997value, kern2020first,zhou2024robustness} such as
\(
    \sup_{s,a}
    d\big(P_t(\cdot\mid s,a),Q_t(\cdot\mid s,a)\big).
\)
By contrast, our bound is expressed directly in terms of the distance between the primitive distributions \(\mu_t\) and \(\tilde\mu_t\), without taking a supremum over all state-action pairs. In this sense, although our perturbation class is narrower than arbitrary MDP kernel perturbations, the resulting bound avoids the worst-case amplification over \((s,a)\), thus enhancing the tractability, and provides a sharper description for random-input-induced perturbations in the integrated MSP--MDP framework.

\section{Quantitative stability with respect to exogenous uncertainty}

We now turn to analyzing problem  \eqref{eq:mixed-MSP-MDP}
when exogenous uncertainty $\xi$ is perturbed whereas endogenous uncertainty is unperturbed. 

\subsection{Holistic stability analysis}

We begin by considering
\begin{eqnarray} 
        \vt(Q) := \underset{\bm x \in \mathcal{X}(\xi)}{\min} \mathbb{E}_{\xi} \left[ \sum\limits_{t = 0}^T C_t(s_t, x_t, \xi_{[t]}, \zeta_t) \right]
        \label{eq:v-xi}
        \end{eqnarray} 
and its perturbation
    \begin{eqnarray}    
        \vt(\tilde{Q}) := \underset{\bm x \in \mathcal{X}(\tilde{\xi})}{\min} \mathbb{E}_{\tilde{\xi}} \left[ \sum\limits_{t = 0}^T C_t(s_t, x_t, \tilde{\xi}_{[t]}, \zeta_t) \right].
        \label{eq:v-xi-ptb}
    \end{eqnarray}
Similar to Section 4,
we write $\mathcal{X}(\xi)$
for the set of feasible policies to problem \eqref{eq:v-xi} 
and $\mathcal{X}(\tilde{\xi})$ 
for the set of feasible policies to problem \eqref{eq:v-xi-ptb}.
Unlike endogenous random variables, the distributions of exogenous random variables are intertemporally dependent. Therefore, we cannot investigate the stability stage by stage. In light of this, we 
first consider the changes in the optimal value and optimal solution set of problem \eqref{eq:mixed-MSP-MDP} when the whole data process \(\xi\) is perturbed to \(\tilde{\xi}\), and its distribution varies from \(Q\) to \(\tilde{Q}\). Let $\Pi(Q,\tilde Q)$ denote the set of all joint probability distributions with marginals $Q$ and $\tilde Q$ that are compatible with the stagewise information structures of both processes. Then we will extend the stability analysis to the general situation under stagewise distribution 
perturbations. 

To establish the quantitative stability, let $\mathscr{Q}_t = \mathscr{P}(\R^{m_{1, t}})$
denote the set of all probability measures in $\R^{m_{1, t}}$ and 
$Q_t, \tilde{Q}_t \in \mathscr{Q}_t$ be the probability measures of $\xi_t$ at stage $t$. Then we need to demonstrate the Lipschitz continuity of the feasible solution set under distribution perturbations. For \(t = 1, 2, \cdots, T\), the decision vector \(x_t\) at stage $t$ depends on \((s_t, x_{t-1}, \xi_{[t]})\). Therefore, if we consider the policy \(\bm x\) 
of problem \eqref{eq:mixed-MSP-MDP}, it would depend on \((s_0, \xi, \zeta)\). Since we are considering the perturbation of the data process \(\xi\) as a whole and here \( s_0 \) and \( \zeta \) are fixed, we use \(\mathcal{X}(\xi)\) to denote the feasible policy set of problem \eqref{eq:mixed-MSP-MDP}. 
We need the following assumption:
    
    \begin{assumption}[Local Lipschitz continuity of \(C_t\), \(g_{t, i}\) and \(S^M_t\)]
    For \(t = 1, 2, \cdots, T\) and any $\xi = (\xi_1, \xi_2, \cdots, \xi_T)$,
    
    \noindent $\mathrm{(C^{\xi})}$ \(C_t(s_t, x_t, \xi_{[t]}, \zeta_t)\) is locally Lipschitz continuous in \((s_t, x_t, \xi_{[t]})\) with
    modulus \(L_{C,t}(\xi_{[t]}, \tilde{\xi}_{[t]})\)
    which is continuous and 
    symmetric in $\xi_{[t]}$ and $ \tilde{\xi}_{[t]}$, and 
 bounded below by a positive constant;
    
    \noindent $\mathrm{(S^{\xi})}$ \(S^M_{t - 1}(s_{t - 1}, x_{t - 1}, \xi_{t - 1}, \zeta_{t - 1})\) is locally Lipschitz continuous 
    in \((s_{t - 1}, x_{t - 1}, \xi_{t - 1})\) with modulus $L_{S,t - 1}(\xi_{[t - 1]}, \tilde{\xi}_{[t - 1]})$ which is continuous,
    symmetric in $\xi_{[t]}$ and $ \tilde{\xi}_{[t]}$, and bounded below by a positive constant;
    
    \noindent $\mathrm{(G^{\xi})}$ \(g_{t,i}(s_t, x_t, x_{t-1}, \xi_{[t]}), i \in I_t\) is locally Lipschitz continuous 
    in \((s_t, x_{t-1}, \xi_{[t]})\) with modulus $L_{g,t}(\xi_{[t]}, \tilde{\xi}_{[t]})$ which is continuous,
    symmetric in $\xi_{[t]}$ and $ \tilde{\xi}_{[t]}$, and bounded below by a positive constant.
    \label{assumption8}
    \end{assumption}

Condition ($S^{\xi}$) is satisfied when 
$S^M_t$ is locally Lipschitz continuous in $(s_t, x_t)$
and globally Lipschitz continuous in $\xi_t$ uniformly 
w.r.t.~$(s_t, x_t)$. Similar comment applies to ($G^{\xi}$) and ($C^{\xi}$). For simplicity, we define $L_{S,t}(\xi, \tilde{\xi}) :=L_{S,t}(\xi_{[t]}, \tilde{\xi}_{[t]})$,
$
L_{g, t}(\xi, \tilde{\xi}) := L_{g, t}(\xi_{[t]}, \tilde{\xi}_{[t]})
$ and $L_{C, t}(\xi, \tilde{\xi}) := L_{C, t}(\xi_{[t]}, \tilde{\xi}_{[t]})$. Further, let \(L_S(\xi, \tilde{\xi}) := \underset{t = 0, 1, \cdots, T - 1}{\max} L_{S,t}(\xi, \tilde{\xi})\)
, \(L_g(\xi, \tilde{\xi}) := \underset{t \in \{1, 2, \cdots, T\}}{\max} \left\{ L_{g,t}(\xi, \tilde{\xi}) \right\}\) and \(L_C(\xi, \tilde{\xi}) := \underset{t \in \{1, 2, \cdots, T\}}{\max} \left\{ L_{C, t}(\xi, \tilde{\xi}) \right\}\). We assume, without loss of generality, all the Lipschitz modulus hereinafter in this part are integrable.
   
As a preparation for the later stability analysis, we show the Lipschitz continuity of the feasible solution set 
mapping 
in 
terms of the Hausdorff distance. As only the distribution of $\xi$ is perturbed here, we omit the expectation w.r.t.~$\zeta$ in what follows for brevity. 
\begin{proposition}[Lipschitz continuity of the feasible set mapping]
Suppose: (a) Assumptions~\ref{Assu:welldefinedness}- \ref{Assu:compactness},
\ref{assumption7} and Assumption~\ref{assumption8} ($G^{\xi}$),($S^{\xi}$) 
hold;
(b) for $t = 1, 2, \cdots, T$ and for each fixed \((x_{t - 1}, \xi_{[t]})\), \(g_{t,i}(s_t, x_t, x_{t-1}, \xi_{[t]}), i \in I_t\) is convex 
in \((s_t, x_t)\). 
Then the set-valued mapping \(\mathcal{X}(\xi)\) is Lipschitz continuous 
in the following sense
\begin{eqnarray}
\mathbb{E}_{\pi(Q, \tilde{Q})} \left[ \mathbb{H}(\mathcal{X}(\xi), \mathcal{X}(\tilde{\xi})) \right] \leq \mathbb{E}_{\pi(Q, \tilde{Q})} \left[ \left(\underset{t \in \{1, 2, \cdots, T\}}{\max} L_{X, t}(\xi, \tilde{\xi})\right)
\Vert \xi - \tilde{\xi} \Vert \right] 
\label{eq:stab-feasibleset-xi}
\end{eqnarray}
 for any $\pi(Q, \tilde{Q}) \in \Pi(Q, \tilde{Q})$, where  $L_{X, t}(\xi, \tilde{\xi})$ is specified in \eqref{eq:L_Xt1} and \eqref{eq:L_Xt2}. If in addition, $L_g({\xi, \tilde{\xi}}) := L_g \max\{1, \Vert \xi \Vert, \Vert \tilde{\xi} \Vert\}$, $L_S({\xi, \tilde{\xi}}) := L_S \max\left\{1, \Vert \xi \Vert, \Vert \tilde{\xi} \Vert\right\}$, and $\xi$ has finite $2T$-th moment, then  
    \begin{eqnarray}
        \mathbb{E}_{\pi^*(Q, \tilde{Q})} \left[\mathbb{H}(\mathcal{X}(\xi), \mathcal{X}(\tilde{\xi})) \right] \leq L_X \Delta_{2T}(Q, \tilde{Q}),
        \label{eq:Lip-feasi-W}
    \end{eqnarray}
    where $\pi^*(Q, \tilde{Q})$ is the optimal coupling of $Q$ and $\tilde{Q}$ defined as in \eqref{eq:def-DeltaL} with $L(\xi, \tilde{\xi}) = \max\{1, \| \xi \|, \| \tilde{\xi} \| \}$, $L_X := \underset{t = 1, 2, \cdots, T}{\max}L_{X, t}$ and $L_{X, t}$ is recursively defined in \eqref{eq:Lip-feasi-FM-LX}.
    \label{Prop-stab-exo-feas}
\end{proposition}
\begin{proof}
    Since the decision \(x_0\) at the initial stage is chosen from a fixed feasible solution set \(\mathcal{X}_0\), we assume without loss of generality that \(x_0\) is a fixed decision, 
    we proceed with the proof 
    from stage \(1\). 
    For \(t = 1\), we show that there exists a positive coefficient \(L_{X, 1}(\xi, \tilde{\xi})\) such that
    \begin{eqnarray}
    \mathbb{E}_{\pi(Q, \tilde{Q})}\left[\mathbb{H}
    \left(\mathcal{X}_1(s_1, x_{0}, \xi_1), \mathcal{X}_1(s_1, x_0, \tilde{\xi}_{1})\right)\right] \leq  \mathbb{E}_{\pi(Q, \tilde{Q})} \left[ 
    L_{X, 1}(\xi, \tilde{\xi}) \Vert \xi_1 - \tilde{\xi}_{1} \Vert \right].  \nonumber
    \end{eqnarray}
Since the Slater condition holds for any \(\xi\), there exists an \(\bar{x}_1 \in \mathcal{X}_1(s_1, x_0, \xi_1)\) such that
    \begin{eqnarray}
        g_{1, i}(s_1, \bar{x}_1, x_0, \xi_1) \leq -\rho, \quad i \in I_1. \nonumber
    \end{eqnarray}
    Meanwhile, for any \(\tilde{x}_1 \in \mathcal{X}_1(s_1, x_0, \tilde{\xi}_{1})\), we have
    \begin{eqnarray}
    g_{1, i}(s_1, \tilde{x}_1, x_0, \tilde{\xi}_{1}) \leq 0, \quad i \in I_1.
    \label{eq:stab-exo-tildex1}
    \end{eqnarray}
    By \eqref{eq:stab-exo-tildex1} and Assumption \ref{assumption8} ($G^{\xi}$), 
    we can derive
    \begin{eqnarray}
    g_{1, i}(s_1, \tilde{x}_1, x_0, \xi_1) \leq L_g({\xi, \tilde{\xi}}) \Vert \xi_1 - \tilde{\xi}_{1} \Vert, \quad i \in I_1. \nonumber
    \end{eqnarray}
    Let 
    \[z_1 = \frac{L_g({\xi, \tilde{\xi}}) \Vert \xi_1 - \tilde{\xi}_{1} \Vert}{\rho + L_g({\xi, \tilde{\xi}}) \Vert \xi_1 - \tilde{\xi}_{1} \Vert} \bar{x}_1 + \frac{\rho}{\rho + L_g({\xi, \tilde{\xi}}) \Vert \xi_1 - \tilde{\xi}_{1} \Vert} \tilde{x}_1.\]
    By the convexity of \(g_{1, i}\),
    \begin{eqnarray}
    g_{1, i}(s_1, z_1, x_0, \xi_1) \leq 0, \quad i \in I_1. \nonumber
    \end{eqnarray}
    Combining this, the definition of \(z_1\) and Assumption \ref{Assu:compactness},
    we have
    \begin{eqnarray}
    d(z_1, \tilde{x}_1) = \frac{L_g({\xi, \tilde{\xi}}) \Vert \xi_1 - \tilde{\xi}_{1} \Vert}{\rho} d(z_1, x_1) \leq \frac{L_g({\xi, \tilde{\xi}}) A}{\rho} \Vert \xi_1 - \tilde{\xi}_{1} \Vert, \nonumber
    \end{eqnarray}
    where \(A\) is defined as in Theorem \ref{Theo-Lipschitz}. Consequently
    \begin{eqnarray}
        d(\tilde{x}_1, \mathcal{X}_1(s_1, x_0, {\xi}_{1})) \leq d(\tilde{x}_1, z_1) \leq \frac{L_g({\xi, \tilde{\xi}}) A}{\rho} \Vert \xi_1 - \tilde{\xi}_{1} \Vert. \nonumber
    \end{eqnarray}
    Since \(\tilde{x}_1\) is arbitrarily chosen from the feasible set \(\mathcal{X}_1(s_1, x_0, \tilde{\xi}_{1})\), it follows that
    \begin{eqnarray}
    \mathbb{D}(\mathcal{X}_1(s_1, x_0, \tilde{\xi}_{1}), \mathcal{X}_1(s_1, x_0, {\xi}_{1})) &\leq& \frac{L_g({\xi, \tilde{\xi}}) A}{\rho} \Vert \xi_1 - \tilde{\xi}_{1} \Vert. \nonumber
    \end{eqnarray}
   Likewise, we can show there exists an \(\bar{\tilde{x}}_1 \in \mathcal{X}_1(s_1, x_0, \tilde{\xi}_{1})\) 
   which satisfies the Slater condition uniformly, 
   and for any \(x_1 \in \mathcal{X}_1(s_1, x_0, \xi_1)\), 
    \[
    \tilde{z}_1 = \frac{L_g({\xi, \tilde{\xi}}) \Vert \xi_1 - \tilde{\xi}_{1} \Vert}{\rho + L_g({\xi, \tilde{\xi}}) \Vert \xi_1 - \tilde{\xi}_{1} \Vert} \bar{\tilde{x}}_1 + \frac{\rho}{\rho + L_g({\xi, \tilde{\xi}}) \Vert \xi_1 - \tilde{\xi}_{1} \Vert} {x}_1
    \]
    satisfies \(d(\tilde{z}_1, \mathcal{X}_1(s_1, x_0, \xi_1)) \leq \frac{L_g({\xi, \tilde{\xi}}) A}{\rho} \Vert \xi_1 - \tilde{\xi}_{1} \Vert\).

    The above two results ensure that the Hausdorff distance between \(\mathcal{X}_1(s_1, x_0, \xi_1)\) and \(\mathcal{X}_1(s_1, x_0, \tilde{\xi}_{1})\) satisfies
    \begin{eqnarray}
    && \mathbb{H}(\mathcal{X}_1(s_1, x_0, \xi_1), \mathcal{X}_1(s_1, x_0, \tilde{\xi}_{1})) \nonumber \\
    & =& \text{max } \left\{\mathbb{D}(\mathcal{X}_1(s_1, x_0, \xi_1), \mathcal{X}_1(s_1, x_0, \tilde{\xi}_{1})), \mathbb{D}(\mathcal{X}_1(s_1, x_0, \tilde{\xi}_{1}), \mathcal{X}_1(s_1, x_0, \xi_1))\right\} \nonumber \\
    & \leq& \frac{L_g({\xi, \tilde{\xi}}) A}{\rho} \Vert \xi_1 - \tilde{\xi}_{1} \Vert := L_{X, 1}({\xi, \tilde{\xi}}) \Vert \xi_1 - \tilde{\xi}_{1} \Vert.
    \label{eq:feasi-Lip-1}
    \end{eqnarray}    
    As the integrated MSP-MDP model includes state transition equations, in addition to considering the Lipschitz continuity of the feasible solution set, we also need to consider the Lipschitz continuity of state transitions. At stage 1, the state variable \(s_1\) is not affected by 
    exogenous random variables. Select any feasible solution \(x_1 \in \mathcal{X}_1(s_1, x_0, \xi_1)\), and let \(\tilde{y}_1\) be the projection of \(x_1\) onto \(\mathcal{X}_1(s_1, x_0, \tilde{\xi}_{1})\). Then, according to the state transition mapping, we have
    \begin{eqnarray} 
    {s}_2 = S^M_1 (s_{1}, x_{1}, {\xi}_{1}, \zeta_1),  \tilde{s}_2 = S^M_1 (s_{1}, \tilde{y}_{1}, \tilde{\xi}_{1}, \zeta_1). \nonumber
    \end{eqnarray}
    Assumption \ref{assumption8} and the established Lipschitz continuity of 
    \(\mathcal{X}_1(\cdot)\) ensure that
    \begin{eqnarray}
    \Vert s_2 - \tilde{s}_2 \Vert & = &\Vert S^M_1(s_1, x_1, \xi_1, \zeta_1) - S^M_1({s}_1, \tilde{y}_1, \tilde{\xi}_{1}, \zeta_1) \Vert \leq L_S({\xi, \tilde{\xi}})(\Vert x_1 - \tilde{y}_1 \Vert + \Vert \xi_1 - \tilde{\xi}_{1} \Vert)   \nonumber\\
    &\leq& L_S({\xi, \tilde{\xi}})(\mathbb{H}(\mathcal{X}_1(s_1, x_0, \xi_1), \mathcal{X}_1(s_1, x_0, \tilde{\xi}_{1})) + \Vert \xi_1 - \tilde{\xi}_{1} \Vert)\nonumber \\ 
    &\leq& L_S({\xi, \tilde{\xi}})(L_{X, 1}({\xi, \tilde{\xi}}) + 1) \Vert \xi_1 - \tilde{\xi}_{1} \Vert. 
    \label{eq:feasi-Lip-s2}
    \end{eqnarray}
    It implies the Lipschitz continuity of the state variable at stage 2. Similarly, we denote the Lipschitz modulus here as \(l_{s, 2}({\xi, \tilde{\xi}}) = L_S({\xi, \tilde{\xi}})(L_{X, 1}({\xi, \tilde{\xi}}) + 1)\).
    
    For \(2 \leq t \leq T\), assume that for any \(k < t\), the following inequality holds:
    \[
    \mathbb{H}(\mathcal{X}_k(s_k, x_{k-1}, \xi_{[k]}), \mathcal{X}_k(\tilde{s}_k, \tilde{y}_{k-1}, \tilde{\xi}_{[k]})) \leq L_{X, k}({\xi, \tilde{\xi}}) \Vert \xi_{[k]} - \tilde{\xi}_{[k]} \Vert.
    \]
    Here, \(s_k = S^M_{k - 1}(s_{k - 1}, x_{k - 1}, \xi_{k - 1}, \zeta_{k - 1})\), \(\tilde{s}_k = S^M_{k - 1}(\tilde{s}_{k - 1}, \tilde{y}_{k - 1}, \tilde{\xi}_{k - 1}, \zeta_{k - 1})\), and \(\tilde{y}_{k}\) is the projection of \({x}_{k}\) onto \(\mathcal{X}_{k}(\tilde{s}_k, \tilde{x}_{k - 1}, \tilde{\xi}_{[k]})\). Similarly, assume that for any \(k \leq t\), we have
    \[
    \Vert s_k - \tilde{s}_k \Vert \leq l_{s, k}({\xi, \tilde{\xi}}) \Vert \xi_{[k - 1]} - \tilde{\xi}_{[k - 1]} \Vert \leq l_{s, k}({\xi, \tilde{\xi}}) \Vert \xi_{[k]} - \tilde{\xi}_{[k]} \Vert.
    \]

    According to Assumption \ref{assumption7}, there exists an \(\bar{x}_t\) such that for any \({\xi_{[t]}}\) and any feasible solution \(\tilde{x}_t \in \mathcal{X}_t(\tilde{s}_t, \tilde{y}_{t - 1}, \tilde{\xi}_{[t]})\), the following holds:
    \begin{eqnarray}
    && g_{t,i}(s_t, \bar{x}_t, x_{t-1}, \xi_{[t]}) \leq -\rho, \quad i \in I_t. \nonumber \\
    && g_{t,i}(\tilde{s}_t, \tilde{x}_t, \tilde{y}_{t-1}, \tilde{\xi}_{[t]}) \leq 0, \quad i \in I_t. \nonumber
    \end{eqnarray}
    Similar to the proof for \(t = 1\), we can then 
    obtain 
    \begin{eqnarray}
    g_{t,i}(s_t, \tilde{x}_t, x_{t-1}, \xi_{[t]}) & \leq& L_g({\xi, \tilde{\xi}})(\Vert s_t - \tilde{s}_t \Vert + \Vert x_{t-1} - \tilde{y}_{t-1} \Vert + \Vert \xi_{[t]} - \tilde{\xi}_{[t]} \Vert).    \nonumber    \\
    & \leq& L_g({\xi, \tilde{\xi}})(l_{s, t}({\xi, \tilde{\xi}}) + L_{X, t-1}({\xi, \tilde{\xi}}) + 1) \Vert \xi_{[t]} - \tilde{\xi}_{[t]} \Vert, \quad i \in I_t. \nonumber
    \end{eqnarray}
    Let
    \begin{eqnarray*}
    z_t &= & \frac{\rho}{\rho + L_g({\xi, \tilde{\xi}})(l_{s, t}({\xi, \tilde{\xi}}) + L_{X, t-1}({\xi, \tilde{\xi}}) + 1) \Vert \xi_{[t]} - \tilde{\xi}_{[t]} \Vert} \tilde{x}_t \nonumber \\
    && + \frac{L_g({\xi, \tilde{\xi}})(l_{s, t}({\xi, \tilde{\xi}})+ L_{X, t-1}({\xi, \tilde{\xi}}) + 1) \Vert \xi_{[t]} - \tilde{\xi}_{[t]} \Vert}{\rho + L_g({\xi, \tilde{\xi}})(l_{s, t}({\xi, \tilde{\xi}}) + L_{X, t-1}({\xi, \tilde{\xi}}) + 1) \Vert \xi_{[t]} - \tilde{\xi}_{[t]} \Vert} \bar{x}_t.
    \end{eqnarray*}
    Using the convexity of \(g_{t, i}\), it is easy to derive that
    \begin{eqnarray}
    &&g_{t,i}(s_t, z_t, x_{t-1}, \xi_{[t]}) \nonumber \\
    &\leq & \frac{\rho g_{t, i}(s_t, \tilde{x}_t, x_{t - 1}, \xi_{[t]})}{\rho + L_g({\xi, \tilde{\xi}})(l_{s, t}({\xi, \tilde{\xi}}) + L_{X, t-1}({\xi, \tilde{\xi}}) + 1) \Vert \xi_{[t]} - \tilde{\xi}_{[t]} \Vert} \nonumber \\
    && - \frac{L_g({\xi, \tilde{\xi}})(l_{s, t}({\xi, \tilde{\xi}}) + L_{X, t-1}({\xi, \tilde{\xi}}) + 1) \Vert \xi_{[t]} - \tilde{\xi}_{[t]} \Vert g_{t, i}(s_t, \bar{x}_t, x_{t - 1}, \xi_{[t]})}{\rho + L_g({\xi, \tilde{\xi}})(l_{s, t}({\xi, \tilde{\xi}}) + L_{X, t-1}({\xi, \tilde{\xi}}) + 1) \Vert \xi_{[t]} - \tilde{\xi}_{[t]} \Vert} = 0.
    \end{eqnarray}
    Then we have
    \[
    \Vert \tilde{x}_t - z_t \Vert \leq \frac{A L_g({\xi, \tilde{\xi}})(l_{s, t}({\xi, \tilde{\xi}}) + L_{X, t-1}({\xi, \tilde{\xi}}) + 1)}{\rho} \Vert \xi_{[t]} - \tilde{\xi}_{[t]} \Vert,
    \]
    and thus 
    \[
    d(\tilde{x}_t, \mathcal{X}_t(s_t, x_{t-1}, \xi_{[t]})) \leq \frac{A L_g({\xi, \tilde{\xi}})(l_{s, t}({\xi, \tilde{\xi}}) + L_{X, t-1}({\xi, \tilde{\xi}}) + 1)}{\rho} \Vert \xi_{[t]} - \tilde{\xi}_{[t]} \Vert.
    \]
    Since \(\tilde{x}_t\) is chosen arbitrarily, we have
    \begin{eqnarray}
    \mathbb{D}(\mathcal{X}_t(\tilde{s}_t, \tilde{y}_{t-1}, \tilde{\xi}_{[t]}), \mathcal{X}_t(s_t, x_{t-1}, \xi_{[t]})) \leq \frac{A L_g({\xi, \tilde{\xi}})(l_{s, t}({\xi, \tilde{\xi}}) + L_{X, t-1}({\xi, \tilde{\xi}}) + 1)}{\rho} \Vert \xi_{[t]} - \tilde{\xi}_{[t]} \Vert. \nonumber
    \end{eqnarray}
    Similarly, it can be shown that
    \[
    \mathbb{D}(\mathcal{X}_t(s_t, x_{t-1}, \xi_{[t]}), \mathcal{X}_t(\tilde{s}_t, \tilde{y}_{t-1}, \tilde{\xi}_{[t]})) \leq \frac{A L_g({\xi, \tilde{\xi}})(l_{s, t}({\xi, \tilde{\xi}}) + L_{X, t-1}({\xi, \tilde{\xi}}) + 1)}{\rho} \Vert \xi_{[t]} - \tilde{\xi}_{[t]} \Vert.
    \]
    Combining the two inequalities above, we obtain
    \begin{eqnarray}
    \mathbb{H}(\mathcal{X}_t(\tilde{s}_t, \tilde{y}_{t-1}, \tilde{\xi}_{[t]}), \mathcal{X}_t(s_t, x_{t-1}, \xi_{[t]})) \leq L_{X, t}({\xi, \tilde{\xi}}) \Vert \xi_{[t]} - \tilde{\xi}_{[t]} \Vert. \nonumber
    \end{eqnarray}
    Also, we denote the Lipschitz modulus here as \(L_{X, t}({\xi, \tilde{\xi}}) = \frac{A L_g({\xi, \tilde{\xi}})(l_{s, t}({\xi, \tilde{\xi}}) + L_{X, t-1}({\xi, \tilde{\xi}}) + 1)}{\rho}\).
    Meanwhile, we consider the state variable. 
    Let \(\tilde{y}_t\) be the projection of \(x_t \in \mathcal{X}_t(s_t, x_{t - 1}, \xi_{[t]})\) onto \(\mathcal{X}_t(\tilde{s}_t, \tilde{y}_{t-1}, \tilde{\xi}_{[t]})\). Then,
    \begin{eqnarray}
    \Vert s_{t + 1} - \tilde{s}_{t + 1} \Vert & =& \Vert S^M_t(s_t, x_t, \xi_{t}, \zeta_{t}) - S^M_t(\tilde{s}_t, \tilde{y}_t, \tilde{\xi}_{t}, \zeta_{t}) \Vert \nonumber \\
    & \leq& L_S({\xi, \tilde{\xi}})(\Vert s_t - \tilde{s}_t \Vert + \Vert x_t - \tilde{y}_t \Vert + \Vert \xi_{[t]} - \tilde{\xi}_{[t]} \Vert) \nonumber \\
    & \leq& L_S({\xi, \tilde{\xi}})(l_{s, t}({\xi, \tilde{\xi}}) + L_{X, t}({\xi, \tilde{\xi}}) + 1) \Vert \xi_{[t]} - \tilde{\xi}_{[t]} \Vert. \nonumber
    \end{eqnarray}
    Let \(l_{s, t + 1}({\xi, \tilde{\xi}}) = L_S({\xi, \tilde{\xi}})(l_{s, t}({\xi, \tilde{\xi}}) + L_{X, t}({\xi, \tilde{\xi}}) + 1)\) and we establish the Lipschitz continuity of the state variable at stage $t$. At this point, the inductive proof is completed. 
    
    We have shown that for any \(t = 1, 2, \cdots, T\),
    \begin{eqnarray}
    \mathbb{H}(\mathcal{X}_t(\tilde{s}_t, \tilde{y}_{t-1}, \tilde{\xi}_{[t]}), \mathcal{X}_t(s_t, x_{t-1}, \xi_{[t]})) \leq L_{X, t}({\xi, \tilde{\xi}}) \Vert \xi_{[t]} - \tilde{\xi}_{[t]} \Vert. \label{eq:nonexpected-Lip-feasible}
    \end{eqnarray}
    Thus,
    \begin{eqnarray}
     \mathbb{E}_{\pi(Q, \tilde{Q})} \left[ \mathbb{H}(\mathcal{X}_t(\tilde{s}_t, \tilde{y}_{t-1}, \tilde{\xi}_{[t]}), \mathcal{X}_t(s_t, x_{t-1}, \xi_{[t]})) \right] \leq \mathbb{E}_{\pi(Q, \tilde{Q})} \left[ L_{X, t}({\xi, \tilde{\xi}})  \Vert \xi_{[t]} - \tilde{\xi}_{[t]} \Vert \right]. \nonumber
    \end{eqnarray}
    
    Combining the following recursive equations for $L_{X, t}({\xi, \tilde{\xi}})$ 
    and $l_{s, t}({\xi, \tilde{\xi}})$
    \begin{eqnarray}
        &&l_{s, t + 1}({\xi, \tilde{\xi}}) =  L_S({\xi, \tilde{\xi}})(l_{s, t}({\xi, \tilde{\xi}}) + L_{X, t}({\xi, \tilde{\xi}}) + 1), \nonumber \\ 
        &&L_{X, t}({\xi, \tilde{\xi}}) 
        = \frac{A L_g({\xi, \tilde{\xi}})(l_{s, t}({\xi, \tilde{\xi}}) + L_{X, t-1}({\xi, \tilde{\xi}}) + 1)}{\rho}, \label{eq:Lip-feasi-induction-t}
    \end{eqnarray}
    we obtain
    \begin{eqnarray}
        L_{X, t}({\xi, \tilde{\xi}}) &= & \left(\frac{A L_g({\xi, \tilde{\xi}})L_S({\xi, \tilde{\xi}})}{\rho} + L_S({\xi, \tilde{\xi}}) + \frac{A L_g({\xi, \tilde{\xi}})}{\rho}\right) L_{X, t - 1}({\xi, \tilde{\xi}}) \nonumber \\
        && - \frac{A L_g({\xi, \tilde{\xi}})L_S({\xi, \tilde{\xi}})}{\rho} L_{X, t - 2}({\xi, \tilde{\xi}}) + \frac{A L_g({\xi, \tilde{\xi}})}{\rho}. \nonumber
    \end{eqnarray}
    Since $L_{X, 0} = 0$ and $L_{X, 1}({\xi, \tilde{\xi}}) = \frac{A L_g({\xi, \tilde{\xi}})}{\rho}$, we can derive a closed-form expression for $L_{X, t}({\xi, \tilde{\xi}})$
    when $1 - \frac{A L_g({\xi, \tilde{\xi}})}{\rho} - L_S({\xi, \tilde{\xi}}) \neq 0$. Concretely,
    \begin{eqnarray}
    L_{X, t}({\xi, \tilde{\xi}}) &= & \frac{\frac{A L_g({\xi, \tilde{\xi}})}{\rho}}{\left(1 - \frac{A L_g({\xi, \tilde{\xi}})}{\rho} - L_S({\xi, \tilde{\xi}})\right)(r_1 - r_2)} 
    \left( (r_2 - 1) r_1^{t+1} - (r_1 - 1) r_2^{t+1} \right) \nonumber \\
    &&+ \frac{\frac{A L_g({\xi, \tilde{\xi}})}{\rho}}{1 - \frac{A L_g({\xi, \tilde{\xi}})}{\rho} - L_S({\xi, \tilde{\xi}})} , \label{eq:L_Xt1}
    \end{eqnarray}
    where $ r_1 $ and $ r_2 $ are the roots of the characteristic equation
    \begin{eqnarray}
        r^2 - \left( \frac{A L_g({\xi, \tilde{\xi}})}{\rho} L_S({\xi, \tilde{\xi}}) + \frac{A L_g({\xi, \tilde{\xi}})}{\rho} + L_S({\xi, \tilde{\xi}}) \right) r + \frac{A L_g({\xi, \tilde{\xi}})}{\rho} L_S({\xi, \tilde{\xi}}) = 0. \nonumber
    \end{eqnarray}
    When $1 - \frac{A L_g({\xi, \tilde{\xi}})}{\rho} - L_S({\xi, \tilde{\xi}}) = 0$, we have
    \begin{eqnarray}
        \hspace{-0.5em} L_{X, t}({\xi, \tilde{\xi}}) = \frac{ (\frac{A L_g({\xi, \tilde{\xi}})}{\rho})^2 L_S({\xi, \tilde{\xi}}) }{ \left(1 - \frac{A L_g({\xi, \tilde{\xi}})}{\rho} L_S({\xi, \tilde{\xi}}) \right)^2 } \left( \left( \frac{A L_g({\xi, \tilde{\xi}})}{\rho} L_S({\xi, \tilde{\xi}}) \right)^t - 1 \right) + \frac{ \frac{A L_g({\xi, \tilde{\xi}})}{\rho} t }{ 1 - \frac{A L_g({\xi, \tilde{\xi}})}{\rho} L_S({\xi, \tilde{\xi}}) }. \label{eq:L_Xt2}
    \end{eqnarray}
    
    With the Lipschitz continuity of the feasible solution sets at individual stages, we can now consider the corresponding properties for the feasible policy. 
    Let \(\bm x = (x_0, x_1, x_2, \cdots, x_T)\) be a realization of the feasible policy for the original problem \eqref{eq:mixed-MSP-MDP} and \(\tilde{\bm y} = (x_0, \tilde{y}_1, \tilde{y}_2, \cdots, \tilde{y}_T)\) be the feasible policy for the problem \eqref{eq:mixed-MSP-MDP} under the perturbed process $\tilde{\xi}$, here \(\tilde{y}_t\) is the projection of \(x_t\) onto \(\mathcal{X}_t(\tilde{s}_t, \tilde{y}_{t - 1}, \tilde{\xi}_{[t]})\). If we define
    \[
    \Vert \bm x \Vert = \underset{t \in \{1, 2, \cdots, T\}}{\max} \mathbb{E}\left[ \Vert x_t \Vert \right],
    \]
    then based on the arguments above, we have
    \[
    d(x_t(\xi_{[t]}), \tilde{y}_t(\tilde{\xi}_{[t]})) \leq L_{X, t}({\xi, \tilde{\xi}}) \Vert \xi - \tilde{\xi} \Vert.
    \]
    Thus,
    \begin{eqnarray}
    \mathbb{E}_{\pi(Q, \tilde{Q})}[\Vert \bm x - \tilde{\bm y} \Vert] = \underset{t \in \{1, 2, \cdots, T\}}{\max} \mathbb{E}_{\pi(Q, \tilde{Q})} \left[ \Vert x_t - \tilde{y}_t \Vert \right] \leq \mathbb{E}_{\pi(Q, \tilde{Q})} \left[ \left( \underset{t \in \{1, 2, \cdots, T\}}{\max} L_{X, t}({\xi, \tilde{\xi}})\right) \Vert \xi - \tilde{\xi} \Vert \right]. \nonumber
    \end{eqnarray}
    In summary, when the stochastic process \(\xi\) is perturbed to \(\tilde{\xi}\), the expected Hausdorff distance between the feasible sets before and after the perturbation is Lipschitz continuous w.r.t.~\(\xi\), and we have
    \begin{eqnarray}
    \mathbb{E}_{\pi(Q, \tilde{Q})} \left[ \mathbb{H}(\mathcal{X}(\xi), \mathcal{X}(\tilde{\xi})) \right] \leq \mathbb{E}_{\pi(Q, \tilde{Q})} \left[ \left(\underset{t \in \{1, 2, \cdots, T\}}{\max}  L_{X, t}({\xi, \tilde{\xi}})\right) \Vert \xi - \tilde{\xi} \Vert\right]. \nonumber
    \end{eqnarray}
    
    Next, we prove \eqref{eq:Lip-feasi-W}. By \eqref{eq:feasi-Lip-1}, 
    \begin{eqnarray}
        \mathbb{H}(\mathcal{X}_1(s_1, x_0, \xi_1), \mathcal{X}_1(s_1, x_0, \tilde{\xi}_1)) &\leq & \frac{A L_g}{\rho} \max\left\{1, \Vert \xi \Vert, \Vert \tilde{\xi} \Vert\right\}\|\xi - \tilde{\xi}\| \nonumber \\
        &:= & L_{X, 1}\max\left\{1, \Vert \xi \Vert, \Vert \tilde{\xi} \Vert\right\} \|\xi - \tilde{\xi}\|. 
        \label{eq:Lip-feasi-1-FM}
    \end{eqnarray}
    It is known from \eqref{eq:feasi-Lip-s2} that 
    \begin{eqnarray}
        \|s_2 - \tilde{s}_2\| &\leq & L_S \max\left\{1, \|\xi\|, \Vert \tilde{\xi} \Vert\right\} (1 + L_{X, 1}\max\{1, \|\xi\|, \Vert \tilde{\xi} \Vert\})\|\xi - \tilde{\xi}\| \nonumber \\
        &\leq & L_S (1 + L_{X, 1}) \max\left\{1, \|\xi\|^2, \|\tilde{\xi}\|^2\right\} \|\xi - \tilde{\xi}\| \nonumber \\
        &:= & l_{s, 2} \max\left\{1, \|\xi\|^2, \|\tilde{\xi}\|^2\right\} \|\xi - \tilde{\xi}\|.
    \end{eqnarray}
    Analogous to the induction argument in \eqref{eq:Lip-feasi-induction-t}, we have
    \begin{eqnarray}
        L_{X, t} = \frac{1}{\rho} A L_g (l_{s, t} + L_{X, t - 1} + 1), l_{s, t + 1} = L_S(L_{X, t} + l_{s, t} + 1). 
        \label{eq:Lip-feasi-FM-LX}
    \end{eqnarray}
    By \eqref{eq:stab-feasibleset-xi},
    \begin{eqnarray}
        \mathbb{E}_{\pi^*(Q, \tilde{Q})} \left[\mathbb{H}(\mathcal{X}(\xi), \mathcal{X}(\tilde{\xi})) \right] &\leq & \left( \underset{t = 1, 2, \cdots, T}{\max}L_{X, t} \right) \mathbb{E}_{\pi^*(Q, \tilde{Q})} \left[ \max\left\{1, \|\xi\|^{2T - 1}, \|\tilde{\xi}\|^{2T - 1}\right\} \|\xi - \tilde{\xi}\| \right] \nonumber \\
        &:= & L_X \mathbb{E}_{\pi^*(Q, \tilde{Q})} \left[ \max\left\{1, \|\xi\|^{2T - 1}, \|\tilde{\xi}\|^{2T - 1}\right\} \|\xi - \tilde{\xi}\| \right] \nonumber \\
        &= & L_X  \Delta_{2T}(Q, \tilde{Q}), \nonumber
    \end{eqnarray}
    where the last equality is due to the existence of finite $2T$-th moment of $\xi$ and the definition of $\Delta_p$ with $p = 2T$. 
\end{proof}

To ease the notation, 
in the remainder of this section, we 
write \(L_X({\xi, \tilde{\xi}})\) for 
\(\underset{t \in \{1, 2, \cdots, T\}}{\max} L_{X, t}({\xi, \tilde{\xi}})\).
Proposition \ref{Prop-stab-exo-feas} extends 
Proposition 3.1 in \cite{kuchler2008stability}
 where
$L_g({\xi, \tilde{\xi}}) = L_g \max \left\{ 1, \Vert \xi \Vert, \Vert \tilde{\xi} \Vert \right\}$ and $ L_S({\xi, \tilde{\xi}}) = 0$.
In the case that $L_S({\xi, \tilde{\xi}}) := L_S$ and $L_g({\xi, \tilde{\xi}}) := L_g$ are constant, 
we can obtain 
a similar result under the Kantorovich metric. 

With the established Lipschitz continuity of the feasible set, we can now consider the stability of the optimal value.  
Specifically, if the optimal value remains stable under small perturbations of the exogenous stochastic process, the optimal solution obtained from solving the original problem can still provide a high-quality solution even if there are some errors or perturbations in exogenous random variables. This is particularly important for dealing with the impact of the random environment variation on problem-solving in practical applications.


\begin{theorem}[Stability of the optimal value]
Under Assumptions \ref{Assu:welldefinedness}-
\ref{Assu:compactness},
\ref{assumption7}, \ref{assumption8} and \ref{Assu:convexity}(c), 
there exists a nonnegative 
\(L_{\vt}(\xi, \tilde{\xi})\) such that
    \begin{eqnarray}
        | \vt(Q) - \vt(\tilde{Q}) | \leq \Delta_{L_{\vt}}(Q, \tilde{Q}),
        \label{eq:stab-value-xi}
    \end{eqnarray}
where $\Delta_{L}$ is defined as in \eqref{eq:def-DeltaL} with $L = L_\vt$,
\bgeqn 
L_{\vt}(\xi, \tilde{\xi}) &:=& T L_C({\xi, \tilde{\xi}}) (L_X({\xi, \tilde{\xi}}) + 1) + \frac{T L_C({\xi, \tilde{\xi}}) (L_{X}({\xi, \tilde{\xi}}) + 1)}{1 - L_S({\xi, \tilde{\xi}})}\nonumber\\
&&- \frac{L_C({\xi, \tilde{\xi}})(L_X({\xi, \tilde{\xi}}) + 1)(1-L_S({\xi, \tilde{\xi}})^{T + 1})}{(1 - L_S({\xi, \tilde{\xi}}))^2} \label{eq:stab-coeff-vt1}
\edeqn 
for $L_S(\xi, \tilde{\xi}) \neq 1$ and 
\bgeqn
L_{\vt}(\xi, \tilde{\xi}) := T L_C({\xi, \tilde{\xi}}) (L_X({\xi, \tilde{\xi}}) + 1) + \frac{T (T - 1) L_C({\xi, \tilde{\xi}}) (L_X({\xi, \tilde{\xi}}) + 1)}{2} \label{eq:stab-coeff-vt2}
\edeqn 
for $L_S({\xi, \tilde{\xi}}) = 1$.
If, in addition, 
 $L_g({\xi, \tilde{\xi}}) := L_g \max\left\{1, \Vert \xi \Vert, \Vert \tilde{\xi} \Vert\right\}$, $L_S({\xi, \tilde{\xi}}) := L_S \max\left\{1, \Vert \xi \Vert, \Vert \tilde{\xi} \Vert\right\}$ and $L_C({\xi, \tilde{\xi}}) := L_C \max\left\{1, \Vert \xi \Vert, \Vert \tilde{\xi} \Vert\right\}$, and $\xi$ has finite $(3T+1)$-th moment, then 
    \begin{eqnarray}
        | \vt(Q) - \vt(\tilde{Q}) | \leq L_{\vt} \Delta_{3T + 1}(Q, \tilde{Q}),
        \label{eq:Lip-optvalue-FM}
    \end{eqnarray}
    where $L_{\vt} := T L_C (L_X + 1) + L_C (L_X + 1) \sum_{t = 1}^T \sum_{k = 1}^{t - 1} L_S^k$.
\label{Theo-stab-exo-optvalue}
\end{theorem}

\begin{proof}
    Let \(\bm x^*(\xi) := (x_0^*, x_1^*(s_1, x_0^*, \xi_1), x_2^*(s_2, x_1^*, \xi_{[2]}), \cdots, x_T^*(s_T, x_{T-1}^*, \xi_{[T]})) \in \mathcal{X}^*(\xi)\) 
    be an optimal policy to problem \eqref{eq:v-xi}, and \( \tilde{\bm x}^* := (\tilde{x}_0^*, \tilde{x}_1^*(s_1, \tilde{x}_0^*, \tilde{\xi}_1), \tilde{x}_2^*(s_2, \tilde{x}_1^*, \tilde{\xi}_{[2]}), \cdots, \tilde{x}_T^*(s_T, \tilde{x}_{T-1}^*, \tilde{\xi}_{[T]})) \in \mathcal{X}^*(\tilde{\xi})\) be an optimal policy to problem \eqref{eq:v-xi-ptb}.
    Fix any \(\pi(Q,\tilde Q)\in\Pi(Q,\tilde Q)\). Let \(y_0:=\tilde x_0^*\in\mathcal X_0\) and
    \(y_t\) be a measurable projection of \(\tilde x_t^*(\tilde s_t,\tilde x_{t-1}^*,\tilde\xi_{[t]})\) onto \(\mathcal X_t(s_t^y,y_{t-1},\xi_{[t]})\)
    for \(t=1,\ldots,T\).
    For every \(A\in\mathcal B(\mathbb R^{n_t})\), where $\mathcal{B}(\mathbb R^{n_t})$ denotes the Borel \(\sigma\)-algebra on
    \(\mathbb R^{n_t}\),
    define
    \bgeqn 
    \label{eq:K_t-def-proof}
    \mathsf K_t^{y,\pi}(A\mid s_t^y,y_{t-1},\xi_{[t]})
    :=\mathbb E_{\pi(Q,\tilde Q)}
    \left[\mathds{1}_A(y_t)\mid s_t^y,y_{t-1},\xi_{[t]}\right],
    \edeqn 
    where $\mathds{1}_A$ denotes the indicator function on the set $A$.
By definition
 \begin{eqnarray}
    \vt(Q) =\mathbb E_{\xi}\left[\sum_{t=0}^T C_t(s_t,x_t^*,\xi_{[t]},\zeta_t)\right]
    =\mathbb E_{\pi(Q,\tilde Q)}
    \left[\sum_{t=0}^T C_t(s_t,x_t^*,\xi_{[t]},\zeta_t)
    \right].
    \label{eq:ex-xt*-yt-proof}
    \end{eqnarray}
Let $z_0 = \tilde{x}_0^*$, $z_t(\cdot), t = 1, 2, \cdots, T$, be a randomized policy, where $z_t(\cdot, s_t^z,z_{t-1},\xi_{[t]})$ follows the distribution $K_t^{z,\pi} \left( \cdot\mid s_t^z,z_{t-1},\xi_{[t]} \right)$.
By \cite[Theorem 4.4.1]{puterman2014markov}
(see also \cite[Lemma 2]{ma2024bayesian}) and the definition of $\mathsf K_t^{y, \pi}$
\eqref{eq:K_t-def-proof},
the optimal value under the deterministic policy  
coincides with the optimal value under the randomized policy, i.e., 
    \begin{eqnarray}
    \mathbb E_{\pi(Q,\tilde Q)}
    \left[\sum_{t=0}^T C_t(s_t,x_t^*,\xi_{[t]},\zeta_t) \right] &= & \min_{z_t(\cdot, s_t^z,z_{t-1},\xi_{[t]}) \sim \mathsf K_t^{z,\pi} \left( \cdot\mid s_t^z,z_{t-1},\xi_{[t]} \right)} \mathbb E_{\pi(Q,\tilde Q)}
    \left[\sum_{t=0}^T C_t(s_t^z,z_t,\xi_{[t]},\zeta_t)\right] \nonumber \\
    &\leq&\mathbb E_{\pi(Q,\tilde Q)}
    \left[\sum_{t=0}^T C_t(s_t^y,y_t,\xi_{[t]},\zeta_t) \right].
    \label{eq:randomized-policy-proof}
    \end{eqnarray}
    The last inequality is due to the fact that  \(\bm y^\pi := (y_0, \bm y_1, \ldots, \bm y_T)\) is a randomized policy, where $\bm y_t$ follows the distribution $\mathsf K_t^{y,\pi}$.     By the definition of \(\Pi(Q,\tilde Q)\) and the measurable projections, the randomized policy \(\bm y^\pi(\cdot)\) is measurable, feasible and nonanticipative.
    Combining \eqref{eq:ex-xt*-yt-proof} with \eqref{eq:randomized-policy-proof},
    we obtain 
    \begin{eqnarray}
    \vt(Q)-\vt(\tilde Q)
    \leq\mathbb E_{\pi(Q,\tilde Q)}
    \left[\sum_{t=0}^T C_t(s_t^y,y_t,\xi_{[t]},\zeta_t)
    -\sum_{t=0}^T C_t(\tilde s_t,\tilde x_t^*,\tilde\xi_{[t]},\zeta_t)\right].
    \label{eq:xt*-yt-proof}
    \end{eqnarray}
    We compare the integrands inside the expectation operators, i.e., the difference of accumulated costs on different scenarios
    \begin{eqnarray} 
    &&  \sum_{t = 0}^T C_t(s_t^y, y_t, \xi_{[t]}, \zeta_t) - \sum_{t = 0}^T C_t(\tilde{s}_t, \tilde{x}_t^*, \tilde{\xi}_{[t]}, \zeta_t)  \nonumber \\
    &\leq &  \sum_{t = 1}^T L_{C, t}({\xi, \tilde{\xi}}) (\Vert s_t^y - \tilde{s}_t \Vert + \Vert y_t - \tilde{x}_t^* \Vert + \Vert \xi_{[t]} - \tilde{\xi}_{[t]} \Vert)   \nonumber \\
    & \leq&  L_C({\xi, \tilde{\xi}}) \sum_{t = 1}^T (\Vert s_t^y - \tilde{s}_t \Vert + \Vert y_t - \tilde{x}_t^* \Vert + \Vert \xi_{[t]} - \tilde{\xi}_{[t]} \Vert) ,   
    \label{eq:stab-exo-optvalue-v}
    \end{eqnarray}
    where \(s_t^y = S^M_{t - 1}(s_{t - 1}^y, y_{t - 1}, \xi_{t - 1}, \zeta_{t - 1}), 1 \leq t \leq T\), and \(s_0^y = s_0\).  
    The first inequality is based on Assumption \ref{assumption8}.  
    The terms inside the bracket at the right-hand side of \eqref{eq:stab-exo-optvalue-v} can be divided into two parts: \( \sum\limits_{t = 1}^T \Vert s_t^y - \tilde{s}_t \Vert \) and \(\sum\limits_{t = 1}^T (\Vert y_t - \tilde{x}_t^* \Vert + \Vert \xi_{[t]} - \tilde{\xi}_{[t]} \Vert)\). First, we consider the second part. It is known from \eqref{eq:nonexpected-Lip-feasible} that 
    \begin{eqnarray}
    && \sum_{t=1}^T\left(\Vert y_t-\tilde{x}_t^*\Vert+\Vert\xi_{[t]}-\tilde{\xi}_{[t]}\Vert\right)
    =\sum_{t=1}^T\left(\underset{x_t\in\mathcal{X}_t(s_t^y,y_{t-1},\xi_{[t]})}{\min}d(x_t,\tilde{x}_t^*)+\Vert\xi_{[t]}-\tilde{\xi}_{[t]}\Vert\right)\nonumber\\
    &\leq&\sum_{t=1}^T\left(\mathbb{H}\left(\mathcal{X}_t(s_t^y,y_{t-1},\xi_{[t]}),\mathcal{X}_t(\tilde{s}_t,\tilde{x}_{t-1}^*,\tilde{\xi}_{[t]})\right)+\Vert\xi_{[t]}-\tilde{\xi}_{[t]}\Vert\right)\nonumber\\
    &\leq&\sum_{t=1}^T\left(L_{X,t}(\xi,\tilde{\xi})+1\right)\Vert\xi_{[t]}-\tilde{\xi}_{[t]}\Vert
    \leq T\left(L_X(\xi,\tilde{\xi})+1\right)\Vert\xi-\tilde{\xi}\Vert.
    \label{eq:stab-exo-x}
    \end{eqnarray}
    Next, we consider the first part. Utilizing \eqref{eq:stab-exo-x} and the assumed Lipschitz continuity of \(S^M_t\), we have
    \begin{eqnarray}
    && \Vert s_t^y - \tilde{s}_t \Vert \leq L_S({\xi, \tilde{\xi}}) (\Vert s_{t - 1}^y - \tilde{s}_{t - 1} \Vert + \Vert y_{t - 1} - \tilde{x}_{t - 1}^* \Vert + \Vert \xi_{[t - 1]} - \tilde{\xi}_{[t - 1]} \Vert) \nonumber \\
    &\leq & L_S({\xi, \tilde{\xi}}) \Vert s_{t - 1}^y - \tilde{s}_{t - 1} \Vert + L_S({\xi, \tilde{\xi}}) (L_{X, t - 1}({\xi, \tilde{\xi}}) + 1) \Vert \xi_{[t - 1]} - \tilde{\xi}_{[t - 1]} \Vert \nonumber \\
    &\leq & L_S({\xi, \tilde{\xi}})^2 (\Vert s_{t - 2}^y - \tilde{s}_{t - 2} \Vert + \Vert y_{t - 2} - \tilde{x}_{t - 2}^* \Vert + \Vert \xi_{[t - 2]} - \tilde{\xi}_{[t - 2]} \Vert) \nonumber\\
    &&+ L_S({\xi, \tilde{\xi}}) (L_{X, t - 1}({\xi, \tilde{\xi}}) + 1) \Vert \xi_{[t - 1]} - \tilde{\xi}_{[t - 1]} \Vert \nonumber \\
    &\leq & \cdots \leq \sum_{k = 1}^{t - 1} L_S({\xi, \tilde{\xi}})^k (L_{X, t - k}({\xi, \tilde{\xi}}) + 1) \Vert \xi_{[t - k]} - \tilde{\xi}_{[t - k]} \Vert \nonumber\\
    &\leq & \sum_{k = 1}^{t - 1} L_S({\xi, \tilde{\xi}})^k (L_{X}({\xi, \tilde{\xi}}) + 1) \Vert \xi - \tilde{\xi} \Vert.
    \label{eq:stab-exo-s}
    \end{eqnarray}
    Combining \eqref{eq:stab-exo-x} and \eqref{eq:stab-exo-s} with \eqref{eq:stab-exo-optvalue-v} 
    gives rise to
    \begin{eqnarray}
        && \sum_{t=0}^T C_t(s_t^y,y_t,\xi_{[t]},\zeta_t) -\sum_{t=0}^T C_t(\tilde{s}_t,\tilde{x}_t^*,\tilde{\xi}_{[t]},\zeta_t) \nonumber \\
        &\leq & L_C({\xi, \tilde{\xi}}) T (L_X({\xi, \tilde{\xi}}) + 1)  \Vert \xi - \tilde{\xi} \Vert + L_C({\xi, \tilde{\xi}}) \sum_{t = 1}^T \sum_{k = 1}^{t - 1} L_S({\xi, \tilde{\xi}})^k (L_X({\xi, \tilde{\xi}}) + 1) \Vert \xi - \tilde{\xi} \Vert \nonumber \\
        &= & \left(T L_C({\xi, \tilde{\xi}}) (L_X({\xi, \tilde{\xi}}) + 1) + \sum_{t = 1}^T \sum_{k = 1}^{t - 1} L_S({\xi, \tilde{\xi}})^k L_C({\xi, \tilde{\xi}}) (L_X({\xi, \tilde{\xi}}) + 1)\right) \Vert \xi - \tilde{\xi} \Vert \nonumber \\
        &= & \left(T L_C({\xi, \tilde{\xi}}) (L_X({\xi, \tilde{\xi}}) + 1) + L_C({\xi, \tilde{\xi}}) (L_X({\xi, \tilde{\xi}}) + 1) \sum_{t = 1}^T \sum_{k = 1}^{t - 1} L_S({\xi, \tilde{\xi}})^k\right) \Vert \xi - \tilde{\xi} \Vert.
        \label{eq:stab-xi-opval}
    \end{eqnarray}
    The inequality still holds when 
    we swap the positions between 
     $\xi$ and $\tilde{\xi}$.
    Therefore, it remains to identify the explicit form of the coefficient \(L_{\vt}(\xi,\tilde{\xi})\), which will then be used to derive \eqref{eq:stab-value-xi}. To this end, we derive the specific forms of \(L_{\vt}({\xi, \tilde{\xi}})\) in \eqref{eq:stab-coeff-vt1}-\eqref{eq:stab-coeff-vt2}.
    Consider the case that \(L_S({\xi, \tilde{\xi}}) \neq 1\). Then
    \begin{eqnarray}
    L_{\vt}({\xi, \tilde{\xi}}) &= & T L_C({\xi, \tilde{\xi}}) (L_X({\xi, \tilde{\xi}}) + 1) + L_C({\xi, \tilde{\xi}}) (L_X({\xi, \tilde{\xi}}) + 1) \sum_{t = 1}^T \sum_{k = 1}^{t - 1} L_S({\xi, \tilde{\xi}})^k \nonumber \\
    &= & T L_C({\xi, \tilde{\xi}}) (L_X({\xi, \tilde{\xi}}) + 1) + L_C({\xi, \tilde{\xi}}) (L_X({\xi, \tilde{\xi}}) + 1) \sum_{t = 1}^T \frac{1 - L_S({\xi, \tilde{\xi}})^t}{1 - L_S({\xi, \tilde{\xi}})} \nonumber \\
    &= & T L_C({\xi, \tilde{\xi}}) (L_X({\xi, \tilde{\xi}}) + 1) + \frac{T L_C({\xi, \tilde{\xi}}) (L_{X}({\xi, \tilde{\xi}}) + 1)}{1 - L_S({\xi, \tilde{\xi}})} \nonumber\\
    &&- \frac{L_C({\xi, \tilde{\xi}})(L_X({\xi, \tilde{\xi}}) + 1)(1-L_S({\xi, \tilde{\xi}})^{T + 1})}{(1 - L_S({\xi, \tilde{\xi}}))^2},
    \end{eqnarray}
    which gives rise to \eqref{eq:stab-coeff-vt1}.
      Likewise, when \(L_S({\xi, \tilde{\xi}}) = 1\),
  we have
  \begin{eqnarray}
    L_{\vt}({\xi, \tilde{\xi}}) = T L_C({\xi, \tilde{\xi}}) (L_X({\xi, \tilde{\xi}}) + 1) + \frac{T (T - 1) L_C({\xi, \tilde{\xi}}) (L_X({\xi, \tilde{\xi}}) + 1)}{2}, 
    \end{eqnarray}
    which is \eqref{eq:stab-coeff-vt2}.
    By combining \eqref{eq:stab-xi-opval} with the definition of $\Delta_L$ with $L = L_\vt$, we obtain \eqref{eq:stab-value-xi}.
    
    Next, we prove \eqref{eq:Lip-optvalue-FM}.
    Under the additional assumptions about Lipschitz modulus, analogous to \eqref{eq:stab-xi-opval}, \eqref{eq:stab-coeff-vt1} and \eqref{eq:stab-coeff-vt2}, we have
    \begin{eqnarray}
        && \sum_{t = 0}^T C_t(s_t^y, y_t, \xi_{[t]}, \zeta_t) - \sum_{t = 0}^T C_t(\tilde{s}_t, \tilde{x}_t^*, \tilde{\xi}_{[t]}, \zeta_t) \nonumber \\
        &\leq & \left( T L_C (L_X + 1) + L_C (L_X + 1) \sum_{t = 1}^T \sum_{k = 1}^{t - 1} L_S^k \right) \max\left\{1, \|\xi\|^{3T}, \|\tilde{\xi}\|^{3T}\right\} \|\xi - \tilde{\xi}\| \nonumber \\
        &:= & L_{\vt} \max\left\{1, \|\xi\|^{3T}, \|\tilde{\xi}\|^{3T}\right\} \|\xi - \tilde{\xi}\|.
    \end{eqnarray}
    Combining with \eqref{eq:xt*-yt-proof}, we obtain
    \begin{eqnarray}
        \vt(Q) - \vt(\tilde{Q}) \leq \mathbb{E}_{\pi^*(Q, \tilde{Q})} \left[ L_{\vt} \max\left\{1, \|\xi\|^{3T}, \|\tilde{\xi}\|^{3T}\right\} \|\xi - \tilde{\xi}\| \right] \leq 
        L_{\vt} \Delta_{3T + 1}(Q, \tilde{Q}).
    \end{eqnarray}
    The conclusion follows by swapping the positions between $\xi$ and $\tilde{\xi}$.
\end{proof}

Theorem \ref{Theo-stab-exo-optvalue} is similar in form to the stability result in 
\cite[Theorem 2.1]{HRS2006stability} for MSP problems.
However, the two results are derived under different settings.
In 
\cite[Theorem 2.1]{HRS2006stability}, the authors  
assume the 
complete recourse condition holds in both 
problems before and after perturbation of $\xi_{[t]}$.
Here,
we require the Lipschitz continuity 
of $C_t, g_t, S^M_t $ in $(s_t, x_t, \xi_{[t]})$ 
and 
the Slater condition. 
Moreover, the perturbation in 
\cite[Theorem 2.1]{HRS2006stability}
is measured by the 
filtration distance which depends on both the whole random
process and the optimal solutions at different stages.
Meanwhile, in \cite[Theorem 6.1]{pflug2012distance},
the authors
use 
the nested distance to measure the perturbation of the whole data process and derive an error bound 
under H\"{o}lder continuity and convexity 
of the objective function. We will come back 
to the details of the differences
in Example~\ref{ex-feasible-unchanged} at the end of the section.
Next, we investigate 
stability of the set of optimal solutions. 
For this purpose, we first show the continuity of the optimal solution set.

\begin{theorem}[Continuity of the optimal solution set]
    Suppose that Assumptions \ref{Assu:convexity},
    \ref{assumption4} and \ref{assumption8} hold and the conditions in Theorem \ref{Theo-stab-exo-optvalue} are satisfied. 
    Let
     $L_{\Sigma}: \Xi_{[T]} \times \Xi_{[T]} \to \mathbb{R}_+$ be an integrable function, $L_X({\xi, \tilde{\xi}})$ defined in Proposition \ref{Prop-stab-exo-feas}, and $L_{\vt}({\xi, \tilde{\xi}})$ defined in Theorem \ref{Theo-stab-exo-optvalue}.
     Then for any $\epsilon > 0$ and $ \pi({Q, \tilde{Q}}) \in \Pi(Q, \tilde{Q})$, if \(
        \mathbb{E}_{\pi(Q, \tilde{Q})} \left[ \frac{L_X({\xi, \tilde{\xi}})}{\epsilon} \Vert \xi - \tilde{\xi} \Vert \right] < 1, \Delta_{L_{\vartheta}}(Q,\tilde Q)
        \leq
        \frac{a(\epsilon)}{2}\)
        and
        \(
        \frac{\Delta_{L_{\Sigma}(L_X + 1)}(Q, \tilde{Q})}{a(2\epsilon) - a(\epsilon)} \leq 1
        \), then
    \begin{eqnarray}
        \mathbb{E}_{\pi({Q, \tilde{Q}})} \left[ \mathbb{H}( \mathcal{X}^*(\tilde{\xi}), \mathcal{X}^*(\xi) )\right] \leq 3\epsilon, 
    \label{eq:continuity-optsolution-xi}
    \end{eqnarray}
    where $a(\epsilon)$ is defined in \eqref{eq:aepsilon}.
    \label{Theo-cont-exo-optsolu}
\end{theorem}
\begin{proof}
    Fix any \(\pi(Q,\tilde Q)\in\Pi(Q,\tilde Q)\) and carry out all scenario-wise projections below on the corresponding coupling space. We divide the proof into the following two cases.

    \textbf{Case 1:} \(\mathcal{X}^*(\xi) = \mathcal{X} (\xi)\). The conclusion holds by Proposition \ref{Prop-stab-exo-feas}.

    \textbf{Case 2:} \(\mathcal{X}^*(\xi) \neq \mathcal{X} (\xi)\). Denote the \(\epsilon\)-neighborhood of the optimal solution set $\mathcal{X}^*(\xi)$ by \(\epsilon \text{-} \mathcal{X}^*(\xi) = \mathcal{X}^*(\xi) + \epsilon \mathbb{B}\), where \(\mathbb{B}\) is the unit open ball. 
    Let
    $\mathcal{X}^{\epsilon}(\xi) = \mathcal{X}(\xi) \setminus \epsilon \text{-} \mathcal{X}^*(\xi)$. 
    Observe that 
    \(\mathcal{X}(\xi)\) is a bounded and closed set under Assumptions \ref{Assu:compactness} and \ref{assumption8}; \(\mathcal{X}^*(\xi)\) is a convex set.
    Thus, \(\mathcal{X}^{\epsilon}(\xi)\) is a compact set as \(\epsilon \text{-} \mathcal{X}^*(\xi)\) is an open set. 
    Denote the minimum value of \(\mathbb{E}_{\xi} \left[ \sum\limits_{t = 0}^T C_t(s_t, x_t, \xi_{[t]}, \zeta_t) \right]\) in \(\mathcal{X}^{\epsilon}(\xi)\) by \(\vt^{\epsilon}(Q) := \mathbb{E}_{\xi} \left[ \sum\limits_{t = 0}^T C_t(s_t^{\epsilon}, x_t^{*, \epsilon}, \xi_{[t]}, \zeta_t)\right]\), where $x^{*, \epsilon}$ is the optimal policy in $\mathcal{X}^{\epsilon}(\xi)$, $s_{t + 1}^{\epsilon} = S^M_t(s_t^{\epsilon}, x_t^{*, \epsilon}, \xi_{t}, \zeta_t)$. 
    Then
    \(\vt^{\epsilon}(Q) > \vt(Q)\).
    Let 
    \begin{eqnarray}
    a(\epsilon) := \vt^{\epsilon}(Q) - \vt(Q).
    \label{eq:aepsilon}
    \end{eqnarray}
For any $\tilde{x}(\tilde{\xi}) \in \mathcal{X}^{\epsilon}(\xi)$, let $s_{t + 1} = S^M_t(s_t, \tilde{x}_t, \xi_{t}, \zeta_t), \tilde{s}_{t + 1} = S^M_t(\tilde{s}_t, \tilde{x}_t, \tilde{\xi}_{t}, \zeta_t)$, by Theorem \ref{Theo-stab-exo-optvalue}, 
    \(
        | \vt(Q) - \vt(\tilde{Q}) | \leq \Delta_{L_{\vt}}(Q, \tilde{Q}).
    \)
Therefore, 
if \(\tilde{\xi}\) 
    is sufficiently close to $\xi$
    such that 
    $\Delta_{L_{\vartheta}}(Q,\tilde Q)
    \leq
    \frac{a(\epsilon)}{2}$,
    then
    \begin{eqnarray}
    \vt(\tilde{Q}) &\leq & \vt(Q) + \frac{a(\epsilon)}{2} = \vt^{\epsilon}(Q) - \frac{a(\epsilon)}{2} < \mathbb{E}_{\xi} \left[ \sum\limits_{t = 0}^T C_t(s_t^{\epsilon}, x_t^{*, \epsilon} \xi_{[t]}, \zeta_t) \right] - \Delta_{L_{\vt}}(Q, \tilde{Q}), \nonumber \\
    &\leq & \mathbb{E}_{\xi} \left[ \sum\limits_{t = 0}^T C_t(s_t^{\epsilon}, x_t^{*, \epsilon}, \xi_{[t]}, \zeta_t) \right] - \left| \mathbb{E}_{\tilde{\xi}} \left[  \sum\limits_{t = 0}^T C_t(\tilde{s}_t, \tilde{x}_t, \tilde{\xi}_{[t]}, \zeta_t) \right] - \mathbb E_{\pi(Q,\tilde Q)}\left[\sum_{t=0}^T C_t(s_t,\tilde x_t,\xi_{[t]},\zeta_t)\right] \right| \nonumber \\
    &\leq & \mathbb{E}_{\xi} \left[ \sum\limits_{t = 0}^T C_t(s_t^{\epsilon}, x_t^{*, \epsilon}, \xi_{[t]}, \zeta_t) \right] - \mathbb E_{\pi(Q,\tilde Q)}\left[\sum_{t=0}^T C_t(s_t,\tilde x_t,\xi_{[t]},\zeta_t)\right] + \mathbb{E}_{\tilde \xi} \left[ \sum\limits_{t = 0}^T C_t(\tilde{s}_t, \tilde{x}_t, \tilde{\xi}_{[t]}, \zeta_t)  \right] \nonumber \\
    &\leq & \mathbb{E}_{\tilde{\xi}} \left[ \sum\limits_{t = 0}^T C_t(\tilde{s}_t, \tilde{x}_t, \tilde{\xi}_{[t]}, \zeta_t)  \right], \nonumber
    \end{eqnarray}
    which implies that 
    \underline{if the optimal policy \(\tilde{\bm x}^*(\tilde{\xi})\) to problem \eqref{eq:v-xi-ptb} lies in \({\mathcal{X}(\xi)} \cap {\mathcal{X}(\tilde{\xi})}\)}, then 
    \(\tilde{\bm x}^* \in \epsilon \text{-} \mathcal{X}^*(\xi)\). 
    By an argument analogous to that preceding \eqref{eq:xt*-yt-proof}, we conclude that the policy induced by these projections is measurable and nonanticipative.
    Thus, when 
    $\frac{2 \Delta_{L_{\vartheta}}(Q,\tilde Q)}{a(\epsilon)}
    \leq
    1$,
    inequality \eqref{eq:continuity-optsolution-xi} must hold. The third inequality is derived by \eqref{eq:stab-exo-optvalue-v} and \eqref{eq:stab-xi-opval}.
    
In what follows, we consider \underline{the case that \(\tilde{\bm x}^*\notin\mathcal X(\xi)\cap\mathcal X(\tilde\xi)\)}. Under the coupling \(\pi(Q,\tilde Q)\), set \(x_0:=\tilde x_0^*\) and for \(t=1,\ldots,T\), let \(x_t\) be a measurable projection of \(\tilde x_t^*\) onto \(\mathcal X_t(s_t,x_{t-1},\xi_{[t]})\). Denote the projected trajectory by \(\bm x:=(x_0,\ldots,x_T)\). As above, these projected decisions induce a randomized policy for the original problem. We prove that \(d(\bm x,\mathcal X^*(\xi))<2\epsilon\).
To this end, we measure the difference in the objective function values of problems \eqref{eq:v-xi} and \eqref{eq:v-xi-ptb}, i.e.,
\begin{eqnarray}
    && \bigg| \sum_{t = 0}^T \big(C_t(s_t, x_t, \xi_{[t]}, \zeta_t) - C_t(\tilde{s}_t, \tilde{x}_t, \tilde{\xi}_{[t]}, \zeta_t)\big) \bigg| \nonumber \\
    &\leq& \sum_{t = 0}^T L_{C, t}({\xi, \tilde{\xi}})\big(\Vert s_t - \tilde{s}_t \Vert + \Vert x_t - \tilde{x}_t \Vert + \Vert \xi_{[t]} - \tilde{\xi}_{[t]} \Vert \big) \nonumber \\
    &\leq & (T + 1) L_C({\xi, \tilde{\xi}}) (\Vert \xi - \tilde{\xi} \Vert + \Vert \bm x - \tilde{\bm x} \Vert) + L_C({\xi, \tilde{\xi}}) \sum_{t = 0}^T \Vert s_t - \tilde{s}_t \Vert ,
    \label{eq:cont-obj-Lip}
\end{eqnarray}
where $\tilde{s}_t = S^M_{t - 1}(\tilde{s}_{t - 1}, \tilde{x}_{t - 1}, \tilde{\xi}_{t - 1}, \zeta_{t - 1})$ and $s_0 = \tilde{s}_0$. 
By Assumption \ref{assumption8},
\begin{eqnarray}
\Vert s_t - \tilde{s}_t \Vert &= & \Vert S^M_{t - 1}(s_{t - 1}, x_{t - 1}, \xi_{t - 1}, \zeta_{t - 1}) - S^M_{t - 1}(\tilde{s}_{t - 1}, \tilde{x}_{t - 1}, \tilde{\xi}_{t - 1}, \zeta_{t - 1}) \Vert \nonumber \\
&\leq & L_S({\xi, \tilde{\xi}})(\Vert s_{t - 1} - \tilde{s}_{t - 1} \Vert + \Vert x_{t - 1} - \tilde{x}_{t - 1} \Vert + \Vert \xi_{t - 1} - \tilde{\xi}_{t - 1} \Vert) \nonumber \\
&\leq & L_S({\xi, \tilde{\xi}})^2(\Vert s_{t - 2} - \tilde{s}_{t - 2} \Vert + \Vert x_{t - 2} - \tilde{x}_{t - 2} \Vert + \Vert \xi_{t - 2} - \tilde{\xi}_{t - 2} \Vert) \nonumber \\
&& + L_S({\xi, \tilde{\xi}})(\Vert x_{t - 1} - \tilde{x}_{t - 1} \Vert + \Vert \xi_{t - 1} - \tilde{\xi}_{t - 1} \Vert) \nonumber \\
&\leq & \cdots \leq \sum_{k = 1}^{t} L_S({\xi, \tilde{\xi}})^k(\Vert x_{t - k} - \tilde{x}_{t - k} \Vert + \Vert \xi_{t - k} - \tilde{\xi}_{t - k} \Vert) \nonumber \\
&\leq & \sum_{k = 1}^{t} L_S({\xi, \tilde{\xi}})^k(\Vert \bm x - \tilde{\bm x} \Vert + \Vert \xi - \tilde{\xi} \Vert) \leq \frac{1 - L_S({\xi, \tilde{\xi}})^{t + 1}}{1 -L_S({\xi, \tilde{\xi}})} (\Vert \bm x - \tilde{\bm x} \Vert + \Vert \xi - \tilde{\xi} \Vert) \nonumber
\end{eqnarray}
provided that $L_S({\xi, \tilde{\xi}}) \neq 1$. 
Consequently 
\begin{eqnarray}
\sum_{t = 0}^T \Vert s_t - \tilde{s}_t \Vert & =& \sum_{t = 1}^T \frac{1 - L_S({\xi, \tilde{\xi}})^{t + 1}}{1 -L_S({\xi, \tilde{\xi}})} (\Vert \bm x - \tilde{\bm x} \Vert + \Vert \xi - \tilde{\xi} \Vert) \nonumber \\
&= & \frac{T - 1 - TL_S({\xi, \tilde{\xi}}) + L_S({\xi, \tilde{\xi}})^{T + 2}}{(1-L_S({\xi, \tilde{\xi}}))^2} (\Vert \bm x - \tilde{\bm x} \Vert + \Vert \xi - \tilde{\xi} \Vert).
\label{eq:xi-st-lip}
\end{eqnarray}
Let $L_{\Sigma, 1}({\xi, \tilde{\xi}}) := \left(T + 1 + \frac{T - 1 - TL_S({\xi, \tilde{\xi}}) + L_S({\xi, \tilde{\xi}})^{T + 2}}{(1-L_S({\xi, \tilde{\xi}}))^2}\right) L_C({\xi, \tilde{\xi}})$.
Combining \eqref{eq:cont-obj-Lip} with \eqref{eq:xi-st-lip}, we obtain
\begin{eqnarray}
&& \text{rhs of } \eqref{eq:cont-obj-Lip} \leq 
L_{\Sigma,1}({\xi, \tilde{\xi}})(\Vert \xi - \tilde{\xi} \Vert + \Vert \bm x - \tilde{\bm x} \Vert).
\label{eq:cont-exo-L_Sigma1}
\end{eqnarray}
In the case that  $L_S({\xi, \tilde{\xi}}) = 1$,
we have
\begin{eqnarray}
    \text{rhs of }  \eqref{eq:cont-obj-Lip} &\leq & (T + 1) L_C({\xi, \tilde{\xi}}) (\Vert \xi - \tilde{\xi} \Vert + \Vert \bm x - \tilde{\bm x} \Vert)  + L_C({\xi, \tilde{\xi}}) \sum_{t = 0}^T \Vert s_t - \tilde{s}_t \Vert \nonumber \\
    &\leq & \left( (T + 1)L_C({\xi, \tilde{\xi}}) + \frac{T (T + 1)}{2} L_C({\xi, \tilde{\xi}}) \right)(\Vert \xi - \tilde{\xi} \Vert + \Vert \bm x - \tilde{\bm x} \Vert) \nonumber \\
    &:= & L_{\Sigma, 2}({\xi, \tilde{\xi}}) (\Vert \xi - \tilde{\xi} \Vert + \Vert \bm x - \tilde{\bm x} \Vert).
    \label{eq:cont-exo-L_Sigma2}
\end{eqnarray}
To ensure the existence of a suitable $\tilde{\xi}$ as above, we need to prove that $a(\epsilon)$ is strictly increasing w.r.t.~$\epsilon$. First, it is immediate from its definition that $a(\epsilon)$ is non-decreasing. Let $\epsilon_2 > \epsilon_1 > 0$. If $a(\epsilon_2) = a(\epsilon_1)$, then there exists a $\bar{y} \in \mathcal{X}^{\epsilon_2}(\xi)$ such that
\begin{eqnarray}
\mathbb{E}_{\xi} \left[ \sum_{t = 0}^T C_t(\bar{s}_t, \bar{y}_t, \xi_{[t]}, \zeta_t) \right] = \vt(Q) + a(\epsilon_2). \nonumber
\end{eqnarray}
Here, \(\bar{s}_t = S^M_{t - 1}(\bar{s}_{t - 1}, \bar{y}_{t - 1}, \xi_{t - 1}, \zeta_{t - 1})\), and \(\bar{s}_0 = s_0\). Suppose the distance from \(\bar{y}\) to \(\mathcal{X}^{*}(\xi)\) is \( \gamma \), then \(\gamma \geq \epsilon_2\). Let \(y\) be the projection of \(\bar{y}\) onto \(\mathcal{X}^*(\xi)\) and define \(\tilde{y} = \frac{\epsilon_2 - \epsilon_1}{d(\bar{y}, y)} y + \left(1 - \frac{\epsilon_2 - \epsilon_1}{d(\bar{y}, y)}\right) \bar{y}\). Then, we have
\begin{eqnarray}
    d(\tilde{y}, y) = \left(1 - \frac{\epsilon_2 - \epsilon_1}{d(\bar{y}, y)}\right) d(y, \bar{y}) \geq \frac{\epsilon_1}{\epsilon_2}d(\bar{y}, y) \geq \epsilon_1. \nonumber
\end{eqnarray}    
By the assumptions, \(\mathbb{E}_{\xi} \left[ \sum\limits_{t = 0}^T C_t({s}_t, {x}_t, \xi_{[t]}, \zeta_t) \right]\) is convex w.r.t.~\(\bm x\). Therefore:
\begin{eqnarray}
    && \mathbb{E}_{\xi} \left[ \sum_{t = 0}^T C_t(\tilde{s}^y_t, \tilde{y}_t, \xi_{[t]}, \zeta_t) \right] \nonumber \\
    &\leq & \frac{\epsilon_2 - \epsilon_1}{d(\bar{y}, y)} \mathbb{E}_{\xi} \left[ \sum_{t = 0}^T C_t({s}^y_t, {y}_t, \xi_{[t]}, \zeta_t) \right] + \left(1 - \frac{\epsilon_2 - \epsilon_1}{d(\bar{y}, y)}\right) \mathbb{E}_{\xi} \left[ \sum_{t = 0}^T C_t(\bar{s}^y_t, \bar{y}_t, \xi_{[t]}, \zeta_t) \right] \nonumber \\
    &= & \frac{\epsilon_2 - \epsilon_1}{d(\bar{y}, y)} \vt(Q) + \left(1 - \frac{\epsilon_2 - \epsilon_1}{d(\bar{y}, y)}\right) (\vt(Q) + a(\epsilon_2)) \nonumber \\
    &= & \vt(Q) + \left(1 - \frac{\epsilon_2 - \epsilon_1}{d(\bar{y}, y)}\right) a(\epsilon_2) \nonumber \\
    &< & \vt(Q) + a(\epsilon_2) = \vt(Q) + a(\epsilon_1). \nonumber
\end{eqnarray}
This contradicts the assumption that \(\vt(Q) + a(\epsilon_1)\) is the optimal value of problem \eqref{eq:mixed-MSP-MDP} on \(\mathcal{X}^{\epsilon_1}(\xi)\). Therefore, we have shown that \(a(\epsilon)\) must be strictly increasing.

Combining the two conclusions above, we know that if a feasible policy \(\bm x \in \mathcal{X}(\xi)\) is chosen such that \(d(\bm x, \mathcal{X}^*({\xi})) \geq 2\epsilon\), then when \(\tilde{\xi}\) is selected such that 
$\frac{\Delta_{L_{\Sigma}(L_X + 1)}(Q, \tilde{Q})}{a(2\epsilon) - a(\epsilon)} \leq 1$, it holds
\begin{eqnarray}
    &&\mathbb{E}_{\tilde{\xi}} \left[ \sum_{t = 0}^T C_t(\tilde{s}^*_t, \tilde{x}^*_t, \tilde{\xi}_{[t]}, \zeta_t) \right] \nonumber \\
    &\geq& \mathbb E_{\pi(Q,\tilde Q)}\left[\sum_{t=0}^T C_t(s_t,x_t,\xi_{[t]},\zeta_t)\right]-\left|\mathbb E_{\tilde{\xi}}\left[\sum_{t=0}^T C_t(\tilde{s}_t^*,\tilde{x}_t^*,\tilde{\xi}_{[t]},\zeta_t)\right]-\mathbb E_{\pi(Q,\tilde Q)}\left[\sum_{t=0}^T C_t(s_t,x_t,\xi_{[t]},\zeta_t)\right]\right|\nonumber\\
    &\geq & \vt(Q) + a(2\epsilon) - \Delta_{L_{\Sigma}(L_X + 1)}(Q, \tilde{Q}) \nonumber \\
    &\geq & \vt(Q) + a(\epsilon) > \vt(\tilde{Q}), \nonumber
\end{eqnarray}
    where $L_{\Sigma}({\xi, \tilde{\xi}}) := \max\left\{ L_{\Sigma, 1}({\xi, \tilde{\xi}}), L_{\Sigma, 2}({\xi, \tilde{\xi}}) \right\}$. 
    The second inequality follows from \eqref{eq:cont-exo-L_Sigma1}, \eqref{eq:cont-exo-L_Sigma2} and Proposition \ref{Prop-stab-exo-feas}. 
    This contradicts the fact that \(\tilde{\bm x}^*\) is an optimal policy for problem \eqref{eq:mixed-MSP-MDP} on \(\mathcal{X}^*(\tilde{\xi})\). Therefore, we must have \(\bm x \in 2\epsilon \text{-} \mathcal{X}^*({\xi})\). Furthermore, by Proposition \ref{Prop-stab-exo-feas}, \(\mathbb{E}_{\pi(Q, \tilde{Q})} [d(\bm x, \tilde{\bm x}^*)] \leq  \mathbb{E}_{\pi(Q, \tilde{Q})} \left[ L_X({\xi, \tilde{\xi}}) \Vert \xi - \tilde{\xi} \Vert \right]\). This means that, when \(\mathbb{E}_{\pi(Q, \tilde{Q})} \left[ \frac{L_X({\xi, \tilde{\xi}})}{\epsilon} \Vert \xi - \tilde{\xi} \Vert \right] < 1\), we have \(\mathbb{E}_{\pi(Q, \tilde{Q})} [d(\bm x, \tilde{\bm x}^*)] < \epsilon\). 
    In summary, when 
    \[
    \mathbb{E}_{\pi(Q, \tilde{Q})} \left[ \frac{L_X({\xi, \tilde{\xi}})}{\epsilon} \Vert \xi - \tilde{\xi} \Vert \right] < 1, \Delta_{L_{\vartheta}}(Q,\tilde Q)
    \leq
    \frac{a(\epsilon)}{2},
    \frac{\Delta_{L_{\Sigma}(L_X + 1)}(Q, \tilde{Q})}{a(2\epsilon) - a(\epsilon)} \leq 1,
    \]
    it must hold that
    \begin{eqnarray}
    \mathbb{E}_{\pi(Q, \tilde{Q})} \left[ d(\tilde{\bm x}^*, \mathcal{X}^*(\xi)) \right] \leq \mathbb{E}_{\pi(Q, \tilde{Q})} \left[ d(\tilde{\bm x}^*, \bm x) \right] + \mathbb{E}_{\pi(Q, \tilde{Q})} \left[ d(\bm x, \mathcal{X}^*(\xi)) \right] < 3\epsilon. \nonumber
    \end{eqnarray}
    Since \(\tilde{\bm x}^*\) is chosen arbitrarily, the inequality above implies that
    \begin{eqnarray}
    \mathbb{E}_{\pi(Q, \tilde{Q})} \left[ \mathbb{D}(\mathcal{X}^*(\tilde{\xi}), \mathcal{X}^*(\xi)) \right] < 3\epsilon. \nonumber
    \end{eqnarray}
    Similarly, we can reach the same conclusion when the positions of \(\xi\) and \(\tilde{\xi}\) are swapped. Therefore,
    \begin{eqnarray}
    \mathbb{E}_{\pi(Q, \tilde{Q})} \left[ \mathbb{H}(\mathcal{X}^*(\tilde{\xi}), \mathcal{X}^*(\xi)) \right] < 3\epsilon. \nonumber
    \end{eqnarray}
    This completes the proof about the continuity of the optimal solution set.
\end{proof}
    Unlike the stability results 
    in \cite{HRS2006stability}, 
    inequality 
    \eqref{eq:continuity-optsolution-xi}
in Theorem~\ref{Theo-cont-exo-optsolu}   
    does not require 
    a growth condition. 
    Moreover,
    the convexity of the objective function is necessary for the continuity of the set of optimal solutions.  
    The next example shows that the 
    continuity may fail without convexity.
    \begin{example}
    Consider the following 
    one-stage stochastic 
    minimization problem:
    \begin{subequations}
    \begin{eqnarray}
    \underset{x \in \mathbb{R}} \min && - \mathbb{E}_{\xi} \left[ | x - \xi | \right] \\
    \textrm{ s.t. } && |x - 1| - 1 \leq 0,
    \end{eqnarray}     
    \label{eq-ex-cont-fanli}
    \end{subequations}
    where
    $\xi$ follows a uniform distribution on $[0, 2]$, i.e. $Q = U(0,2)$.  
    Problem \eqref{eq-ex-cont-fanli} satisfies all conditions in 
    Theorem~\ref{Theo-cont-exo-optsolu} except for 
    convexity of the objective function.  
    First, the constraint function in \eqref{eq-ex-cont-fanli} is Lipschitz continuous 
    in \(x\) with a Lipschitz modulus of 1, and it is also convex 
    in \(x\).  
    Second, since the objective function and the constraint function are independent of \(s\) in \eqref{eq-ex-cont-fanli}, 
    the Lipschitz continuity of
    the state transition function holds automatically.
    Third, the objective function is Lipschitz continuous 
    in $x$
    with modulus $1$.
    Fourth, for any \(\xi\), there exists a feasible solution \(x = 1\) such that \(|x - 1| - 1 \leq -1\). Therefore, the Slater condition holds uniformly.  
    Finally, the feasible solution set of problem \eqref{eq-ex-cont-fanli} is obviously bounded.
    However, the objective function is concave 
    in \(x\). To see this, we can obtain a closed form of the objective function by straightforward calculation
$$
\mathbb{E}_{\xi}[|x-\xi|] 
    = \frac{1}{2}\left( \int_{0}^{x} (x-t)\,dt + \int_{x}^{2} (t-x)\,dt \right)  \\
    = \frac{x^{2} + (2-x)^{2}}{4} = \frac{(x-1)^{2}+1}{2}. 
    $$
    Then the set of optimal solutions
    to problem \eqref{eq-ex-cont-fanli}  is $\mathcal{X}^*(\xi) = \{0, 2\}$.
    However, the optimal solution set changes drastically 
    with 
    a small perturbation in \(\xi\). For instance, if the distribution $Q$ is perturbed to $U(-\delta, 2)$ for any $\delta > 0$,
  then
    $$
    \mathbb{E}_{\tilde{\xi}}[|x-\tilde{\xi}|] 
    = \frac{1}{2+\delta} \left( \int_{-\delta}^{x} (x-t)\,dt + \int_{x}^{2} (t-x)\,dt \right) 
    = \frac{(x+\delta)^{2} + (2-x)^{2}}{2(2+\delta)}.
    $$
    The set of optimal solutions 
    reduces to a singleton
    $\mathcal{X}^*(\tilde{\xi}) = \{ 2 \}$. 
    Consequently,
    $\mathbb{H}(\mathcal{X}^*(\xi), \mathcal{X}^*(\tilde{\xi})) = 2$
    for any $\delta>0$. 
    The failure of stability is
    caused by disconnectedness 
    of the set of optimal solutions to the original problem.
    \end{example}
    It is possible to strengthen 
    Theorem \ref{Theo-cont-exo-optsolu}
    by deriving 
    an error bound for 
    $\mathbb{H}(\mathcal{X}^*(\xi), \mathcal{X}^*(\tilde{\xi}))$ in terms 
    of perturbations of $\xi$. The next theorem states this.

\begin{theorem}[Quantitative stability of the set of optimal solutions]
    Let 
    $\mathcal{X}^*(\xi)$ and $\mathcal{X}^*(\tilde{\xi})$
    be the sets of optimal solutions 
    to problems \eqref{eq:v-xi} and \eqref{eq:v-xi-ptb}. 
    Assume: (a) Assumptions 
    \ref{assumption7} and 
    \ref{assumption8} hold and the conditions in Theorem \ref{Theo-stab-exo-optvalue} are satisfied, (b)
    problem \eqref{eq:mixed-MSP-MDP} satisfies the \(\nu\)-th order growth condition, i.e., 
    there exist constants \(\beta>0\) and \(\nu\geq1\) such that
    \begin{eqnarray}
    \mathbb{E}_{\xi}\left[ \sum_{t = 0}^T C_t(s_t, x_t, \xi_{[t]}, \zeta_t) \right] - \vt(Q) \geq \beta \left( \mathbb{E}_{\xi} \left[ d(\bm x, \mathcal{X}^*(\xi)) \right] \right)^{\nu}, \quad \forall  \bm x \in \mathcal{X}(\xi)
    \label{eq:growth-xi-nu}
    \end{eqnarray}
    holds for both $\xi$ and its perturbation $\tilde{\xi}$.
    Then 
    \begin{eqnarray}
    && \mathbb{E}_{\pi^*(Q, \tilde{Q})} \left[ \mathbb{H}(\mathcal{X}^*(\xi), \mathcal{X}^*(\tilde{\xi}))\right]  \leq \Delta_{L_X}(Q, \tilde{Q}) + \left( \frac{1}{\beta} \Delta_{L_\vt + L_{\Sigma}(L_X + 1)}(Q, \tilde{Q}) \right)^{\frac{1}{\nu}}, 
    \label{eq:stab-optimalsolu-xi}
    \end{eqnarray}
    where $\pi^*(Q, \tilde{Q})$ is the optimal coupling of $Q$ and $\tilde{Q}$ defined as in \eqref{eq:def-DeltaL} with $L(\xi, \tilde{\xi}) = L_X(\xi, \tilde{\xi})\}$; $L_X({\xi, \tilde{\xi}})$, $L_{\vt}({\xi, \tilde{\xi}})$ and $L_{\Sigma}({\xi, \tilde{\xi}})$ are
    defined respectively 
    in Proposition~\ref{Prop-stab-exo-feas},
    Theorem~\ref{Theo-stab-exo-optvalue} and Theorem~\ref{Theo-cont-exo-optsolu}.   If, in addition, 
    $L_g({\xi, \tilde{\xi}}) := L_g \max\left\{1, \Vert \xi \Vert, \Vert \tilde{\xi} \Vert\right\}$, $L_S({\xi, \tilde{\xi}}) := L_S \max\left\{1, \Vert \xi \Vert, \Vert \tilde{\xi} \Vert\right\}$ and $L_C({\xi, \tilde{\xi}}) := L_C \max\left\{1, \Vert \xi \Vert, \Vert \tilde{\xi} \Vert\right\}$, and $\xi$ has finite $(3T+1)$-th moment, then
       \begin{eqnarray}
           \hspace{-0.5em} \mathbb{E}_{\pi^*(Q, \tilde{Q})} \left[ \mathbb{H}(\mathcal{X}^*(\xi), \mathcal{X}^*(\tilde{\xi})) \right] \leq L_{X} \Delta_{2T}(Q, \tilde{Q}) + \left( \frac{1}{\beta} (L_{\vt} + L_{\Sigma} L_X + L_{\Sigma}) \Delta_{3T + 1}(Q, \tilde{Q}) \right)^{\frac{1}{\nu}},
           \label{eq:stab-exo-optsolu-FM}
       \end{eqnarray}
       where $\pi^*(Q, \tilde{Q})$ is the optimal coupling of $Q$ and $\tilde{Q}$ defined as in \eqref{eq:def-DeltaL} with $L(\xi, \tilde{\xi}) = \max\{1, \| \xi \|, \| \tilde{\xi} \| \}$; $L_{\Sigma} := \max \left\{ L_{\Sigma,1}, L_{\Sigma,2} \right\}$,  $L_{\Sigma,1}$ and $L_{\Sigma,2}$ are specified in \eqref{eq:stab-optsolu-xi-FM-Lsigma1} and \eqref{eq:stab-optsolu-xi-FM-Lsigma2}.
    \label{Theo-stab-exo}
\end{theorem}
\begin{proof}
    By Theorem \ref{Theo-stab-exo-optvalue}, 
    \begin{eqnarray}
        | \vt(Q) - \vt(\tilde{Q}) | \leq \Delta_{L_{\vt}}(Q, \tilde{Q}).
        \label{eq:stab-diff-v}
    \end{eqnarray}
    Let \(\bm x^*\) be an optimal policy of problem \eqref{eq:mixed-MSP-MDP}, with \(s_t^*=S^M_{t-1}(s_{t-1}^*,\) \(x_{t-1}^*,\xi_{t-1},\zeta_{t-1})\). Under the coupling \(\pi^*(Q,\tilde Q)\), set \(\tilde x_0:=x_0^*\) and, recursively, let \(\tilde x_t\) be a measurable projection of \(x_t^*\) onto \(\mathcal X_t(\tilde s_t,\tilde x_{t-1},\tilde\xi_{[t]})\), where \(\tilde s_t=S^M_{t-1}(\tilde s_{t-1},\tilde x_{t-1},\tilde\xi_{t-1},\zeta_{t-1})\) and \(\tilde s_0=s_0\). Denote the projected trajectory by \(\tilde{\bm x}:=(\tilde x_0,\ldots,\tilde x_T)\). As discussed in the proof of Theorem \ref{Theo-stab-exo-optvalue}, these projected decisions induce a randomized policy for the perturbed problem.
    By \eqref{eq:cont-obj-Lip}-\eqref{eq:cont-exo-L_Sigma2},
\begin{eqnarray}
    && \bigg|  \sum_{t = 0}^T C_t(\tilde{s}_t, \tilde{x}_t, \tilde{\xi}_{[t]}, \zeta_{t}) - \sum_{t = 0}^T C_t(s_t^*, x_t^*, \xi_{[t]}, \zeta_{t}) \bigg| \nonumber \\
    &\leq & L_{\Sigma}({\xi, \tilde{\xi}}) (\Vert \tilde{\bm x} - \bm x^* \Vert + \Vert \xi - \tilde{\xi} \Vert)
    \leq  L_{\Sigma}({\xi, \tilde{\xi}})(L_X({\xi, \tilde{\xi}}) + 1) \Vert \xi - \tilde{\xi} \Vert, \nonumber
\end{eqnarray}
which implies, by the definition of $\Delta_L$ in \eqref{eq:def-DeltaL},
\begin{eqnarray}
    && \bigg| \mathbb{E}_{\pi^*(Q,\tilde Q)} \left[ \sum_{t = 0}^T C_t(\tilde{s}_t, \tilde{x}_t, \tilde{\xi}_{[t]}, \zeta_t) \right] - \vt(Q) \bigg|
    \leq \Delta_{L_{\Sigma}(L_X + 1)}(Q,\tilde Q).
    \label{eq:stab-diff-vtildex-xstar}
\end{eqnarray}
A combination of \eqref{eq:stab-diff-v} and \eqref{eq:stab-diff-vtildex-xstar}  yields
\begin{eqnarray}
    && \bigg|\mathbb{E}_{\pi^*(Q,\tilde Q)} \left[ \sum_{t = 0}^T C_t(\tilde{s}_t,\tilde{x}_t,\tilde{\xi}_{[t]},\zeta_t) \right]-\vt(\tilde Q)\bigg| \nonumber\\
    &\leq& \bigg|\mathbb{E}_{\pi^*(Q,\tilde Q)} \left[ \sum_{t = 0}^T C_t(\tilde{s}_t,\tilde{x}_t,\tilde{\xi}_{[t]},\zeta_t) \right]-\vt(Q)\bigg|+\left|\vt(Q)-\vt(\tilde Q)\right| \nonumber\\
    &\leq& \Delta_{L_\vt+L_{\Sigma}(L_X+1)}(Q,\tilde Q).
    \label{eq:stab-diff-vtildex-ctildexi}
\end{eqnarray}
Since \(\nu\geq1\), the growth condition \eqref{eq:growth-xi-nu} also applies to the randomized policy induced by \(\tilde{\bm x}\). Assume for the sake of a contradiction that
\begin{eqnarray}
\mathbb{E}_{\pi^*(Q,\tilde Q)}\left[d(\tilde{\bm x},\mathcal X^*(\tilde\xi))\right]>\left(\frac{1}{\beta}\Delta_{L_\vt+L_{\Sigma}(L_X+1)}(Q,\tilde Q)\right)^{\frac{1}{\nu}}.
\label{eq:contra-optsolu}
\end{eqnarray}
Then by 
\eqref{eq:growth-xi-nu} (replacing 
$\xi$ with $\tilde{\xi}$)
and \eqref{eq:contra-optsolu},
\begin{eqnarray}
\bigg| \mathbb{E}_{\pi^*(Q,\tilde Q)}\left[\sum_{t=0}^T C_t(\tilde s_t,\tilde x_t,\tilde\xi_{[t]},\zeta_t)\right]-\vt(\tilde Q)\bigg|>\Delta_{L_\vt+L_{\Sigma}(L_X+1)}(Q,\tilde Q),\nonumber
\end{eqnarray}
which contradicts \eqref{eq:stab-diff-vtildex-ctildexi}.
Thus
\begin{eqnarray}
\mathbb{E}_{\pi^*(Q,\tilde Q)}\left[d(\tilde{\bm x},\mathcal X^*(\tilde\xi))\right]\leq\left(\frac{1}{\beta}\Delta_{L_\vt+L_{\Sigma}(L_X+1)}(Q,\tilde Q)\right)^{\frac{1}{\nu}}.
\label{eq:optsolu-projection}
\end{eqnarray}
Since \(\tilde{x}_t\) is the projection of \(x_t^*\) onto $\mathcal{X}_t(\tilde{s}_t, \tilde{x}_{t - 1}, \tilde{\xi}_{[t]})$, by Proposition \ref{Prop-stab-exo-feas}
\begin{eqnarray}
    \Vert \bm x^* - \tilde{\bm x} \Vert \leq \mathbb{H}(\mathcal{X}(\xi), \mathcal{X}(\tilde{\xi})) \leq L_X({\xi, \tilde{\xi}}) \Vert \xi - \tilde{\xi} \Vert. \label{eq:optsolu-x-tildex}
\end{eqnarray}
Since $x^*$ is chosen arbitrarily, combining \eqref{eq:optsolu-projection} and \eqref{eq:optsolu-x-tildex}, we have 
for any \(\tilde{\bm x}^* \in \mathcal{X}^*(\tilde{\xi})\), 
\begin{eqnarray}
    && \mathbb{E}_{\pi^*(Q, \tilde{Q})} \left[ d(\tilde{\bm x}^*, \mathcal{X}^*(\xi)) \right] \leq \Delta_{L_X}(Q, \tilde{Q}) + \left( \frac{1}{\beta} \Delta_{L_\vt + L_{\Sigma}(L_X + 1)}(Q, \tilde{Q}) \right)^{\frac{1}{\nu}}. \nonumber
\end{eqnarray}
Since \(\tilde{\bm x}^*\) 
is chosen arbitrarily from \(\mathcal{X}^*(\tilde{\xi})\), then
\begin{eqnarray}
    && \mathbb{E}_{\pi^*(Q, \tilde{Q})} \left[ \mathbb{D}(\mathcal{X}^*(\tilde{\xi}), \mathcal{X}^*(\xi)) \right] \leq \Delta_{L_X}(Q, \tilde{Q}) + \left( \frac{1}{\beta} \Delta_{L_\vt + L_{\Sigma}(L_X + 1)}(Q, \tilde{Q}) \right)^{\frac{1}{\nu}}. \nonumber
\end{eqnarray}
By swapping 
the positions between  
\(\xi\) and \(\tilde{\xi}\), we obtain \eqref{eq:stab-optimalsolu-xi}.

Next, we can deduce that \eqref{eq:cont-exo-L_Sigma1} and \eqref{eq:cont-exo-L_Sigma2} hold with the coefficients replaced by 
    \begin{eqnarray}
       L_{\Sigma, 1} &= & L_C\left(T + 1 + \sum_{t = 1}^T \sum_{k = 1}^{t} L_S^k\right) = L_C \frac{T - 1 - TL_S + L_S^{T + 2}}{(1-L_S)^2} 
       \label{eq:stab-optsolu-xi-FM-Lsigma1}
    \end{eqnarray}
    and
    \begin{eqnarray}
       L_{\Sigma, 2} &= & L_C\left(T + 1 + \sum_{t = 1}^T \sum_{k = 1}^{t} L_S^k\right) = L_C \left(T + 1 + \frac{T (T + 1)}{2}\right)
       \label{eq:stab-optsolu-xi-FM-Lsigma2}
    \end{eqnarray}    
    respectively. Let $L_{\Sigma} := \max\left\{ L_{\Sigma,1}, L_{\Sigma,2} \right\}$. Then analogous to \eqref{eq:optsolu-projection}, we obtain
    \begin{eqnarray}
        \mathbb{E}_{\pi^*(Q, \tilde{Q})} \left[ d( \tilde{\bm x}, \mathcal{X}^*(\tilde{\xi}) ) \right] &\leq & \left( \frac{1}{\beta} \Delta_{L_\vt + L_{\Sigma}(L_X + 1)}(Q, \tilde{Q}) \right)^{\frac{1}{\nu}} \nonumber \\
        &= & \left(
        \frac{1}{\beta} \left( L_\vt + L_{\Sigma}(L_X+1) \right)
        \sup_{f\in \mathcal{F}_{3T+1}(\Xi)}
        \left|
        \int_{\Xi} f(\xi) Q(\mathrm d\xi)
        -
        \int_{\Xi} f(\xi) \tilde Q(\mathrm d\xi)
        \right|
        \right)^{\frac{1}{\nu}} \nonumber \\
        &= & \left( \frac{1}{\beta} (L_{\vt} + L_{\Sigma} L_X + L_{\Sigma}) \Delta_{3T + 1}(Q, \tilde{Q}) \right)^{\frac{1}{\nu}}.
        \label{eq:stab-optsolu-FM-projtoopt}
    \end{eqnarray}
    Combining \eqref{eq:stab-optsolu-FM-projtoopt} and Proposition \ref{Prop-stab-exo-feas}, we obtain \eqref{eq:stab-exo-optsolu-FM} by the same argument as in the proof of Theorem \ref{Theo-stab-exo}. 
\end{proof}

\subsection{Stagewise perturbations and interactions}

The stability results established in the preceding discussions are based on the perturbation of the whole stochastic process $\xi$.
In practice, it might be desirable to 
consider perturbations at each stage and
their impact on the value function locally.
It will also be interesting to investigate 
inter-stage effects of these perturbations, e.g.,
propagation of the effect of perturbation
at the current stage 
on the value functions at later stages. 
We begin with a quantitative stability result 
on the optimal value function.

\begin{theorem}[Quantitative stability of the optimal value]
Assume that (a) Assumptions \ref{assumption7} and 
\ref{assumption8} hold; 
(b) the conditions in Theorem \ref{Theo-stab-exo-optvalue} hold with $L_C({\xi, \tilde{\xi}}) := L_C, L_S({\xi, \tilde{\xi}}) := L_S$ and $L_g({\xi, \tilde{\xi}}) := L_g$
being independent of $\xi$; 
(c) for \(t = 2, \cdots, T\) and given the distribution \(Q_{[1:t-1]}\) of \(\xi_{[t - 1]}\), \(\mathbb{E}_{\xi_{[t-1]}} \left[ \dd_K(Q_t(\xi_t | \xi_{[t - 1]}), \tilde{Q}_t(\tilde{\xi}_t | \xi_{[t - 1]})) \right] < \infty\) when the distribution \(Q_t \in \mathscr{Q}_t\) 
of \(\xi_t\) is perturbed to \(\tilde{Q}_t \in \mathscr{Q}_t\);
(d) the conditional distributions satisfy the Lipschitz continuity condition:
    \begin{subequations}
    \begin{eqnarray}
    && \dd_K(Q_{t} (\xi_t | \xi_{[t-1]}), Q_t ({\xi}_t | \tilde{\xi}_{[t-1]})) \leq L_{Q_{t}} \Vert \tilde{\xi}_{[t-1]} - \xi_{[t-1]} \Vert,  \label{eq:P_wrt_xi1} \\
    && \dd_K(\tilde{Q}_{t} (\tilde{\xi}_t | \xi_{[t-1]}), \tilde{Q}_t (\tilde{\xi}_t | \tilde{\xi}_{[t-1]})) \leq L_{Q_{t}} \Vert \tilde{\xi}_{[t-1]} - \xi_{[t-1]} \Vert. \label{eq:P_wrt_xi2}
    \end{eqnarray}        
    \label{eq:P_wrt_xi}
    \end{subequations}
    Then
    \begin{itemize}
    \item[(i)] for $t = 1, 2, \cdots, T$, $v_t$ is Lipschitz continuous w.r.t. $(s_{t - 1}, x_{t - 1}, \xi_{[t]})$, i.e., there exists a constant $L_{v, t} > 0$ such that for any $(s_{t - 1}, x_{t - 1}, \xi_{[t]})$ and $(\hat{s}_{t - 1}, \hat{x}_{t - 1}, \hat{\xi}_{[t]})$,
    \begin{eqnarray}
    && v_t(s_{t - 1}, x_{t - 1}, \xi_{[t]}, \zeta_{t - 1}) - v_t(\hat{s}_{t - 1}, \hat{x}_{t - 1}, \hat{\xi}_{[t]}, \zeta_{t - 1}) \nonumber \\
    &\leq & L_{v, t}(\Vert s_{t - 1} - \hat{s}_{t - 1} \Vert + \Vert x_{t - 1} - \hat{x}_{t - 1} \Vert + \Vert \xi_{[t]} - \hat{\xi}_{[t]} \Vert) 
    \label{eq:Lip-xi-v_t}
    \end{eqnarray}
    where $L_{v, t} := (L_{C, t} + \max\{ L_{v, t + 1}, L_{v, t + 1} L_{Q_{t + 1}} \})(\frac{L_g A}{\rho} + 1)(L_S + 1)$;
    \item[(ii)] there exists a series of constants \(L_{\xi, t} > 0, t = 1, \cdots, T\) such that
    \begin{eqnarray}
    |\vt(Q) - \vt(\tilde{Q})| \leq \sum_{t = 1}^T L_{\xi,t} \mathbb{E}_{\xi_{[t - 1]}} \left[ \dd_K(Q_t(\xi_t | \xi_{[t - 1]}), \tilde{Q}_t(\tilde{\xi}_t | \xi_{[t - 1]})) \right]. 
    \label{eq:stab-xi-stagewise-val}
    \end{eqnarray}
    \end{itemize}
    \label{Theo-stab-exo-stage}
\end{theorem}
\begin{proof}
    \underline{Part (i)}.
    We 
    prove \eqref{eq:Lip-xi-v_t}
    by induction
    from $t=T$ backward to $t=1$.
    For any $(s_{T - 1}, x_{T - 1}, \xi_{[T]})$ and $(\hat{s}_{T - 1}, \hat{x}_{T - 1}, \hat{\xi}_{[T]})$, 
    \begin{eqnarray}
    && v_T(s_{T - 1}, x_{T - 1}, \xi_{[T]}, \zeta_{T - 1}) - v_T(\hat{s}_{T - 1}, \hat{x}_{T - 1}, \hat{\xi}_{[T]}, \zeta_{T - 1}) \nonumber \\
    &= & C_T(s_T, x_T^*, \xi_{[T]}, \zeta_{T}) - C_T(\hat{s}_T, \hat{x}_T^*, \hat{\xi}_{[T]}, \zeta_T) \nonumber \\
    &\leq & C_T(s_T, y_T, \xi_{[T]}, \zeta_T) - C_T(\hat{s}_T, \hat{x}_T^*, \hat{\xi}_{[T]}, \zeta_T) \nonumber \\
    &\leq & L_{C, T} (\Vert s_T - \hat{s}_T \Vert + \Vert y_T - \hat{x}^*_T \Vert + \Vert \xi_{[T]} - \hat{\xi}_{[T]} \Vert) \nonumber \\
    &\leq & L_{C, T}(\Vert s_T - \hat{s}_T \Vert + \frac{L_g A}{\rho}(\Vert s_T - \hat{s}_T \Vert + \Vert \xi_{[T]} - \hat{\xi}_{[T]} \Vert + \Vert x_{T - 1} - \hat{x}_{T - 1} \Vert) + \Vert \xi_{[T]} - \hat{\xi}_{[T]} \Vert) \nonumber \\
    &\leq & L_{C, T}(L_S + \frac{L_g A}{\rho} L_S + \frac{L_g A}{\rho} + 1) (\Vert s_{T - 1} - \hat{s}_{T - 1} \Vert + \Vert x_{T - 1} - \hat{x}_{T - 1} \Vert + \Vert \xi_{[T]} - \hat{\xi}_{[T]} \Vert), \nonumber
    \end{eqnarray}
    where $x^*_T$ and $\hat{x}^*_T$ represent the optimal solutions of problem \eqref{eq:value-T} at stage $T$ before and after the distribution perturbation of $\xi_{[T]}$, 
    respectively, and $y_T$ is the projection of $\hat{x}^*_T$ onto $\mathcal{X}_T(s_T, x_{T - 1}, \xi_{[T]})$. The second inequality follows from condition (b), the third inequality follows from Proposition \ref{Prop-stab-exo-feas}, the last inequality follows from condition (b). The inequality 
    \eqref{eq:Lip-xi-v_t}
    follows by setting 
    $L_{v, T} := L_{C, T}(L_S + 1)(\frac{L_g A}{\rho} + 1)$. 
    Next,   
    assume that \eqref{eq:Lip-xi-v_t} holds 
    for $t \geq k + 1$.
    We prove  \eqref{eq:Lip-xi-v_t}
    for $t = k$. Analogous to the case $t=T$, we can use conditions (b) and Proposition \ref{Prop-stab-exo-feas} to establish
    \begin{eqnarray}
    && v_k(s_{k - 1}, x_{k - 1}, \xi_{[k]}, \zeta_{k - 1}) - v_k(\hat{s}_{k - 1}, \hat{x}_{k - 1}, \hat{\xi}_{[k]}, \zeta_{k - 1}) \nonumber \\
    &= & C_k(s_k, x_k^*, \xi_{[k]}, \zeta_k) - C_k(\hat{s}_k, \hat{x}_k^*, \hat{\xi}_{[k]}, \zeta_k) + \mathbb{E}_{\xi_{k + 1} | \xi_{[k]}} \left[ 
    v_{k + 1}(s_k, x_k^*, \xi_{[k + 1]}, \zeta_{k}) \right] \nonumber \\ 
    && - \mathbb{E}_{\hat{\xi}_{k + 1} | \hat{\xi}_{[k]}} \left[ 
    v_{k + 1}(\hat{s}_k, \hat{x}_k^*, \hat{\xi}_{[k + 1]}, \zeta_{k}) \right] \nonumber \\
    &\leq & C_k(s_k, y_k, \xi_{[k]}, \zeta_k) - C_k(\hat{s}_k, \hat{x}_k^*, \hat{\xi}_{[k]}, \zeta_k) + \mathbb{E}_{\xi_{k + 1} | \xi_{[k]}} \left[ 
    v_{k + 1}(s_k, y_k, \xi_{[k + 1]}, \zeta_{k}) \right] \nonumber \\ 
    && - \mathbb{E}_{\hat{\xi}_{k + 1} | \hat{\xi}_{[k]}} \left[ 
    v_{k + 1}(\hat{s}_k, \hat{x}_k^*, \hat{\xi}_{[k + 1]}, \zeta_{k}) \right] \nonumber \\
    &\leq & L_{C, k} (\Vert s_k - \hat{s}_k \Vert + \Vert y_k - \hat{x}_k^* \Vert + \Vert \xi_{[k]} - \hat{\xi}_{[k]} \Vert)  + \mathbb{E}_{\xi_{k + 1} | \xi_{[k]}} \left[ v_{k + 1}(s_k, y_k, \xi_{[k + 1]}, \zeta_{k}) \right] \nonumber \\
    && - \mathbb{E}_{\hat{\xi}_{k + 1} | \hat{\xi}_{[k]}} \left[ 
    v_{k + 1}(\hat{s}_k, \hat{x}_k^*, \hat{\xi}_{[k + 1]}, \zeta_{k}) \right],
    \label{eq:stab-exo-stage-Lip}
    \end{eqnarray}
    where $y_k$ is the projection of $\hat{x}^*_k$ onto $\mathcal{X}_k(s_k, x_{k - 1}, \xi_{[k]})$.
    By 
    induction,
    \begin{eqnarray}
    && v_{k + 1}(s_k, y_k, \xi_{[k + 1]}, \zeta_{k}) - v_{k + 1}(\hat{s}_k, \hat{x}_k^*, \hat{\xi}_{[k + 1]}, \zeta_{k}) \nonumber \\
    &\leq & L_{v, k + 1}(\Vert s_k - \hat{s}_k \Vert + \Vert y_k - \hat{x}_k^* \Vert + \Vert \xi_{[k + 1]} - \hat{\xi}_{[k + 1]} \Vert), \nonumber
    \end{eqnarray}
    which implies  $v_{k+1}$ is Lipschitz continuous in $(s_{k}, x_{k}, \xi_{[k]})$ with modulus $L_{v, k+1}$. 
    Thus
    \begin{eqnarray} 
    && \left|\mathbb{E}_{\xi_{k + 1} | \xi_{[k]}} \left[ v_{k + 1}(s_k, y_k, \xi_{[k + 1]}, \zeta_{k}) \right] 
    - \mathbb{E}_{\hat{\xi}_{k + 1} | \hat{\xi}_{[k]}} \left[ 
    v_{k + 1}(\hat{s}_k, \hat{x}_k^*, \hat{\xi}_{[k + 1]}, \zeta_{k}) \right]\right| \nonumber \\
    &\leq & \left|\mathbb{E}_{\xi_{k + 1} | \xi_{[k]}} \left[ v_{k + 1}(s_k, y_k, \xi_{[k + 1]}, \zeta_{k}) \right] 
    - \mathbb{E}_{\xi_{k + 1} | \xi_{[k]}} \left[ 
    v_{k + 1}(\hat{s}_k, \hat{x}_k^*, \xi_{[k + 1]}, \zeta_{k}) \right]\right| \nonumber \\
    && + \left| \mathbb{E}_{\xi_{k + 1} | \xi_{[k]}} \left[v_{k + 1}(\hat{s}_k, \hat{x}_k^*, \xi_{[k + 1]}, \zeta_{k} ) \right] - \mathbb{E}_{\hat{\xi}_{k + 1} | \hat{\xi}_{[k]}} \left[v_{k + 1}(\hat{s}_k, \hat{x}_k^*, \hat{\xi}_{[k + 1]}, \zeta_{k}) \right]\right| \nonumber \\
    &\leq & L_{v, k + 1} (\Vert s_k - \hat{s}_k \Vert + \Vert y_k - \hat{x}_k^* \Vert) + L_{v, k + 1} \dd_K(Q_{k + 1}(\xi_{k + 1} | \xi_{[k]}), Q_{k + 1}({\xi}_{k + 1} | \hat{\xi}_{[k]})).
    \label{eq:Kanto-4.53}
    \end{eqnarray} 
    The second inequality is obtained by the definition of the Kantorovich metric and the Lipschitz continuity of $v_{k + 1}$. Thus 
    \begin{eqnarray}
    \text{rhs of } \eqref{eq:stab-exo-stage-Lip} 
    &\leq & L_{C, k} (\Vert s_k - \hat{s}_k \Vert + \Vert y_k - \hat{x}_k^* \Vert + \Vert \xi_{[k]} - \hat{\xi}_{[k]} \Vert) \nonumber \\
    && + L_{v, k + 1} (\Vert s_k - \hat{s}_k \Vert + \Vert y_k - \hat{x}_k^* \Vert) + L_{v, k+1} L_{Q_{k + 1}} \Vert \xi_{[k]} - \hat{\xi}_{[k]} \Vert \nonumber \\
    &\leq & (L_{C, k} + {\max}\{ L_{v, k + 1}, L_{v, k + 1} L_{Q_{k + 1}} \}) (\Vert s_k - \hat{s}_k \Vert + \Vert y_k - \hat{x}_k^* \Vert + \Vert \xi_{[k]} - \hat{\xi}_{[k]} \Vert) \nonumber \\
    &\leq & (L_{C, k} + {\max}\{ L_{v, k + 1}, L_{v, k + 1} L_{Q_{k + 1}} \})(\frac{L_g A}{\rho} + 1)(L_S + 1) \nonumber \\
    && (\Vert s_{k - 1} - \hat{s}_{k - 1} \Vert + \Vert x_{k - 1} - \hat{x}_{k - 1} \Vert + \Vert \xi_{[k]} - \hat{\xi}_{[k]} \Vert). \nonumber
    \end{eqnarray}
    The first inequality follows from
    \eqref{eq:P_wrt_xi}, while the other two inequalities follow from Proposition \ref{Prop-stab-exo-feas} and the Lipschitz property of $S^M_{k - 1}$. 
    This shows inequality \eqref{eq:Lip-xi-v_t} holds for $t = k$ by setting
    \[
    L_{v, k} := (L_{C, k} + \max\{ L_{v, k + 1}, L_{v, k + 1} L_{Q_{k + 1}} \})(\frac{L_g A}{\rho} + 1)(L_S + 1).
    \]
    
    \underline{Part (ii)}. 
    When the distribution of \(\xi\) is perturbed from \(Q\) to \(\tilde{Q}\), the distribution of \(\xi_t\) at stage \(t\) is correspondingly perturbed from \(Q_t (\xi_t \mid \xi_{[t-1]})\) to \(\tilde{Q}_t (\tilde{\xi}_t \mid \tilde{\xi}_{[t-1]})\).
    Let \(x_t^*\left(s_t^*, x_{t-1}^*, \xi_{[t]}\right)\) and \(\tilde{x}_t^*\left(\tilde{s}_t^*, \tilde{x}_{t - 1}^*, \tilde{\xi}_{[t]}\right)\) be optimal solutions of problem \eqref{eq:value-t} at stage \(t\) before and after perturbations.
    Recall that \(s_0^* = \tilde{s}_0^* = s_0\) by setup and
    \begin{eqnarray}
    s_t^* &= & S^M_{t - 1}(s_{t - 1}^*, x_{t - 1}^*, \xi_{[t - 1]}, \zeta_{t - 1}), \quad t = 1, 2, \cdots, T \nonumber \\
    \tilde{s}_t^* &= & S^M_{t - 1}(\tilde{s}_{t - 1}^*, \tilde{x}_{t - 1}^*, \tilde{\xi}_{[t - 1]}, \zeta_{t - 1}), \quad t = 1, 2, \cdots, T. \nonumber
    \end{eqnarray}
   We consider the difference between the optimal values of \eqref{eq:v-xi} and \eqref{eq:v-xi-ptb}
    \begin{eqnarray}
     \vt(Q) - \vt(\tilde{Q})
     &= & \vt(Q) - \vt((\tilde{Q}_T, Q_{[1 : T - 1]})) +
    \vt((\tilde{Q}_T, Q_{[1 : T - 1]})) - \vt((\tilde{Q}_{[T - 1 : T]}, Q_{[1 : T - 2]}))\nonumber\\ 
    &&+ 
    \cdots 
+ \vt((\tilde{Q}_{[2 : T]}, Q_{1})) - \vt(\tilde{Q}),
    \label{eq:opt-value-xi} 
    \end{eqnarray}
    where 
    \begin{eqnarray} 
    \vt\left(\tilde{Q}_{[t + 1 : T]}, Q_{[1 : t]}\right)
     &:= & \underset{\bm x \in \mathcal{X}}{\min} \quad \mathbb{E}_{\xi_1 \cdots, \xi_{t}, \tilde{\xi}_{t + 1}, \cdots, \tilde{\xi}_T} \Big[ C_0(s_0, x_0, \zeta_0) \nonumber \\
     && + \sum_{k = 1}^{t} C_k(s_k, x_k, \xi_{[k]}, \zeta_k) + \sum_{k = t + 1}^T C_k(s_k, x_k, \tilde{\xi}_{[k]}, \zeta_k)\Big].
     \label{eq:v-xi-ptbd}
    \end{eqnarray} 
     Let $\tilde{\bm x}^{*,t + 1}:= (\tilde{x}_0^{*,t + 1}, \cdots, \tilde{x}_T^{*,t + 1})$ 
    be an optimal policy to \eqref{eq:v-xi-ptbd}, and $\mathcal{X}^{*, t + 1}$ be the set of optimal policies to \eqref{eq:v-xi-ptbd}.
    Note that  $\tilde{\bm x}^{*, 1} = \tilde{\bm x}^*$, $\mathcal{X}^{*, 1} = \mathcal{X}^{*}(\tilde{\xi})$.
    Let $\tilde{v}_{t + 1}(s_{t}, x_{t}, \xi_{[t + 1]}, \zeta_{t}) :=  v_{t + 1}(s_{t}, x_{t}, (\xi_{[t]}, \tilde{\xi}_{t + 1}), \zeta_{t})$.
    For $k = t, t-1, \cdots, 1$, define
    \begin{eqnarray*}
    &&\tilde{v}_k(s_{k - 1}, x_{k - 1}, \xi_{[k]}, \zeta_{k - 1}) \\
    &:= & \underset{x_k \in \mathcal{X}_k(s_k, x_{k - 1}, \xi_{[k]})}{\text{min}} \mathbb{E}_{\zeta_k} \left[ C_k(s_k, x_k, \xi_{[k]}, \zeta_k) + \mathbb{E}_{\xi_{[k + 1]} | \xi_{[k]}} \left[ \tilde{v}_{k + 1}(s_k, x_k, \xi_{[k + 1]}, \zeta_k) \right] \right].  
    \end{eqnarray*}
    In what follows, we estimate
$\vt((\tilde{Q}_{[t + 1 : T]}, Q_{[1 : t]})) - \vt((\tilde{Q}_{[t : T]}, Q_{[1 : t - 1]}))$ for $t = T, T - 1, 
\cdots,1$.  
    Let \(y_T\) be the projection of \(\tilde{x}^{*,T}_T\) onto \(\mathcal{X}_T(s_T, x_{T - 1}, \xi_{[T]})\). 
    By the Lipschitz continuity of \(C_T(s_T, x_T, \xi_{[T]}, \zeta_T)\) 
    in \((s_T, x_T, \xi_{[T]})\) 
and Proposition \ref{Prop-stab-exo-feas}, we have
    \begin{eqnarray}
    && v_T(s_{T - 1}, x_{T - 1}, \xi_{[T]}, \zeta_{T - 1}) - v_T(s_{T - 1}, x_{T - 1}, \tilde{\xi}_{[T]}, \zeta_{T - 1}) \nonumber \\
    &= & \mathbb{E}_{\zeta_T} \left[ C_T(s_T, x^*_{T}, \xi_{[T]}, \zeta_T) - C_T(\tilde{s}_T, \tilde{x}^{*, T}_T, \tilde{\xi}_{[T]}, \zeta_T) \right] \nonumber \\
    &\leq & \mathbb{E}_{\zeta_T} \left[ C_T(s_T, y_{T}, \xi_{[T]}, \zeta_T) - C_T(\tilde{s}_T, \tilde{x}^{*,T}_T, \tilde{\xi}_{[T]}, \zeta_T) \right] \nonumber \\
    &\leq & L_{C, T}(\Vert s_T - \tilde{s}_T \Vert + \Vert y_T - \tilde{x}^{*,T}_T \Vert + \Vert \xi_{[T]} - \tilde{\xi}_{[T]} \Vert) \nonumber \\
    &\leq & L_{C, T}(L_S + L_{X, T} + 1) \Vert \xi_{[T]} - \tilde{\xi}_{[T]} \Vert \nonumber \\
    &= & L_{C, T}(L_S + L_{X, T} + 1) \Vert \xi_{T} - \tilde{\xi}_{T} \Vert, \nonumber
    \end{eqnarray}
    where the second inequality holds due to its independence of $\zeta_T$; 
    the last equality holds 
    since only the distribution of $\xi_T$ is perturbed. 
    By swapping 
    the positions between 
    $\xi$ and $\tilde{\xi}$,
    we obtain the Lipschitz continuity of $v_T(\cdot)$  in  $\xi_{T}$ with modulus $L_{C, T}(L_S + L_{X, T} + 1)$.    
By the definition of the Kantorovich metric 
    \begin{eqnarray}
    && \mathbb{E}_{\xi_T | \xi_{[T - 1]}} \left[ v_T(s_{T - 1}, x_{T - 1}, \xi_{[T]}, \zeta_{T - 1})\right] - \mathbb{E}_{\tilde{\xi}_T | \xi_{[T - 1]}} \left[ v_T(s_{T - 1}, x_{T - 1}, \tilde{\xi}_{[T]}, \zeta_{T - 1}) \right] \nonumber \\
    &\leq & L_{C, T}(L_S + L_{X, T} + 1) \dd_K\big(Q_T(\xi_T \mid \xi_{[T - 1]}), \tilde{Q}_T(\tilde{\xi}_T \mid \xi_{[T - 1]})\big), \forall (s_{T - 1}, x_{T - 1}).
    \end{eqnarray}
    Similar to the derivation in \eqref{eq:stab-optvalue-zeta-step1},
    we can establish for 
    $t=1,\cdots,T-1$, 
    \begin{eqnarray}
    &&  \mathbb{E}_{\xi_t | \xi_{[t - 1]}} \left[ v_t(s_{t - 1}, x_{t - 1}, \xi_{[t]}, \zeta_{t - 1}) - \tilde{v}_t(s_{t - 1}, x_{t - 1}, {\xi}_{[t]}, \zeta_{t - 1}) \right] \nonumber \\
    &\leq & L_{C, T}(L_S + L_{X, T} + 1) \mathbb{E}_{\xi_{[t : T - 1]} | \xi_{[t - 1]}} \left[ \dd_K(Q_T(\xi_T \mid \xi_{[T - 1]}), \tilde{Q}_T(\tilde{\xi}_T \mid \xi_{[T - 1]})) \right]. 
    \label{eq:vt-stagexi-T}
    \end{eqnarray}
    By \eqref{eq:vt-stagexi-T} with $t=1$, we have
    \begin{eqnarray}
    && \vt(Q) - \vt((\tilde{Q}_T, Q_{[1 : T - 1]})) \nonumber \\
    &= & \mathbb{E}_{\zeta_0} \left[ C_0(s_0, x_0^*, \zeta_0) - C_0 (s_0, \tilde{x}_0^{*, T}, \zeta_0) + \mathbb{E}_{\xi_1} \left[  v_1(s_0, x_0^*, \xi_1, \zeta_0) - \tilde{v}_1(s_0, \tilde{x}_0^{*, T}, \xi_1, \zeta_0) \right] \right]\nonumber \\
    &\leq & \mathbb{E}_{\zeta_0, \xi_1} \left[ v_1(s_0, \tilde{x}_0^{*, T}, \xi_1, \zeta_0) - \tilde{v}_1(s_0, \tilde{x}_0^{*, T}, \xi_1, \zeta_0)\right] \nonumber \\
    &\leq & L_{C, T}(L_S + L_{X, T} + 1) \mathbb{E}_{\xi_{[T - 1]}} \left[ \dd_K(Q_T(\xi_T \mid \xi_{[T - 1]}), \tilde{Q}_T(\tilde{\xi}_T \mid \xi_{[T - 1]})) \right]. 
    \label{eq:optvalue-stage-T}
    \end{eqnarray}
    We are now ready to estimate
$\vt((\tilde{Q}_{[t + 1 : T]}, Q_{[1 : t]})) - \vt((\tilde{Q}_{[t : T]}, Q_{[1 : t - 1]}))$. Observe that 
    \begin{eqnarray}
    && v_t(s_{t - 1}, x_{t - 1}, \xi_{[t]}, \zeta_{t - 1}) - v_t(s_{t - 1}, x_{t - 1}, (\xi_{[t - 1]}, \tilde{\xi}_{t}), \zeta_{t - 1}) \nonumber \\
    &= & C_t(s_t, \tilde{x}_t^{*, t + 1}, \xi_{[t]}, \zeta_t) + \mathbb{E}_{\tilde{\xi}_{t + 1}| \xi_{[t]}} \left[ v_{t + 1}(s_t, \tilde{x}_t^{*, t + 1}, ({\xi}_{[t]}, \tilde{\xi}_{t + 1}), \zeta_{t}) \right] \nonumber \\
    && - C_t(s_t, \tilde{x}_t^{*, t}, (\xi_{[t - 1]}, \tilde{\xi}_{t}), \zeta_t) - \mathbb{E}_{\tilde{\xi}_{t + 1}| (\xi_{[t - 1]}, \tilde{\xi}_t)} \left[ v_{t + 1}(s_t, \tilde{x}_t^{*, t}, ({\xi}_{[t - 1]}, \tilde{\xi}_{[t : t + 1]}), \zeta_{t}) \right] \nonumber \\
    &\leq & C_t(s_t, y_t, \xi_{[t]}, \zeta_t) + \mathbb{E}_{\tilde{\xi}_{t + 1}| \xi_{[t]}} \left[ v_{t + 1}(s_t, y_t, ({\xi}_{[t]}, \tilde{\xi}_{t + 1}), \zeta_{t}) \right] \nonumber \\
    && - C_t(s_t, \tilde{x}_t^{*, t}, (\xi_{[t - 1]}, \tilde{\xi}_{t}), \zeta_t) - \mathbb{E}_{\tilde{\xi}_{t + 1}| (\xi_{[t - 1]}, \tilde{\xi}_t)} \left[ v_{t + 1}(s_t, \tilde{x}_t^{*, t}, ({\xi}_{[t - 1]}, \tilde{\xi}_{[t : t + 1]}), \zeta_{t}) \right] \nonumber \\
    &\leq & L_{C, t}(\Vert \xi_{t} - \tilde{\xi}_{t} \Vert + \Vert y_t - \tilde{x}_t^{*, t} \Vert) \nonumber \\
    && + \mathbb{E}_{\tilde{\xi}_{t + 1}| \xi_{[t]}} \left[ v_{t + 1}(s_t, y_t, ({\xi}_{[t]}, \tilde{\xi}_{t + 1}), \zeta_{t}) \right] - \mathbb{E}_{\tilde{\xi}_{t + 1}| (\xi_{[t - 1]}, \tilde{\xi}_t)} \left[ v_{t + 1}(s_t, \tilde{x}_t^{*, t}, ({\xi}_{[t - 1]}, \tilde{\xi}_{[t : t + 1]}), \zeta_{t}) \right] \nonumber \\
    &\leq & L_{C, t}(L_{X, t} + 1) \Vert \xi_{t} - \tilde{\xi}_{t} \Vert  + \mathbb{E}_{\tilde{\xi}_{t + 1}| \xi_{[t]}} \left[ v_{t + 1}(s_t, y_t, ({\xi}_{[t]}, \tilde{\xi}_{t + 1}), \zeta_{t}) \right] \nonumber \\
    && - \mathbb{E}_{\tilde{\xi}_{t + 1}| (\xi_{[t - 1]}, \tilde{\xi}_t)} \left[ v_{t + 1}(s_t, \tilde{x}_t^{*, t}, ({\xi}_{[t - 1]}, \tilde{\xi}_{[t : t + 1]}), \zeta_{t}) \right],
    \label{eq:stab-t:T disturbed}
    \end{eqnarray}
    where $y_t$ represents the projection of $\tilde{x}_t^{*, t}$ onto $\mathcal{X}_t(s_t, x_{t - 1}, \xi_{[t]})$. 
    By the Lipschitz continuity of $v_t$ w.r.t.~$(s_{t - 1}, x_{t - 1}, \xi_{[t]})$ proved in (i), we can further quantify the right-hand side of \eqref{eq:stab-t:T disturbed}.
    Analogous to 
    \eqref{eq:Kanto-4.53},
    we can obtain
    \begin{eqnarray}
    \text{rhs of } \eqref{eq:stab-t:T disturbed}
    &\leq & L_{C, t} (L_{X, t} + 1) \Vert \xi_{[t]} - \tilde{\xi}_{[t]} \Vert + L_{v, t + 1} \Vert y_t - \tilde{x}^*_t \Vert + L_{v, t + 1} L_{Q_{t + 1}} \Vert \xi_{[t]} - \tilde{\xi}_{[t]} \Vert \nonumber \\
    &\leq & (L_{C, t} (L_{X, t} + 1) + L_{v, t + 1}(L_{X, t} + L_{Q_{t + 1}})) \Vert \xi_{[t]} - \tilde{\xi}_{[t]} \Vert. \nonumber
    \end{eqnarray}
    By using the definition of the Kantorovich metric 
    \begin{eqnarray}
    && \mathbb{E}_{\xi_t | \xi_{[t - 1]}} \left[ v_t(s_{t - 1}, x_{t - 1}, \xi_{[t]}, \zeta_t) \right] - \mathbb{E}_{\tilde{\xi}_t | \xi_{[t - 1]}} \left[ v_t(s_{t - 1}, x_{t - 1}, \tilde{\xi}_{[t]}, \zeta_t) \right] \nonumber \\
    &\leq & (L_{C, t} (L_{X, t} + 1) + L_{v, t + 1}(L_{X, t} + L_{Q_{t + 1}})) \dd_K(Q_t(\xi_t | \xi_{[t - 1]}), \tilde{Q}_t(\tilde{\xi}_t | \xi_{[t - 1]})). \nonumber
    \end{eqnarray}
    Thus
    \begin{eqnarray}
    && \vt((\tilde{Q}_{[t + 1 : T]}, Q_{[1 : t]})) - \vt((\tilde{Q}_{[t : T]}, Q_{[1 : t - 1]})) \nonumber \\
    &= & C_0(s_0, \tilde{x}_0^{*, t+1}, \zeta_0) - C_0(s_0, \tilde{x}_0^{*, t}, \zeta_0) + \mathbb{E}_{\xi_1}\left[ v_1(s_0, \tilde{x}_0^{*, t+1}, \xi_{1}, \zeta_0) - \tilde{v}_1(s_0, \tilde{x}_0^{*,t}, \xi_1, \zeta_0) \right] \nonumber \\
    &\leq& \mathbb{E}_{\xi_1}\left[ v_1(s_0, \tilde{x}_0^{*,t}, \xi_{1}, \zeta_0) - \tilde{v}_1(s_0, \tilde{x}_0^{*,t}, \xi_1, \zeta_0) \right] \nonumber \\
    &= & \mathbb{E}_{\xi_1}\left[ C_1(s_1, \tilde{x}_1^{*, t+1}, \xi_{1}, \zeta_1) - C_1(s_1, \tilde{x}_1^{*,t}, {\xi}_1, \zeta_1) \right. \nonumber \\
    && \left. + \mathbb{E}_{\xi_2 | \xi_1}\left[ v_2(s_1, \tilde{x}_1^{*,t+1}, \xi_{[2]}, \zeta_1) - \tilde{v}_2(s_1, \tilde{x}_1^{*,t}, \xi_{[2]}, \zeta_1) \right] \right] \nonumber \\
    &\leq & \mathbb{E}_{\xi_{[2]}}\left[ v_2(s_1, \tilde{x}_1^{*,t+1}, \xi_{[2]}, \zeta_1) - \tilde{v}_2(s_1, \tilde{x}_1^{*,t}, \xi_{[2]}, \zeta_1) \right] \leq \cdots \nonumber \\
    &\leq & \mathbb{E}_{\xi_{[t - 1]}}\left[ v_{t - 1}(s_{t - 2}, \tilde{x}_{t - 2}^{*, t+1}, \xi_{[t - 1]}, \zeta_{t - 2}) - \tilde{v}_{t - 1}(s_{t - 2}, \tilde{x}_{t - 2}^{*,t}, \xi_{[t - 1]}, \zeta_{t - 2}) \right] \nonumber \\
    &\leq & \mathbb{E}_{\xi_{[t - 1]}}\left[ \mathbb{E}_{\xi_t | \xi_{[t - 1]}} \left[ v_t(s_{t - 1}, \tilde{x}_{t - 1}^{*,t}, \xi_{[t]}, \zeta_{t - 1}) \right] - \mathbb{E}_{\tilde{\xi}_t | \xi_{[t - 1]}} \left[ v_t(s_{t - 1}, \tilde{x}_{t - 1}^{*,t}, \tilde{\xi}_{[t]}, \zeta_{t - 1}) \right] \right] \nonumber \\
    &\leq & \mathbb{E}_{\xi_{[t - 1]}}\Big[ (L_{C, t} (L_{X, t} + 1) + L_{v, t + 1}(L_{X, t} + L_{Q_{t + 1}}))  \dd_K(Q_t(\xi_t | \xi_{[t - 1]}), \tilde{Q}_t(\tilde{\xi}_t | \xi_{[t - 1]})) \Big]. \nonumber
    \end{eqnarray}
    By setting 
    $L_{\xi,t} := L_{C, t} (L_{X, t} + 1) + L_{v, t + 1}(L_{X, t} + L_{Q_{t + 1}})$ and summing up $\vt((\tilde{Q}_{[t + 1 : T]}, Q_{[1 : t]})) - \vt((\tilde{Q}_{[t : T]}, Q_{[1 : t - 1]}))$ 
    for $1 \leq t \leq T$ (see \eqref{eq:opt-value-xi}), 
    we arrive at
    \eqref{eq:stab-xi-stagewise-val}.
\end{proof}

    Analyzing the perturbation to exogenous random process stage-by-stage 
    has two main advantages:
    First, it helps us understand how a small perturbation at a specific stage affects the optimal value and decisions in later stages, which is hard to see if we treat $\xi$ as a whole. 
    Second, it allows us to break down the effects of random disturbances into two parts: the impact caused by 
    the inaccurate distribution model based on historical information, and the impact from variations in that historical information itself. This decomposition provides a clearer  description of how different stages interact in a complex and dynamic
    way.
    
    The assumptions in Theorem \ref{Theo-stab-exo-stage} are reasonable and often automatically satisfied in typical scenarios. Consider  
    for example, 
    a merchant trying to predict the random demand $\xi_t$ for a product in the $t$-th month in the future,
    the demand not only depends on the sale's current situation 
    but also is affected by the average demand over the past few months. In this case, the average demand in the past can be used as a predictor, which introduces some correlation among 
    exogenous random variables at different stages. Take the Gaussian distribution as an example for random demand, we assume that the demand distributions under different historical conditions are $P$ and $Q$ (the subscript $t$ is omitted here for simplicity), respectively. Concretely,
    \[
    P = \mathcal{N}\left( \frac{1}{t - 1} \sum_{k=1}^{t - 1} \xi_k, \Sigma \right), \quad Q = \mathcal{N}\left( \frac{1}{t - 1} \sum_{k=1}^{t - 1} \tilde{\xi}_k, \Sigma \right),
    \]
    where $\Sigma$ is the common covariance matrix, $\xi_k$ and $\tilde{\xi}_k, 1 \leq k \leq t - 1$, are historical realizations of the random demands $\boldsymbol{\xi}_k$ and $\tilde{\boldsymbol{\xi}}_k$.
    Let $p(x)$ and $q(x)$ be the probability density functions of $P$ and $Q$ respectively. 
    \footnote{In the case when random variables are discretely distributed, we can treat 
    $p$ and $q$ as probability distribution functions and 
    replace the subsequent integrations with summations. }
    Then
    \[
    q(\xi_t) = p\left(\xi_t + \frac{1}{t - 1} \sum_{k=1}^{t - 1} \xi_k - \frac{1}{t - 1} \sum_{k=1}^{t - 1} \tilde{\xi}_k\right).
    \]
    According to the definition of the Kantorovich metric, we have
     \begin{eqnarray}
    \dd_K(P, Q) &= & \underset{h \in \mathcal{F}_1(\Xi)}{\sup} \int_{\mathbb{R}} h(\boldsymbol{\xi}_t) dP(\boldsymbol{\xi}_t) - \int h(\boldsymbol{\xi}_t) dQ(\boldsymbol{\xi}_t) \nonumber \\
    &= & \underset{h \in \mathcal{F}_1(\Xi)}{\sup} \int_{\mathbb{R}} h(\boldsymbol{\xi}_t) (p(\boldsymbol{\xi}_t) - q(\boldsymbol{\xi}_t)) d \boldsymbol{\xi}_t \nonumber \\
    & =& \underset{h \in \mathcal{F}_1(\Xi)}{\sup} \int_{\mathbb{R}} h(\boldsymbol{\xi}_t) \left(p(\boldsymbol{\xi}_t) - p\left(\boldsymbol{\xi}_t + \frac{1}{t - 1} \sum_{k=1}^{t - 1} \xi_k - \frac{1}{t - 1} \sum_{k=1}^{t - 1} \tilde{\xi}_k\right)\right) d \boldsymbol{\xi}_t \nonumber \\
    & =& \underset{h \in \mathcal{F}_1(\Xi)}{\sup} \int_{\mathbb{R}} p(\boldsymbol{\xi}_t) \left(h(\boldsymbol{\xi}_t) - h\left(\boldsymbol{\xi}_t - \frac{1}{t - 1} \sum_{k=1}^{t - 1} \xi_k + \frac{1}{t - 1} \sum_{k=1}^{t - 1} \tilde{\xi}_k\right)\right) d \boldsymbol{\xi}_t \nonumber \\
    & \leq& \underset{h \in \mathcal{F}_1(\Xi)}{\sup} \int_{\mathbb{R}} p(\boldsymbol{\xi}_t) 
    \left\|\frac{1}{t - 1} \sum_{k=1}^{t - 1} \xi_k - \frac{1}{t - 1} \sum_{k=1}^{t - 1} \tilde{\xi}_k 
    \right\|d\boldsymbol{\xi}_t \nonumber \\
    & =& \left\| \frac{1}{t - 1} \sum_{k=1}^{t - 1} \xi_k - \frac{1}{t - 1} \sum_{k=1}^{t - 1} \tilde{\xi}_k \right\| \leq \Vert \xi_{[t - 1]} - \tilde{\xi}_{[t - 1]} \Vert. \nonumber
    \end{eqnarray}
    This example shows that the assumptions such as \eqref{eq:P_wrt_xi} in Theorem \ref{Theo-stab-exo-stage} are reasonable. Similar assumptions have also been adopted in studies like \cite{wozabalstability}.
\begin{remark}
The quantitative stability results w.r.t. the exogenous process \(\xi\) can be
understood from two different but complementary perspectives.

First, Theorem~\ref{Theo-stab-exo-optvalue} treats
\(\xi=(\xi_1,\ldots,\xi_T)\) as a whole stochastic process. In this case, the
perturbation is measured directly w.r.t. the joint laws \(Q\) and
\(\tilde Q\) of the whole exogenous process. The resulting estimates take the
form
\[
    |\vartheta(\xi)-\vartheta(\tilde{\xi})|
    \leq
    \Delta_{L_\vt}(Q,\tilde Q),
\]
where \(\Delta_{L_\vt}\) denotes the weighted transport discrepancy. This upper bound is compact and global: it compares two entire
exogenous stochastic processes without decomposing the perturbation into stagewise components.

Second, Theorem~\ref{Theo-stab-exo-stage} provides a stagewise description of the same
type of perturbation. Instead of comparing the joint laws \(Q\) and
\(\tilde Q\) directly, it compare the conditional distributions
\[
    Q_t(\cdot\mid \xi_{[t-1]})
    \quad \text{and} \quad
    \tilde Q_t(\cdot\mid \xi_{[t-1]}),
    \qquad t=1,\ldots,T .
\]
In particular, the optimal-value variation is bounded by
\[
    |\vartheta(\xi)-\vartheta(\tilde{\xi})|
    \leq
    \sum_{t=1}^T
    L_{\xi,t}\,
    \mathbb E_{\xi_{[t-1]}}
    \left[
        \dd_K\left(
        Q_t(\xi_t\mid \xi_{[t-1]}),
        \tilde Q_t(\tilde \xi_t\mid \xi_{[t-1]})
        \right)
    \right].
\]
This stagewise estimation is more informative when the exogenous process is
history-dependent. It separates the effect of each conditional
distribution perturbation and shows how a local perturbation at stage \(t\)
propagates to later stages through the Lipschitz constants of the conditional
laws. Therefore, the whole-process estimate is more concise and distributional,
whereas the stagewise estimate is more diagnostic and better reflects the
intertemporal propagation mechanism of exogenous uncertainty.
\end{remark}
    Next, we consider the quantitative stability of the optimal solution set of problem \eqref{eq:mixed-MSP-MDP} w.r.t.~the distribution perturbation of $\xi_t, 1 \leq t \leq T$. 
    Since the optimal solution $x_t^*$ at stage $t$ depends on $(s_t, x_{t-1}, \xi_{[t]})$, it is not possible to directly compare the corresponding optimal solution sets at stage $t$ 
    after perturbing the distribution of $\xi_t$. In light of this, we examine the expected change in the distance between the optimal 
    policy sets before and after the perturbation.
    
\begin{theorem}[Quantitative stability of the optimal solution set]
    Assume that (a) the conditions in Theorem \ref{Theo-stab-exo-stage} are satisfied; (b) problem \eqref{eq:v-xi} satisfies the $\nu$-th (\(\nu\geq1\)) order growth condition, i.e., there exists a constant $\beta > 0$ such that 
    \begin{eqnarray}
    \mathbb{E}_{\xi} \left[ \sum_{t = 0}^T C_t(s_t, x_t, \xi_{[t]}, \zeta_t) \right] - \mathbb{E}_{\xi} \left[ \sum_{t = 0}^T C_t(s^*_t, x^*_t, \xi_{[t]}, \zeta_t) \right] \geq \beta \left( \mathbb{E}_{\xi} \left[ d(\bm x, \mathcal{X}^*(\xi)) \right] \right)^\nu,
    \label{eq:growth-stage}
    \end{eqnarray}
    for both $\xi$ and its perturbation $\tilde{\xi}$,
    where $s_t = S^M_{t - 1}(s_{t - 1}, x_{t - 1}, \xi_{t - 1}, \zeta_{t - 1})$, $s^*_t = S^M_{t - 1}(s^*_{t - 1}, x^*_{t - 1}, \xi_{t - 1}, \zeta_{t - 1})$, and $s_0 = s^*_0$. Then
    \begin{eqnarray}
    && \mathbb{E}_{\pi^*(Q, \tilde{Q})} \left[ \mathbb{H}(\mathcal{X}^*(\xi), \mathcal{X}^*(\tilde{\xi})) \right] \nonumber \\
    &\leq & \sum_{t = 1}^T L_X {\max} \left\{ 1, L_{Q_{t + 1}}, L_{Q_{t + 1}}L_{Q_{t + 2}}, \cdots, \prod_{i = 1}^{T -t} L_{Q_{t + i}} \right\} \mathbb{E}_{\xi_{[t - 1]}} \left[ \dd_K(Q_t(\xi_t | \xi_{[t - 1]}), \tilde{Q}_t(\tilde{\xi}_t | \xi_{[t - 1]}))\right] \nonumber \\
            && + \sum_{t = 1}^T \Bigg( \left( \frac{L_{\xi,t}}{\beta} + \frac{1}{\beta}(L_X + 1)L_{\Sigma} {\max} \left\{ 1, L_{Q_{t + 1}}, L_{Q_{t + 1}}L_{Q_{t + 2}}, \cdots, \prod_{i = 1}^{T -t} L_{Q_{t + i}} \right\}\right) \times \nonumber \\
        && \qquad \mathbb{E}_{\xi_{[t - 1]}} \left[ \dd_K(Q_t(\xi_t | \xi_{[t - 1]}), \tilde{Q}_t(\tilde{\xi}_t | \xi_{[t - 1]}))\right] \Bigg)^{\frac{1}{\nu}},
    \label{eq:stab-exo-stage-optsolu}
    \end{eqnarray}
    where $L_X, L_{\Sigma}, L_{\xi, t}, L_{Q_t}$ are defined as in Theorem \ref{Theo-stab-exo-stage}, $\pi^*(Q, \tilde{Q})$ is the optimal coupling of $Q$ and $\tilde{Q}$ in \eqref{eq:def-DeltaL}.
    \label{Theo-stab-exo-optsolu-stage}
\end{theorem}

\begin{proof}
    Let $\vt(\tilde{Q}_T, Q_{[1 : T - 1]})$ be defined as in \eqref{eq:v-xi-ptbd} with $t = T - 1$.
    By \eqref{eq:optvalue-stage-T},
    \begin{eqnarray}
    |\vt(Q) - \vt(\tilde{Q}_T, Q_{[1 : T - 1]})| \leq L_{\xi, T} \mathbb{E}_{\xi_{[T - 1]}} \left[ \dd_K(Q_T(\xi_T \mid \xi_{[T - 1]}), \tilde{Q}_T(\tilde{\xi}_T \mid \xi_{[T - 1]})) \right]. \nonumber
    \end{eqnarray}
    Let \(\tilde{\bm x}^{*,T}\) be optimal for problem \eqref{eq:v-xi-ptbd} with \(t=T-1\). Under \(\pi^*(Q,\tilde Q)\), let \(\bm y:=(y_0,\ldots,y_T)\) be obtained recursively by projecting \(\tilde{\bm x}^{*,T}\) onto the original feasible sets. By the argument before \eqref{eq:xt*-yt-proof}, these projected decisions induce a measurable, feasible, and nonanticipative randomized policy for the original problem.  Since \(\nu\geq1\), we know that the growth condition \eqref{eq:growth-stage} also holds for this randomized policy.
    Assume for the sake of a contradiction that
    \begin{eqnarray}
    \mathbb E_{\pi^*(Q,\tilde Q)}\left[d(\bm y,\mathcal X^*(\xi))\right]>\left(\frac{1}{\beta}\big((L_X+1)L_{\Sigma}+L_{\xi,T}\big)\mathbb E_{\xi_{[T-1]}}\left[\dd_K\left(Q_T(\xi_T\mid\xi_{[T-1]}),\tilde Q_T(\tilde\xi_T\mid\xi_{[T-1]})\right)\right]\right)^{\frac{1}{\nu}}.
    \label{eq:stage-optsolution-fanzheng}
\end{eqnarray}
    By Proposition \ref{Prop-stab-exo-feas},
    \begin{eqnarray}
    && \left|\vt(\tilde Q_T,Q_{[1:T-1]})-\mathbb E_{\pi^*(Q,\tilde Q)}\left[\sum_{t=0}^T C_t(s_t^y,y_t,\xi_{[t]},\zeta_t)\right]\right| \nonumber\\
    &\leq& L_{\Sigma}\mathbb E_{\pi^*(Q,\tilde Q)}\left[\|\bm y-\tilde{\bm x}^{*,T}\|+\|\xi-\tilde\xi\|\right]\leq L_{\Sigma}(L_X+1)\mathbb E_{\pi^*(Q,\tilde Q)}\left[\|\xi_T-\tilde\xi_T\|\right] \nonumber\\
    &\leq& L_{\Sigma}(L_X+1)\mathbb E_{\xi_{[T-1]}}\left[\dd_K\left(Q_T(\xi_T\mid\xi_{[T-1]}),\tilde Q_T(\tilde\xi_T\mid\xi_{[T-1]})\right)\right],
    \label{vtildeandCy}
    \end{eqnarray}
    where the first inequality follows from \eqref{eq:cont-obj-Lip}-\eqref{eq:cont-exo-L_Sigma2}. The second inequality follows from Proposition \ref{Prop-stab-exo-feas} and the last inequality is obtained by the definition of Kantorovich metric.
    By \eqref{eq:optvalue-stage-T},
    \begin{eqnarray}
    && |\vt(\tilde{Q}_T, Q_{[1 : T - 1]}) - \vt(Q)| \leq L_{\xi,T} \mathbb{E}_{\xi_{[T - 1]}} \left[ \dd_K(Q_T(\xi_T \mid \xi_{[T - 1]}), \tilde{Q}_T(\tilde{\xi}_T \mid \xi_{[T - 1]}))\right].
    \label{vtildeandv}
    \end{eqnarray}
    Thus, 
    \begin{eqnarray}
    && \mathbb E_{\pi^*(Q,\tilde Q)}\left[\sum_{t=0}^T C_t(s_t^y,y_t,\xi_{[t]},\zeta_t)\right]-\vt(Q)\nonumber\\
    &\geq& \beta\left(\mathbb E_{\pi^*(Q,\tilde Q)}\left[d(\bm y,\mathcal X^*(\xi))\right]\right)^\nu \quad \text{(by \eqref{eq:growth-stage})}\nonumber\\
    &>& \big((L_X+1)L_{\Sigma}+L_{\xi,T}\big)\mathbb E_{\xi_{[T-1]}}\left[\dd_K\left(Q_T(\xi_T\mid\xi_{[T-1]}),\tilde Q_T(\tilde\xi_T\mid\xi_{[T-1]})\right)\right] \quad \text{(by \eqref{eq:stage-optsolution-fanzheng})}\nonumber\\
    &\geq& \left|\vt(\tilde Q_T,Q_{[1:T-1]})-\vt(Q)\right|+\left|\vt(\tilde Q_T,Q_{[1:T-1]})-\mathbb E_{\pi^*(Q,\tilde Q)}\left[\sum_{t=0}^T C_t(s_t^y,y_t,\xi_{[t]},\zeta_t)\right]\right| \quad \text{(by \eqref{vtildeandCy}--\eqref{vtildeandv})}\nonumber\\
    &\geq& \mathbb E_{\pi^*(Q,\tilde Q)}\left[\sum_{t=0}^T C_t(s_t^y,y_t,\xi_{[t]},\zeta_t)\right]-\vt(Q).\nonumber
    \end{eqnarray}
    is a contradiction.
    Let $\mathcal{X}^{*}(\xi)$ be the set of optimal
    policies of problem \eqref{eq:v-xi} and $\mathcal{X}^{*,T}(\tilde{\xi}_T, \xi_{[1 : T - 1]})$ be the set of optimal
    policies of \eqref{eq:v-xi-ptbd} with $t = T - 1$.
    Thus,
    we must have
    \begin{eqnarray}
    &&\mathbb E_{\pi^*(Q,\tilde Q)}\left[\mathbb D\left(\mathcal X^{*,T}(\tilde\xi_T,\xi_{[1:T-1]}),\mathcal X^*(\xi)\right)\right]\nonumber\\
    &\leq& \mathbb E_{\pi^*(Q,\tilde Q)}\left[\sup_{\tilde{\bm x}\in\mathcal X^{*,T}(\tilde\xi_T,\xi_{[1:T-1]})}\left\{d(\tilde{\bm x}, \bm y)+d(\bm y,\mathcal X^*(\xi))\right\}\right]\nonumber\\
    &\leq& L_X\mathbb E_{\xi_{[T-1]}}\left[\dd_K\left(Q_T(\xi_T\mid\xi_{[T-1]}),\tilde Q_T(\tilde\xi_T\mid\xi_{[T-1]})\right)\right]\nonumber\\
    &&+\left(\frac{1}{\beta}\big((L_X+1)L_{\Sigma}+L_{\xi,T}\big)\mathbb E_{\xi_{[T-1]}}\left[\dd_K\left(Q_T(\xi_T\mid\xi_{[T-1]}),\tilde Q_T(\tilde\xi_T\mid\xi_{[T-1]})\right)\right]\right)^{\frac{1}{\nu}}.\nonumber
\end{eqnarray}
    where $y$ is the projection of $\tilde{x}$ on $\mathcal{X}(\xi)$. Likewise, 
    we can show that
    \begin{eqnarray}
    &&\mathbb{E}_{\pi^*(Q, \tilde{Q})} \left[ \mathbb{D}(\mathcal{X}^*(\xi), \mathcal{X}^{*, T}(\tilde{\xi}_{T}, \xi_{[1:T - 1]})) \right]\nonumber\\
    &\leq & L_X \mathbb{E}_{\xi_{[T - 1]}} \left[ \dd_K(Q_T(\xi_T \mid \xi_{[T - 1]}), \tilde{Q}_T(\tilde{\xi}_T \mid \xi_{[T - 1]})) \right] \nonumber \\
    && + \left(  \frac{1}{\beta} ((L_{X} + 1) L_{\Sigma} + L_{\xi, T}) \mathbb{E}_{\xi_{[T - 1]}} \left[ \dd_K(Q_T(\xi_T \mid \xi_{[T - 1]}), \tilde{Q}_T(\tilde{\xi}_T \mid \xi_{[T - 1]})) \right] \right)^{\frac{1}{\nu}}. \nonumber
    \end{eqnarray}
    Combining the two inequalities above, we obtain
    \begin{eqnarray}
    && \mathbb{E}_{\pi^*(Q, \tilde{Q})} \left[ \mathbb{H}(\mathcal{X}^{*}(\xi), \mathcal{X}^{*,T}(\tilde{\xi}_T, \xi_{[1 : T - 1]})) \right] \nonumber\\
    &=& \mathbb{E}_{\pi^*(Q, \tilde{Q})} \left[ \max\left\{ \mathbb{D}(\mathcal{X}^{*}(\xi), \mathcal{X}^{*,T}(\tilde{\xi}_{T}, \xi_{[1:T - 1]})), \mathbb{D}(\mathcal{X}^{*,T}(\tilde{\xi}_{T}, \xi_{[1:T - 1]}), \mathcal{X}^{*}(\xi)) \right\} \right] \nonumber \\
    &\leq & L_X \mathbb{E}_{\xi_{[T - 1]}} \left[ \dd_K(Q_T(\xi_T \mid \xi_{[T - 1]}), \tilde{Q}_T(\tilde{\xi}_T \mid \xi_{[T - 1]})) \right] \nonumber \\
    && + \left(  \frac{1}{\beta} ((L_{X} + 1) L_{\Sigma} + L_{\xi, T}) \mathbb{E}_{\xi_{[T - 1]}} \left[ \dd_K(Q_T(\xi_T \mid \xi_{[T - 1]}), \tilde{Q}_T(\tilde{\xi}_T \mid \xi_{[T - 1]})) \right] \right)^{\frac{1}{\nu}}.
    \label{eq:optsolution-hausdorff-stage-T}
    \end{eqnarray}
    Likewise, we can show by induction that for $t=1,2,\cdots, T - 1$
    \begin{eqnarray}
        \left| \vt((\tilde{Q}_{[t + 1 : T]}, Q_{[1 : t]})) - \vt((\tilde{Q}_{[t : T]}, Q_{[1 : t - 1]})) \right| \leq L_{\xi,t} \mathbb{E}_{\xi_{[t - 1]}} \left[ \dd_K(Q_t(\xi_t | \xi_{[t - 1]}), \tilde{Q}_t(\tilde{\xi}_t | \xi_{[t - 1]})) \right] 
        \label{eq:stab-xi-stage-t}
    \end{eqnarray}
    and subsequently
    \begin{eqnarray}
    && \mathbb{E}_{\pi^*(Q, \tilde{Q})} \left[ \mathbb{H}(\mathcal{X}^{*, t + 1}(\tilde{\xi}_{[t + 1 : T]}, \xi_{[1 : t]}), \mathcal{X}^{*, t}(\tilde{\xi}_{[t : T]}, \xi_{[1 : t - 1]}))
    \right] 
    \nonumber \\
    &\leq & L_X {\max} \left\{ 1, L_{Q_{t + 1}}, L_{Q_{t + 1}}L_{Q_{t + 2}}, \cdots, \prod_{i = 1}^{T -t} L_{Q_{t + i}} \right\} \mathbb{E}_{\xi_{[t - 1]}} \left[ \dd_K(Q_t(\xi_t | \xi_{[t - 1]}), \tilde{Q}_t(\tilde{\xi}_t | \xi_{[t - 1]}))\right] \nonumber \\
    && + \Bigg( \left( \frac{L_{\xi,t}}{\beta} + \frac{1}{\beta}(L_X + 1)L_{\Sigma} {\max} \left\{ 1, L_{Q_{t + 1}}, L_{Q_{t + 1}}L_{Q_{t + 2}}, \cdots, \prod_{i = 1}^{T -t} L_{Q_{t + i}} \right\}\right) \times \nonumber \\
    && \qquad \mathbb{E}_{\xi_{[t - 1]}} \left[ \dd_K(Q_t(\xi_t | \xi_{[t - 1]}), \tilde{Q}_t(\tilde{\xi}_t | \xi_{[t - 1]}))\right] \Bigg)^{\frac{1}{\nu}}. \nonumber
    \end{eqnarray}
    we omit the details.
    Summarizing the discussions above, we 
    have 
    \begin{eqnarray}
        &&\mathbb{E}_{\pi^*(Q, \tilde{Q})} \left[ \mathbb{H}(\mathcal{X}^*(\xi), \mathcal{X}^*(\tilde{\xi})) \right]\nonumber\\
        &\leq& \mathbb{E}_{\pi^*(Q, \tilde{Q})} \left[ \mathbb{H}(\mathcal{X}^*(\xi), \mathcal{X}^{*,T}(\tilde{\xi}_T, \xi_{[1 : T]})) + \cdots + \mathbb{H}(\mathcal{X}^{*, 2}(\tilde{\xi}_{2 : T}, \xi_1), \mathcal{X}^*(\tilde{\xi})) \right] \nonumber \\
        &\leq & \sum_{t = 1}^T L_X {\max} \left\{ 1, L_{Q_{t + 1}}, L_{Q_{t + 1}}L_{Q_{t + 2}}, \cdots, \prod_{i = 1}^{T -t} L_{Q_{t + i}} \right\} \mathbb{E}_{\xi_{[t - 1]}} \left[ \dd_K(Q_t(\xi_t | \xi_{[t - 1]}), \tilde{Q}_t(\tilde{\xi}_t | \xi_{[t - 1]}))\right] \nonumber \\
        && + \sum_{t = 1}^T \Bigg( \left( \frac{L_{\xi,t}}{\beta} + \frac{1}{\beta}(L_X + 1)L_{\Sigma} {\max} \left\{ 1, L_{Q_{t + 1}}, L_{Q_{t + 1}}L_{Q_{t + 2}}, \cdots, \prod_{i = 1}^{T -t} L_{Q_{t + i}} \right\}\right) \times \nonumber \\
    && \qquad \mathbb{E}_{\xi_{[t - 1]}} \left[ \dd_K(Q_t(\xi_t | \xi_{[t - 1]}), \tilde{Q}_t(\tilde{\xi}_t | \xi_{[t - 1]}))\right] \Bigg)^{\frac{1}{\nu}}.
        \label{eq:optsolu-sum}
    \end{eqnarray}
    which is \eqref{eq:stab-exo-stage-optsolu}.
\end{proof}

It might be helpful to explain how the established stability results may be positioned within the existing literature
of stability analysis in stochastic programming. 
First, in the case that $T = 1, \nu = 1$, 
\bgeqn 
\mathcal{X}^*(\xi) =
\Big\{ \left(x_0^*, x_1^*\left(S_0^M(s_0, x_0^*, \zeta_0), x_0^*, \xi\right)\right): x_0 \in \mathcal{X}_0^*(\xi)\Big\},
\edeqn
which is singleton when $\mathcal{X}_0^*(\xi)$ is a singleton. By \eqref{eq:stab-exo-stage-optsolu},
    \begin{eqnarray}
    && \mathbb{E}_{\pi^*(Q, \tilde{Q})} \left[ \mathbb{H}(\mathcal{X}^*(\xi), \mathcal{X}^*(\tilde{\xi})) \right]\leq  \bigg( \frac{L_{\xi,1}}{\beta} + \left(L_X + \frac{1}{\beta}(L_X + 1)L_{\Sigma}\right) \bigg) \dd_K(Q_1, \tilde{Q}_1).
    \label{eq:stab-exo-stage-optsolu-2stage}
    \end{eqnarray}
Moreover, since  $\mathcal{X}^*_0(\xi)$ 
and $\mathcal{X}^*_0(\tilde{\xi})$
depend only on the distribution $Q$ and $\tilde{Q}$,
then 
\begin{eqnarray}
    \mathbb{E}_{\pi^*(Q, \tilde{Q})} \left[ \mathbb{H}(\mathcal{X}_0^*(\xi), \mathcal{X}_0^*(\tilde{\xi})) \right] &\leq & \mathbb{E}_{\pi^*(Q, \tilde{Q})} \left[\mathbb{H}(\mathcal{X}^*(\xi), \mathcal{X}^*(\tilde{\xi})) \right] \nonumber \\
    &\leq &  \bigg( \frac{L_{\xi,1}}{\beta} + \left(L_X + \frac{1}{\beta}(L_X + 1)L_{\Sigma}\right) \bigg) \dd_K(Q_1, \tilde{Q}_1).
    \label{eq:twostage-X0star}
\end{eqnarray}
\eqref{eq:twostage-X0star} strengthens a similar inequality 
established by R\"omisch and Schultz 
in \cite{romisch1993stability} 
 where the rhs is $L^*\dd_K(P, Q)^{\frac{1}{2}}$
 under the strong convexity 
 of the objective function and 
 polyhedral structure of the feasible set.

\begin{example}
The constants $L_t$, $1\leq t\leq T$, in \eqref{eq:Lip-Lt} may grow geometrically fast with the remaining horizon length \(T - t\). This is not simply because the derived bound is not tight, but can be intrinsic to the model when the state dynamics is not contractive $(L_S > 1)$. 
To see this, consider a one-dimensional problem.

For \(k=0,1,\ldots,T-1\),
let \(C_k(s_k,x_k)=b |x_k|\),  \(C_T(s_T,x_T)=|s_T|+b |x_T|\), where \(b>0\) is a constant; 
\(\mathcal X_k=[-1,1]\) with $\mathcal{X}_T = [-1,1]$, let \(s_{k+1} = S^M_k(s_{k}, x_{k}, \xi_{k}, \zeta_k) = a s_k\), where \(a > 1\) is a constant. 
Then the Slater condition holds uniformly with \(\bar x_k=0\). 
Observe that 
the state dynamics $S^M_k$ does not depend on $x_k$ and 
every nonzero $x_k$ incurs an additional cost, the optimal decision is \(x_k^*=0\) for each \(k\).

Consider two different state values \(s_t\) and \(\tilde s_t\) at stage \(t\). 
Then, \(s_T = a^{T-t} s_t\). 
For ease of the exposition, let \(v_t(s_t)\) 
be defined as a one-dimensional specialization of \eqref{eq:value-T}--\eqref{eq:value-t} 
for a given state \(s_t\). 
It follows that \(v_t(s_t)=a^{T-t} |s_t|\). Thus,
\(
    |v_t(s_t)-v_t(\tilde s_t)|
    =
    a^{T-t}\bigl||s_t|-|\tilde s_t|\bigr|
    \leq
    a^{T-t}|s_t-\tilde s_t|.
\)
This shows that the difference in the value function can indeed be amplified by the factor \(a^{T-t}\). The difference arising at an early stage passes through state transitions and may consequently grow geometrically with the number of remaining stages.

The same example can also be interpreted as a perturbation of endogenous uncertainty. Let \(s_t=0\), \(s_{t+1}=a s_t+\zeta_t\), and \(s_{k+1}=a s_k\), for \(k=t+1,\ldots,T-1\). Let \(P=(P_0,\ldots,P_T)\) and \(\tilde P=(\tilde P_0,\ldots,\tilde P_T)\) denote the original and perturbed distributions of the endogenous uncertainties, respectively. Consider the case that the distributions at all stages other than \(t\) are unchanged and concentrated at zero, that is, \(P_k=\tilde P_k=\delta_0\) for all \(k\neq t\), while only the distribution at stage \(t\) is perturbed. Here, \(\vartheta(P)\) and \(\vartheta(\tilde P)\) denote the optimal values defined as in \eqref{eq:v-zeta} under \(P\) and \(\tilde P\), respectively. Since \(P_t\) is the distribution of \(\zeta_t\), then \(\vartheta(P)=a^{T-t-1}\mathbb E_{P_t}[|\zeta_t|]\). Let \(P_t=\delta_0\) and \(\tilde P_t=\delta_\epsilon\), where \(\epsilon>0\). Then, \(\dd_K(P_t,\tilde P_t)=\epsilon\) and
\(
    |\vartheta(P)-\vartheta(\tilde P)|
    =
    a^{T-t-1}\epsilon.
\)
Thus, the difference between the optimal values may also grow geometrically with the remaining horizon length when the distribution of endogenous uncertainty is perturbed.
For perturbations of exogenous uncertainty, a similar argument applies. By setting \(s_{t+1}=a s_t+\xi_t\) and keeping the other settings unchanged, we obtain the same result.


A similar horizon-dependent geometric factor also appears in the coefficients of existing stability bounds for multistage stochastic programs. In particular, K\"uchler~\cite{kuchler2008stability} derives in Theorem~2 the coefficient
\(C=D+(T-1)^2BM^{T-1}\),
which is subsequently used in the stability results. Here, \(D\) is an approximation constant, \(B\) is the Lipschitz modulus of the stagewise costs, and \(M\) is the Lipschitz modulus of the stagewise feasible set mappings. The term \(M^{T-1}\) reflects the recursive propagation of the differences across stages, showing that such geometric dependence on the horizon length is not specific in this paper.
\end{example}

Let us now turn to 
the general case that $T>1$.
The most notable 
research papers on stability analysis 
are
\cite{HRS2006stability,pflug2012distance}. 
Here we highlight the main differences between 
the derived results in this paper and those established in the two papers.
First, Pflug and Pichler (\cite{pflug2012distance}) established the quantitative stability of the optimal value 
of multistage convex stochastic optimization problems
under the 
nested distance.
Their focus is on 
the effect of perturbations of the probability distributions 
of the entire data process 
without specifically considering 
the dynamic changes and interactions 
of distributions at different stages.
Heitsch and R\"{o}misch (\cite{HRS2006stability})
investigated the quantitative stability of the optimal value for multistage stochastic linear programming problems by introducing the filtration distance under the assumptions of relatively complete recourse and locally bounded objective functions.
Similar to \cite{pflug2012distance}, 
the paper 
focuses on the stability of the optimal value
against perturbation of the
whole filtration process. 
Since error bounds are established in terms of the 
 filtration distance
which depends on decision variables,
it is difficult to verify the conditions
and figure out the bounds in some cases.


In this paper, we take a different approach
by 
decomposing 
perturbations of the whole 
exogenous random process into 
stagewise 
perturbations.
In doing so, 
we 
quantify the effect 
on the optimal value function and optimal decisions by 
the perturbation at each stage. This enables us not only
to capture 
the dynamic changes and interactions 
of the distributions of random variables at 
different stages, but also derive 
error bounds recursively and establish 
the overall impact by accumulating them.

\subsection{Illustrative examples}

The next two examples explain the main advantages
of the stability results in terms of tightness of error bounds
and applicability. For simplicity, whenever an $n$-dimensional vector $a$ is involved, 
the notation $a + 1$ stands for $a + e$, 
where $e$ denotes the all-ones vector.

\begin{example}[Linear problem with feasible set independent of $\xi$]
\label{ex-feasible-unchanged}
    Consider
    \begin{subequations}
    \begin{eqnarray}
        \underset{x}{\min} && \quad \mathbb{E}_{\mu} \left[ e^{\top} ( s_0 + x_0 + \zeta_0 ) + \sum\limits_{t = 1}^T e^{\top} ( s_t + x_t + \xi_t + \zeta_t ) \right] \label{eq-ex-feasible-unchanged-a}\\
        \rm{s.t.} && \quad s_{t + 1} = M_1 s_t + M_2 x_t + N_1 \xi_t + N_2 \zeta_t, \quad t = 1, \cdots, T - 1, \label{eq-ex-feasible-unchanged-b}\\
        && \quad -d \leq x_t \leq d, \quad t = 0,1, \cdots, T, \label{eq-ex-feasible-unchanged-c} \\
        && \quad s_1 = M_1 s_0 + M_2 x_0 + N_2 \zeta_0, \quad s_0 = 0. \label{eq-ex-feasible-unchanged-d}
    \end{eqnarray}        
    \label{eq-ex-feasible-unchanged}
\end{subequations}
To ease the exposition, we restrict the discussions to the case that
$T=2, n = 2$, 
$e = (1, 1)^\top$, $d = 5$, $\xi_0 = 0$,
$M_1 = M_2 = 
    \begin{pmatrix}
    1 & 0 \\
    0 & 1 
    \end{pmatrix},
    N_1 = N_2 = \begin{pmatrix}
    1 & 0 \\
    0 & 1
    \end{pmatrix},
    $
 \begin{eqnarray*}
    \zeta_t &\sim& U(-2,0) \times U(-2,0), \quad t = 0,1,2, \\
    \xi_1 &\sim& U(-2,0) \times U(-2,0), \\
    \xi_2 &=& \xi_1 + \nu_1, \quad \nu_1 \sim U(-1,1) \times U(-1,1),    
    \end{eqnarray*}
where $U(a,b)$ denotes the uniform distribution 
of a random variable over interval     \([a,b] \subset \mathbb{R}\). Consequently, we can write the problem \eqref{eq-ex-feasible-unchanged} as
 \begin{subequations}
    \begin{eqnarray}
        \underset{x}{\min} && \quad \mathbb{E}_{\mu} \left[ (1, 1)^{\top} (x_0 + \zeta_0 ) + \sum\limits_{t = 1}^2 (1, 1)^{\top} ( s_t + x_t + \xi_t + \zeta_t ) \right] \label{eq-ex-feasible-unchanged-value-a}\\
        \rm{s.t.} && \quad s_{t + 1} =  s_t + x_t + \xi_t + \zeta_t, \quad t = 1, \cdots, T - 1, \label{eq-ex-feasible-unchanged-value-b}\\
        && \quad |x_t| - 5\leq 0, \quad t = 0,1, \cdots, T, \label{eq-ex-feasible-unchanged-value-c} \\
        && \quad s_1 = s_0 + x_0 + \zeta_0. \label{eq-ex-feasible-unchanged-value-d}
    \end{eqnarray}        
    \label{eq-ex-feasible-unchanged-value}
\end{subequations}
For problem \eqref{eq-ex-feasible-unchanged-value},
the Lipschitz continuity of the underlying functions 
w.r.t.~$\xi_{[t]}$ in Assumption \ref{assumption8} 
and w.r.t.~$\zeta_{[t]}$ in Assumption \ref{assumption6}
is satisfied with 
  $L_C = 2$, $L_S = 1$, $L_g = 0$.
  Our focus will be on stability w.r.t.~perturbation of  
  $\xi_{[t]}$ and thus we don't need Assumption \ref{assumption6}.
  The Slater condition in Assumption \ref{assumption7} is also 
  satisfied with $\bar{x}_t = (0, 0)$ with $\rho = 5$.
Note that the constraint \eqref{eq-ex-feasible-unchanged-value-c} can be equivalently 
written as $|x_t|^\alpha - 5^\alpha \leq 0$ for any positive 
constant $\alpha$. Moreover, by restricting $\alpha$ to an even positive number, $g_t(s_t, x_t, x_{t - 1}, \xi_{[t]}) = |x_t|^\alpha - 5^\alpha$ is convex in $x_t$, thus satisfying condition (b) in Proposition \ref{Prop-stab-exo-feas}.
Under the equivalent form of the constraint, 
the Slater condition is satisfied with $\bar{x}_t = (0, 0)$ with $\rho = 5^\alpha$.

Next, we consider perturbation of $Q$.
Let $\xi_1$ be perturbed from $U(-2, 0) \times U(-2, 0)$ to $U(-1.98, 0) \times U(-1.98, 0)$ and 
the distribution of $\nu_1$ be perturbed 
from $U(-1, 1) \times U(-1, 1)$ to 
$U(-0.98, 1) \times U(-0.98, 1)$.
    Then by the monotonicity of the objective function and the state transition function
    in
    \eqref{eq-ex-feasible-unchanged-value-a}, \eqref{eq-ex-feasible-unchanged-value-b} and \eqref{eq-ex-feasible-unchanged-value-d}  w.r.t. $x$, 
    we can deduce
    the optimal solutions to problem \eqref{eq-ex-feasible-unchanged} before and after the perturbations 
    are both $x_0 = x_1 = x_2 = (-5, -5)$ with 
    optimal values $\vt(Q) = -78.0$ and $\vt(\tilde{Q}) = -77.92$, respectively. On the other hand,     
    \begin{eqnarray*}
    \dd_K(Q_1,\tilde{Q}_1)
    &=& \sup_{f \in \mathcal{F}_1(\Xi)} \left( \int_{[-2, -2]}^{[0,0]} f dQ_1 - \int_{[-1.98, -1.98]}^{[0,0]} f d\tilde{Q}_1 \right) \nonumber \\
    &= & \int_{[-2, -2]}^{[0,0]} \|\xi_1\| dQ_1 - \int_{[-1.98, -1.98]}^{[0,0]} \|\tilde{\xi}_1\| d\tilde{Q}_1 =\tfrac{2}{3}\times 2 \;-\; \tfrac{2}{3} \times 1.98 = \frac{1}{75}.
    \end{eqnarray*}
    Likewise, 
    we can obtain $\dd_K(Q_2(\xi_2 | \xi_1), \tilde{Q}_2(\tilde{\xi}_2 | \xi_1)) = \frac{1}{75}$. 
    The nested distance between $Q$ and $\tilde{Q}$ (\cite{pflug2012distance} Definition 5.1 ) 
    is
    \begin{eqnarray}
    \label{eq:nested-distance}
        \dd_{\rm{Nested}}(Q, \tilde{Q}) =  
        &\underset{\pi\in\mathcal P(\Xi\times\tilde\Xi)}{\min} & \int d(\xi_1,\tilde\xi_1) + d(\xi_2, \tilde{\xi}_2) \pi(\mathrm d\xi,\mathrm d\tilde\xi) \nonumber \\
    &\textrm{s.t.} & \pi \left[ A\times\tilde\Xi \, \middle|\, \mathcal F_t \otimes \tilde{\mathcal F}_t \right](\xi,\tilde\xi)
    = P\left[ A \,\middle|\, \mathcal F_t \right](\xi), \big(A\subset\Xi_T,\ t = 1, 2\big), \nonumber\\
    && \pi \left[ \Xi\times B \,\middle|\, \mathcal F_t \otimes \tilde{\mathcal F}_t \right](\xi,\tilde\xi)
    = \tilde P\!\left[ B \,\middle|\, \tilde{\mathcal F}_t \right](\tilde\xi), \big(B\subset\tilde{\Xi}_T,\ t = 1, 2\big) \nonumber \\
    = && \hspace{-3.5em}\dd_K(Q_1, \tilde{Q}_1) + \dd_K(Q_2(\xi_2), \tilde{Q}_2(\tilde{\xi}_2)) \nonumber \\
    = && \hspace{-3.5em}\dd_K(Q_1, \tilde{Q}_1) + \left( \dd_K(Q_2(\xi_2 | \xi_1), \tilde{Q}_2(\tilde{\xi}_2 | \xi_1)) + \dd_K(Q_1, \tilde{Q}_1) \right) = 0.04. \nonumber
    \end{eqnarray}
By Theorem \ref{Theo-stab-exo-optvalue}, we can figure out
the Lipschitz modulus of 
the overall objective function 
$(1, 1)^\top(s_0 + x_0 + \zeta_0) + \sum_{t = 1}^2 (1, 1)^\top(s_t + x_t + \xi_t + \zeta_t)$
in $\xi$ is
$L_1 = 4.0$.
By Theorem \ref{Theo-Pflug} (\cite{pflug2012distance}),
\bgeqn 
     | \vt(Q) - \vt(\tilde{Q}) | \leq L_1 \dd_{\rm{Nested}}(Q, \tilde{Q}) = 0.16. \label{eq-results-pflug}
\edeqn
Next, we compare the above error bound with the one derived via 
\eqref{eq:stab-xi-stagewise-val}
   in Theorem~\ref{Theo-stab-exo-stage}.
To this end, we consider the equivalent 
formulation of the box constraint
\eqref{eq-ex-feasible-unchanged-value-c}
$|x_t|^\alpha-5^\alpha\leq 0$.
    By Theorem \ref{Theo-stab-exo-stage}, we can figure out the Lipschitz modulus 
    $L_{X, 1} = \frac{10}{5^\alpha}$, $L_{X, 2} = \frac{20}{5^\alpha} + \frac{200}{5^{2\alpha}}$, $L_{v, 2} = 4 + \frac{40}{5^\alpha}$, $L_{\xi, 2} = 2 + \frac{40}{5^\alpha} + \frac{400}{5^{2\alpha}}$, $L_{\xi, 1} = 6 + \frac{100}{5^\alpha} + \frac{400}{5^{2\alpha}}$. 
    We can then plug them into \eqref{eq:stab-xi-stagewise-val}
    to obtain
    \begin{eqnarray}
        | \vt(Q) - \vt(\tilde{Q}) | &\leq & L_{\xi, 1} \dd_K(Q_1, \tilde{Q}_1) + L_{\xi, 2} \dd_K\left(Q_2(\xi_2 | \xi_1), \tilde{Q}_2(\tilde{\xi}_2 | \xi_1)\right) \nonumber \\
        &= & \frac{8}{75} + \frac{28}{15 \times 5^\alpha} + \frac{32}{3 \times 5^{2\alpha}}.
        \label{eq-ex-bound-alpha}
    \end{eqnarray}
    Since $\alpha$ 
    can be 
    arbitrarily large,
we obtain    $L_{\xi, 1} \to 6, L_{\xi, 2} \to 2$.
The error bound in the above inequality converges to: 
    \bgeqn 
    | \vt(Q) - \vt(\tilde{Q}) | \leq L_{\xi, 1} \dd_K(Q_1, \tilde{Q}_1) + L_{\xi, 2} \dd_K\left(Q_2(\xi_2 | \xi_1), \tilde{Q}_2(\tilde{\xi}_2 | \xi_1)\right) =  \frac{8}{75},
\edeqn 
    which provides a 
    tighter bound than \eqref{eq-results-pflug}. However, if $\alpha=1,2$, 
then the rhs of \eqref{eq-ex-bound-alpha}
exceeds 0.16, which means 
the error bound \eqref{eq-results-pflug} is tighter. 
From the example, we can see that there is no definite conclusion which stability result may provide a sharper bound -- it depends on the problem structure. However,  it might be fair to say that 
it is relatively easier to calculate/estimate 
the Kantorovich metric.

    Another advantage of Theorem \ref{Theo-stab-exo-stage} is that it is applicable to general feasible sets depending on $\xi$. We will illustrate this in the next example and
    compare the quantitative stability result in Theorem \ref{Theo-stab-exo-stage} with the bound presented in Heitsch et al.~\cite{HRS2006stability}.
    To this end, we 
    add to problem \eqref{eq-ex-feasible-unchanged} an additional linear constraint
    \begin{eqnarray}
        A x_t  + B x_{t - 1} + C s_t + D \xi_t - \kappa \leq 0, \quad t = 1, 2, \cdots, T, \label{EX-linear-1c}
    \end{eqnarray}
where
    $\xi, \zeta$ are all the same as those in Example \ref{ex-feasible-unchanged}. Moreover, let 
    $A = \begin{pmatrix}
    -10 & 0 \\
    0 & -10
    \end{pmatrix},
    B = \begin{pmatrix}
    1 & 0 \\
    0 & 1
    \end{pmatrix},
    C = \begin{pmatrix}
    1 & 0 \\
    0 & 1
    \end{pmatrix},
    D = \begin{pmatrix}
    1 & 0 \\
    0 & 1
    \end{pmatrix}$, $\kappa = 11$. 
    Specifically, we consider
    \begin{subequations}
    \begin{eqnarray}
        \underset{x}{\min} && \quad \mathbb{E}_{\mu} \left[ (1, 1)^{\top} (x_0 + \zeta_0 ) + \sum\limits_{t = 1}^2 (1, 1)^{\top} ( s_t + x_t + \xi_t + \zeta_t ) \right] \label{eq-ex-linear-a}\\
        \rm{s.t.} && \quad s_{t + 1} =  s_t + x_t + \xi_t + \zeta_t, \quad t = 1, \cdots, T - 1, \label{eq-ex-linear-b}\\
        && \quad -10 x_t + x_{t - 1} + s_t + \xi_t - 11 \leq 0, \label{eq-ex-linear-c} \\ 
        && \quad -5 \leq x_t \leq 5, \quad t = 0,1, \cdots, T, \label{eq-ex-linear-d} \\
        && \quad s_1 = s_0 + x_0 + \zeta_0. \label{eq-ex-linear-e}
    \end{eqnarray}        
    \label{eq-ex-linear}
\end{subequations}
    
    For problem \eqref{eq-ex-linear}, the induced problem at stage $t = 2$ can be represented as
\begin{subequations}
\begin{eqnarray}
    &\underset{x_2 \in \mathcal{X}_2(s_2, x_1, \xi_{[2]})}{\min} & \mathbb{E}_{\zeta_2} \left[ e^{\top} ( s_2 + x_2 + \xi_2 + \zeta_2 ) \right] \nonumber \\
    &\rm{s.t.} & -10 x_{2, 1} + x_{1, 1} + s_{2, 1} + \xi_{2, 1} - 11 \leq 0, \nonumber \\
    && -10 x_{2, 2} + x_{1, 2} + s_{2, 2} + \xi_{2, 2} - 11 \leq 0. \nonumber \\
    && |x_{2, 1}| \leq 5, |x_{2, 2}| \leq 5. \nonumber
\end{eqnarray}
\end{subequations}
We can figure out the optimal solution 
of the problem with $x_{2} = \big(\frac{1}{10}(x_{1, 1} + s_{2, 1} + \xi_{2, 1} - 11)$, $\frac{1}{10}(x_{1, 2} + s_{2, 2} + \xi_{2, 2} - 11)\big)^{\top}$, and the corresponding
value function 
\begin{eqnarray}
    v_2(s_2, x_1, \xi_{[2]}) = e^{\top} ( 1.1 s_2 + 1.1 \xi_2 + 0.1 x_1) - 4.2. \nonumber
\end{eqnarray}    
Then the dynamic problem at $t = 1$
is
\begin{subequations}
\begin{eqnarray}
    &\underset{x_1 \in \mathcal{X}_1(s_1, x_0, \xi_{1})}{\min} & \mathbb{E}_{\zeta_1} \left[ e^{\top} ( s_1 + x_1 + \xi_1 + \zeta_1 ) + \mathbb{E}_{\xi_2 | \xi_1} \left[ v_2(s_2, x_1, \xi_{[2]}) \right] \right] \nonumber \\
    &\rm{s.t.} & x_{1, 1} - x_{0, 1} + s_{1, 1} + \xi_{1, 1} - 11 \leq 0, \nonumber \\
    && x_{1, 2} - x_{0, 2} + s_{1, 2} + \xi_{1, 2} - 11 \leq 0. \nonumber \\
    && |x_{1, 1}| \leq 5, |x_{1, 2}| \leq 5, \nonumber \\
    && s_2 = s_1 + x_1 + \xi_1 + \zeta_1. \nonumber
\end{eqnarray}
\end{subequations}
The optimal solution is $x_1 = (\frac{1}{10}(x_{0, 1} + s_{1, 1} + \xi_{1, 1} - 11), \frac{1}{10}(x_{0, 2} + s_{1, 2} + \xi_{1, 2} - 11))^\top$, and the value function is 
\begin{eqnarray}
    v_1(s_1, x_0, \xi_1) &= & e^{\top} ( 2.32 s_1 + 3.42 \xi_1 + 0.22 x_0) - 4.2 - 4.2 - 2.64 - 2.2 \nonumber \\
    &= & e^{\top} ( 2.32 s_1 + 3.42 \xi_1 + 0.22 x_0) - 13.24. \nonumber
\end{eqnarray}
The optimal value of problem \eqref{eq-ex-linear} 
can thus be determined by the following problem:
\begin{eqnarray}
    && \underset{x_0}{\min} \; \mathbb{E}_{\zeta_0} \left[ e^\top(s_0 + x_0 + \zeta_0) + \mathbb{E}_{\xi_1} \left[ v_1(s_1, x_0, \xi_1)\right] \right] \nonumber \\
    &= & \mathbb{E}_{\zeta_0} \left[ e^\top(s_0 + x_0^* + \zeta_0) + \mathbb{E}_{\xi_1} \left[ e^{\top} ( 2.32 s_1 + 3.42 \xi_1 + 0.22 x_0^*) - 13.24\right] \right] \nonumber \\
    &= & \mathbb{E}_{\zeta_0} \left[ e^\top(s_0 + x_0^* + \zeta_0) +  e^{\top} ( 2.32 s_1 + 0.22 x_0^*) - 20.08 \right] \nonumber \\
    &= & \mathbb{E}_{\zeta_0} \left[ e^\top(3.32 s_0 + 3.54 x_0^* + 3.32 \zeta_0) - 20.08 \right] \nonumber \\
    &= & \mathbb{E}_{\zeta_0} \left[ -35.4 - 6.64 - 20.08 \right] = - 62.12. \nonumber
\end{eqnarray}
The optimal solution of the problem above is $x_0^* = -5$, $x_1^*(s_1, x_0, \xi_1) = \frac{1}{10}(s_1 + x_0 + \xi_1 - 11)$, $x_2^*(s_2, x_1, \xi_2) = \frac{1}{10}(s_2 + x_1 + \xi_2 - 11)$, and the optimal value of problem \eqref{eq-ex-linear} is $-62.12$.

We consider  the same  perturbation of probability distributions as in 
\eqref{eq-ex-feasible-unchanged-value}, 
the optimal solution $\tilde{x}^*$ remains the same, that is, $\tilde{x}^* = x^*$,  and the optimal value after the perturbation becomes
$-62.0296$.
For problem \eqref{eq-ex-linear}, we can 
figure out that
$L_S = 1$, $\rho \geq 40$, $A = 10$, $L_g = 1$, $L_C = 2$, $L_{X, 1} = \frac{1}{4}$, $L_{X, 2} = \frac{5}{8}$, $L_{\xi, 1} = \frac{35}{4}$, $L_{\xi, 2} = \frac{13}{4}$. 
By Theorem \ref{Theo-stab-exo-stage}, 
\bgeqn 
| \vt(Q) - \vt(\tilde{Q}) | \leq L_{\xi, 1} \dd_K(Q_1, \tilde{Q}_1) + L_{\xi, 2} \dd_K\left(Q_2(\xi_2 | \xi_1), \tilde{Q}_2(\tilde{\xi}_2 | \xi_1)\right) = 0.16. \label{eq-thispaper-linear}
\edeqn 
Next, we 
calculate the error bound based on the stability result in
\cite{HRS2006stability}. 
To this end, we need to calculate the filtration distance (see Appendix \ref{appendixproof})
$$
\dd_{\rm{Filt}}(Q, \tilde{Q}) := \underset{\epsilon \in (0, \alpha]}{\sup} \underset{x \in l_{\epsilon}(\xi), \tilde{x} \in l_{\epsilon}(\tilde{\xi})}{\inf} \sum_{t=1}^{2} 
\max\Big\{ 
\big\| x_t - \mathbb{E}[x_t \mid \tilde{\mathcal{F}}_t] \big\|,\;
\big\| \tilde{x}_t - \mathbb{E}[\tilde{x}_t \mid \mathcal{F}_t] \big\| 
\Big\}.
$$ 
At stage $0$, $\mathbb{E}\left[ x^*_0 - \mathbb{E}\left[ x^*_0 \right] \right] = 0$. At $t = 1$, we have
\begin{eqnarray}
    && \mathbb{E}_{\xi_1, \zeta_0}\left[ \left\| x_1^* - \mathbb{E}_{\tilde{\xi}_1, \zeta_0} \left[ x_1^* \right] \right\| \right] \nonumber \\
    &= & 
    \mathbb{E}_{\xi_1, \zeta_0}\left[ \left\| x_1^* - \frac{1}{10} \mathbb{E}_{\tilde{\xi}_1, \zeta_0} \left[ s_1 + x_0 + \tilde{\xi}_1 - 11\right] \right\| \right] \nonumber \\
    &= & \mathbb{E}_{\xi_1, \zeta_0}\left[ \left\| x_1^* - \frac{1}{10} \mathbb{E}_{\tilde{\xi}_1, \zeta_0} \left[ s_0 + 2x_0 + \zeta_0 + \tilde{\xi}_1 - 11\right] \right\| \right] \nonumber \\
    &= & \mathbb{E}_{\xi_1, \zeta_0}\left[ \left\| x_1^* + \frac{22.99}{10} \right\| \right] \nonumber \\
    &= & \frac{1}{10} \mathbb{E}_{\xi_1, \zeta_0}\left[ \left\| s_1 + x_0 + \xi_1 - 11 + 22.99  \right\| \right] \nonumber \\
    &= & \frac{1}{10} \mathbb{E}_{\xi_1, \zeta_0}\left[ \left\| 1.99 + \zeta_0 + \xi_1 \right\| \right] \nonumber \\
    &\approx & \frac{1}{15}.  \nonumber
\end{eqnarray}
At $t = 2$, we obtain 
\begin{eqnarray}
    && \mathbb{E}_{\xi_{[2]}, \zeta_1, \zeta_0}\left[ \left\| x_2^* - \mathbb{E}_{\tilde{\xi}_{[2]}, \zeta_1, \zeta_0} \left[ x_2^* \right] \right\| \right] \nonumber \\
    &= & \mathbb{E}_{\xi_{[2]}, \zeta_1, \zeta_0}\left[ \left\| x_2^* - \mathbb{E}_{\tilde{\xi}_{[2]}, \zeta_1, \zeta_0} \left[ \frac{1}{10} (s_2 + x_1^* + \tilde{\xi}_2 - 11) \right]\right\| \right] \nonumber \\
    &= & \mathbb{E}_{\xi_{[2]}, \zeta_1, \zeta_0}\left[ \left\| x_2^* - \mathbb{E}_{\tilde{\xi}_{[2]}, \zeta_1, \zeta_0} \left[ \frac{1}{10} (s_1 + 2 x_1^* + \tilde{\xi}_1 + \zeta_1 + \tilde{\xi}_2 - 11) \right] \right\| \right] \nonumber \\
    &= & \mathbb{E}_{\xi_{[2]}, \zeta_1, \zeta_0}\left[ \left\| x_2^* - \mathbb{E}_{\tilde{\xi}_{[2]}, \zeta_1, \zeta_0} \left[ \frac{1}{10} (s_0 + x^*_0 + \zeta_0 + 2 x_1^* + \tilde{\xi}_1 + \zeta_1 + \tilde{\xi}_2 - 11) \right] \right\| \right] \nonumber \\
    &= & \mathbb{E}_{\xi_{[2]}, \zeta_1, \zeta_0}\left[ \left\| x_2^* + 2.4568 \right\| \right] \nonumber \\
    &= & \mathbb{E}_{\xi_{[2]}, \zeta_1, \zeta_0}\left[ \left\| \frac{1}{10} \big(s_0 + x^*_0 + \zeta_0 + \frac{1}{5}(s_0 + 2x_0^* + \xi_1 + \zeta_0 - 11) + \xi_1 + \zeta_1 + \xi_2 - 11\big) + 2.4568 \right\| \right] \nonumber \\
    &= & \mathbb{E}_{\xi_{[2]}, \zeta_1, \zeta_0} \left[ \left\| \frac{3}{25} \zeta_0 + \frac{3}{25} \xi_1 + \frac{\zeta_1}{10} + \frac{\xi_2}{10} + 0.4368 \right\| \right] \nonumber \\
    &\approx & 0.136. \nonumber
\end{eqnarray}
By the definition of the filtration distance, we obtain that 
\bgeqn 
\dd_{\rm{Filt}}(Q, \tilde{Q}) \approx 0.203. 
\label{eq-filtration-HRS}
\edeqn
Moreover,
\begin{eqnarray}
    \| \xi - \tilde{\xi} \|_{W_3} := W_3(Q, \tilde{Q}) = 0.01(2 + 2 + 12 + 12)^{\frac{1}{3}} \approx 0.030, 
    \label{eq-Wasserstein3-HRS}
\end{eqnarray}
where $W_3$ stands for the 3-Wasserstein metric. Consequently, we can approximately calculate the bound provided by \cite{HRS2006stability} 
\begin{eqnarray}
    |\vt(Q) - \vt(\tilde{Q})| \leq L (\| \xi - \tilde{\xi} \|_{W_3} + \dd_{\rm{Filt}}(Q, \tilde{Q})) =18 \times (0.203 + 0.030) \approx 4.19.
\end{eqnarray}
    We can see that the bound in \eqref{eq-thispaper-linear} is tighter.
    
    Note that even 
    the filtration distance before and after perturbation is $0$, 
    we can 
    still show that the error bound provided by
    Theorem ~\ref{Theo-stab-exo-stage} could be tighter.   
    To see this, we set $A = 
    \begin{pmatrix}
    1 & 0 \\
    0 & 1
    \end{pmatrix},
    B = \begin{pmatrix}
    -1 & 0 \\
    0 & -1 
    \end{pmatrix}
    $. 
    The resulting optimal solution at stage $0$ is $x_0 = (-5, -5)^\top$, and the optimal value is $\vt(Q) = - 78$. 
    If we perturb the distribution of each component of $\xi_1$ from $U(-2, 0)$ to $U(-1.98, 0)$ and that of $\xi_2 | \xi_1$ from $\xi_1 + U(-1, 1)$ to $\xi_1 + U(-0.98, 1)$, we obtain that the optimal value after the perturbation is $\vt(\tilde{Q}) = -77.92$.
    Under this setting, 
    $L_C = 2$, $L_S = 1$, $A = 10$, $L_g = 1$, $\rho = 10$, $L_{X, 1} = 1$, $L_{X, 2} = 4$, $L_{v, 2} = 4$, $L_{Q, 1} = 1$, $L_{\xi, 1} = 12$, $L_{\xi_2} = 10$. 
    The bound provided by Theorem \ref{Theo-stab-exo-stage} is $\frac{22}{75}$.
    On the other hand, the bound derived from \cite{HRS2006stability} is approximately $ 18 * 0.030 = 0.540$.

    Moreover, if $\kappa$ increases, the bound provided by Theorem~\ref{Theo-stab-exo-stage}
    can  be further tightened. Especially, 
    when $\kappa$ is sufficiently large, 
    $\rho$ becomes large, 
    the bound provided in this paper 
    approaches the real gap between the optimal values before and after the perturbation. 

    From this example, we envisage that the error bound 
    derived in Theorem~\ref{Theo-stab-exo-stage} is in general tighter than the one based on the filtration distance albeit theoretical evidence is yet to be established. 
    This is because 
    the sum of the filtration distance and the Wasserstein metric is usually large. 
    \end{example}

    Finally, we note that 
    the stability results in 
    \cite{HRS2006stability} and \cite{pflug2012distance} 
    are established under some specified problem structure. 
    For example, Heitsch et al.~\cite{HRS2006stability} require the problem to be linear, whereas 
    Pflug and Pichler \cite{pflug2012distance} require the feasible set to be deterministic. 
    In contrast,
    the stability results in this paper are not subject to  these restrictions.
    The next example illustrates this.

\begin{example}[Nonlinear problem]
    Consider problem \eqref{eq-ex-feasible-unchanged} with additional nonlinear constraints
    \begin{eqnarray}
        \| A x_t \circ x_t  + B x_{t - 1} + C s_t + D \xi_t \| - \kappa \leq 0, \quad t = 1, 2, \cdots, T,
    \end{eqnarray}        
    where $a \circ b$ denotes the Hadamard product of $a$ and $b$.
    Let $A = 
    \begin{pmatrix}
    1 & 0 \\
    0 & 1
    \end{pmatrix},
    B = \begin{pmatrix}
    -1 & 0 \\
    0 & -1 
    \end{pmatrix}
    $. Then the problem becomes
    \begin{subequations}
    \begin{eqnarray}
        \underset{x \in \mathcal{X}}{\min} && \quad \mathbb{E}_{\mu} \left[ (1, 1)^{\top} (x_0 + \zeta_0 ) + \sum\limits_{t = 1}^T (1, 1)^{\top} ( s_t + x_t + \xi_t + \zeta_t ) \right] \label{eq:obj-norm}\\
        \rm{s.t.} && \quad s_{t + 1} = s_t + x_t + \xi_t + \zeta_t, \quad t = 1, \cdots, T - 1, \label{eq:trans-bilinear} \\
        && \quad x_{t, i}^2 - x_{t - 1, i} + s_{t, i} + \xi_{t, i} - 11 \leq 0, \quad t = 1, 2, \cdots, T; \quad i = 1, 2, \label{eq:constraint-norm}\\
        && \quad -5 \leq x_t \leq 5,  \quad t = 0,1, \cdots, T,\\
        && \quad s_1 = s_0 + x_0 + \zeta_0.
    \end{eqnarray}        
    \label{ex-stab-nonlinear}
    \end{subequations}
    The dynamic problem at $t = 2$ can be written as 
\begin{subequations}
\begin{eqnarray}
     &\underset{x_2
     }{\min}& \mathbb{E}_{\zeta_2} \left[ s_{2, 1} + x_{2, 1} + \xi_{2, 1} + \zeta_{2, 1} + s_{2, 2} + x_{2, 2} + \xi_{2, 2} + \zeta_{2, 2} \right] \nonumber \\
    &\rm{s.t.}& | x_{2, 1} | \leq 5, |x_{2, 2}| \leq 5, \nonumber \\
    && x_{2, 1}^2 - x_{1, 1} + s_{2, 1} + \xi_{2, 1} \leq 11, \nonumber \\
    && x_{2, 2}^2 - x_{1, 2} + s_{2, 2} + \xi_{2, 2} \leq 11. \nonumber
\end{eqnarray}
\end{subequations}
The optimal solution is 
\begin{eqnarray*}
    x_{2, 1}^*(s_2, x_1, \xi_{[2]}) &= & -\sqrt{11 + x_{1, 1} - s_{2, 1} - \xi_{2, 1}}, \\
    x_{2, 2}^*(s_2, x_1, \xi_{[2]}) &= & -\sqrt{11 + x_{1, 2} - s_{2, 2} - \xi_{2, 2}}
\end{eqnarray*}
and
the value function at stage $2$ is
\begin{eqnarray}
    &&v_2(s_2, x_1, \xi_{[2]}) \nonumber \\
    &= &  s_{2, 1} - \sqrt{11 + x_{1, 1} - s_{2, 1} - \xi_{2, 1}} + \xi_{2, 1} - 1 + s_{2, 2} - \sqrt{11 + x_{1, 2} - s_{2, 2} - \xi_{2, 2}} + \xi_{2, 2} - 1  . \nonumber
\end{eqnarray}
We consider the dynamic problem at $t = 1$. Due to the decomposibility, we simply 
consider the dynamic problem at $t = 1$ corresponding to the first component, which is
\begin{eqnarray}
    &\underset{x_1
    }{\min}& \mathbb{E}_{\zeta_1, \xi_2 | \xi_1} \left[ s_{1, 1} + x_{1, 1} + \xi_{1, 1} + \zeta_{1, 1} + s_{2, 1} - \sqrt{11 + x_{1, 1} - s_{2, 1} - \xi_{2, 1}} + \xi_{2, 1} - 1 \right] \nonumber \\
    &\rm{s.t.}& |x_{1, 1}| \leq 5, x_{1, 1}^2 - x_{0, 1} + s_{1, 1} + \xi_{1, 1} \leq 11, \nonumber \\
    && s_{2, 1} = s_{1, 1} + x_{1, 1} + \xi_{1, 1} + \zeta_{1, 1}. \nonumber
\end{eqnarray}
The objective function can be reformulated as
$$
2 s_{1, 1} + 2 x_{1, 1} + 2 \xi_{1, 1} + 2 \zeta_{1, 1} - \sqrt{11 - s_{1, 1} - \xi_{1, 1} - \zeta_{1, 1} - \xi_{2, 1}} + \xi_{2, 1} - 1.
$$
We can figure out the optimal solution with $x_{1, 1}^* = - \sqrt{11 + x_{0, 1} - s_{1, 1} - \xi_{1, 1}}$. Similarly, we obtain  $x_{1, 2}^* = - \sqrt{11 + x_{0, 2} - s_{1, 2} - \xi_{1, 2}}$.
Consequently, 
\begin{eqnarray}
    && v_{1, 1}(s_1, x_0, \xi_1) \nonumber \\
    &= & 
    2 s_{1, 1} - 2 \sqrt{11 + x_{0, 1} - s_{1, 1} - \xi_{1, 1}}  + 3 \xi_{1, 1} - 3 - \mathbb{E}_{\zeta_{1, 1}, \nu_{1, 1}} \left[ \sqrt{11 - s_{1, 1} - 2 \xi_{1, 1} - \zeta_{1, 1} - \nu_{1, 1}} \right] 
    \nonumber \\
    &= & 2 s_{1, 1} - 2 \sqrt{11 + x_{0, 1} - s_{1, 1} - \xi_{1, 1}}  + 3 \xi_{1, 1} - 3 \nonumber \\
    && - \frac{1}{15} \left( (14 - s_{1,1} - 2\xi_{1,1})^{5/2} - 2(12 - s_{1,1} - 2\xi_{1,1})^{5/2} + (10 - s_{1,1} - 2\xi_{1,1})^{5/2} \right). \label{eq:v_11}
\end{eqnarray}
Likewise, we have
\begin{eqnarray}
    && v_{1, 2}(s_1, x_0, \xi_1) = 2 s_{1, 2} - 2 \sqrt{11 + x_{0, 2} - s_{1, 2} - \xi_{1, 2}}  + 3 \xi_{1, 2} - 3 \nonumber \\
    && - \frac{1}{15} \left( (14 - s_{1,2} - 2\xi_{1,2})^{5/2} - 2(12 - s_{1,2} - 2\xi_{1,2})^{5/2} + (10 - s_{1,2} - 2\xi_{1,2})^{5/2} \right). \nonumber
\end{eqnarray}
Thus the value function at stage $1$ can be represented as 
\begin{eqnarray}
    v_1(s_1, x_0, \xi_1) =  v_{1, 1}(s_1, x_0, \xi_1) + v_{1, 2}(s_1, x_0, \xi_1) . \nonumber
\end{eqnarray}
w.r.t.~the first component, the dynamic problem at stage $0$ can be written as
\begin{eqnarray}
    &\underset{x_{0, 1}}{\min}& \mathbb{E}_{\zeta_0, \xi_1} \left[ s_{0, 1} + x_{0, 1} + \zeta_{0, 1} + v_{1, 1}(s_1, x_0, \xi_1) \right] \nonumber \\
    &\rm{s.t. }& |x_{0, 1}| \leq 5, \nonumber \\
    && s_1 = s_0 + x_0 + \zeta_0. \nonumber
\end{eqnarray}
The optimal solution to this problem is $x_0^* = -5$. The dynamic problem w.r.t.~the second variable can be similarly solved and we then obtain the optimal value of problem \eqref{ex-stab-nonlinear}  
\begin{eqnarray*}
&&2 \left( -24 - \frac{2}{15}(15^{\frac{5}{2}} - 2 \cdot 13^{\frac{5}{2}} + 11^{\frac{5}{2}}) - \frac{1}{1890}\left(2 \cdot 19^{\frac{9}{2}} + 2 \cdot 21^{\frac{9}{2}} - 3 \cdot 17^{\frac{9}{2}} - 3 \cdot 23^{\frac{9}{2}} + 15^{\frac{9}{2}} + 25^{\frac{9}{2}} \right) \right) \\
&\approx & -71.36    
\end{eqnarray*}
If we perturb the distribution of each component of $\xi_1$ from $U(-2, 0)$ to $U(-1.98, 0)$ and that of $\xi_2 | \xi_1$ from $\xi_1 + U(-1, 1)$ to $\xi_1 + U(-0.99, 0.99)$,  we can find that the optimal value after the perturbation is approximately equal to -71.28. 
On the other hand, 
$\dd_K(Q_2(\xi_2 | \xi_1), \tilde{Q}_2(\tilde{\xi}_2 | \xi_1)) = \frac{1}{150}$.
Since $L_{C} = 2$, $L_S = 1$, $A = 10$, $L_g = 1$, $\rho = 5$, $L_{X, 1} = 2$, $L_{X, 2} = 12$, $L_{v, 1} = 84$, $L_{v, 2} = 12$, $L_{Q, 1} = 1$, $L_{\xi, 1} = 42$, $L_{\xi, 2} = 26$,
then we can apply Theorem \ref{Theo-stab-exo-stage} 
to establish 
\begin{eqnarray}
    | \vt(Q) - \vt(\tilde{Q}) | \leq L_{\xi, 1} \dd_K(Q_1, \tilde{Q}_1) + L_{\xi, 2} \dd_K\left(Q_2(\xi_2 | \xi_1), \tilde{Q}_2(\tilde{\xi}_2 | \xi_1)\right) = \frac{11}{15}. 
    \label{eq-nonlinear-theostage}
\end{eqnarray}

\end{example}
\section{Concluding Remarks}

In this paper, 
we consider 
an 
integrated MSP-MDP framework which combines  
multistage stochastic programming
and 
MDP.
The integrated model covers emerging multistage decision-making problems involving both exogenous and endogenous uncertainties.
Under some moderate conditions, we derive 
dynamic recursive formulations of the problem,
investigate the stability of the optimal 
values and optimal solutions 
when the underlying random processes are perturbed locally and globally. 
The new stability results provide theoretical grounding of the integrated MSP-MDP model and complement the existing stability results 
in the literature of MSP/MDP.

The research may be extended in a few directions. First, it may deserve 
further explorations on the advantages and 
disadvantages of the established
stability results in comparison with 
the existing ones in terms of tightness and applicability. Second, it will be interesting to extend the stability results, at least some of them, to infinite time horizon cases because a large class of MDP problems in the literature are infinite-horizon.
Third, extending the stability analysis to multistage risk minimization problems might be another promising direction. 
Stability analysis of risk minimization problems are mostly focused on one-stage and two-stage decision making problems, see e.g~\cite{claus2017weak, wang2020robust, pichler2022quantitative,guo2021statistical}, it might be interesting to extend the research to multistage setting.
Fourth,  
we may use the established stability results 
to develop appropriate computational methods for solving the integrated MSP-MDP problem. 
We leave all these for future research.

\bibliography{sample}
\bibliographystyle{plain}
\appendix

\section{Supplementary results and proofs}
\label{appendixproof}

\begin{proposition}[Random lower semicontinuity of $v_t(s_{t - 1}, x_{t - 1}, \xi_{[t]}, \zeta_{t - 1})$]
Assume that Assumptions \ref{Assu:welldefinedness} 
 and 
\ref{Assu:compactness} hold, and for any $t = 1, 2, \cdots, T$, $i = 1, 2, \cdots, I_t$, $g_{t, i}$ is lower semicontinuous w.r.t. $(s_t, x_t)$, $S^M_t(s_{t}, x_t, \xi_t, \omega_t)$ is continuous w.r.t.~$(s_t, x_t, \xi_t, \omega_t)$, \( C_t\left(s_t, x_t, \xi_{[t]}, \zeta_t\right) \) is lower semicontinuous w.r.t. $(s_t, x_t)$.
Then
\( v_t(s_{t - 1}, x_{t - 1}, \xi_{[t]}, \zeta_{t - 1}) \) is random lower semicontinuous,  
i.e., \( v_t(s_{t - 1}, x_{t - 1}, \xi_{[t]}, \zeta_{t - 1}) \) is lower semicontinuous in $(s_{t - 1}, x_{t - 1})$, and measurable in $(s_{t - 1}, x_{t - 1}, \xi_{[t]}, \zeta_{t - 1})$. 
\label{Prop-randomlower}
\end{proposition}

\begin{proof}
We proceed the proof in four steps.
    
\underline{Step 1}.   
We show that \( v_T(s_{T - 1}, x_{T - 1}, \xi_{[T]}, \zeta_{T - 1}) \) is measurable w.r.t. \( (s_{T - 1}, x_{T - 1}, \xi_{[T]}, \zeta_{T - 1}) \). 
Since \( g_{T,i}(s_T, x_T, x_{T - 1}, \xi_{[T]}), i \in I_t, \) is measurable w.r.t.~\( (s_T, x_T, x_{T - 1}, \xi_{[T]}) \) under the lower semicontinuous condition, it follows that
    \begin{eqnarray}
    Z_T := \{ (s_T, x_T, x_{T - 1}, \xi_{[T]}) \mid g_{T,i}(s_T, x_T, x_{T - 1}, \xi_{[T]}) \leq 0, i \in I_T\} \nonumber
    \end{eqnarray}
    is a 
    measurable set. 
    We aim to prove that the set of feasible solutions at stage \( T \),
    i.e., the set-valued mapping 
    \( \mathcal{X}_T : \mathbb{R}^{\hat{n}_T} \times \mathbb{R}^{n_{T - 1}} \times \mathbb{R}^{m_{1, [T]}} \rightrightarrows \mathbb{R}^{n_T} \)
    is weakly measurable (\cite{aliprantis2006infinite}), that is, for any open set \( U \subset \mathbb{R}^{n_T} \), 
    {\begin{eqnarray}
    \mathcal{X}_T^{-1}(U) = \left\{ (s_T, x_{T - 1}, \xi_{[T]}) \mid \mathcal{X}_T(s_T, x_{T - 1}, \xi_{[T]}) \cap U \neq \emptyset \right\}
    \end{eqnarray}}
    is measurable. 
    Observe that
    \begin{eqnarray}
    \mathcal{X}_T^{-1}(U) &= & \left\{ (s_T, x_{T - 1}, \xi_{[T]}) \mid \exists x_T \in U, \text{ such that } g_{T,i}(s_T, x_{T}, x_{T - 1}, \xi_{[T]}) \leq 0, i \in I_T \right\} \nonumber \\
    &= & 
    \Pi\left((U \times (\mathbb{R}^{\hat{n}_T} \times \mathbb{R}^{n_{T - 1}} \times \mathbb{R}^{m_{1, [T]}})) \cap Z_T\right), \nonumber
\end{eqnarray}
where \( \Pi \)
denotes
the projection onto \(\mathbb{R}^{\hat{n}_T} \times \mathbb{R}^{n_{T - 1}} \times \mathbb{R}^{m_{1, [T]}}\). Since \( U \times (\mathbb{R}^{\hat{n}_T} \times \mathbb{R}^{n_{T - 1}} \times \mathbb{R}^{m_{1, [T]}}) \) is an open set, it is measurable; meanwhile, we know that \( Z_T \) is a measurable set, so \( (U \times (\mathbb{R}^{\hat{n}_T} \times \mathbb{R}^{n_{T - 1}} \times \mathbb{R}^{m_{1, [T]}})) \cap Z_T \) is also a measurable set. By the measurable projection theorem (\cite{dellacherie1979probabilities}), it follows that \( \mathcal{X}_T^{-1}(U) \) is measurable. Thus, \( \mathcal{X}_T \) is weakly measurable.

\underline{Step 2}. We show that \( v_T(s_{T - 1}, x_{T - 1}, \xi_{[T]}, \zeta_{T - 1}) \) is lower semicontinuous w.r.t. $(s_{T - 1}, x_{T - 1})$.
By the
continuity of \( S^M_{T-1} \) 
and the lower semicontinuous condition, the feasible set of problem \eqref{eq:value-T} at stage $T$ is measurable w.r.t.~\( (s_{T}, x_{T - 1}, \xi_{[T]}, \zeta_{T - 1}) \). 
Combining this with Theorem~18.19 in \cite{aliprantis2006infinite}, we can deduce that
\begin{eqnarray}
    && v_T(s_{T - 1}, x_{T - 1}, \xi_{[T]}, \zeta_{T - 1}) = \underset{x_T \in \mathcal{X}_T(s_T, x_{T - 1}, \xi_{[T]})}{\min} \mathbb{E}_{\zeta_T} \left[ C_T(s_T, x_T, \xi_{[T]}, \zeta_T) \right] \nonumber \\
    &= & \underset{x_T \in \mathcal{X}_T(s_T, x_{T - 1}, \xi_{[T]})}{\min} \mathbb{E}_{\zeta_T} \left[ C_T(S^M_{T - 1}(s_{T - 1}, x_{T - 1}, \xi_{T - 1}, \zeta_{T - 1}), x_T, \xi_{[T]}, \zeta_T) \right] \nonumber
\end{eqnarray}
is measurable 
in \( (s_{T - 1}, x_{T - 1}, \xi_{[T]}, \zeta_{T - 1}) \).
The attainability of the optimum 
can be established as follows. 
Let  $C_T^m$ denote 
the infimum of $C_T(s_T, x_T, \xi_{[T]}, \zeta_T)$ 
over $\mathcal{X}_T(s_T, x_{T - 1}, \xi_{[T]})$.
Assumption \ref{Assu:continuity} 
ensures that there
exists a sequence of feasible solustions 
$\{x_{T, i}\} \subset \mathcal{X}_T(s_T, x_{T - 1}, \xi_{[T]})$
such that 
$C_T(s_T, x_{T, i}, \xi_{[T]}, \zeta_T) \to C_T^m$. 
By the compactness of $\mathcal{X}_T(s_T, x_{T - 1}, \xi_{[T]})$ under 
Assumption~\ref{Assu:compactness}, 
$\{x_{T, i}\}$ has a convergent subsequence $x_{T, ij} \to x_{T, 0}$ such that $C_T(s_T, x_{T, ij}, \xi_{[T]}, \zeta_T) \to C_T^m$. 
By the lower semicontinuity of $C_T$, 
$C_T(s_T, x_{T, 0}, \xi_{[T]}, \zeta_T) = C_T^m$, which means 
that $x_{T,0} \in \mathcal{X}_T(s_T, x_{T - 1}, \xi_{[T]})$ is an optimal solution.

Next, we show the lower semicontinuity of $v_T(\cdot)$.
Under Assumption \ref{Assu:compactness}, we can assume that for any feasible tuple 
\( (s_{T-1}, x_{T-1}) \), there exists a sequence of feasible tuples 
\( (s_{T-1,n}, x_{T-1,n}) \) such that  
    $$
    \lim _{n \rightarrow \infty}\left(s_{T-1, n}, x_{T-1, n}\right)=\left(s_{T-1}, x_{T-1}\right).
    $$
    To prove the 
    lower semicontinuity of \( v_T(\cdot) \), it suffices to show that
    $$
    \liminf _{n \rightarrow \infty} v_T\left(s_{T-1, n}, x_{T-1, n}, \xi_{[T]}, \zeta_{T - 1}\right) \geqslant v_T\left(s_{T-1}, x_{T-1}, \xi_{[T]}, \zeta_{T - 1}\right).
    $$  
    Let  
    \[
    s_{T, n} = S^M_{T - 1}\left(s_{T-1, n}, x_{T-1, n}, \xi_{T-1}, \zeta_{T - 1}\right), 
    s_T = S^M_{T - 1}\left(s_{T-1}, x_{T-1}, \xi_{T-1}, \zeta_{T - 1}\right).
    \]  
    By the continuity of the state transition mapping, 
    we have \( s_{T, n} \rightarrow s_T \).
    Now, we consider the value functions:
    {\fontsize{9pt}{10pt}\selectfont \begin{subequations}
    \begin{eqnarray}
        v_T\left(s_{T-1}, x_{T-1}, \xi_{[T]}, \zeta_{T - 1}\right) &= &  \min _{x_T \in \mathcal{X}_T (s_T, x_{T-1}, \xi_{[T]})} \mathbb{E}_{\zeta_T} \left[ C_T\left(s_T, x_T, \xi_{[T]}, \zeta_T\right) \right],  \\
        v_T\left(s_{T-1, n}, x_{T-1, n}, \xi_{[T]}, \zeta_{T - 1}\right) &= & \min _{x_T \in \mathcal{X}_T (s_{T,n}, x_{T-1,n}, \xi_{[T]})} \mathbb{E}_{\zeta_T} \left[ C_T\left(s_{T, n}, x_T, \xi_{[T]}, \zeta_T\right) \right]. 
        \label{eq:Prop-random-v_Tn}
    \end{eqnarray}   
    \label{eq:Prop-random-v_T}
    \end{subequations}}
    Assumption \ref{Assu:compactness} and the lower semicontinuity of \( g_{T, i}, i \in I_T, \) w.r.t. $(s_T, x_{T})$ imply that \( \mathcal{X}_T (s_T, x_{T-1}, \xi_{[T]}) \) and \( \mathcal{X}_T (s_{T,n}, x_{T-1,n}, \xi_{[T]}) \) are compact sets. Combining with the lower semicontinuity of \( C_T \) in $(s_T, x_T)$, 
    we can show that the optimums
    in \eqref{eq:Prop-random-v_T} 
    are attainable. Let 
    \( x_T^* \) and \( x_{T, n}^* \)
    denote the optimal solutions. Then
    \begin{eqnarray}
    v_T\left(s_{T-1, n}, x_{T-1, n}, \xi_{[T]}, \zeta_{T - 1}\right) 
    &= & \mathbb{E}_{\zeta_T} \left[ C_T\left(s_{T, n}, x_{T, n}^*, \xi_{[T]}, \zeta_T\right) \right]. \nonumber
\end{eqnarray}
By the lower semicontinuity of \( C_T(s_T,x_T,\xi_{[T]}, \zeta_T) \), there exists a subsequence \( \left(s_{T, n_j}, x_{T, n_j}^*\right) \) such that  
$$
\lim _{j \rightarrow \infty} \mathbb{E}_{\zeta_T} \left[ C_T\left(s_{T, n_j}, x_{T, n_j}^*, \xi_T, \zeta_T\right) \right] = c,
$$  
where \( c \) is the infimum of the sequence \( \left\{ \mathbb{E}_{\zeta_T} \left[ C_T(s_{T, n}, x_{T, n}^*, \xi_{[T]}, \zeta_T) \right] \right\} \). Since the feasible solution set is compact, there exists a convergent subsequence\( \left\{ \left(s_{T, n_{j i}}, x_{T, n_{j i}}^*\right) \right\} \) such that  
$
x_{T, n_{j i}}^* \rightarrow y_T^*.
$
Again, using the lower semicontinuity of the objective function, we obtain  
\begin{eqnarray}
\lim _{j \rightarrow \infty} \mathbb{E}_{\zeta_T} \left[ C_T\left(s_{T, n_j}, x_{T, n_j}^*, \xi_T, \zeta_T\right) \right] \geqslant \mathbb{E}_{\zeta_T} \left[ C_T\left(s_T, y_T^*, \xi_T, \zeta_T\right) \right] \nonumber \geqslant \mathbb{E}_{\zeta_T} \left[ C_T\left(s_T, x_T^*, \xi_T, \zeta_T\right) \right]. \nonumber
\end{eqnarray}  
Thus, by choosing a subsequence that attains the lower limit, we have
\begin{eqnarray}
&& \liminf _{n \rightarrow \infty} v_T\left(s_{T-1, n}, x_{T-1, n}, \xi_{[T]}, \zeta_{T - 1}\right) =  \liminf _{j \rightarrow \infty} \mathbb{E}_{\zeta_T} \left[ C_T\left(s_{T, n_j}, x_{T, n_j}^*, \xi_T, \zeta_T\right) \right] \nonumber \\
&\geqslant & \mathbb{E}_{\zeta_T} \left[ C_T\left(s_T, x_T^*, \xi_T, \zeta_T\right) \right] = v_T(s_{T - 1}, x_{T - 1}, \xi_{[T]}, \zeta_{T - 1}). \nonumber
\end{eqnarray}
It shows that \( v_T \) is lower semicontinuous w.r.t.~\( (s_{T - 1}, x_{T - 1}) \).  
The measurability of \( v_T \) follows directly from above. Therefore, the random lower semicontinuity of \( v_T \) is established.

\underline{Step 3}.  We show that \( v_t(s_{t - 1}, x_{t - 1}, \xi_{[t]}, \zeta_{t - 1}) \) is measurable. We do by induction. 

By the measure-preserving property of the expectation operator, we can conclude that \( \mathbb{E}_{\xi_T | \xi_{[T - 1]}} [ 
v_T(s_{T - 1}, x_{T - 1}, \xi_{[T]}, \zeta_{T - 1}) ] \) is measurable w.r.t.~\( (s_{T - 1}, x_{T - 1}, \xi_{[T - 1]}, \zeta_{T - 1}) \). 
This and Assumption~\ref{Assu:continuity} 
ensure 
that the objective function   
$C_{T - 1}(s_{T - 1}, x_{T - 1}, \xi_{[T - 1]}, \zeta_{T - 1}) +  \mathbb{E}_{\xi_T | \xi_{[T - 1]}} \big[ $ \\ $ v_T(s_{T - 1}, x_{T - 1}, \xi_{[T]}, \zeta_{T - 1}) \big]$ 
is also measurable w.r.t. \( (s_{T - 1}, x_{T - 1}, \xi_{[T - 1]}, \zeta_{T - 1}) \). 
By using a similar argument 
as that for stage \( T \), 
we can show that \( v_{T - 1}(s_{T-2}, x_{T-2}, \xi_{[T-1]}, \zeta_{T - 2}) \) is also measurable w.r.t. \( (s_{T-2}, x_{T-2}, \xi_{[T-1]}, \zeta_{T - 2}) \).

By recursively applying the above argument 
for \( t = 1, 2, \cdots, T - 1 \), we can conclude that $v_t, 1 \leq t \leq T$ is also a measurable function, and the expectation operators in problems \eqref{eq:value-t} are all well-defined. 

\underline{Step 4}. We show that \( v_t(s_{t - 1}, x_{t - 1}, \xi_{[t]}, \zeta_{t - 1}) \) is lower semicontinuous w.r.t. $(s_{t - 1}, x_{t - 1})$.
Assume that the random lower semicontinuity holds 
at stages after \( t+1 \), we now prove that it also holds 
at stage \( t \). 
To this end, consider {$v_t\left(s_{t-1}, x_{t-1}, \xi_{[t]}, \zeta_{t - 1}\right)$.}
According to the induction hypothesis, 
\( v_{t+1}\left(s_t, x_t, \xi_{[t+1]}, \zeta_{t}\right) \) is lower semicontinuous w.r.t.~\( \left(s_t, x_t\right) \). Since the expectation operator preserves lower semicontinuity, then \( \mathbb{E}_{\xi_{t+1} \mid \xi_{[t]}, \zeta_{t}} v_{t+1}\left(s_t, x_t, \xi_{[t+1]}, \zeta_{t}\right) \) is also lower semicontinuous w.r.t.~\( \left(s_t, x_t\right) \). Similar to the argument for stage \( T \),  we can prove that \( v_t\left(s_{t-1}, x_{t-1}, \xi_{[t]}, \zeta_{t - 1}\right) \) is a lower semicontinuous function w.r.t.~\( \left(s_{t-1}, x_{t-1}\right) \).
We can also show that $v_t\left(s_{t-1}, x_{t-1}, \xi_{[t]}, \zeta_{t - 1}\right)$ is a measurable function. 
This ensures  the lower semicontinuity of \( v_t \).

Since we first consider the lower semicontinuity of $v_t$, then the measurability of $v_{t - 1}$, the minimization operator in problem \eqref{eq:value-t} is always well-defined, i.e., the optimal solution always exists. The proof is completed.
\end{proof}

To illustrate the reasonability of the growth condition in Theorem \ref{Theo-stab-exo}, we define the objective functions of  problem \eqref{eq:mixed-MSP-MDP} under two feasible policies $\bm x$ and $\bm y$ as 
   \begin{eqnarray*}
    F(\bm x, \xi, \zeta) &:= & C_0(s_0, x_0, \zeta_0) + \sum\limits_{t = 1}^T C_t(s_t^{\bm x}, x_t, \xi_{[t]}, \zeta_t), \\
    F(\bm y, \xi, \zeta) &:= & C_0(s_0, y_0, \zeta_0) + \sum\limits_{t = 1}^T C_t(s_t^{\bm y}, y_t, \xi_{[t]}, \zeta_t).
   \end{eqnarray*}

\begin{proposition}[strong convexity of $F(\bm x, \xi, \zeta)$]
    \label{Prop-strong-convexity}
    Let Assumptions \ref{Assu:convexity} and \ref{assumption4} hold, and for $t = 1, 2, \cdots, T$, $C_t(s_t, x_t, \xi_{[t]}, \zeta_t)$ is strongly convex w.r.t. $x_t$, i.e., there exists a constant $\mu_t > 0$ such that
    \begin{eqnarray*}
        && \alpha C_t(s_t^{\bm x}, x_t, \xi_{[t]}, \zeta_t) + (1 - \alpha) C_t(s_t^{\bm y}, y_t, \xi_{[t]}, \zeta_t) \\ 
        &\geq &  C_t(\alpha s_t^{\bm x} + (1- \alpha) s_t^{\bm y}, \alpha x_t + (1 - \alpha) y_t, \xi_{[t]}, \zeta_t) + \mu_t \alpha (1 - \alpha) \Vert x_t - y_t \Vert^2, \quad \forall (s_t^{\bm x}, x_t), (s_t^{\bm y}, y_t).  
    \end{eqnarray*}
    Then the objective function  of  problem \eqref{eq:mixed-MSP-MDP} is strongly convex in $\bm x$, i.e., there exists a constant $\mu > 0$ such that
    \begin{eqnarray*}
        \alpha F(\bm x, \xi, \zeta) + (1 - \alpha)F(\bm y, \xi, \zeta) \geq F(\alpha \bm x + (1 - \alpha) \bm y, \xi, \zeta) + \mu \alpha (1- \alpha) \Vert \bm x - \bm y \Vert^2.
    \end{eqnarray*}
    Furthermore, if $\bm x^*$ is the optimal solution to problem \eqref{eq:mixed-MSP-MDP}, then 
    \begin{eqnarray*}
\mathbb{E}_{\mu} \left[ F(\bm x, \xi, \zeta)\right] \geq \mathbb{E}_{\mu} \left[ F(\bm x^*, \xi, \zeta) + \mu \Vert \bm x - \bm x^* \Vert^2 \right].
   \end{eqnarray*}
    \label{Prop-strongly-convex}
\end{proposition}
\begin{proof}
   To prove the strong convexity of $ F(\bm x, \xi, \zeta)$ w.r.t.~decision variables,
   it suffices to show that 
   for any two feasible policies $\bm x, \bm y,$
   \begin{eqnarray*}
       \alpha F(\bm x, \xi, \zeta) + (1 - \alpha) F(\bm y, \xi, \zeta) - \alpha (1 - \alpha) \mu \Vert \bm x - \bm y \Vert^2 \geq F(\alpha \bm x + (1 - \alpha) \bm y, \xi, \zeta). 
   \end{eqnarray*}
   From the convexity of $C_t$ w.r.t. $(s_t, x_t)$ and its strong convexity w.r.t. $x_t$, we have
   \begin{eqnarray}
       &&\alpha F(\bm x, \xi, \zeta) + (1 - \alpha) F(\bm y, \xi, \zeta) \nonumber \\
       &\geq & C_0(s_0, \alpha x_0 + (1 - \alpha) y_0, \zeta_0) + \sum\limits_{t = 1}^T C_t(\alpha s_t^{\bm x} +(1 - \alpha)s_t^{\bm y}, \alpha x_t + (1 - \alpha) y_t, \xi_{[t]}, \zeta_t) \nonumber \\
       && + \sum\limits_{t = 0}^T \mu_t \alpha (1 - \alpha) \Vert x_t - y_t \Vert^2, \label{eq:strongly-convex-1}
   \end{eqnarray}
   where $s_t^{\bm x}= S^M_{t - 1}(s_{t - 1}, x_{t - 1}, \xi_{t - 1}, \zeta_{t - 1})$, $s_t^{\bm y}= S^M_{t - 1}(s_{t - 1}, y_{t - 1}, \xi_{t - 1}, \zeta_{t - 1})$. Let $z = \alpha \bm x + (1 - \alpha) \bm y$ and $s_t^{\bm z} = S^M_{t - 1}(s_{t - 1}, z_{t - 1}, \xi_{t - 1}, \zeta_{t - 1})$ be the state at stage $t$ associated with $\bm z$. At stage 0, $s_0^{\bm x} = s_0^{\bm y} = s_0^{\bm z} = s_0$. Next, we compare $s_t^{\bm z}$ with $\alpha s_t^{\bm x} + (1 - \alpha) s_t^{\bm y}$. For $t = 1$, it is known from the convexity of $S^M_0$ that 
   \begin{eqnarray*}
       s_1^{\bm z} = S^M_0(s_0, z_0, \zeta_0) \leq \alpha S_0^M(s_0, x_0, \zeta_0) + (1 - \alpha) S_0^M(s_0, y_0, \zeta_0) = \alpha s_1^{\bm x} + (1 - \alpha) s_1^{\bm y}.
   \end{eqnarray*}
   Assume that for any $ 1 \leq k < t \leq T-1$, we have $s_k^{\bm z} \leq \alpha s_k^{\bm x} + (1 - \alpha) s_k^{\bm y}$. Then by the convexity and monotonicity of $S^M_{t - 1}$, we have
   \begin{eqnarray*}
       s_t^{\bm z} &= & S^M_{t - 1}(s_{t - 1}^{\bm z}, z_{t - 1}, \xi_{t - 1}, \zeta_{t - 1}) \leq S^M_{t - 1}(\alpha s_{t - 1}^{\bm x} + (1 - \alpha) s_{t - 1}^{\bm y}, z_{t - 1}, \xi_{t - 1}, \zeta_{t - 1}) \\
       &\leq & \alpha S^M_{t - 1}(s_{t - 1}^{\bm x}, x_{t - 1}, \xi_{t - 1}, \zeta_{t - 1}) + (1 - \alpha) S^M_{t - 1}(s_{t - 1}^{\bm y}, y_{t - 1}, \xi_{t - 1}, \zeta_{t - 1}) = \alpha s_t^{\bm x} + (1 - \alpha) s_t^{\bm y}.
   \end{eqnarray*}
    The principle of induction implies that we have shown $s_t^{\bm z} \leq \alpha s_t^{\bm x} + (1 - \alpha) s_t^{\bm y}$ for $t = 1, 2, \cdots, T$. With this and  the monotonicity of $C_t$ w.r.t. $s_t$, we have 
   \begin{eqnarray*}
       \eqref{eq:strongly-convex-1} &\geq & C_0(s_0, \alpha x_0 + (1 - \alpha) y_0, \zeta_0) + \sum\limits_{t = 1}^T C_t(s_t^{\bm z}, \alpha x_t + (1 - \alpha) y_t, \xi_{[t]}, \zeta_t) \\
       && + \sum\limits_{t = 0}^T \mu_t \alpha (1 - \alpha) \Vert x_t - y_t \Vert^2 \\
       &= & C_0(s_0, z_0, \zeta_0) + \sum\limits_{t = 1}^T C_t(s_t^{\bm z}, z_t, \xi_{[t]}, \zeta_t) + \sum\limits_{t = 0}^T \mu_t \alpha (1 - \alpha) \Vert x_t - y_t \Vert^2 \\
       &\geq & \underset{t = 0,1, \cdots, T}{\min} \mu_t \alpha (1 - \alpha) \Vert \bm x - \bm y \Vert^2 + F(\bm z, \xi, \zeta)\\
       &= & \mu \alpha (1 - \alpha) \Vert \bm x - \bm y \Vert^2 + F(\bm z, \xi, \zeta),
   \end{eqnarray*}
   where $\mu = \underset{t = 0,1, \cdots, T}{\min} \mu_t$. This establishes the strong convexity of the objective function  of  problem \eqref{eq:mixed-MSP-MDP}.
   Then for any feasible solution \( \bm x \) and the optimal solution \( \bm x^* \) (which must be unique due to the strong convexity) to problem \eqref{eq:mixed-MSP-MDP}, we have
   \begin{eqnarray}
       &&\mathbb{E}_{\mu} \left[ \alpha F(\bm x, \xi, \zeta) + (1 - \alpha) F(\bm x^*, \xi, \zeta) \right] \geq \mathbb{E}_{\mu} \left[ F(\alpha \bm x + (1 - \alpha)\bm x^*, \xi, \zeta)  + \mu \alpha (1 - \alpha) \Vert \bm x - \bm x^* \Vert^2 \right] \nonumber \\
       &\geq & \mathbb{E}_{\mu} \left[ F(\bm x^*, \xi, \zeta)  + \mu \alpha (1 - \alpha) \Vert \bm x - \bm x^* \Vert^2 \right]. \label{eq:strongly-convex-2}
   \end{eqnarray}
   This means that for any $0 < \alpha < 1$, we obtain 
   \begin{eqnarray*}
       \mathbb{E}_{\mu} \left[ F(\bm x, \xi, \zeta)\right] \geq \mathbb{E}_{\mu} \left[ F(\bm x^*, \xi, \zeta) + \mu (1 - \alpha) \Vert \bm x - \bm x^* \Vert^2 \right].
   \end{eqnarray*}
   In other words, the second-order growth condition
   \begin{eqnarray*}
       \mathbb{E}_{\mu} \left[ F(\bm x, \xi, \zeta)\right] \geq \mathbb{E}_{\mu} \left[ F(\bm x^*, \xi, \zeta) + \mu \Vert \bm x - \bm x^* \Vert^2 \right]
   \end{eqnarray*}
   holds by redefining $\mu (1 - \alpha)$ as $\mu$.
\end{proof}
Proposition \ref{Prop-strong-convexity} tells us that, when the strong convexity holds, the growth condition in Theorem \ref{Theo-stab-exo} holds for $\nu = 2$.

\section{Existing quantitative stability results}

To facilitate reading,
we include 
two  well-known 
quantitative stability results in 
multistage stochastic programming by Pflug and Pichler~\cite{pflug2012distance} and Heitsch et al.~\cite{HRS2006stability}. 

Pflug and Pichler~\cite{pflug2012distance} consider the following 
problem
\begin{eqnarray}\label{eq:MSP}
v(\mathbb{P}) \;=\; \inf_{x \in \mathbb{X},\, x \triangleleft \mathcal{F}}
   \; \mathbb{E}_{\mathbb{P}}\bigl[ H(\xi, x) \bigr],
\end{eqnarray}
where $x \triangleleft \mathcal{F}$ means that the decision vector 
$x=(x_t)_{t \in T}$ is adapted to the filtration $\mathcal{F}=(\mathcal{F}_t)_{t \in T}$,
and enforces the non-anticipativity constraint: 
decisions at stage $t$ can only depend on the information revealed up to time $t$, but not on future information, 
$\mathbb{X}$ denotes the feasible set independent of $\xi$,
$\mathbb{P}$ is a nested distribution which is a distribution that includes both the values of a 
stochastic process and the associated information structure. Formally, it is the 
distribution of a value-and-information structure $(\Omega,\mathcal{F},\mathbb{P},\xi)$, 
where $\xi=(\xi_t)_{t\in T}$ is adapted to the filtration 
$\mathcal{F}=(\mathcal{F}_t)_{t\in T}$.
\begin{theorem}[{\cite[Theorem 6.1]{pflug2012distance}}]
\label{Theo-Pflug}
Let $\mathbb{P}$ and $\tilde{\mathbb{P}}$ be two nested distributions. 
Assume that $\mathbb{X}$ is convex, 
the objective function $H(\xi, x)$ is convex in $x$ for any fixed~$\xi$, 
and $H$ is uniformly H\"older continuous in $\xi$ with exponent $\beta \leq 1$ and constant $L_\beta$, i.e.
\begin{eqnarray}
\bigl| H(\xi,x) - H(\tilde{\xi},x) \bigr| 
   \le L_\beta \Big( \sum_{t\in T} d(\xi_t,\tilde{\xi}_t) \Big)^{\beta},
   \qquad \forall x\in\mathbb{X}.
   \label{eq-ex-Holder-Pflug}
\end{eqnarray}
Then 
\begin{eqnarray}
|v(\mathbb{P})-v(\tilde{\mathbb{P}})| \le L_\beta \, \dd_{\rm{Nested}}(\mathbb{P},\tilde{\mathbb{P}})^{\beta}, 
\qquad \forall \beta \le 1.
\label{eq-ex-result-Pflug}
\end{eqnarray}
\end{theorem}
Next, we recall a quantitative stability result from Heitsch et al.~\cite{HRS2006stability}, which provides upper bound on the variation of the value function to distribution perturbations w.r.t.~the filtration distance. Heitsch et al.~consider the problem
\begin{eqnarray}
\min_{x_1, x_2, \dots, x_T} && \mathbb{E}_{\xi} \left[ \sum_{t=1}^T \langle b_t(\xi_t), x_t \rangle \right] \nonumber \\
\textrm{s.t.} && x_t \in X_t, \quad x_t \text{ is } \mathcal{F}_t\text{-measurable}, \quad t = 1, \dots, T, \nonumber \\
&& A_{t,0} x_t + A_{t,1}(\xi_t) x_{t-1} = h_t(\xi_t), \quad t = 2, \dots, T, \nonumber
\end{eqnarray}
where $b_t(\xi_t)$ is a random cost vector depending affinely on $\xi_t$, specifying the linear cost matrix of the stage-$t$ decision $x_t$, $A_{t,0}$ is a fixed matrix representing the deterministic technology coefficients at stage $t$, $A_{t,1}(\xi_t)$ is a random matrix depending affinely on $\xi_t$, linking the stage-$t$ decision $x_t$ with the previous stage decision $x_{t-1}$, $h_t(\xi_t)$ is a random right-hand side vector depending 
linear on $\xi_t$. Following the notation in \cite{HRS2006stability}, we 
let $F(\xi,x) := \mathbb{E}\!\left[\sum_{t=1}^T \langle b_t(\xi_t),x_t\rangle\right]$ and $X(\xi)$ denote the feasible set 
of the multistage stochastic programming problem. The optimal value function is then defined as
\bgeqn
v(\xi) := \min_{x \in X(\xi)} F(\xi,x).
\label{eq-ex-Romisch-optvalue}
\edeqn 

\begin{theorem}[{\cite[Theorem 2.1]{HRS2006stability}}]
\label{Theo-Romisch}
Assume the following conditions hold:
\begin{itemize}
  \item[(A1)]
  There exists a $\delta > 0$ such that for any $\tilde{\xi}$ with $\|\tilde{\xi}-\xi\|_{W_r} \le \delta$, and any 
  feasible prefix decisions $x_1,\dots,x_{t-1}$,  $t=2,\dots,T,$ the $t$-th stage feasibility set
  \[
  X_t(x_{t-1};\tilde{\xi}_t) = \left\{x_t \in X_t \mid A_{t,0}x_t + A_{t,1}(\tilde{\xi}_t)x_{t-1} = h_t(\tilde{\xi}_t)\right\}
  \]
  is nonempty.
  \item[(A2)]
  The optimal value $v(\tilde{\xi})$ is finite for all $\tilde{\xi}$ in a neighborhood of $\xi$, and there exist $\alpha > 0$, $\delta > 0$ and a bounded set $B \subset L^{r'}(\Omega,\mathcal{F},P;\mathbb{R}^m)$ such that the $\alpha$-level set
  \[
  l_\alpha(F(\tilde{\xi},\cdot)) := \left\{x \in X(\tilde{\xi}) \mid F(\tilde{\xi},x) \le v(\tilde{\xi})+\alpha\right\}
  \]
  is nonempty and contained in $B$ for all $\tilde{\xi}$ with $\|\tilde{\xi}-\xi\|_{W_r} \le \delta$.
  \item[(A3)]
  $\xi \in L^r(\Omega,\mathcal{F},P;\mathbb{R}^s)$ for some $r \ge 1$.
\end{itemize}
If, in addition, $X_1$ is bounded, then there exist constants $L,\alpha,\delta > 0$ such that
\bgeqn
|v(\xi)-v(\tilde{\xi})| \;\le\; L\big(\|\xi-\tilde{\xi}\|_{W_r} + \dd_{\rm{Filt}}(\xi,\tilde{\xi})\big)
\label{eq-ex-result-Romisch}
\edeqn
for all $\tilde{\xi}\in L^r(\Omega,\mathcal{F},P;\mathbb{R}^s)$ with $\|\xi-\tilde{\xi}\|_{W_r} \le \delta$, where $\dd_{\rm{Filt}}(\xi,\tilde{\xi})$ denotes the filtration distance between $\xi$ and $\tilde{\xi}$,
\bgeqn 
\dd_{\rm{Filt}}(\xi, \tilde{\xi}) := \underset{\epsilon \in (0, \alpha]}{\sup} \underset{x \in l_{\epsilon}(\xi), \tilde{x} \in l_{\epsilon}(\tilde{\xi})}{\inf} \sum_{t=1}^{T} 
\max\Big\{
\big\| x_t - \mathbb{E}[x_t \mid \tilde{\mathcal{F}}_t] \big\|,\;
\big\| \tilde{x}_t - \mathbb{E}[\tilde{x}_t \mid \mathcal{F}_t] \big\| 
\Big\}.
\label{eq-ex-filtration}
\edeqn 
\end{theorem}

\end{document}